\documentclass[sn-mathphys-num]{sn-jnl}%

\usepackage{graphicx}%
\usepackage{multirow}%
\usepackage{amsmath,amssymb,amsfonts}%
\usepackage{amsthm}%
\usepackage{mathrsfs}%
\usepackage[title]{appendix}%
\usepackage{xcolor}%
\usepackage{textcomp}%
\usepackage{manyfoot}%
\usepackage{booktabs}%
\usepackage{algorithm}%
\usepackage{algorithmicx}%
\usepackage{algpseudocode}%
\usepackage{listings}%
\usepackage{adjustbox}

\usepackage{fullpage}

\def\E{\mathbb{E}}
\def\R{\mathbb{R}}
\def\Z{\mathbb{Z}}

\usepackage{comment}

\usepackage{subcaption}
\newcommand{\Fcal}{{\cal F}}
\DeclareMathOperator{\tr}{tr}
\newcommand{\defeq}{:=}

\newtheorem{definition}{Definition}
\newtheorem{thm}{Theorem}
\newtheorem{lem}{Lemma}

\newtheorem{assum}{Assumption}
\newtheorem{cond}{Condition}
\newtheorem{test}{Test}
\begin{document}

\title[On the Convergence and Complexity of the Stochastic Central Finite-Difference Based Gradient Estimation Methods]{On the Convergence and Complexity of the Stochastic Central Finite-Difference Based Gradient Estimation Methods}

\author*[1]{\fnm{Raghu} \sur{Bollapragada}}\email{raghu.bollapragada@utexas.edu}

\author[1]{\fnm{Cem} \sur{Karamanli}}\email{cem.karamanli@utexas.edu}

\affil[1]{\orgdiv{Operations Research and Industrial Engineering}, \orgname{The University of Texas at Austin}, \orgaddress{\city{Austin}, \state{TX}, \country{USA}}}

\abstract{
This paper presents an algorithmic framework for solving unconstrained stochastic optimization problems using only stochastic function evaluations. We employ central finite-difference based gradient estimation methods %
to approximate the gradients and dynamically control the accuracy of these approximations by adjusting the sample sizes used in stochastic realizations.
Extending the analysis of \cite{bollapragada2024derivative} to nonconvex functions and central finite-difference based methods, we analyze the theoretical properties of the proposed framework. 
Our analysis yields sublinear convergence results to the neighborhood of the solution, and establishes the optimal worst-case iteration complexity ($\mathcal{O}(\epsilon^{-1})$) and sample complexity ($\mathcal{O}(\epsilon^{-2})$) for each gradient estimation method to achieve an $\epsilon$-accurate solution. Finally, we demonstrate the performance of the proposed framework and the quality of the gradient estimation methods through numerical experiments on nonlinear least squares problems.
}

\maketitle

\section{Introduction}
\label{sec:intro}
We consider unconstrained stochastic optimization problems of the form 
\begin{equation} \label{eq:stochasticProblem}
	\min_{x \in \mathbb{R}^d} F(x) = \E_{\zeta}\left[f(x, \zeta)\right],
\end{equation}
where $F : \R^{d} \to \R$ is a continuously differentiable function,  
$\zeta$ is a random variable with associated probability space $(\Xi,\Fcal,P)$, $f : \R^{d} \times 
\Xi \to \R$, and $\E_{\zeta}[\cdot]$ denotes the expectation with respect to $P$. We consider derivative-free optimization (DFO) settings where gradient information is unavailable, and only stochastic realizations of the objective functions, obtained through a zeroth-order oracle, are available. Such problem settings arise in different applications, including simulation optimization \citep{blanchet2019convergence,pasupathy2018sampling,shashaani2018astro} %
and reinforcement learning \citep{bertsekas2019reinforcement,choromanski2018optimizing,choromanski2018structured, mania2018simple}.

Gradient estimation methods, which involve estimating gradients through function evaluations and incorporating these estimators into optimization algorithms, are widely recognized within DFO settings. Recently, \cite{bollapragada2024derivative} proposed a unified algorithmic  
framework that encompasses stochastic 
gradient estimation methods for solving \eqref{eq:stochasticProblem}. This framework incorporates forward finite-difference methods for gradient estimation and controls the accuracy of these estimators by utilizing common random number (CRN) settings and adaptively choosing the number of stochastic realizations (sample sizes) employed in these estimators.

In this paper, we extend this framework to include central finite-difference methods, which are an alternative to forward finite-difference methods and have the potential to achieve superior performance under specific %
settings, such as %
the standard finite-difference method employed within CRN settings \citep{l1998budget,larson2019derivative}. Specifically, we propose iterative algorithms that incorporate gradient estimators computed using subsampled functions defined as
\begin{equation}
\label{eq:subfuncdef}
F_{S_k}(x) \defeq \frac{1}{|S_k|} \sum_{\zeta_i \in S_k} f(x, \zeta_i)\qquad \forall x\in \R^d,
\end{equation}
where $S_k = \{\zeta_1, \ldots, \zeta_{|S_k|}\}$ is a set of random realizations at each iteration $k$. By utilizing a set of vectors $T_k$, we derive the general central finite-difference based gradient estimators within CRN settings as follows
\begin{equation}\label{eq:gen_grad_est}
g_{S_k, T_k}(x_k) \defeq \gamma_k \sum_{u_j \in T_k} \left( \frac{F_{S_k}(x_k + \nu u_j) - F_{S_k}(x_k - \nu u_j)}{2\nu} \right)u_j,
\end{equation}
where $\nu > 0$ is the sampling radius, and $\gamma_k > 0$ is the scaling coefficient. 

Within this framework, we consider central finite-difference variants of various methods, encompassing standard finite-difference methods (cFD) \citep{blum1954multidimensional,kiefer1952stochastic}, Gaussian smoothing methods (cGS) \citep{nesterov2017random}, sphere smoothing methods (cSS) \citep{berahas2021theoretical}, randomized coordinate finite-difference methods (cRC) \citep{wright2015coordinate}, and randomized subspace finite-difference methods (cRS) \citep{berahas2021theoretical,kozak2021stochastic}. These methods differ in their choices of $T_k$ and $\gamma_k$ \citep[Table 2.1]{bollapragada2024derivative}. The $u_j$ vectors are canonical vectors $e_j$ for cFD and cRC, sampled from a multivariate standard normal distribution ($\mathcal{N}(0,I)$) for cGS, drawn from a uniform distribution over the surface of a unit sphere ($\mathcal{U}(\mathcal{S}(0,1))$) for cSS, and random orthonormal vectors for cRS. The value of $\gamma_k$ is $1$ for cFD, $1/|T_k|$ for cGS, and $d/|T_k|$ for cSS, cRC, and cRS.

The general iterative update rule within this framework, employing the gradient estimators defined in \eqref{eq:gen_grad_est}, is given as
\begin{equation} \label{eq:iter}
x_{k+1} = x_k - \alpha_k g_{S_k,T_k}(x_k),
\end{equation}
where $\alpha_k$ is the step size. The efficiency of this update rule depends on the accuracy of the gradient estimators. \cite{bollapragada2024derivative} incorporated an adaptive sampling approach to control the gradient estimation accuracy as the algorithm progresses by adjusting the sample sizes ($|S_k|$) at each iteration. The key idea behind adaptive sampling approaches is that inaccurate gradient approximations are sufficient when the current iterate is far away from the solution, and the accuracy in these approximations should increase as the iterates approach the optimal solution. Such approaches retain the optimal theoretical convergence properties of their deterministic counterparts while being efficient. In this work, we adapt these strategies to the central finite-difference based methods and extend their analysis to nonconvex problem settings.

The paper is organized as follows. A literature review and a summary of our notation are presented in the remainder of Section~\ref{sec:intro}. Mathematical preliminaries, including assumptions and sampling conditions, are discussed in Section~\ref{sec:assumandcond}. In Section~\ref{sec:theory}, we provide theoretical convergence and %
complexity results, while Section~\ref{sec:numericalExp} presents the numerical results. Finally, in Section~\ref{sec:conclusions}, we make some concluding remarks.

\subsection{Literature Review}
Several gradient estimation methods \citep{kiefer1952stochastic,nesterov2017random,flaxman2005online,berahas2021theoretical} %
have been proposed for both deterministic and stochastic optimization, as extensively reviewed in \citep{conn2009introduction,larson2019derivative}. \cite{pasupathy2018sampling} analyzed the convergence properties of the fundamental gradient estimation method introduced by \cite{kiefer1952stochastic} under different sampling rates. \cite{ghadimi2013stochastic} provided optimal worst-case sample complexities for both convex and nonconvex functions, showing $\mathcal{O}(\epsilon^{-2})$ complexity to achieve an $\epsilon$-accurate solution for stochastic approximation based gradient estimation methods, with different accuracy definitions for convex and nonconvex problems. Gradient estimation methods have also been integrated into quasi-Newton approaches \citep{berahas2019derivative,bollapragada2023adaptive,marrinan2023zeroth}. Additionally, adaptive sampling approaches, well-established in stochastic optimization \citep{byrd2012sample,bollapragada2018adaptive,bollapragada2018progressive,bollapragada2019exact,xie2020constrained,BOLLAPRAGADA2023239,berahas2022adaptive}, have recently been applied in DFO settings \citep{shashaani2018astro,bollapragada2023adaptive,bollapragada2024derivative}. \cite{bollapragada2023adaptive} generalized the norm condition \citep{byrd2012sample} and the practical inner-product condition \citep{bollapragada2018adaptive} to standard finite-difference based gradient estimation methods, while \cite{bollapragada2024derivative} extended these conditions to other forward finite-difference based gradient estimation methods.

\subsection{Notation}
The set of nonnegative integers and positive integers is denoted by $\Z_{+} \defeq \{0,1,2,\dots\}$ and $\Z_{++} \defeq \{1,2,\dots\}$ respectively. The set of real numbers (scalars) is denoted by $\R$, the set of $d$-dimensional vectors is denoted by $\R^d$, and the set of $m$-by-$d$ matrices is denoted by $\R^{m\times d}$. The transpose of a matrix $A \in \R^{m \times d}$ is denoted by $A^T \in \R^{d\times m}$. Throughout the paper, the $\ell_2$ vector norm or matrix norm is denoted by $\|\cdot\|$.

\section{Preliminaries}
\label{sec:assumandcond}
In this section, we outline the mathematical preliminaries, including the assumptions made throughout the paper and the conditions used to determine the sample sizes in the stochastic approximations at each iteration. We begin by stating the assumptions about the objective function.

\begin{assum}\label{assum:Lipschitzstochf}
    The objective function $F:\R^d \rightarrow \R$ is twice continuously differentiable function and is bounded below. That is, there exists $F^* > -\infty$ such that $F^* \defeq \inf_{x \in \R^d} F(x)$. Furthermore, the stochastic function $f(\cdot,\zeta):\R^d \rightarrow \R$ is also twice continuously differentiable with Lipschitz continuous gradients and Hessians with Lipschitz constants $L_f < \infty$ and $M_f < \infty$, respectively. That is, for every $\zeta$,
    \begin{align*}
	\| \nabla f(x,\zeta) - \nabla f(y,\zeta) \| & \leq L_f \| x - y \|, \quad \text{and} \\
        \| \nabla^2 f(x,\zeta) - \nabla^2 f(y,\zeta) \| & \leq M_f \| x - y \| \quad \forall x,y \in \R^d.
    \end{align*}
\end{assum}
Assumption~\ref{assum:Lipschitzstochf} implies that the objective function $F$ has Lipschitz continuous gradients and Hessians with Lipschitz constants $L_F \leq L_f$ and $M_F \leq M_f$, respectively, which is a common assumption in central finite-difference settings \citep{berahas2021theoretical}. While it is possible to relax the assumption on the smoothness of stochastic functions and require only $F$ to be smooth \citep{bollapragada2023adaptive}, such assumptions on the stochastic functions are useful in providing the complexity analysis \citep{bollapragada2024derivative}. Moreover, under Assumption~\ref{assum:Lipschitzstochf}, it follows from the second-order Taylor expansion that for every $\nu > 0$, $\zeta$, and $x, u \in \mathbb{R}^d$,
\begin{align}
    f (x + \nu u, \zeta) - f (x - \nu u,\zeta) &\leq 2\nu u ^T \nabla f (x,\zeta) + \frac{M_f}{3} \nu^3 \| u\|^3, \label{eq:Taylorstochf} \\
    F (x + \nu u) - F (x - \nu u) &\leq 2\nu u ^T \nabla F (x) + \frac{M_F}{3} \nu^3 \| u\|^3. \label{eq:TaylorF}
\end{align}
We also assume that the variance in the stochastic gradients is bounded, a standard assumption in stochastic optimization literature \citep{bottou2018optimization}.

\begin{assum}\label{assum:boundedvarinstochgrad}
There exist constants $\beta_1,\beta_2 \geq 0$ such that
\begin{align*}
    \E_{\zeta_i}[\| \nabla f (x, \zeta_i) - \nabla F (x)\|^2] \leq \beta_1 \|\nabla F (x)\|^2 + \beta_2, \quad \forall x \in \R^d.
\end{align*} 
\end{assum}

The next assumption pertains to the independence of sets $S_k$ and $T_k$ sampled at each iteration $k$.

\begin{assum}\label{assum:sampling}
	At every iteration $k$, the sample set $S_k$ consists of independent and identically distributed (i.i.d.)\ samples of $\zeta$. That is, for all $x \in \R^d$ and $k\in \Z_+$, 
    \begin{equation*}
	\E_{\zeta_i}[f(x,\zeta_i)] = F(x), \qquad \forall \zeta_i \in S_k.
    \end{equation*}
    Moreover, the vector set $T_k$ comprises sampled vectors that are chosen independently of the sample set $S_k$. 
\end{assum}

We also define the conditional expectations with respect to these random subsets, which are the only sources of randomness in the iterates generated by \eqref{eq:iter}. That is, we define the filtrations $\mathcal{F}_{k} = \sigma (x_0,\{T_1,S_1\},\{T_2,S_2\},\cdots,\{T_{k-1},S_{k-1}\})$ and $\mathcal{F}_{k + \frac{1}{2}} = \sigma (x_0,\{T_1,S_1\},\{T_2,S_2\},\cdots,\{T_{k-1},S_{k-1}\},T_k)$,
\begin{align*}
    \E_{k}[\cdot] = \E_{T_k}[\cdot] \defeq \E[\cdot| \mathcal{F}_{k}], \quad \text{and} \quad
    \E_{S_k}[\cdot] \defeq \E\left[\cdot\Big| \mathcal{F}_{k+\frac{1}{2}}\right].
\end{align*}
Moreover, we define the following expected quantities
\begin{align}
g_{T_k}(x_k) \defeq \E_{S_k}[g_{S_k,T_k}(x_k)], \quad \text{and} \quad
g(x_k) \defeq \E_{T_k} [g_{T_k}(x_k)].
\end{align}
Using these definitions, we can decompose the error in the gradient estimator into three terms: 
\begin{equation*}
    g_{S_k,T_k} (x_k) - \nabla F(x_k) = \underbrace{g_{S_k,T_k} (x_k) - g_{T_k} (x_k)}_{\text{function sampling error}} + \underbrace{g_{T_k} (x_k) - g (x_k)}_{\text{vector sampling error}} + \underbrace{g(x_k) - \nabla F(x_k),}_{ \text{bias}}
\end{equation*}
where the function sampling error depends on the choice of $S_k$, vector sampling error depends on the choice of $T_k$, and the bias term, arising from the absence of gradient information, 
depends on the choice of $\nu$. While we choose a constant $|T_k|$ and $\nu$ throughout the iterations, $|S_k|$ is adaptively chosen to control the function sampling error, thereby controlling the gradient estimation error. 
Guided by this principle,  \cite{bollapragada2024derivative} proposed the following theoretical condition that generalizes the well-known \emph{norm condition} in the derivative-based methods \citep{byrd2012sample,bollapragada2018adaptive}.   

\begin{cond}\label{cond:theoreticalnormcond} (Theoretical Norm Condition) \citep[Condition 3]{bollapragada2024derivative}
\begin{align}\label{eq:theoreticalnormcond}
    \frac{\E_{T_k}[\E_{\zeta_i}[\| g_{\zeta_i, T_k}(x_k) - g_{T_k}(x_k)\|^2]]}{|S_k|} \leq \theta^2 \E_{T_k}[\|g_{T_k}(x_k)\|^2] , \quad \theta > 0.
\end{align}
\end{cond}
We choose the sample sizes $|S_k|$ at each iteration such that this condition is satisfied. Evaluating this condition requires computing population (expectation) quantities that may not be available in practice, and we provide a practical test that employs sampled quantities to overcome this limitation in Section~\ref{sec:numericalExp}.

\section{Theoretical Results}
In this section, we provide theoretical convergence guarantees and worst-case complexity results for the iterates generated by \eqref{eq:iter} with sample sizes $|S_k|$ satisfying Condition~\ref{cond:theoreticalnormcond}. 
\label{sec:theory}

\subsection{Convergence Results}
We first establish a descent lemma that provides an upper bound on the expected decrease in the function value per iteration.
\begin{lem} \label{lem:generalconv}
For any $x_0 \in \R^d$, let $\{x_k: k\in \Z_{++}\}$ be generated by iteration \eqref{eq:iter} with $|S_{k}|$ satisfying Condition \ref{cond:theoreticalnormcond} for a given constant $\theta > 0$. Suppose that Assumptions \ref{assum:Lipschitzstochf}, \ref{assum:boundedvarinstochgrad}, and \ref{assum:sampling} hold. Then, for any $k\in \Z_{+}$,  if $\alpha_k$ satisfies
\begin{equation}\label{eq:alphaupper}
    0 < \alpha_k \leq \bar \alpha_k, 
\end{equation}
we have 
\begin{equation}\label{eq:generalconvres}
\E_k[F(x_{k+1})] \leq F(x_k) - \frac{\alpha_k}{4} \|\nabla F(x_k)\|^2 + \alpha_k \chi_k \quad \forall k\in \Z_{+},
\end{equation}
where $\bar \alpha_k, \chi_k > 0$ are 
given in Table~\ref{tbl:unified}. 
\end{lem}
\begin{proof}
Using the result of \cite[Lemma 3.2]{bollapragada2024derivative}, it follows that 
\begin{align*}
     \E_{k}[F(x_{k+1})]
    & \leq  F(x_k) + \frac{\alpha_k}{2} \|\delta_k\|^2 - \frac{\alpha_k}{2} \|\nabla F(x_k)\|^2 \\
    & + L_F \alpha_k^2 (1 + \theta^2)\left(\|\delta_k\|^2 + \|\nabla F (x_k)\|^2 + \frac{1}{2}(\E_{T_k}[\|g_{T_k}(x_{k}) - g(x_k)\|^2]) \right),
\end{align*}
where $\delta_k \defeq g(x_k) - \nabla F(x_k) $. 
\citep[Section A.2]{bollapragada2024derivative} Table~\ref{tbl:unified} presents the established upper bounds on $\| \delta_k\|$ from existing literature. Using these results, we derive the bounds on per-iteration decrease in expected function values for each gradient estimation method. 
\begin{enumerate}
\item \textit{cFD:} Note that $T_k$ is deterministic and hence $g_{T_k}(x_k) = g(x_k)$. Thus, by \citep[Lemma 3.2]{bollapragada2024derivative},
\begin{align*}
\E_{k}[F(x_{k+1})] &\leq  F(x_k) - \alpha_k\left(\frac{1}{2} - L_F \alpha_k (1 + \theta^2)\right) \|\nabla F(x_k)\|^2  \nonumber + \alpha_k\left(\frac{1}{2} + L_F \alpha_k (1 + \theta^2)\right) \|\delta_k\|^2 \\
&\leq F(x_k) - \frac{\alpha_k}{4} \|\nabla F(x_k)\|^2 + \frac{3\alpha_k}{4} \|\delta_k\|^2 \\
&\leq F(x_k) - \frac{\alpha_k}{4} \|\nabla F(x_k)\|^2 + \frac{\alpha_k d M_F^2 \nu^4}{48},
\end{align*}
where the second inequality is due to \eqref{eq:alphaupper} and $\bar \alpha_k$ given in  Table~\ref{tbl:unified}, and the third inequality is due to the bound on $\|\delta_k\|^2$ given in Table~\ref{tbl:unified}.

\item \textit{cGS:} An upper bound on the covariance term is given as,
\begin{align}\label{eq:covarianceupper}
    Var(g_{T_k}(x_k)) = \E_{T_k}[(g_{T_k}(x_k) - g(x_k))(g_{T_k}(x_k) - g(x_k))^T] \preceq \kappa(x_k) I,
\end{align}
where $\kappa(x_k) = \frac{3}{|T_k|}(3\|\nabla F(x_k)\|^2 + \frac{M_F^2 \nu^4}{36}(d+2)(d+4)(d+6) )$ (see \citep[Lemma 2.4]{berahas2021theoretical}). Therefore, using \eqref{eq:covarianceupper}, we can bound the variance term as 
\begin{align}\label{eq:varianceupperkappax}
    \E_{T_k}[\| g_{T_k}(x_k) - g(x_k) \|^2] \leq d \kappa(x_k).
\end{align}

Now, by \citep[Lemma 3.2]{bollapragada2024derivative}, we have,
\begin{align*}
\E_{k}[F(x_{k+1})] & \leq  F(x_k) + \frac{\alpha_k}{2} \|\delta_k\|^2 - \frac{\alpha_k}{2} \|\nabla F(x_k)\|^2
+ L_F \alpha_k^2 (1 + \theta^2) [ \|\delta_k\|^2 + \|\nabla F (x_k)\|^2 \\
& \quad + \frac{1}{2}(\E_{T_k}[\|g_{T_k}(x_{k}) - g(x_k)\|^2]) ] \\
& \leq  F(x_k) + \frac{\alpha_k}{2} \|\delta_k\|^2 - \frac{\alpha_k}{2} \|\nabla F(x_k)\|^2 + L_F \alpha_k^2 (1 + \theta^2) [ \|\delta_k\|^2 + \|\nabla F (x_k)\|^2 \\
& \quad + \frac{3d}{2|T_k|}(3\|\nabla F(x_k)\|^2 + \frac{M_F^2 \nu^4}{36}(d+2)(d+4)(d+6) )] \\
& \leq F(x_k) + \frac{\alpha_k}{2} d^2 M_F^2 \nu^4 - \frac{\alpha_k}{2} \|\nabla F(x_k)\|^2 \nonumber + L_F \alpha_k^2 (1 + \theta^2)[d^2 M_F^2 \nu^4 + \|\nabla F (x_k)\|^2 \\
& \quad + \frac{3d}{2|T_k|}(3\|\nabla F(x_k)\|^2 + \frac{M_F^2 \nu^4}{36}(d+2)(d+4)(d+6) )] \\
& \leq F(x_k) - \frac{\alpha_k}{4} \|\nabla F(x_k)\|^2 
+ \frac{72|T_k|d + 216d^2 + (d+2)(d+4)(d+6)}{96(|T_k| + 4.5d)} \alpha_k d M_F^2 \nu^4,
\end{align*}
where the second inequality is due to \eqref{eq:varianceupperkappax},
the third inequality is due to the bound on $\|\delta_k\|^2$ given in Table~\ref{tbl:unified},  and the fourth inequality is due to \eqref{eq:alphaupper} and $\bar \alpha_k$ given in Table~\ref{tbl:unified}. 

\item \textit{cSS:} A similar bound on the covariance term \eqref{eq:covarianceupper} is given in \citep[Lemma 2.9]{berahas2021theoretical} for the cSS method where $\kappa(x_k) = \frac{3}{|T_k|}(\frac{3d}{d+2}\|\nabla F(x_k)\|^2 + \frac{d M_F^2 \nu^4}{36} )$ in \eqref{eq:varianceupperkappax}. Now, by \citep[Lemma 3.2]{bollapragada2024derivative}, we have,
\begin{align*}
\E_k[F(x_{k+1})] & \leq  F(x_k) + \frac{\alpha_k}{2} \|\delta_k\|^2 - \frac{\alpha_k}{2} \|\nabla F(x_k)\|^2 + L_F \alpha_k^2 (1 + \theta^2)[\|\delta_k\|^2 + \|\nabla F (x_k)\|^2 \\
& \quad + \frac{1}{2}(\E_{T_k}[\|g_{T_k}(x_{k}) - g(x_k)\|^2]) ] \\
& \leq  F(x_k) + \frac{\alpha_k}{2} \|\delta_k\|^2 - \frac{\alpha_k}{2} \|\nabla F(x_k)\|^2 + L_F \alpha_k^2 (1 + \theta^2)[\|\delta_k\|^2 + \|\nabla F (x_k)\|^2 \\
& \quad + \frac{3d}{2|T_k|}(\frac{3d}{d+2}\|\nabla F(x_k)\|^2 + \frac{d M_F^2 \nu^4}{36} )] \\
& \leq  F(x_k) + \frac{\alpha_k}{2} M_F^2 \nu^4 - \frac{\alpha_k}{2} \|\nabla F(x_k)\|^2 + L_F \alpha_k^2 (1 + \theta^2)[M_F^2 \nu^4 + \|\nabla F (x_k)\|^2 \\
& \quad + \frac{3d}{2|T_k|}(\frac{3d}{d+2}\|\nabla F(x_k)\|^2 + \frac{d M_F^2 \nu^4}{36} )] \\
& \leq  F(x_k) + \frac{\alpha_k}{2} M_F^2 \nu^4 - \frac{\alpha_k}{2} \|\nabla F(x_k)\|^2 + L_F \alpha_k^2 (1 + \theta^2)[M_F^2 \nu^4 + \|\nabla F (x_k)\|^2 \\
& \quad + \frac{3d}{2|T_k|}(3\|\nabla F(x_k)\|^2 + \frac{d M_F^2 \nu^4}{36} )] \\
& \leq F(x_k) - \frac{\alpha_k}{4} \|\nabla F(x_k)\|^2 + \frac{72|T_k| + 216d + d^2}{96(|T_k| + 4.5d)} \alpha_k M_F^2 \nu^4  ,
\end{align*}
where the second inequality is due to \eqref{eq:varianceupperkappax},
the third inequality is due to the bound on $\|\delta_k\|^2$ given in Table~\ref{tbl:unified}, the fourth inequality is because $\frac{d}{d+2} \leq 1$, and the fifth inequality is due to \eqref{eq:alphaupper} and $\bar \alpha_k$ given in Table~\ref{tbl:unified}. 

\item \textit{cRC:} 
Consider 
\begin{align} \label{eq:RCgenconvstep1}
    & \E_{T_k}[\|g_{T_k}(x_{k}) - g(x_k)\|^2] \nonumber \\ & \quad = \E_{T_k} \Big[ \sum_{e_j \in T_k} \Big(\frac{d}{|T_k|}[\nabla^{FD} F(x_k)]_{e_j} - [\nabla^{FD} F(x_k)]_{e_j} \Big)^2 + \sum_{e_j \notin T_k} \Big( [\nabla^{FD} F(x_k)]_{e_j} \Big)^2 \Big]  \nonumber \\
    & \quad = \E_{T_k} \Big[ \sum_{e_j \in T_k} \Big(\frac{d - |T_k|}{|T_k|}\Big)^2[\nabla^{FD} F(x_k)]^2_{e_j} + \sum_{e_j \notin T_k} [\nabla^{FD} F(x_k)]^2_{e_j} \Big]  \nonumber \\
    & \quad = \sum_{j = 1}^d \Big[ \frac{|T_k|}{d} \Big( \frac{d - |T_k|}{|T_k|} \Big)^2 [\nabla^{FD} F(x_k)]^2_{e_j} + \Big(1 - \frac{|T_k|}{d}\Big) [\nabla^{FD} F(x_k)]^2_{e_j} \Big]  \nonumber \\
    & \quad = \frac{d - |T_k|}{|T_k|} \| \nabla^{FD} F(x_k) \|^2 \leq \frac{2d - 2|T_k|}{|T_k|} \| \delta_k \|^2 + \frac{2d - 2|T_k|}{|T_k|} \| \nabla F(x_k) \|^2,
\end{align}
where the inequality is due to $(a+b)^2 \leq 2a^2 + 2b^2$ for any $a,b \in \R$.
Now, by \citep[Lemma 3.2]{bollapragada2024derivative} and \eqref{eq:RCgenconvstep1}, we have,
\begin{align*}
\E_{k}[F(x_{k+1})] & \leq  F(x_k) + \frac{\alpha_k}{2} \|\delta_k\|^2 - \frac{\alpha_k}{2} \|\nabla F(x_k)\|^2 + L_F \alpha_k^2 (1 + \theta^2)\left( \frac{d}{|T_k|} \|\delta_k\|^2 + \frac{d}{|T_k|} \|\nabla F (x_k)\|^2 \right) \\
& \leq F(x_k) - \frac{\alpha_k}{4} \|\nabla F(x_k)\|^2 + \frac{3\alpha_k}{4} \|\delta_k\|^2 \\
& \leq F(x_k) - \frac{\alpha_k}{4} \|\nabla F(x_k)\|^2 + \frac{\alpha_k d M_F^2 \nu^4}{48},
\end{align*}
where the second inequality is due to \eqref{eq:alphaupper} and $\bar \alpha_k$ given given in Table~\ref{tbl:unified}, and the third inequality is due to the bound on $\|\delta_k\|^2$ given in Table~\ref{tbl:unified}.

\item \textit{cRS:} 
From the proof of \citep[Lemma 3.2]{bollapragada2024derivative}, combining \citep[Equation 3.1]{bollapragada2024derivative} and \citep[Equation 3.2]{bollapragada2024derivative}, we get
\begin{align}\label{eq:RSLemma4step1}
    \E_k[F(x_{k+1})] & \leq F(x_k) + \frac{\alpha_k}{2}\|\delta_k\|^2 - \frac{\alpha_k}{2}\|\nabla F(x_k)\|^2 + \frac{L_F \alpha_k^2}{2}  (1 + \theta^2) \E_{T_k}[\|g_{T_k}(x_k)\|^2]. 
\end{align}
We now analyze the terms $\|\delta_k\|^2$ and $\E_{T_k}[\|g_{T_k}(x_k)\|^2]$.  
From \citep[Equation 2.15]{bollapragada2024derivative}, we have
\begin{align}
    \|\delta_k\|^2 &= \|\E_{\Tilde T_k}[U_{\Tilde T_k} b_{\Tilde T_k} - \nabla F(x_k)]\|^2 \leq \E_{\Tilde T_k}[\|U_{\Tilde T_k} b_{\Tilde T_k} - \nabla F(x_k)\|^2] = \E_{\Tilde T_k}[\| b_{\Tilde T_k} - U_{\Tilde T_k}^T \nabla F(x_k)\|^2] \label{eq:ortheq},
\end{align}
where the inequality is due to the Jensen's inequality, and the second equality is due to $U_{\Tilde T_k} U_{\Tilde T_k}^T = U_{\Tilde T_k}^TU_{\Tilde T_k} = I$ and $\|U_{\Tilde T_k}\| = 1$. 
Consider any $u_j^{(k)} \in \tilde{T}_k$. %
By Assumption \ref{assum:Lipschitzstochf} and \eqref{eq:TaylorF}, we have for all $j=1,\cdots,d$ that
\begin{align*}
    F(x_k + \nu u_j^{(k)}) - F(x_k - \nu u_j^{(k)}) \leq 2 \nu \nabla F(x_k)^T u_j^{(k)} + \frac{M_F \nu^3 \| u_j^{(k)}\|^3}{3}. 
\end{align*}
Therefore,
\begin{align*}
    \frac{F(x_k + \nu u_j^{(k)}) - F(x_k - \nu u_j^{(k)})}{2\nu} - \nabla F(x_k)^T u_j^{(k)} \leq \frac{M_F \nu^2}{6}, 
\end{align*}
for all $j=1,\cdots,d$. By concatenating the above inequalities, we have  
\begin{align} \label{eq:RSproxydeltak}
b_{\Tilde T_k} - U_{\Tilde T_k}^T \nabla F(x_k) \leq (\frac{M_F \nu^2}{6})\vec{1}, 
\end{align}
where $\vec{1} \in \R^d$ is a vector composed of $1$'s. Therefore, we have
\begin{align}
    \| b_{\Tilde T_k} - U_{\Tilde T_k}^T \nabla F(x_k)\|^2 \leq (\frac{\sqrt{d} M_F \nu^2 }{6})^2.
    \label{eq:boundgTk}
\end{align}
Combining this inequality with \eqref{eq:ortheq}, we obtain
\begin{equation} \label{eq:boundondeltakRS}
    \|\delta_k\|^2 \leq (\frac{\sqrt{d} M_F \nu^2}{6})^2.
\end{equation}

Consider 
\begin{align}
    & \E_{T_k}[\| g_{T_k}(x_k)\|^2] = \E_{T_k}\left[\left\| \frac{d}{|T_k|} U_{T_k} b_{T_k} \right\|^2\right] = \frac{d^2}{|T_k|^2}\E_{T_k}[\| U_{T_k} b_{ T_k}\|^2] \nonumber \\
    & \quad = \frac{d^2}{|T_k|^2}\E_{T_k}[\| b_{ T_k}\|^2] = \frac{d}{|T_k|}\E_{\Tilde T_k}[\| b_{\Tilde T_k}\|^2] = \frac{d}{|T_k|}\E_{\Tilde T_k}[\| U_{\Tilde T_k} b_{\Tilde T_k}\|^2], \label{eq:RSEtkgtk}
\end{align}
where we used the fact that $U_{T_k}$ and $U_{\Tilde T_k}$ have orthonormal columns. Therefore, we have 
\begin{align}
    \E_{T_k}[\| g_{T_k}(x_k)\|^2] %
    & \leq \frac{2d}{|T_k|}\E_{\Tilde T_k}[\| U_{\Tilde T_k} b_{\Tilde T_k} - \nabla F(x_k) \|^2] + \frac{2d}{|T_k|}\| \nabla F(x_k) \|^2 \nonumber \\
    & \leq \frac{2d}{|T_k|}(\frac{M_F \nu^2 \sqrt{d}}{6})^2 + \frac{2d}{|T_k|}\| \nabla F(x_k) \|^2, \label{eq:RSEtkgtk2}
\end{align}
where the first inequality is due to $(a+b)^2 \leq 2a^2 + 2b^2$, and the second inequality is due to \eqref{eq:ortheq} and \eqref{eq:RSproxydeltak}. Combining \eqref{eq:RSLemma4step1}, \eqref{eq:boundondeltakRS}, and \eqref{eq:RSEtkgtk2}, we obtain
\begin{align*}
    \E_{k}[F(x_{k+1})]
    & \leq F(x_k) + \frac{\alpha_k}{2} \| \delta_k \|^2 - \frac{\alpha_k}{2} \|\nabla F(x_k)\|^2 + L_F \alpha_k^2  \frac{d}{|T_k|} (1 + \theta^2) ( \frac{d M_F^2 \nu^4}{36} + \|\nabla F (x_k)\|^2) \\
    & \leq F(x_k) - \frac{\alpha_k}{4} \|\nabla F(x_k)\|^2 + \frac{\alpha_k d M_F^2 \nu^4}{48},
\end{align*}
where the second inequality is due to \eqref{eq:alphaupper} and $\bar \alpha_k$ given in Table~\ref{tbl:unified}. %

\end{enumerate}
\end{proof}

Lemma \ref{lem:generalconv} demonstrates that the function values are expected to decrease at any iteration $k$, provided the gradient ($\|\nabla F(x_k)\|^2$) is sufficiently large. However, the term $\chi_k$ hinders the progress. In fact, when a fixed step size ($\alpha_k = \bar \alpha$) and a fixed number of directions ($|T_k| = N$) are employed, the $\chi_k$ term becomes constant, and the iterates converge to a neighborhood of the solution. The following theorem establishes the convergence rate and the size of the neighborhood for nonconvex functions.
\begin{thm} \label{thm:sublconv}  (Sublinear Convergence). 
    Suppose the conditions for Lemma \ref{lem:generalconv} hold. Let $\bar \alpha$ and $\bar \chi$ be the values obtained by choosing %
    $|T_k| = N$ in Table \ref{tbl:unified}. If $\alpha_k = \bar \alpha$, for all $k \in \Z_{+}$, then the sequence $\{ x_k: k \in \Z_{+} \}$ converges sublinearly to a neighborhood of the solution. That is,
    \begin{equation}\label{eq:sublconv}
        \min_{0 \leq k \leq K-1} \E[\| \nabla F(x_k) \|^2] \leq \frac{4(F(x_0) - F^*)}{K \bar \alpha} + 4 \bar \chi.
    \end{equation}
    Moreover, for any $p \in (0,1]$, we have with probability at least $1-p$ that  
    \begin{equation}\label{eq:highprobconv}
        \min_{0 \leq k \leq K-1} \| \nabla F(x_k) \|^2 \leq \frac{4(F(x_0) - F^*)}{K \bar \alpha p} + \frac{4 \bar \chi}{p}.
    \end{equation}
\end{thm}
\begin{proof}
Rearranging \eqref{eq:generalconvres} yields
\begin{align*}
    \|\nabla F(x_k)\|^2 \leq \frac{4(F(x_k) - \E_k[F(x_{k+1})])}{\bar \alpha} + 4 \bar \chi.
\end{align*}
Taking the total expectation and averaging the terms from $k=0$ to $k=K-1$ yields
\begin{align}\label{eq:avg_bnd}
    \frac{1}{K} \sum_{k = 0}^{K-1} \E[\|\nabla F(x_k)\|^2] & \leq \frac{1}{K} 
 \sum_{k = 0}^{K-1} \bigg[ \frac{4\E[F(x_k) - F(x_{k+1})]}{\bar \alpha} + 4 \bar \chi \bigg] = \frac{4\E[F(x_0) - F(x_{K})]}{K \bar \alpha} + 4 \bar \chi.
\end{align}
Using the above inequality, along with the fact that $\min_{0 \leq k \leq K-1} \E[\| \nabla F(x_k) \|^2] \leq \frac{1}{K} \sum_{k = 0}^{K-1} \E[\|\nabla F(x_k)\|^2] $, and $F(x_{K}) \geq F^*$ due to Assumption \ref{assum:Lipschitzstochf}, we obtain \eqref{eq:sublconv}. Moreover, using the 
fact that $\min_{0 \leq k \leq K-1} \| \nabla F(x_k) \|^2 \leq \frac{1}{K} \sum_{k = 0}^{K-1} \|\nabla F(x_k)\|^2 $, we get
\begin{align*}
    & P \bigg( \min_{0 \leq k \leq K-1} \| \nabla F(x_k) \|^2 > \frac{4(F(x_0) - F^*)}{K \bar \alpha p} + \frac{4 \bar \chi}{p} \bigg) \\
    &\leq P \bigg( \frac{1}{K} \sum_{k = 0}^{K-1} \|\nabla F(x_k)\|^2 > \frac{4(F(x_0) - F^*)}{K \bar \alpha p} + \frac{4 \bar \chi}{p} \bigg) \\ &\leq \frac{\E[\frac{1}{K} \sum_{k = 0}^{K-1} \|\nabla F(x_k)\|^2]}{\frac{4(F(x_0) - F^*)}{K \bar \alpha p} + \frac{4 \bar \chi}{p}} \leq \frac{\frac{4\E[F(x_0) - F(x_{K})]}{K \bar \alpha} + 4 \bar \chi}{\frac{4(F(x_0) - F^*)}{K \bar \alpha p} + \frac{4 \bar \chi}{p}}\leq p.
\end{align*}
where the second inequality is due to Markov's inequality, the third inequality is due to \eqref{eq:avg_bnd}, and the last inequality is due to the fact that $F(x_{K}) \geq F^*$, which completes the proof.
\end{proof}
Theorem \ref{thm:sublconv} provides convergence results both in expectation and in probability. The difference between different gradient estimation methods in terms of convergence behavior is reflected in the step size $\bar \alpha$ and neighborhood $\bar \chi$ parameters (see Table~\ref{tbl:unified}). We note that similar observations to those made in forward finite-difference methods \citep{bollapragada2024derivative} are found here. cRC and cRS methods have convergence rates $\frac{d}{N}$ times worse compared to cFD. Similarly, convergence rates of cGS and cSS methods are $\frac{N+4.5d}{N}$ times worse than that of cFD. Regarding the size of convergence neighborhoods, we observe that cFD, cRC, and cRS methods have the same neighborhood size. Furthermore, assuming that $N$ is small and $d$ is large, cSS has a neighborhood size similar to cFD, whereas cGS has a neighborhood size that is $d^2$ times larger compared to cFD.

\begin{table}[ht!]
\centering
\def\arraystretch{1.7}
\caption{
Properties and convergence results for different gradient estimation methods. $\bar \delta$ serves as an upper bound on $\|\delta_k\|$, where $\delta_k \defeq g(x_k) - \nabla F(x_k)$. $\bar \alpha_k$ and $\chi_k$ values summarize the results of Lemma \ref{lem:generalconv}. The $\nu = \hat \nu$ values ensure that the convergence neighborhood $4 \bar \chi$ (as stated in Theorem \ref{thm:sublconv}) equals $ \frac{\epsilon p}{2}$. Here, $\Omega_1 \defeq \frac{1}{4(1+\theta^2)L_F}$, $\Omega_2 \defeq \frac{d M_F^2 \nu^4}{48}$, and $\Omega_3 \defeq \sqrt[4]{\frac{6 \epsilon p}{ M_F^2 d}}$.
}
\begin{tabular}{|c|c|c|c|c|}
\hline
Method & $ \bar \delta $ & $\bar \alpha_k$ & $\chi_k$ & $\hat \nu$ \\ \hline
cFD & $\frac{\sqrt{d} M_F \nu^2 }{6}$ & $ \Omega_1 $ & $ \Omega_2 $ & $\Omega_3$ \\ 
cGS & $d M_F \nu^2$ & $ \frac{|T_k|}{(|T_k| + 4.5d)} \Omega_1$ & $\frac{72|T_k|d + 216d^2 + (d+2)(d+4)(d+6)}{2(|T_k| + 4.5d)} \Omega_2$ & $ \sqrt[4]{\frac{2(N + 4.5d)}{72Nd + 216d^2 + (d+2)(d+4)(d+6)}} \Omega_3 $ \\ 
cSS & $M_F \nu^2$ & $ \frac{|T_k|}{(|T_k| + 4.5d)} \Omega_1$ & $\frac{72|T_k| + 216d + d^2}{2d(|T_k| + 4.5d)} \Omega_2 $ & $ \sqrt[4]{\frac{2d(N + 4.5d)}{72N + 216d + d^2}} \Omega_3 $ \\ 
cRC & $\frac{\sqrt{d} M_F \nu^2 }{6}$ & $\frac{|T_k|}{d} \Omega_1$ & $ \Omega_2$ & $\Omega_3$ \\ 
cRS & $\frac{\sqrt{d} M_F \nu^2 }{6}$ & $\frac{|T_k|}{d} \Omega_1$ & $ \Omega_2$ & $\Omega_3$ \\
\hline
\end{tabular}
\label{tbl:unified}
\end{table}

\subsection{Complexity Results}
We now present the iteration and sample complexity results for the central finite-difference based gradient estimation methods, providing bounds on the total number of iterations and stochastic function evaluations required to achieve an $\epsilon$-accurate solution, respectively. The accuracy measure is defined as follows. 

\begin{definition}
    A random iterate $x_k$ is said to be an $\epsilon$-accurate solution if it satisfies $ \| \nabla F(x_k) \|^2 \leq \epsilon$ with probability at least $1-p$, where $p\in[0,1)$.
\end{definition}

Next, we observe that by employing $\hat \nu$ values in $\bar \delta$ in Table~\ref{tbl:unified}, we can ensure that%
\begin{align*}
\|\delta_k\| = \|g(x_k) - \nabla F(x_k)\| \leq \frac{\sqrt{\epsilon p}}{2} \leq \frac{\sqrt{\epsilon}}{2},
\end{align*}
for each gradient estimation method. Suppose that $x_k$ satisfies $\|g(x_k)\|^2 \leq \epsilon'$ with $\epsilon' \defeq \epsilon / 4$. Then, by using the above inequality, we can guarantee that %
\begin{align*}
    \|\nabla F(x_k)\| \leq \|g(x_k)\| + \|\delta_k \| \leq \sqrt{\epsilon},
\end{align*}
where the first inequality is by the triangle inequality (i.e. $\|a + b\| \leq \|a\| + \|b\|$, for any $a,b \in \R^n$). Therefore, for any iteration $k$, if $\|g(x_k)\|^2 \leq \epsilon'$ is satisfied, then $x_k$ is an $\epsilon$-accurate solution. This observation provides a means to upper bound the sample sizes employed at each iteration before achieving an $\epsilon$-accurate solution, as stated in the following lemma.
\begin{lem}\label{lem:boundonSk}
Suppose the conditions of Theorem~\ref{thm:sublconv} hold. 
For all iterations $k\in \mathbb{Z}_{+}$ such that $\|g(x_k)\|^2 > \epsilon'$, we have
\begin{align}
    |S_k| \leq b_1 + \frac{ b_2 }{\epsilon'} = b_1 + \frac{ 4b_2 }{\epsilon} \label{eq:boundonSk},
\end{align}
where $b_1,b_2 \in \R$ are constants that depend on the gradient estimation method. Table \ref{tbl:compres} summarizes the order results for $b_1$ and $b_2$ values.
\end{lem}
\begin{proof}
\citep[Section A.4]{bollapragada2024derivative}
Without loss of generality and choosing the minimum sample size that satisfies Condition~\ref{cond:theoreticalnormcond} at each iteration $k \in \Z_{+}$, we have
\begin{align}
    |S_k| = \left \lceil \frac{\E_{T_k} [\E_{\zeta_i}[\| g_{\zeta_i, T_k}(x_k) - g_{T_k}(x_k)\|^2]]}{\theta^2 \E_{T_k} [\|g_{T_k} (x_k)\|^2]} \right \rceil
    &\leq  1 + \frac{\E_{T_k} [\E_{\zeta_i}[\| g_{\zeta_i, T_k}(x_k) - g_{T_k}(x_k)\|^2]]}{\theta^2 \E_{T_k} [\|g_{T_k} (x_k)\|^2]}
    \label{eq:Skappx1} \\
    & \leq 1 + \frac{\E_{T_k} [\E_{\zeta_i}[\| g_{\zeta_i, T_k}(x_k) - g_{T_k}(x_k)\|^2]]}{\theta^2 \|\E_{T_k} [g_{T_k} (x_k)]\|^2} \label{eq:Skappx2},
\end{align}
where the second inequality is by Jensen's inequality.
We use the weaker bound in \eqref{eq:Skappx2} whenever the bound in \eqref{eq:Skappx1} is difficult to compute. We employ $\epsilon'$ bound to lower bound the denominator in \eqref{eq:Skappx2}. That is, 
\begin{align}\label{eq:gxkandepsilonprime}
    \epsilon' < \| g(x_k) \|^2  =  \|\E_{T_k} [g_{T_k}(x_k)] \|^2. %
\end{align}
We utilize the results developed in Section~\ref{sec:CFDproofofboundedvar} to analyze the variance term (numerator) in \eqref{eq:Skappx1} and \eqref{eq:Skappx2}. 
We now develop bounds on sample sizes for each gradient estimation method. 
\begin{enumerate}
\item \textit{cFD:} Note that $T_k$ is not stochastic and \eqref{eq:Skappx1} can be written as
\begin{align}\label{eq:SkappxFD}
    |S_k| = 1 + \frac{ \E_{\zeta_i}[\| \nabla^{FD} f(x_k,\zeta_i) - \nabla^{FD} F(x_k)\|^2]}{\theta^2 \|\nabla^{FD} F(x_k)\|^2}. %
\end{align}
Considering the numerator in \eqref{eq:SkappxFD}, using \eqref{eq:defofci} and by \citep[Table 2.1]{bollapragada2024derivative}, we have
\begin{align}\label{eq:SkappxFDstep1}
    & \E_{\zeta_i} [ \| \nabla^{FD} f(x_k,\zeta_i) - \nabla^{FD} F(x_k)\|^2] \nonumber \\
    & \quad = \E_{\zeta_i}\Big[\sum_{i=1}^d \Big(\frac{f(x_k + \nu e_j,\zeta_i) - f(x_k - \nu e_j,\zeta_i)}{2\nu} - \frac{F(x_k + \nu e_j) - F(x_k - \nu e_j)}{2\nu} \Big)^2 \Big] \nonumber \\
    & \quad = \E_{\zeta_i}[\sum_{i=1}^d c_{i,j}^2]
    ~\leq~ \E_{\zeta_i}\Big[\sum_{i=1}^d \Big( 2\Big( e_j^T [\nabla f (x_k,\zeta_i) - \nabla F (x_k)] \Big)^2 + \frac{2}{9}M_f^2 \nu^4 \|e_j\|^6 \Big) \Big] \nonumber \\
    & \quad = 2\E_{\zeta_i}[\| \nabla f (x_k,\zeta_i) - \nabla F (x_k)\|^2] + \frac{2}{9}M_f^2 \nu^4 d \nonumber \\
    & \quad \leq 2\beta_1 \|\nabla F (x_k)\|^2 + 2\beta_2 + \frac{2}{9}M_f^2 \nu^4 d  \nonumber \\
    & \quad \leq 4\beta_1 \|\nabla^{FD} F (x_k)\|^2 + 4\beta_1 \|\nabla F (x_k) - \nabla^{FD} F (x_k)\|^2 + 2\beta_2 + \frac{2}{9}M_f^2 \nu^4 d  \nonumber \\
    & \quad \leq 4\beta_1 \|\nabla^{FD} F (x_k)\|^2 + 2\beta_2 + \frac{(2+\beta_1)}{9}M_f^2 \nu^4 d,
\end{align}
where the first inequality is due to \eqref{eq:boundedvarstep2}, the second inequality is due to Assumption \ref{assum:boundedvarinstochgrad}, the third inequality is due to $(a+b)^2 \leq 2a^2 + 2b^2$ for any $a,b \in \R$, and the last inequality is by \citep[Lemma 3.2]{bollapragada2024derivative}, $\bar \delta$ given in Table \ref{tbl:unified} and the fact that $M_F \leq M_f$.
Substituting \eqref{eq:SkappxFDstep1} in \eqref{eq:SkappxFD}, and using \eqref{eq:gxkandepsilonprime},
we have \eqref{eq:boundonSk} satisfied with $b_1 =1 + \frac{4 \beta_1}{\theta^2} $, $b_2 = \frac{(2+\beta_1) M_f^2 \nu^4 d}{9\theta^2} + \frac{2\beta_2}{\theta^2} $.
\vspace{0.5em}
\item \textit{cGS:}
Taking expectation $\E_{T_k}[\cdot]$ on both sides in \eqref{eq:boundedvarstep1}, we get
\begin{align} \label{eq:GSandSSvarstep1}
\E_{T_k} \left [ \E_{\zeta_i} [  \|g_{\zeta_i, T_k}(x_k) - g_{T_k}(x_k)\|^2 ] \right] &\leq \E_{T_k} \left[ \gamma_k^2|T_k| \E_{\zeta_i} \left[ \sum_{u_j \in T_k} \left\| c_{i,j} u_j \right\|^2 \right] \right] \nonumber \\
&= \E_{T_k} \left[ \gamma_k^2|T_k| \E_{\zeta_i} \left[ \sum_{u_j \in T_k} \tr(c_{i,j}^2 u_j u_j^T) \right] \right] \nonumber \\
&= \gamma_k^2 |T_k|^2 \tr \Big( \E_{u_j} [ \E_{\zeta_i} [c_{i,j}^2 u_j u_j^T]] \Big),
\end{align}
where the last equality is due to the linearity of $\tr(\cdot)$ and $\E[\cdot]$ operations, as well as independence of $u_j \in T_k$ vectors. 
By \eqref{eq:boundedvarstep2}, \eqref{eq:GSandSSvarstep1} and by \citep[Table 2.1]{bollapragada2024derivative}, we have
\begin{align}\label{eq:SkappxGSstep1}
    & \E_{T_k} \left[ \E_{\zeta_i} \left[  \|g_{\zeta_i, T_k}(x_k) - g_{T_k}(x_k)\|^2 \right] \right] \nonumber \\
    & \quad \leq 2 \E_{\zeta_i} \left[ \tr \left( \E_{u_j} \left[ \left( u_j^T [\nabla f (x_k,\zeta_i) - \nabla F (x_k) ] \right)^2 u_j u_j^T\right] \right) \right] + \frac{2}{9} \tr \left( \E_{u_j} \left[ M_f^2 \nu^4 \|u_j\|^6 u_j u_j^T \right] \right) \nonumber \\
    & \quad = 2 \E_{\zeta_i} [ \tr ( \| \nabla f (x_k,\zeta_i) - \nabla F (x_k)] \|^2 I + 2[\nabla f (x_k,\zeta_i) - \nabla F (x_k)][\nabla f (x_k,\zeta_i) - \nabla F (x_k)]^T )] \nonumber \\
    & \quad \quad + \frac{2}{9} \tr \left( M_f^2 \nu^4(d+2)(d+4)(d+6)I \right) \nonumber \\
    & \quad = 2 \E_{\zeta_i} \Big[ (d+2)\| \nabla f (x_k,\zeta_i) - \nabla F (x_k)] \|^2 \Big] + \frac{2}{9}M_f^2 \nu^4 d(d+2)(d+4)(d+6) \nonumber \\
    & \quad \leq 2 (d+2) \beta_1 \|\nabla F (x_k)\|^2 + 2 (d+2) \beta_2 + \frac{2}{9} M_f^2 \nu^4 d(d+2)(d+4)(d+6) \nonumber \\
    & \quad \leq 4 (d+2) \beta_1 \|\nabla F^{GS}_{\nu} (x_k)\|^2 + 4 (d+2) \beta_1 \|\nabla F (x_k) - \nabla F^{GS}_{\nu} (x_k)\|^2 + 2 (d+2) \beta_2 \nonumber \\
    & \quad \quad + \frac{2}{9} M_f^2 \nu^4 d(d+2)(d+4)(d+6) \nonumber \\
    & \quad \leq 4 (d+2) \beta_1 \|\nabla F^{GS}_{\nu} (x_k)\|^2 + 2 (d+2) \beta_2 + 2 M_f^2 \nu^4 d(d+2) \bigg[ \frac{1}{9}(d+4)(d+6) + 2 \beta_1 d \bigg]
\end{align}
where the first equality is due to \citep[Equation 2.18]{berahas2021theoretical}, the second inequality is due to Assumption \ref{assum:boundedvarinstochgrad}, the third inequality is due to $(a+b)^2 \leq 2a^2 + 2b^2$ for any $a,b \in \R$, and the last inequality is by \citep[Lemma 3.2]{bollapragada2024derivative}, $\bar \delta$ given in Table \ref{tbl:unified} and $M_F \leq M_f$. By \eqref{eq:Skappx2}, \eqref{eq:gxkandepsilonprime} and \eqref{eq:SkappxGSstep1}, we have \eqref{eq:boundonSk} satisfied with $b_1 =1 + \frac{4 (d+2) \beta_1}{\theta^2} $, $b_2 = \frac{2 (d+2) \beta_2 + 2 M_f^2 \nu^4 d(d+2) [ \frac{1}{9}(d+4)(d+6) + 2 \beta_1 d ]}{\theta^2} $.
\vspace{0.5em}
\item \textit{cSS:}
By \eqref{eq:boundedvarstep2}, \eqref{eq:GSandSSvarstep1} and by \citep[Table 2.1]{bollapragada2024derivative}, we have
\begin{align}\label{eq:SkappxSSstep1}
    & \E_{T_k} \left[ \E_{\zeta_i} \left[  \|g_{\zeta_i, T_k}(x_k) - g_{T_k}(x_k)\|^2 \right] \right] \nonumber \\
    & \quad \leq 2d^2 \E_{\zeta_i} \left[ \tr \left( \E_{u_j} \left[ \left( u_j^T [\nabla f (x_k,\zeta_i) - \nabla F (x_k) ] \right)^2 u_j u_j^T\right] \right) \right] + \frac{2d^2}{9} \tr \Big( \E_{u_j} \Big[ \Big( M_f^2 \nu^4 \|u_j\|^6 u_j u_j^T \Big] \Big) \nonumber \\
    & \quad = \frac{2d}{d+2} \E_{\zeta_i} [ \tr ( \| \nabla f (x_k,\zeta_i) - \nabla F (x_k)] \|^2 I + 2[\nabla f (x_k,\zeta_i) - \nabla F (x_k)][\nabla f (x_k,\zeta_i) - \nabla F (x_k)]^T )] \nonumber \\
    & \quad \quad + \frac{2d}{9} \tr \Big( M_f^2 \nu^4I \Big) \nonumber \\
    & \quad = 2 d \E_{\zeta_i} \Big[ \| \nabla f (x_k,\zeta_i) - \nabla F (x_k)] \|^2 \Big] + \frac{2 d^2}{9} M_f^2 \nu^4 \nonumber \\
    & \quad \leq 2 d \beta_1 \|\nabla F (x_k)\|^2 + 2 d \beta_2 + \frac{2 d^2}{9} M_f^2 \nu^4 \nonumber \\
    & \quad \leq  4 d \beta_1 \|\nabla F^{SS}_{\nu} (x_k)\|^2 + 4 d \beta_1 \|\nabla F (x_k) - \nabla F^{SS}_{\nu} (x_k)\|^2 + 2 d \beta_2 + \frac{2 d^2}{9} M_f^2 \nu^4 \nonumber \\
    & \quad \leq  4 d \beta_1 \|\nabla F^{SS}_{\nu} (x_k)\|^2 + 2 d \beta_2 + 2 d M_f^2 \nu^4 (\frac{d}{9} + 2\beta_1),
\end{align}
where the first equality is due to \citep[Equation 2.39]{berahas2021theoretical}, the second inequality is due to Assumption \ref{assum:boundedvarinstochgrad}, the third inequality is due to $(a+b)^2 \leq 2a^2 + 2b^2$ for any $a,b \in \R$, and the last inequality is by \citep[Lemma 3.2]{bollapragada2024derivative}, $\bar \delta$ given in Table \ref{tbl:unified} and $M_F \leq M_f$. By \eqref{eq:Skappx2}, \eqref{eq:gxkandepsilonprime} and \eqref{eq:SkappxSSstep1} we have \eqref{eq:boundonSk} satisfied with $b_1 = 1 + \frac{4 d \beta_1}{\theta^2} $, $b_2 = \frac{2 d \beta_2 + 2 d M_f^2 \nu^4 (\frac{d}{9} + 2\beta_1)}{\theta^2} $.
\vspace{0.5em}
\item \textit{cRC:}
Considering the numerator in \eqref{eq:Skappx1}, using \eqref{eq:defofci} and by \citep[Table 2.1]{bollapragada2024derivative}, we have
\begin{align}\label{eq:SkappxRCstep1}
    & \E_{T_k} \left[ \E_{\zeta_i} \left[ \| [\nabla^{FD} f(x_k,\zeta_i) - \nabla^{FD} F(x_k)]_{T_k}\|^2 \right] \right] \nonumber \\
    & \quad = \frac{d}{N} \sum_{j=1}^d \E_{\zeta_i}[c_{i,j}^2] \nonumber \\
    & \quad \leq \frac{d}{N} 
 \E_{\zeta_i}\left[\sum_{j=1}^d \left( 2\left( e_j^T [\nabla f (x_k,\zeta_i) - \nabla F (x_k)] \right)^2 + \frac{2}{9} M_f^2 \nu^4 \|e_j\|^6 \right) \right] \nonumber \\
    & \quad = \frac{2d}{N}\E_{\zeta_i}[\| \nabla f (x_k,\zeta_i) - \nabla F (x_k)\|^2] + \frac{2d^2}{9N} M_f^2 \nu^4  \nonumber \\
    & \quad \leq \frac{2\beta_1 d}{N} \|\nabla F (x_k)\|^2 + \frac{2\beta_2 d}{N} + \frac{2d^2}{9N} M_f^2 \nu^4 \nonumber \\
    & \quad \leq \frac{4\beta_1 d}{N} \|\nabla^{FD} F (x_k)\|^2 + \frac{4\beta_1 d}{N} \|\nabla F (x_k) - \nabla^{FD} F (x_k)\|^2 + \frac{2\beta_2 d}{N} + \frac{2d^2}{9N} M_f^2 \nu^4 \nonumber \\
    & \quad \leq \frac{4\beta_1 d}{N} \|\nabla^{FD} F (x_k)\|^2 + \frac{2\beta_2 d}{N} + (2+\beta_1)\frac{M_f^2 \nu^4 d^2}{9N},
\end{align}
where the first inequality is due to \eqref{eq:boundedvarstep2}, the second inequality is due to Assumption \ref{assum:boundedvarinstochgrad}, the third inequality is due to $(a+b)^2 \leq 2a^2 + 2b^2$ for any $a,b \in \R$, and the last inequality is by \citep[Lemma 3.2]{bollapragada2024derivative}, $\bar \delta$ given in Table \ref{tbl:unified} and $M_F \leq M_f$. Note that the denominator in \eqref{eq:Skappx1} can be bounded from below as follows,
\begin{align}
    \E_{T_k}[\| g_{T_k}(x_k)\|^2] &= \frac{d^2}{N^2} \sum_{j = 1}^d \frac{N}{d} ([\nabla^{FD} F(x_k)]_{e_j})^2 = \frac{d}{N} \|\nabla^{FD} F (x_k)\|^2 > \frac{d\epsilon'}{N} \label{eq:denomRC}.
\end{align}
Therefore by \eqref{eq:gxkandepsilonprime}, \eqref{eq:SkappxRCstep1} and \eqref{eq:denomRC}, we have \eqref{eq:boundonSk} satisfied with $b_1 = 1 + \frac{4 \beta_1}{\theta^2} $, $b_2 = \frac{(2+\beta_1) M_f^2 \nu^4 d}{9\theta^2} + \frac{2\beta_2}{\theta^2} $.
\vspace{0.5em}
\item \textit{cRS:}
Considering \eqref{eq:defofci}, \eqref{eq:Skappx1} and by \citep[Table 2.1]{bollapragada2024derivative}, taking the expectation $\E_{T_k}[\cdot]$ on both sides in \eqref{eq:boundedvarstep1}, we get
\begin{align}\label{eq:SkappxRSstep1}
    & \E_{T_k} \left[ \E_{\zeta_i} \left[  \|g_{\zeta_i, T_k}(x_k) - g_{T_k}(x_k)\|^2 \right] \right] = \E_{T_k} \left[ \E_{\zeta_i} \left[ \| \frac{d}{N} \sum_{u_j \in T_k} c_{i,j} u_j  \|^2 \right] \right] \nonumber \\
    & \quad = \frac{d^2}{N^2} \E_{T_k} \left[ \E_{\zeta_i} \left[ \sum_{u_j \in T_k} c_{i,j}^2 \right] \right] = \frac{d}{N} \E_{\Tilde T_k} \left[ \E_{\zeta_i} \left[ \sum_{u_j \in \Tilde T_k} c_{i,j}^2 \right] \right] \nonumber \\
    & \quad \leq \frac{d}{N} 
 \E_{\zeta_i} \left[ \E_{\Tilde T_k}\left[\sum_{u_j \in \Tilde T_k} \left[ 2\left( u_j^T [\nabla f (x_k,\zeta_i) - \nabla F (x_k)] \right)^2 + \frac{2}{9} M_f^2 \nu^4 \|u_j\|^6 \right] \right] \right]\nonumber \\
    & \quad = \frac{2d}{N}\E_{\zeta_i}[\| \nabla f (x_k,\zeta_i) - \nabla F (x_k)\|^2] + \frac{2d^2}{9N} M_f^2 \nu^4 \nonumber \\
    & \quad \leq \frac{2\beta_1 d}{N} \|\nabla F (x_k)\|^2 + \frac{2\beta_2 d}{N} + \frac{2d^2}{9N} M_f^2 \nu^4 \nonumber \\
    & \quad \leq \frac{4\beta_1 d}{N} \|\E_{\Tilde T_k}[U_{\Tilde T_k} b_{\Tilde T_k}]\|^2 + \frac{4\beta_1 d}{N} \|\nabla F (x_k) - \E_{\Tilde T_k}[U_{\Tilde T_k} b_{\Tilde T_k}]\|^2 + \frac{2\beta_2 d}{N} + \frac{2d^2}{9N} M_f^2 \nu^4  \nonumber \\
    & \quad \leq \frac{4\beta_1 d}{N} \E_{\Tilde T_k}[\|U_{\Tilde T_k} b_{\Tilde T_k}\|^2] + \frac{4\beta_1 d}{N} \E_{\Tilde T_k}[\|\nabla F (x_k) - U_{\Tilde T_k} b_{\Tilde T_k}\|^2] + \frac{2\beta_2 d}{N} + \frac{2d^2}{9N} M_f^2 \nu^4 \nonumber \\
    & \quad \leq \frac{4\beta_1 d}{N} \E_{\Tilde T_k}[\|U_{\Tilde T_k} b_{\Tilde T_k}\|^2] + \frac{2\beta_2 d}{N} + (2+\beta_1)\frac{M_f^2 \nu^4 d^2}{9N},
\end{align}
where the first inequality is due to \eqref{eq:boundedvarstep2}, the second inequality is due to Assumption \ref{assum:boundedvarinstochgrad}, the third inequality is due to $(a+b)^2 \leq 2a^2 + 2b^2$ for any $a,b \in \R$, the fourth inequality is by Jensen's inequality, the last inequality is by \citep[Lemma 3.2]{bollapragada2024derivative}, $\bar \delta$ given in Table \ref{tbl:unified} and $M_F \leq M_f$, and the last equality is due to $\| v^T U_{\Tilde T_k} \|^2 = \| v \|^2$ for any $v \in \R^d$. 
From \eqref{eq:RSEtkgtk}, we have
\begin{align}
    \E_{T_k}[\| g_{T_k}(x_k)\|^2] = \frac{d}{N}\E_{\Tilde T_k}[\| U_{\Tilde T_k} b_{\Tilde T_k}\|^2] > \frac{d\epsilon'}{N} \label{eq:denomRS}.
\end{align}
Therefore, by \eqref{eq:Skappx1}, \eqref{eq:gxkandepsilonprime}, \eqref{eq:SkappxRSstep1}, and \eqref{eq:denomRS}, we have \eqref{eq:boundonSk} satisfied with $b_1 = 1 + \frac{4 \beta_1}{\theta^2} $, $b_2 = \frac{(2+\beta_1) M_f^2 \nu^4 d}{9\theta^2} + \frac{2\beta_2}{\theta^2} $.
\end{enumerate}
Since $\theta$ is a user-defined hyperparameter, we do not consider it in our complexity analysis (i.e. we use the fact that $\theta = \mathcal{O}(1)$). %
\end{proof}
Now, we present the complexity results based on these sample size bounds.
\begin{thm} \label{thm:compres}(Complexity Results). 
    Suppose the conditions for Theorem~\ref{thm:sublconv} hold. Then, the number of iterations (iteration complexity) and the total number of stochastic function evaluations (sample complexity) required to obtain an $\epsilon$-accurate solution are given as follows. 
    \begin{align}\label{eq:KepsandWepsdef}
        & K_{\epsilon} = \frac{8 (F(x_0) - F^*) }{\epsilon p \bar \alpha} = \mathcal{O}(\epsilon^{-1}), \quad \text{and} \quad W_{\epsilon} \defeq \sum_{k = 0}^{K_{\epsilon}} 2 |T_k| |S_k| \leq 2N K_{\epsilon} \left(b_1 + \frac{ 4 b_2 }{\epsilon}\right) = \mathcal{O}(\epsilon^{-2}).
    \end{align}
    The results are detailed in Table \ref{tbl:compres} for each gradient estimation method.
\end{thm}
\begin{proof}
    Substituting the values for $\nu=\hat{\nu}$ given in Table~\ref{tbl:unified} leads to $\frac{4 \bar \chi}{p} = \frac{\epsilon}{2}$. From \eqref{eq:highprobconv}, it follows that for any $K \geq K_{\epsilon} = \frac{8 (F(x_0) - F^*) }{\epsilon p \bar \alpha}$,
    \vspace{-1em}
    \begin{align*}
         \min_{0 \leq k \leq K-1} \| \nabla F(x_k) \|^2 \leq \frac{4(F(x_0) - F^*)}{K \bar \alpha p} + \frac{4 \bar \chi}{p} \leq \frac{\epsilon}{2} + \frac{\epsilon}{2} = \epsilon. 
    \end{align*}
    Plugging in the values for $\bar \alpha$ from Table~\ref{tbl:unified} and for $b_1$ and $b_2$ from Table~\ref{tbl:compres} completes the proof.
\end{proof}
Theorem \ref{thm:compres} establishes that the iteration complexities of the methods are $\mathcal{O} (\epsilon^{-1})$, and the sample complexities are $\mathcal{O} (\epsilon^{-2})$, matching the optimal worst-case complexity results for nonconvex problem settings \citep{bottou2018optimization,ghadimi2013stochastic}. Similar observations to those made in forward finite-differences methods \citep{bollapragada2024derivative} are found here. The results in Table~\ref{tbl:compres} highlight that the cFD method exhibits the best iteration complexity compared to other methods concerning the dependency on problem dimension $d$. Moreover, the complexities of cGS and cSS methods are $\frac{N+d}{N}$ times worse than that of cFD, while the complexities of cRC and cRS methods are $\frac{d}{N}$ times worse than cFD. In terms of sample complexities, cFD, cRC, and cRS methods match the optimal complexity, also in terms of the dependency on $d$ \citep{ghadimi2013stochastic}. Assuming $\epsilon^{-2}$ is the dominating term and $N$ is small, cGS and cSS methods exhibit complexities that are $d$ times worse than the cFD method.

\begin{table}[htp]
\centering
\def\arraystretch{1.7}
\caption{Complexity results for various gradient estimation methods. Here, $\beta_1$ and $\beta_2$ are constants defined in Assumption~\ref{assum:boundedvarinstochgrad}. The order results specify only dependencies on $d$, $L_F$, $M_f$, $N$, $\nu$, $p$, $\beta_1$, and $\beta_2$. For $K_{\epsilon}$ and $W_{\epsilon}$, dependencies on $\beta_1$ and $\beta_2$ are removed.
}
\begin{tabular}{|c|c|c|c|c|}
\hline
Method & $ b_1 $ & $ b_2 $ & $ K_{\epsilon} $ & $W_{\epsilon}$ \\ \hline
cFD
& $\mathcal{O} (\beta_1)$
& $\mathcal{O} (\beta_1 M_f^2 \nu^4 d + \beta_2)$
& $\mathcal{O} (\frac{L_F}{\epsilon p})$
& $\mathcal{O} (\frac{L_F d}{\epsilon^{2} p})$\\
cGS
& $\mathcal{O} (\beta_1 d)$
& $\mathcal{O} ((\beta_1 + d)M_f^2 \nu^4 d^3 + \beta_2 d)$
& $\mathcal{O} (\frac{L_F}{\epsilon p} \frac{N+d}{N})$
& $\mathcal{O} (\frac{L_F d (N+d)}{\epsilon^{2} p} + \frac{L_F (N+d)^2 d^2}{\epsilon(N+d^2)})$ \\
cSS 
& $\mathcal{O} (\beta_1 d)$
& $\mathcal{O} ((\beta_1 + d)M_f^2 \nu^4 d + \beta_2 d )$
& $\mathcal{O} (\frac{L_F}{\epsilon p} \frac{N+d}{N})$
& $\mathcal{O} (\frac{L_F d (N+d)}{\epsilon^{2} p} + \frac{L_F (N+d)^2 d^2}{\epsilon(N+d^2)})$ \\
cRC 
& $\mathcal{O} (\beta_1)$
& $\mathcal{O} (\beta_1 M_f^2 \nu^4 d + \beta_2)$
& $\mathcal{O} (\frac{L_F}{\epsilon p} \frac{d}{N})$
& $\mathcal{O} (\frac{L_F d}{\epsilon^{2} p})$ \\
cRS
& $\mathcal{O} (\beta_1)$
& $\mathcal{O} (\beta_1 M_f^2 \nu^4 d + \beta_2)$
& $\mathcal{O} (\frac{L_F}{\epsilon p} \frac{d}{N})$
& $\mathcal{O} (\frac{L_F d}{\epsilon^{2} p})$ \\
\hline
\end{tabular}
\label{tbl:compres}
\end{table}

\section{Numerical Experiments}
\label{sec:numericalExp}
In this section, we evaluate the empirical performance of the proposed methods on nonlinear least squares (NLLS) problems obtained by introducing stochastic Gaussian noise of the form $\zeta \sim \mathcal{N}(0,\sigma^2 I_p)$ with $\sigma = 10^{-3}$ to the nonlinear functions $\phi : \mathbb{R}^d \rightarrow \mathbb{R}^p$ from the CUTEr \citep{gould2003cuter} collection of optimization problems.\footnote{We also evaluate the empirical performance of the proposed methods on binary classification logistic regression problems (see Appendix~\ref{sec:additionalPlots}.)} We consider two different nonlinear functions $\phi$: Bdqrtic $(d = 50, p = 92)$ and Cube $(d = 20, p = 30)$, with two different error terms: absolute error and relative error. The resulting stochastic objective functions are defined as follows
\begin{equation*}
\quad f_{\rm abs}(x,\zeta) \defeq \sum_{j=1}^{p}(\phi_j(x) + \zeta_j)^2 - \sigma^2, \quad \text{and} \quad 
f_{\rm rel}(x,\zeta) \defeq \frac{1}{1 + \sigma^2} \sum_{j=1}^{p}\phi^2_j(x) (1 + \zeta_j)^2. 
\end{equation*}

We employ the following practical test to approximately satisfy the theoretical Condition \ref{cond:theoreticalnormcond} by approximating the population (expectation) quantities with sampled quantities in our implementation.
\begin{test}\label{test:practicalnormtest} (Practical Norm Test) \citep[Test 1]{bollapragada2024derivative}
\begin{align}\label{eq:practicalnormtest}
    \frac{Var_{\zeta_i \in S_k^v}\left( g_{\zeta_i, T_k}(x_k)\right) }{|S_k|} \leq \theta^2  \|{g_{S_k,T_k}(x_k)}\|^2, \quad \theta>0,
\end{align}
where $S_k^v \subseteq S_k$ is any subset of $S_k$,
and $Var_{\zeta_i \in S_k^v}\left( g_{\zeta_i, T_k}(x_k)\right) = \frac{1}{|S_k^v| - 1} \sum_{\zeta_i \in S_k^v} \| g_{\zeta_i, T_k}(x_k) - g_{S_k,T_k}(x_k) \|^2 $.
\end{test}
We evaluate this test at each iteration, and if the test fails, we append the set $S_k$ with additional samples such that $S_k$ satisfies $$ |S_k| = \left\lceil\frac{Var_{\zeta_i \in S_k^v}\left( g_{\zeta_i, T_k}(x_k)\right) }{\theta^2  \|{g_{S_k,T_k}(x_k)}\|^2} \right\rceil.$$ Although this approach involves approximations, the accuracy of these approximations improves with increasing sample sizes $|S_k|$. Moreover, this approach is more efficient than selecting a fixed, yet large, number of samples throughout the algorithm.

In our implementation, we set $\theta=0.9$, initial sample size $|S_0| = 2$, and treat the number of sampled directions $N$, the sampling radius $\nu$, and the step size $\alpha$ as tunable hyperparameters.  We consider the worst-case performance of each combination of $N$, $\nu$, and $\alpha$ values for each method, across three random runs. Finally, we select the best combination yielding the smallest optimality gap ($F(x_k) - F^*$) after the methods have utilized a fixed budget of function evaluations. We conduct an additional 17 random runs for the best combinations. The legends of the figures in this section indicate the method and hyperparameters in $``(\text{method},N,\nu,\alpha)"$ format. We exclude the cRS method in reporting the results as its corresponding results are similar to those of the cRC method. For more detailed information about the implementation, refer to \cite{bollapragada2024derivative}.

Figures~\ref{fig:193bestvsbest} and~\ref{fig:203bestvsbest} present results for the Bdqrtic and Cube functions. The solid lines represent the median performances, while the bands around the lines indicate the 35th and 65th quantiles across 20 random runs. The hyperparameters $(N,\nu,\alpha)$ are independently tuned for each method. The tuned cRC and cFD performed similarly since the tuned $N$ was equal to $d$.
In Figs.~\ref{fig:19abs3bestvsbestoptgap},~\ref{fig:19rel3bestvsbestoptgap},~\ref{fig:20abs3bestvsbestoptgap}, and~\ref{fig:20rel3bestvsbestoptgap}, we report the optimality gap $(F(x_k) - F^*)$ with respect to the number of stochastic function evaluations. We observe that the optimal tuned number of directions ($N$) is smaller for smoothing methods (cGS and cSS) compared to the cRC method. Furthermore, while the smoothing methods initially exhibit superior performance, the standard finite-difference method (cFD) ultimately matches their performance as the number of function evaluations increases due to their superior accuracy in gradient estimation. In Figs.~\ref{fig:19abs3bestvsbestbatch},~\ref{fig:19rel3bestvsbestbatch},~\ref{fig:20abs3bestvsbestbatch}, and~\ref{fig:20abs3bestvsbestbatch}, we report the batch size or sample size ($|S_k|$) with respect to iterations. Here, the smoothing methods tend to increase the sample sizes at a faster rate than the cRC and cFD methods. We conjecture that this behavior is due to the high variance in the smoothing methods with smaller $N$ values, necessitating larger sample sizes.

We also analyze the effect of the number of directions $(N)$ on the empirical performance of cGS, cSS, and cRC methods in Fig.~\ref{fig:193numdirsens}, which reports the optimality gap with respect to the number of function evaluations. The hyperparameters $\nu$ and $\alpha$ are tuned for each $N$ and method combination. The solid lines represent the mean performance, while the bands depict the minimum and maximum values across three random runs. We observe that the behavior across different $N$ values is problem-specific. Typically, smaller values of $N$ lead to high variance in gradient estimation, while larger values result in high per-iteration function evaluations, rendering both inefficient. An optimal $N$ exists that achieves the best performance for the smoothing methods.

Results corresponding to the other datasets are available in Appendix~\ref{sec:additionalPlots}.

\begin{figure}[ht]
\centering
\begin{subfigure}{0.45\textwidth}
  \centering
  \includegraphics[width=1\linewidth]{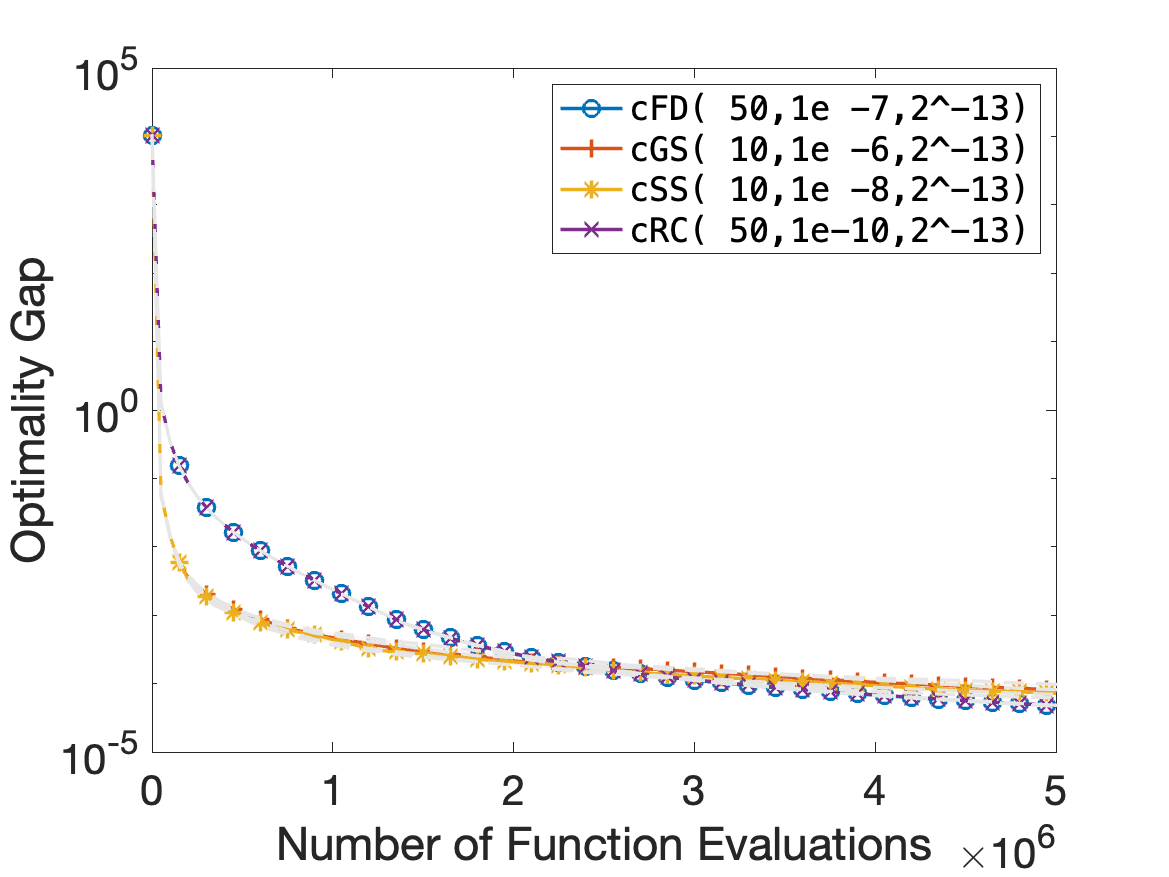}
  \caption{Optimality Gap}
  \label{fig:19abs3bestvsbestoptgap}
\end{subfigure}%
\begin{subfigure}{0.45\textwidth}
  \centering
  \includegraphics[width=1\linewidth]{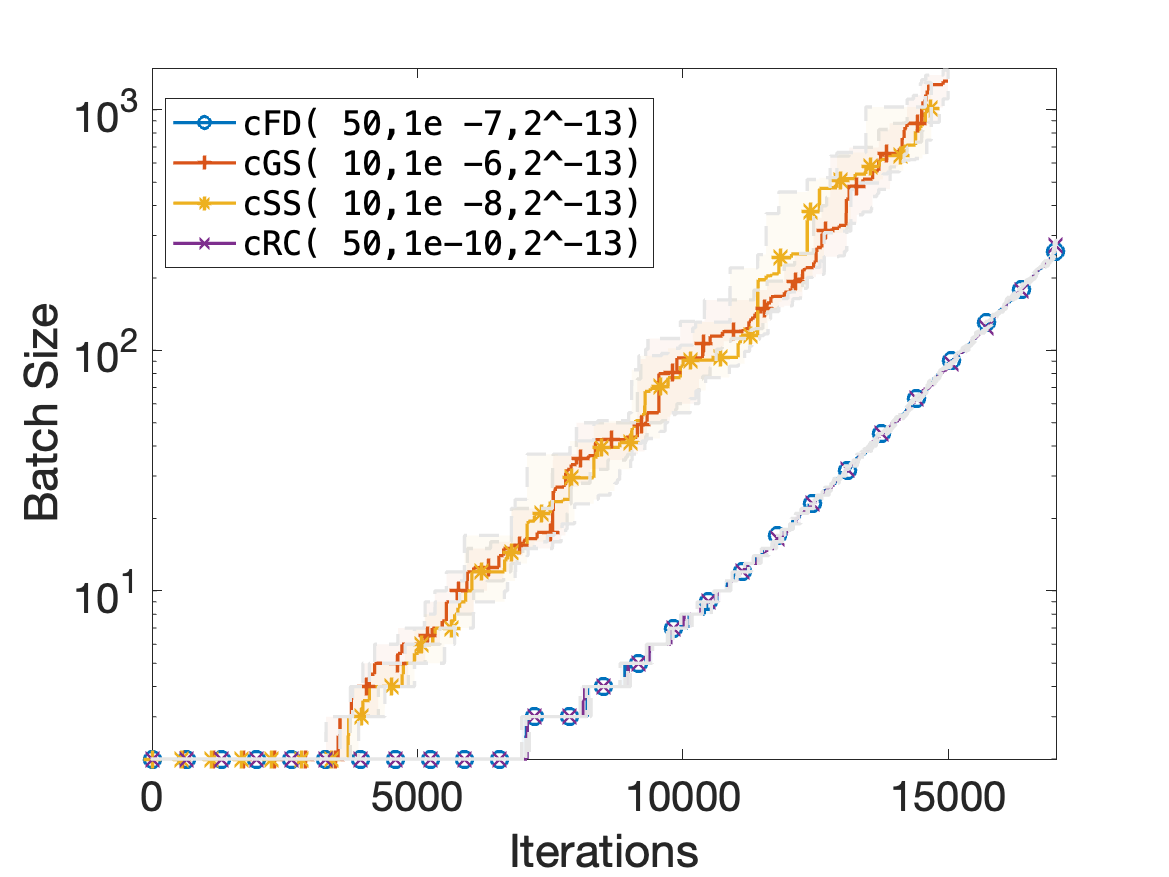}
  \caption{Batch Size}
  \label{fig:19abs3bestvsbestbatch}
\end{subfigure}
\begin{subfigure}{0.45\textwidth}
  \centering
  \includegraphics[width=1\linewidth]{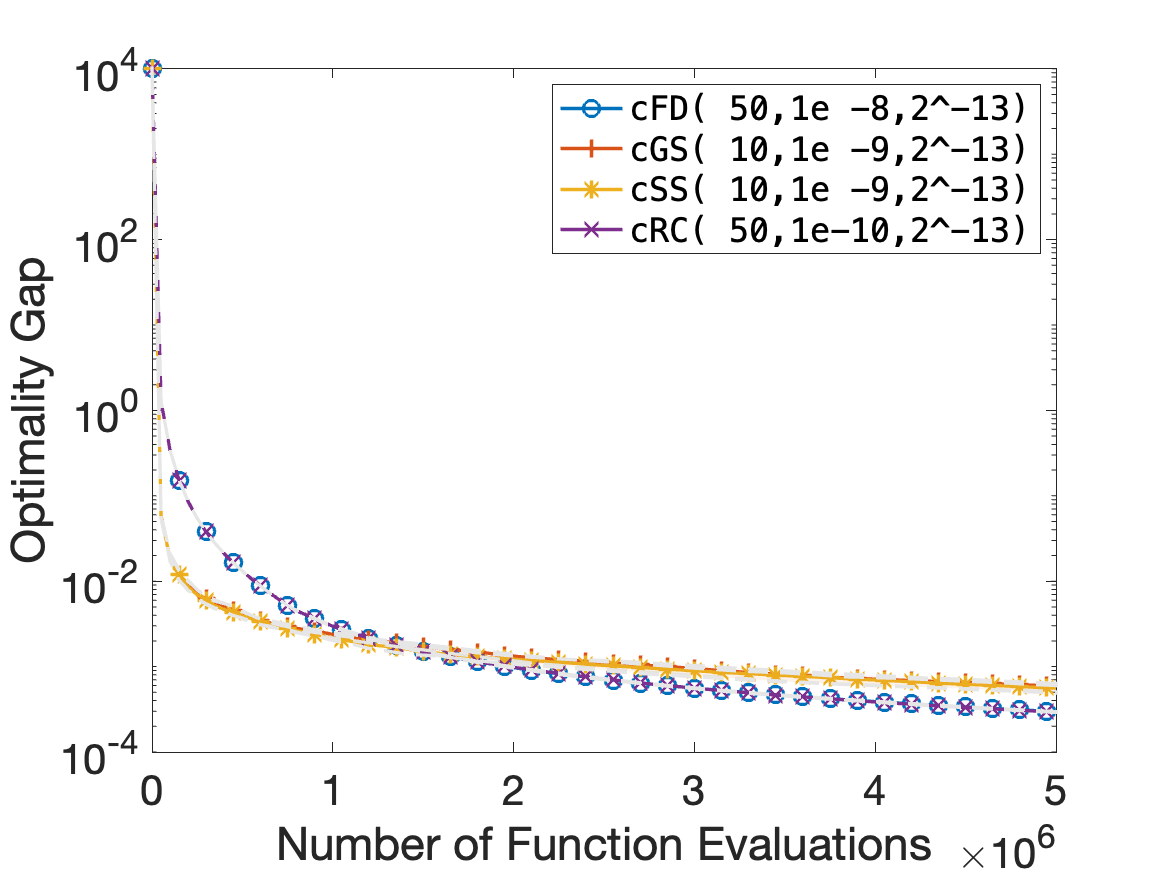}
  \caption{Optimality Gap}
  \label{fig:19rel3bestvsbestoptgap}
\end{subfigure}%
\begin{subfigure}{0.45\textwidth}
  \centering
  \includegraphics[width=1\linewidth]{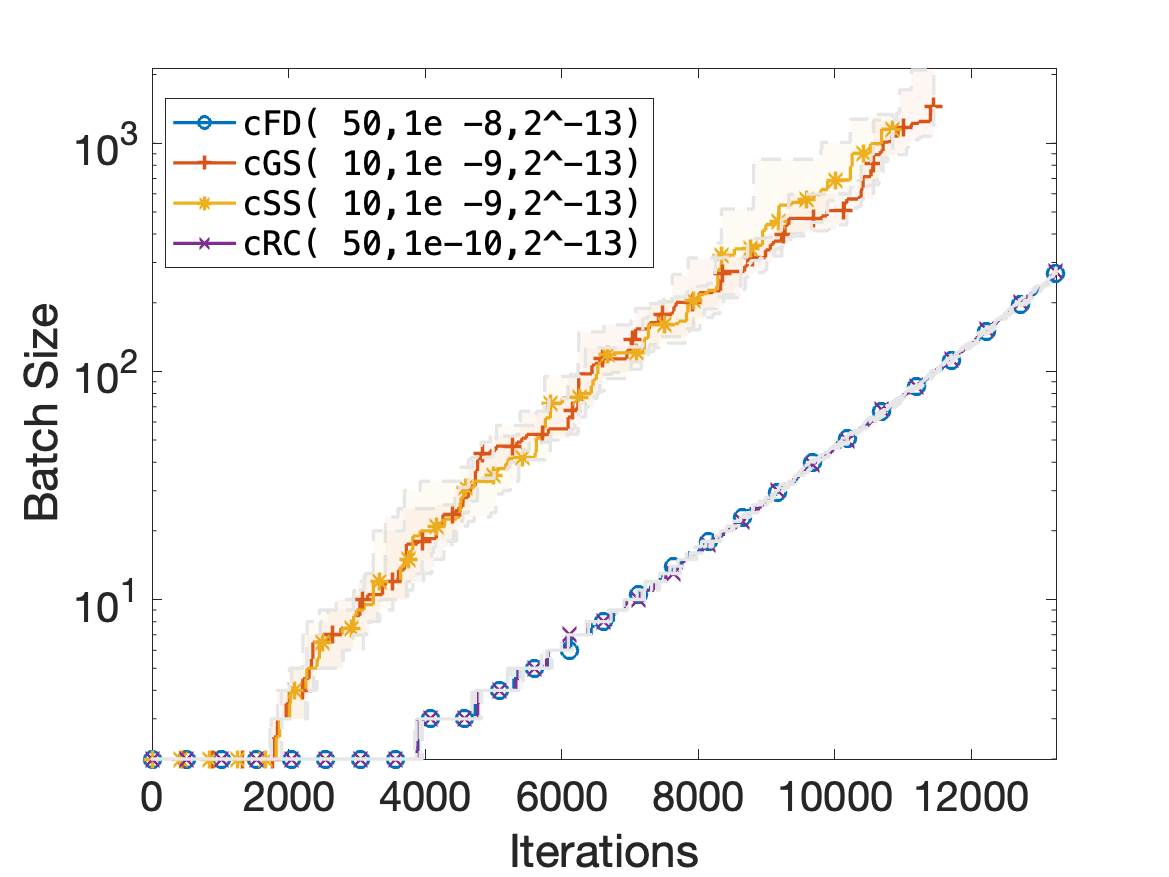}
  \caption{Batch Size}
  \label{fig:19rel3bestvsbestbatch}
\end{subfigure}
\caption{Performance of different gradient estimation methods using the tuned hyperparameters on the Bdqrtic function with $\sigma = 10^{-3}$. Top row: absolute error, bottom row: relative error.}
\label{fig:193bestvsbest}
\end{figure}
\begin{figure}[ht]
\centering
\begin{subfigure}{0.45\textwidth}
  \centering
  \includegraphics[width=1\linewidth]{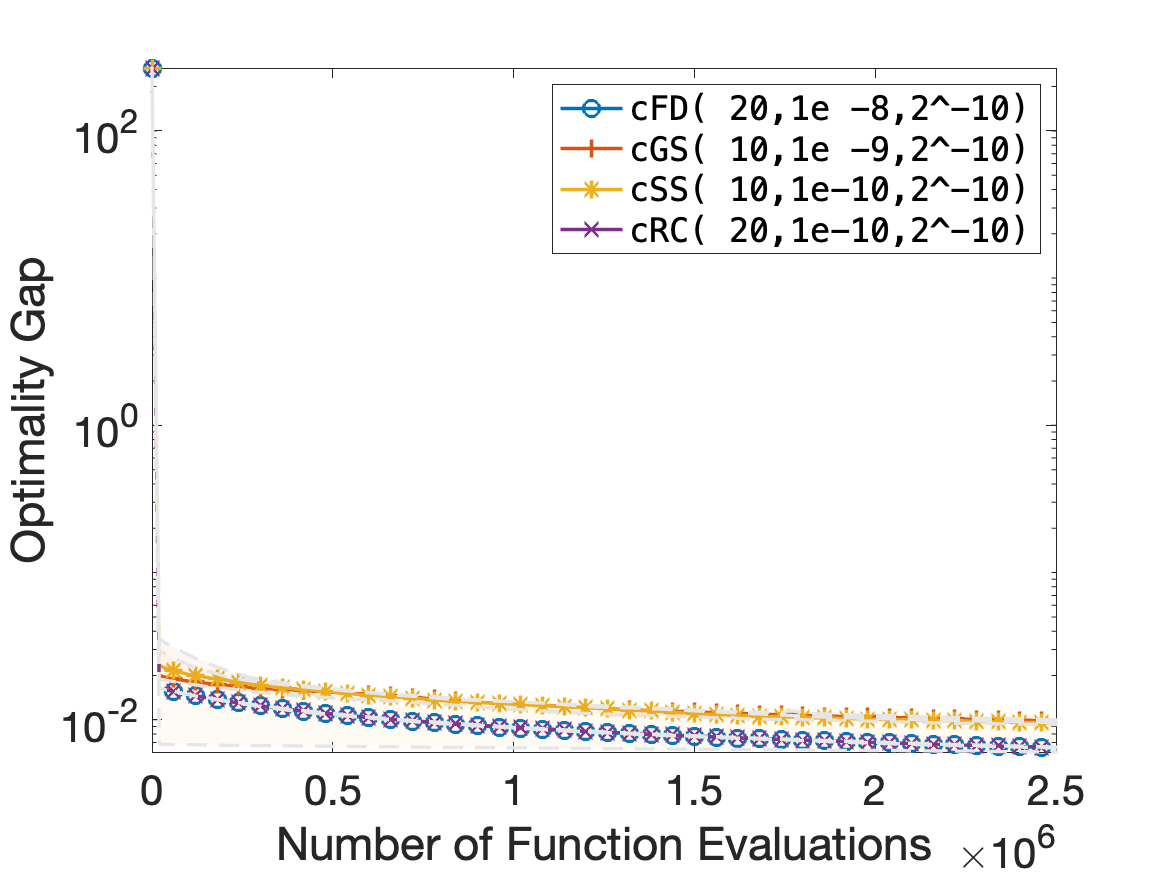}
  \caption{Optimality Gap}
  \label{fig:20abs3bestvsbestoptgap}
\end{subfigure}%
\begin{subfigure}{0.45\textwidth}
  \centering
  \includegraphics[width=1\linewidth]{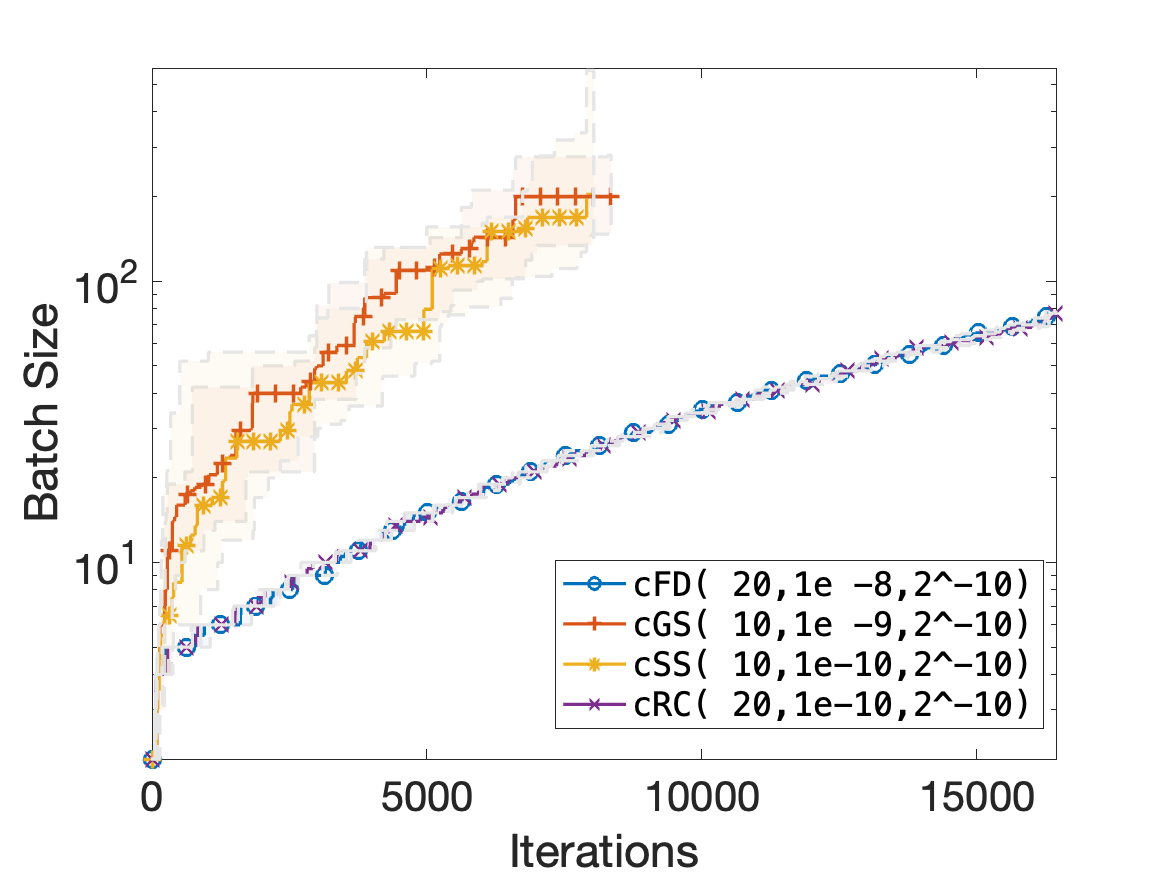}
  \caption{Batch Size}
  \label{fig:20abs3bestvsbestbatch}
\end{subfigure}
\begin{subfigure}{0.45\textwidth}
  \centering
  \includegraphics[width=1\linewidth]{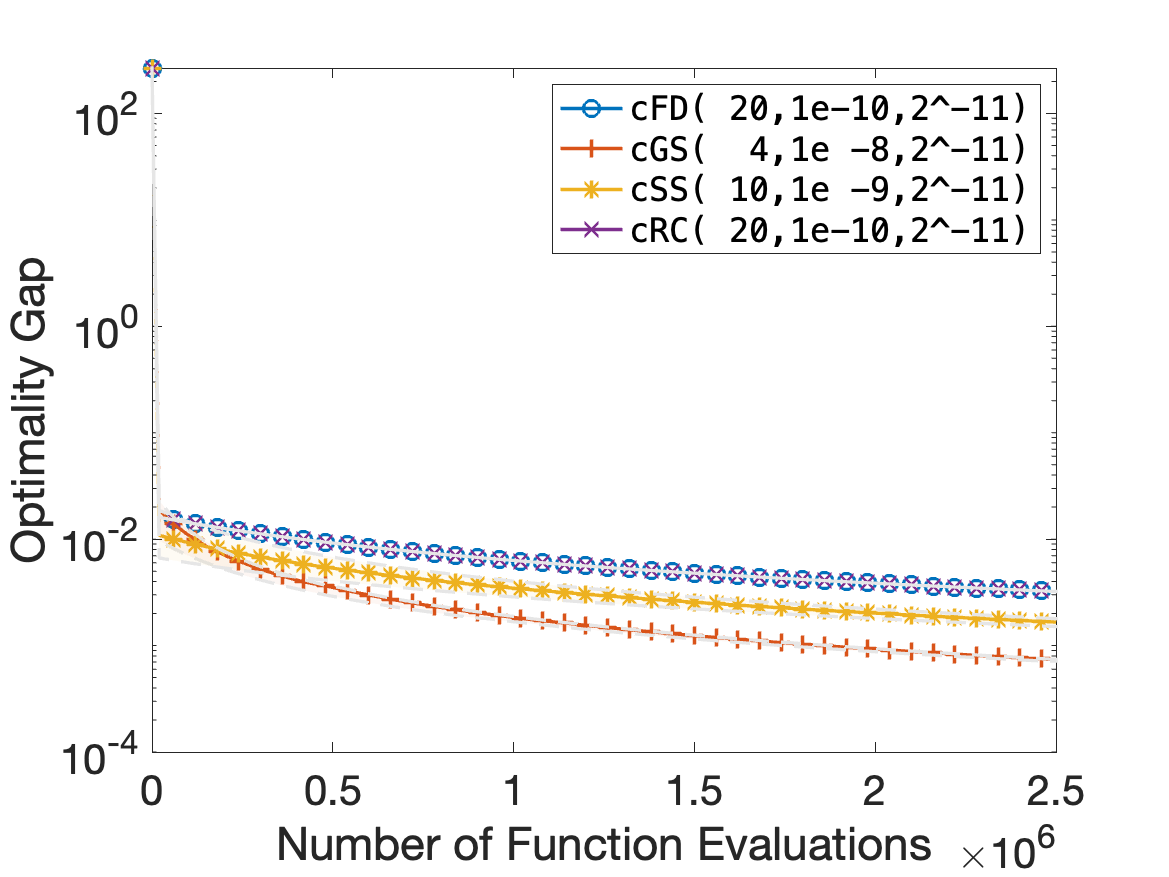}
  \caption{Optimality Gap}
  \label{fig:20rel3bestvsbestoptgap}
\end{subfigure}%
\begin{subfigure}{0.45\textwidth}
  \centering
  \includegraphics[width=1\linewidth]{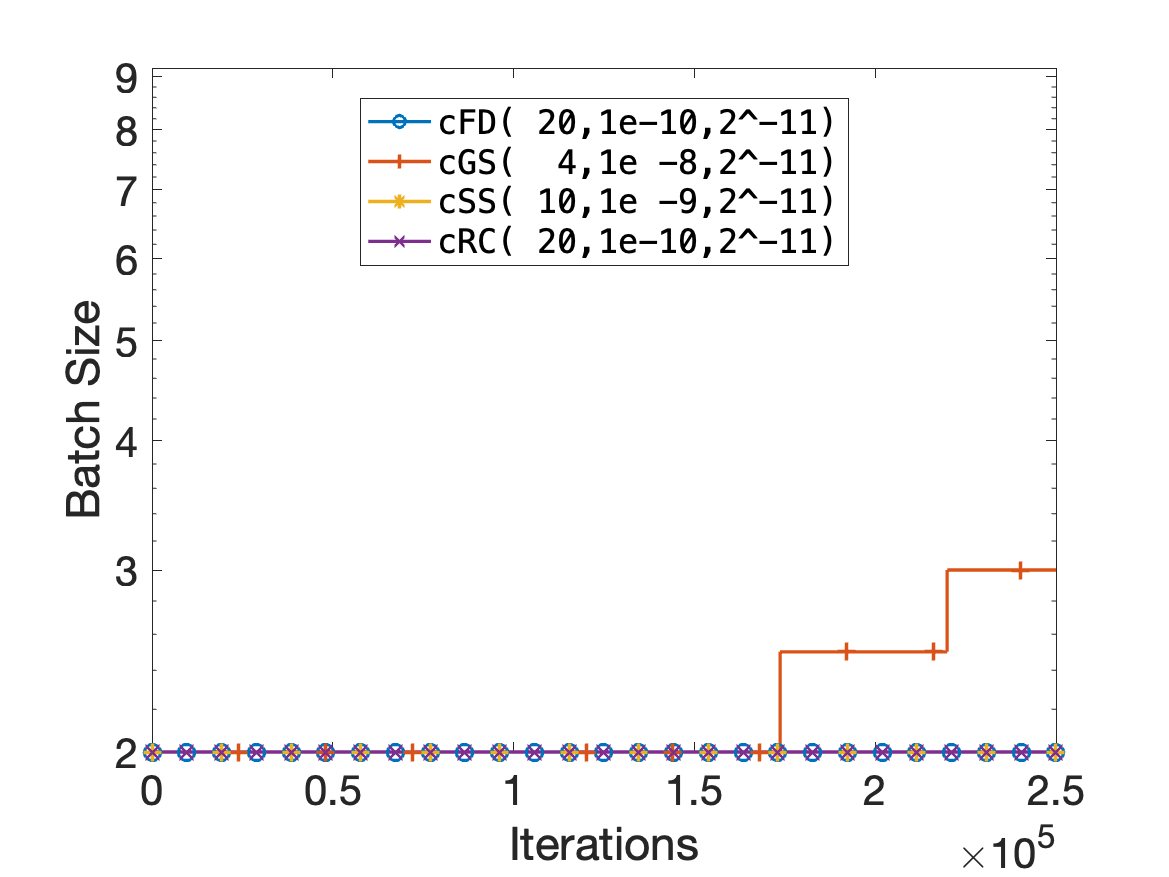}
  \caption{Batch Size}
  \label{fig:20rel3bestvsbestbatch}
\end{subfigure}
\caption{Performance of different gradient estimation methods using the tuned hyperparameters on the Cube function with $\sigma = 10^{-3}$. Top row: absolute error, bottom row: relative error.}
\label{fig:203bestvsbest}
\end{figure}

\begin{figure}[ht]
\centering
\begin{subfigure}{0.33\textwidth}
  \centering
  \includegraphics[width=1\linewidth]{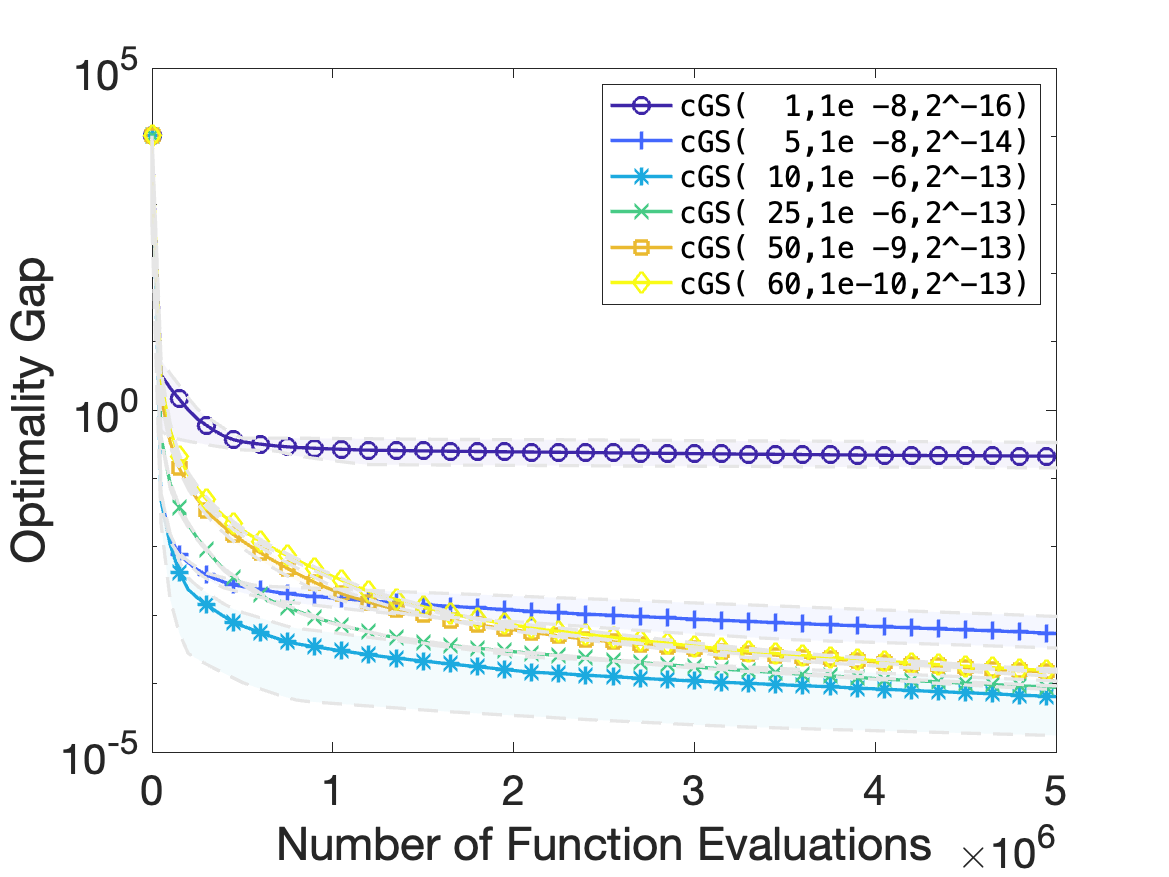}
  \caption{Performance of cGS}
  \label{fig:19abs3numdirsensGSCFD}
\end{subfigure}%
\begin{subfigure}{0.33\textwidth}
  \centering
  \includegraphics[width=1\linewidth]{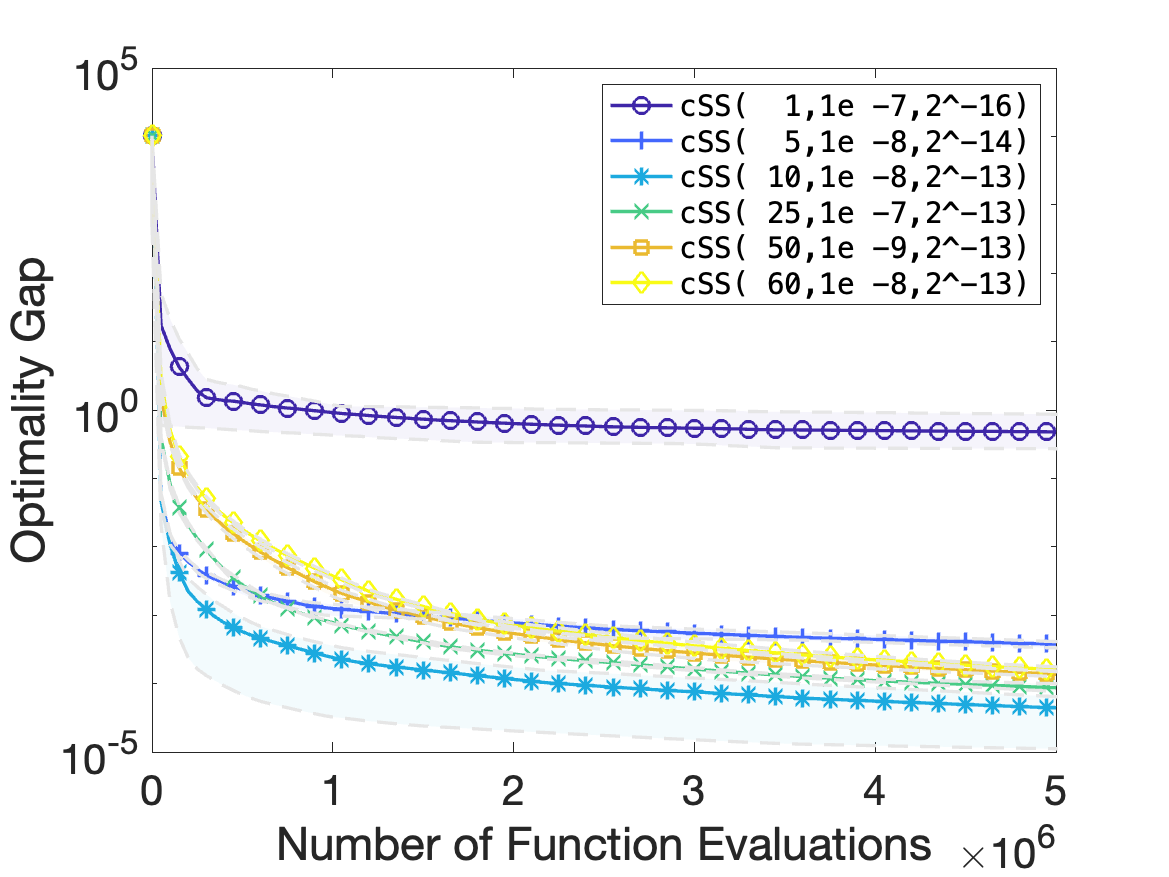}
  \caption{Performance of cSS}
  \label{fig:19abs3numdirsensSSCFD}
\end{subfigure}%
\begin{subfigure}{0.33\textwidth}
  \centering
  \includegraphics[width=1\linewidth]{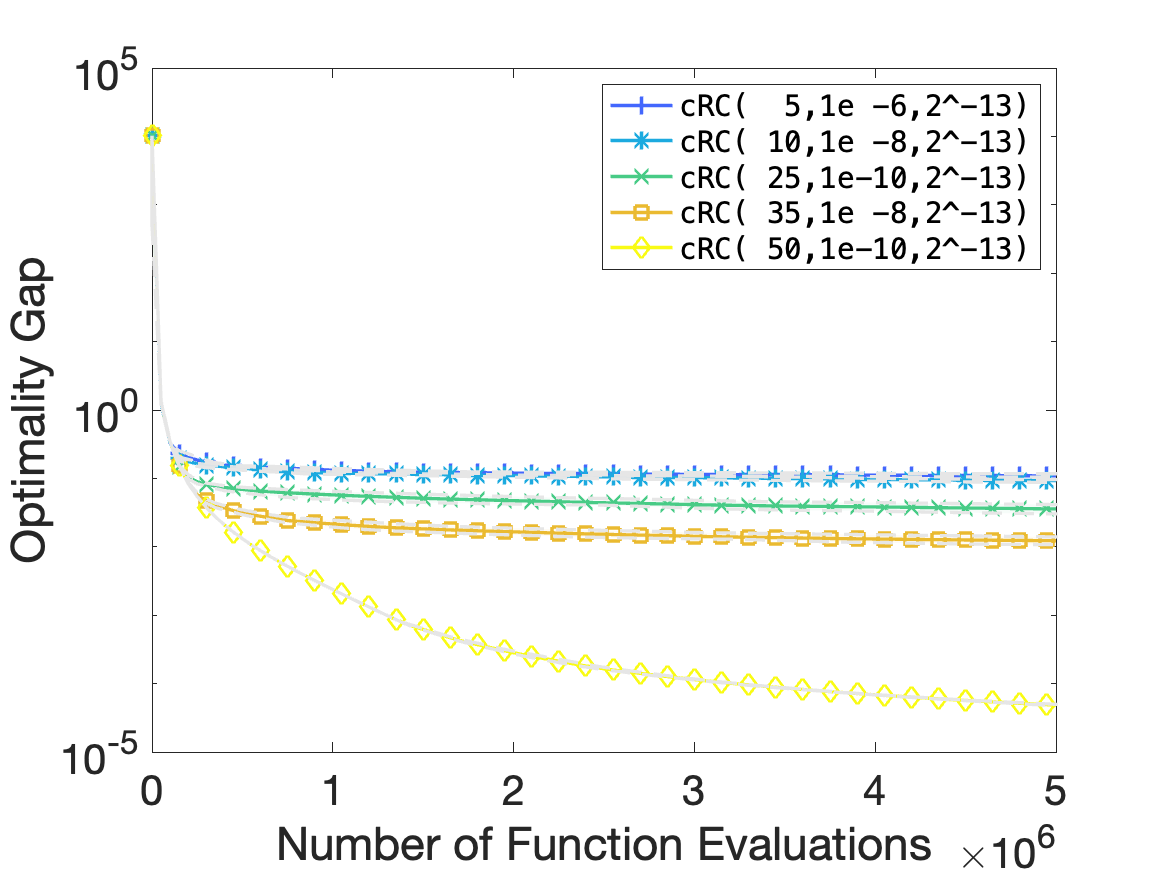}
  \caption{Performance of cRC}
  \label{fig:19abs3numdirsensRCCFD}
\end{subfigure}
\begin{subfigure}{0.33\textwidth}
  \centering
  \includegraphics[width=1\linewidth]{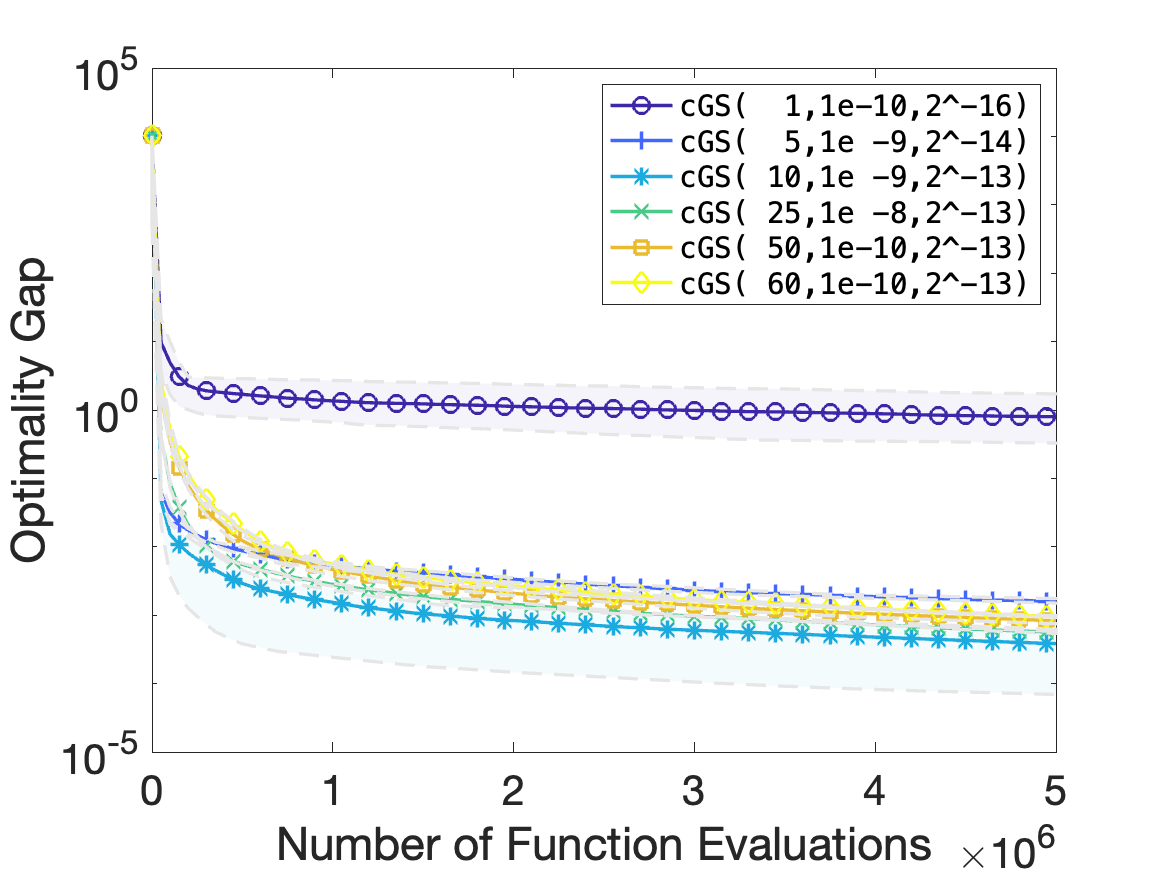}
  \caption{Performance of cGS}
  \label{fig:19rel3numdirsensGSCFD}
\end{subfigure}%
\begin{subfigure}{0.33\textwidth}
  \centering
  \includegraphics[width=1\linewidth]{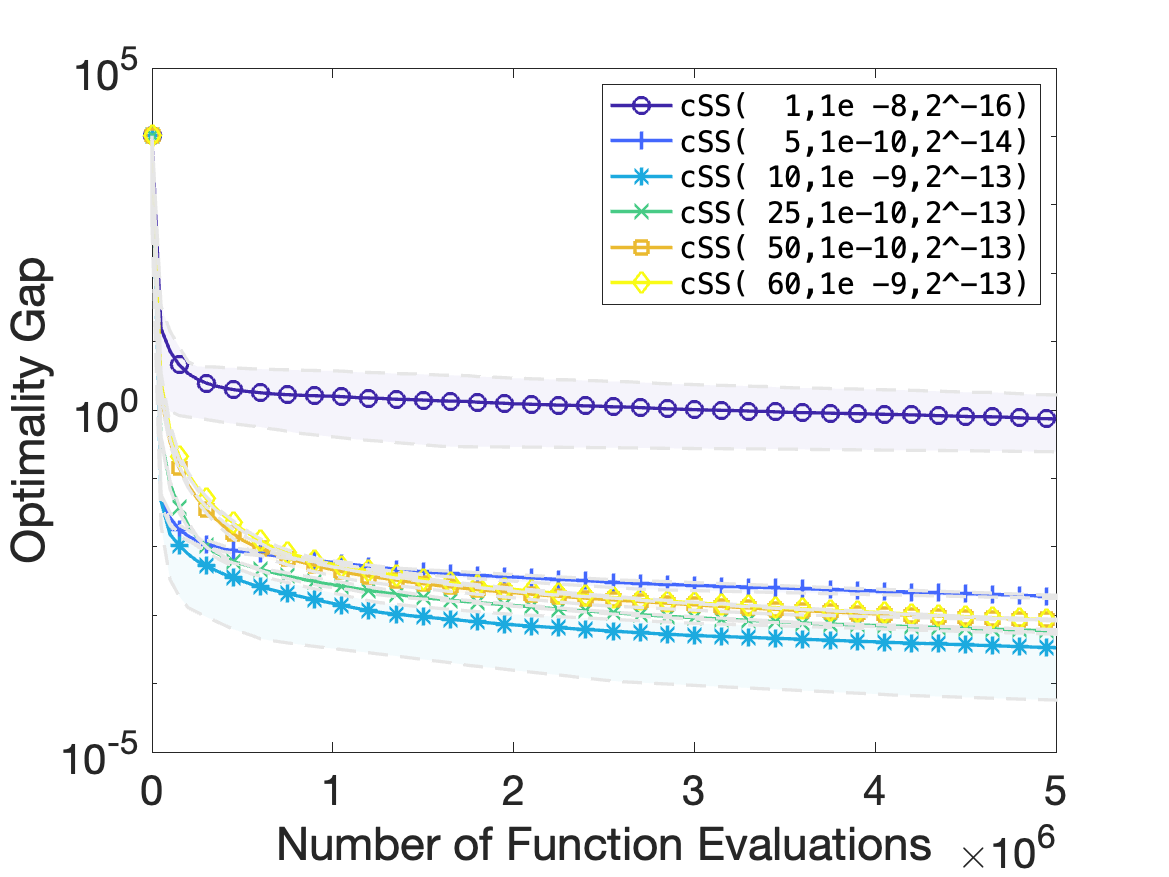}
  \caption{Performance of cSS}
  \label{fig:19rel3numdirsensSSCFD}
\end{subfigure}%
\begin{subfigure}{0.33\textwidth}
  \centering
  \includegraphics[width=1\linewidth]{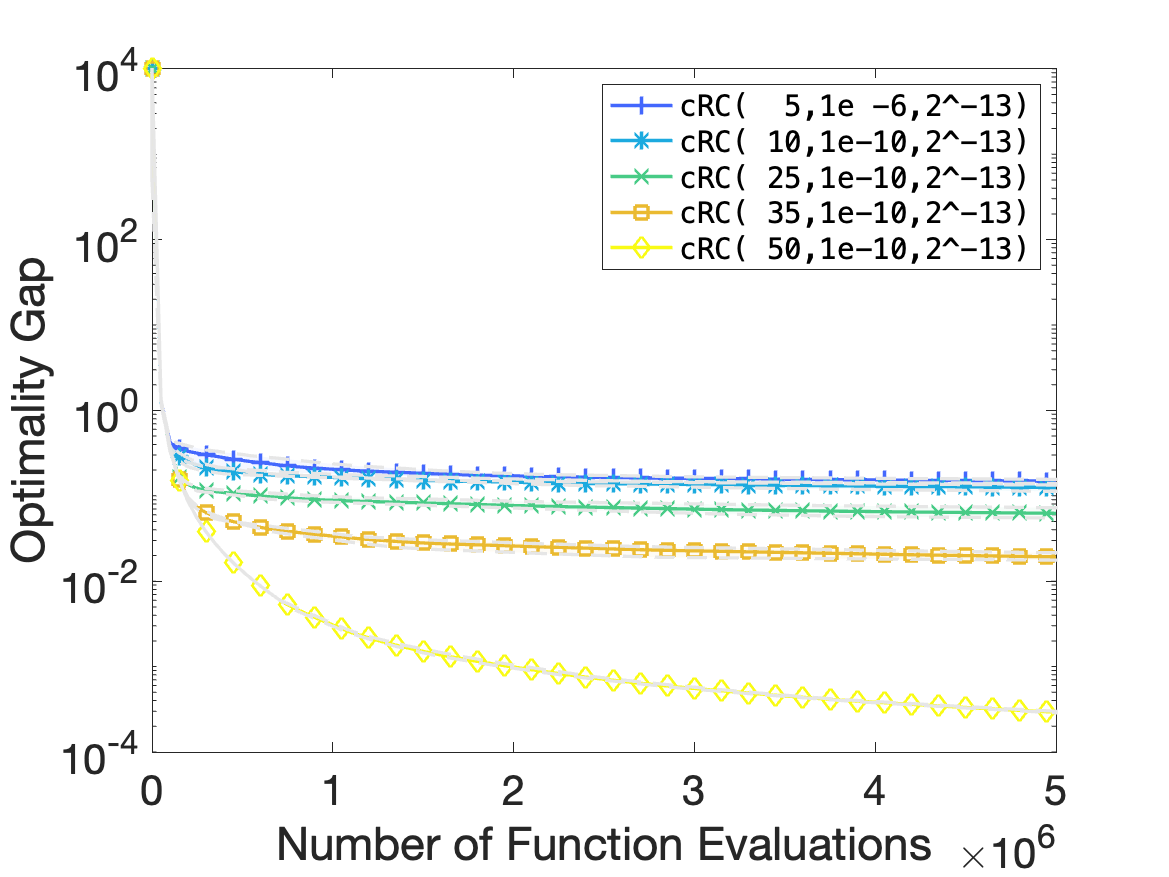}
  \caption{Performance of cRC}
  \label{fig:19rel3numdirsensRCCFD}
\end{subfigure}
\caption{The effect of number of directions $N$ on the performance of different randomized gradient estimation methods on the Bdqrtic function with $ \sigma = 10^{-3} $. Sampling radius $\nu$ and step size $\alpha$ are tuned for each method and $N$ combination to achieve the best performance. Top row: absolute error, bottom row: relative error.}
\label{fig:193numdirsens}
\end{figure}

\newpage
\section{Final Remarks}
\label{sec:conclusions}

Gradient estimation methods that utilize stochastic function evaluations to estimate gradients and incorporate these estimators into standard optimization methods offer scalability to large-dimensional problems and are of interest from both theoretical and practical perspectives. We introduce a unified algorithmic framework for solving stochastic optimization problems using central finite-difference based gradient estimation methods. The variance in these estimators is controlled by adaptively selecting the sample sizes employed in the stochastic approximations. We establish sublinear convergence to a neighborhood for nonconvex functions and demonstrate that this framework achieves optimal worst-case iteration and sample complexities of $\mathcal{O}(\epsilon^{-1})$ and $\mathcal{O}(\epsilon^{-2})$, respectively. Our numerical results on nonlinear least squares problems illustrate the effectiveness of this approach. Furthermore, this framework can be extended to explore other potentially new or hybrid variants of existing central finite-difference methods to further enhance efficiency. It can also be seamlessly integrated into more sophisticated algorithms such as quasi-Newton or accelerated methods.

\section*{Acknowledgments}

The authors thank Dr. Stefan Wild of the Lawrence Berkeley National Laboratory for his invaluable insights that contributed to the preliminary version of this paper.

\bibliography{references}

\appendix
\newpage
\section{Appendix}
\label{sec:appendix}

\subsection{Bounded Variance in \eqref{eq:theoreticalnormcond}}
\label{sec:CFDproofofboundedvar}

\citep[Section A.1]{bollapragada2024derivative}. Let us define
\begin{align} \label{eq:defofci}
    c_{i,j} \defeq \Big( \frac{f(x_k + \nu u_j,\zeta_i) - f(x_k - \nu u_j,\zeta_i)}{2\nu} - \frac{F(x_k + \nu u_j) - F(x_k - \nu u_j)}{2\nu} \Big).
\end{align}
By \eqref{eq:gen_grad_est} and Assumption \ref{assum:sampling} we have
\begin{align} \label{eq:boundedvarstep1}
& \E_{\zeta_i} \Bigg[  \|g_{\zeta_i, T_k}(x_k) - g_{T_k}(x_k)\|^2 \Bigg] %
 = \E_{\zeta_i} \Bigg[ \Big\| \gamma_k\sum_{u_j \in T_k} c_{i,j} u_j \Big \|^2 \Bigg] \nonumber \\
 & \leq \gamma_k^2 |T_k|\sum_{u_j \in T_k} \E_{\zeta_i} [ \| c_{i,j} u_j \|^2 ] = \gamma_k^2 |T_k| \sum_{u_j \in T_k} \E_{\zeta_i} [ c_{i,j}^2] \|u_j\|^2 ,
\end{align}
where the inequality is due to $ (a_1 + a_2 + \dots + a_n)^2 \leq n(a_1^2 + a_2^2 + \dots + a_n^2)$. Now, consider
\begin{align} 
c_{i,j}^2 & = \Big( \frac{f(x_k + \nu u_j,\zeta_i) - f(x_k - \nu u_j,\zeta_i)}{2\nu} - \frac{F(x_k + \nu u_j) - F(x_k - \nu u_j)}{2\nu} \Big)^2 \nonumber \\
& = \Big( u_j^T [\nabla f (x_k,\zeta_i) - \nabla F (x_k)] \nonumber \\
& \quad + \frac{f(x_k + \nu u_j,\zeta_i) - f(x_k - \nu u_j,\zeta_i) - 2\nu u_j^T \nabla f (x_k,\zeta_i)}{2\nu} \nonumber \\
& \quad - \frac{F(x_k + \nu u_j) - F(x_k - \nu u_j) - 2\nu u_j^T \nabla F (x_k)}{2\nu} \Big)^2 \nonumber \\
& \leq 2\Big( u_j^T [\nabla f (x_k,\zeta_i) - \nabla F (x_k)] \Big)^2 \nonumber \\
& \quad + 4 \Big(\frac{f(x_k + \nu u_j,\zeta_i) - f(x_k - \nu u_j,\zeta_i) - 2\nu u_j^T \nabla f (x_k,\zeta_i)}{2\nu} \Big)^2 \nonumber \\
& \quad + 4 \Big( \frac{F(x_k + \nu u_j) - F(x_k - \nu u_j) - 2\nu u_j^T \nabla F (x_k)}{2\nu} \Big)^2 \nonumber \\
& \leq 2\Big( u_j^T [\nabla f (x_k,\zeta_i) - \nabla F (x_k)] \Big)^2 + \frac{1}{9}(M_F^2 + M_f^2) \nu^4 \|u_j\|^6 \nonumber \\
& \leq 2\Big( u_j^T [\nabla f (x_k,\zeta_i) - \nabla F (x_k)] \Big)^2 + \frac{2}{9}M_f^2 \nu^4 \|u_j\|^6 \label{eq:boundedvarstep2} \\
& \leq 2 \| \nabla f (x_k,\zeta_i) - \nabla F (x_k)\|^2 \|u_j\|^2 + \frac{2}{9}M_f^2 \nu^4 \|u_j\|^6 \label{eq:boundedvarstep3}.
\end{align}
where the first inequality is due to $(a+b)^2 \leq 2 a^2 + 2 b^2$, the second inequality is due to Assumption \ref{assum:Lipschitzstochf}, the third inequality is due to $M_F \leq M_f$, and the last inequality is due to $a^Tb \leq \|a\| \|b\|$. By \eqref{eq:boundedvarstep3} we have
\begin{align}\label{eq:boundedvarstep4}
\E_{\zeta_i}[c_{i,j}^2] & \leq 2 \E_{\zeta_i}[\| \nabla f (x_k,\zeta_i) - \nabla F (x_k)\|^2] \|u_j\|^2 + \frac{2}{9}M_f^2 \nu^4 \|u_j\|^6 \nonumber \\
& \leq 2(\beta_1 \|\nabla F (x_k)\|^2 + \beta_2) \|u_j\|^2 + \frac{2}{9}M_f^2 \nu^4 \|u_j\|^6,
\end{align}
where the second inequality is due to Assumption \ref{assum:boundedvarinstochgrad}.
Finally, by \eqref{eq:boundedvarstep1} and \eqref{eq:boundedvarstep4} we have
\begin{align}\label{eq:boundedvarstep5}
    & \E_{\zeta_i} \Bigg[  \|g_{\zeta_i, T_k}(x_k) - g_{T_k}(x_k)\|^2 \Bigg] \leq 2 \gamma_k^2 |T_k| \sum_{u_j \in T_k} \Big[(\beta_1 \|\nabla F (x_k)\|^2 + \beta_2) \|u_j\|^4 + \frac{1}{9}M_f^2 \nu^4 \|u_j\|^8\Big].
\end{align}
Therefore, for all iterations $k$ where $\|\nabla F(x_k)\| < \infty$, and either for deterministic or finite realizations of random vectors $u_j$ (i.e. $\|u_j\| < \infty$ for all $u_j \in T_k$) with $\gamma_k \in \mathbb{Z}_{++}$ and $|T_k| \in \mathbb{Z}_{++}$, we have
\begin{align*}
    \E_{\zeta_i}[\| g_{\zeta_i, T_k}(x_k) - g_{T_k}(x_k)\|^2] < \infty. %
\end{align*} 

\subsection{Additional Plots}
\label{sec:additionalPlots}

In this section, we present additional plots demonstrating the performance on nonlinear least squares problems (Section \ref{sec:add_plots_CFD_NLLS}) and binary classification logistic regression problems (Section \ref{sec:add_plots_CFD_logreg}). We also include ``Comparison of cSS and cRC" plots \citep{bollapragada2024derivative}, where we compare these two methods using different numbers of directions $N$, with the other hyperparameters tuned for each $N$ and method combination.

\subsubsection{Nonlinear Least Squares Problems}
\label{sec:add_plots_CFD_NLLS}

\paragraph{$\quad ~$  Chebyquad Function $(d = 30, p = 45)$ with Relative Error, $\sigma = 10^{-3}$}

\begin{figure}[H]
\centering
\begin{subfigure}{0.33\textwidth}
  \centering
  \includegraphics[width=1.1\linewidth]{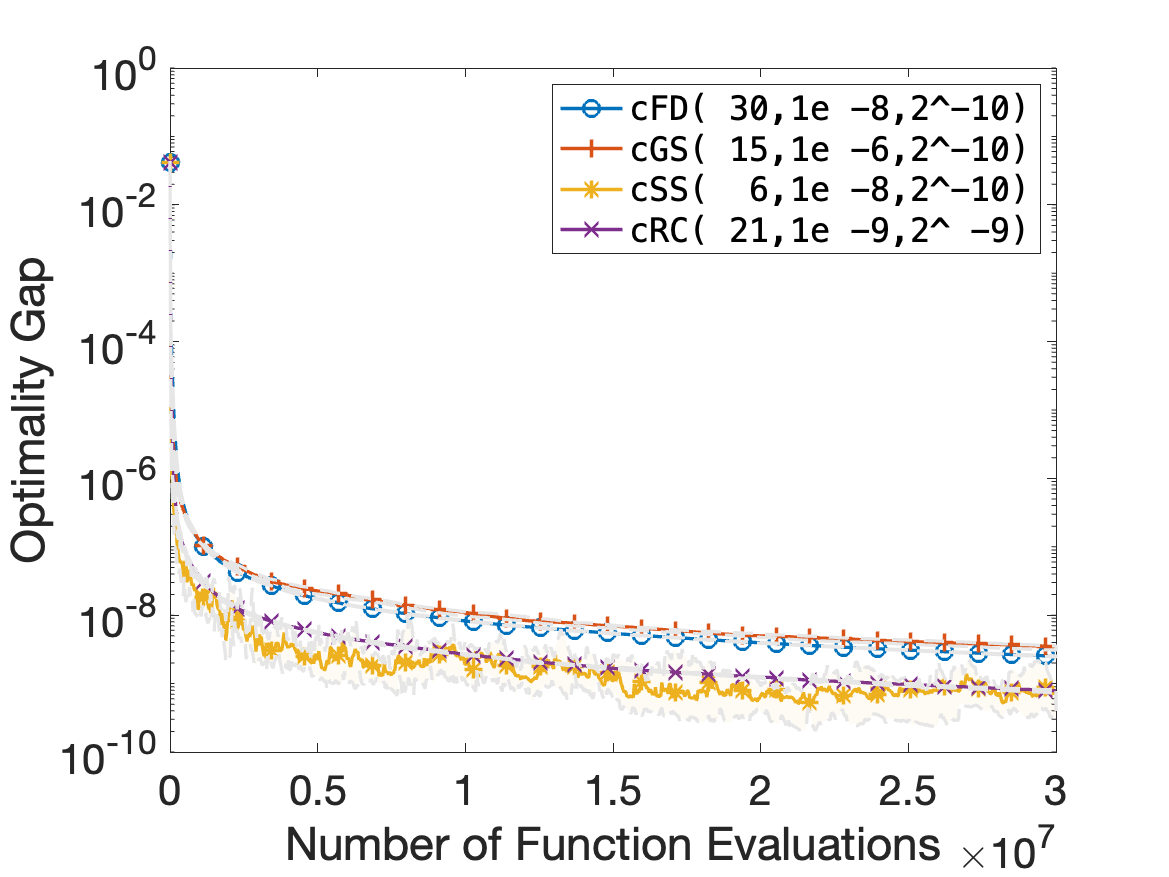}
  \caption{Optimality Gap}
\end{subfigure}%
\begin{subfigure}{0.33\textwidth}
  \centering
  \includegraphics[width=1.1\linewidth]{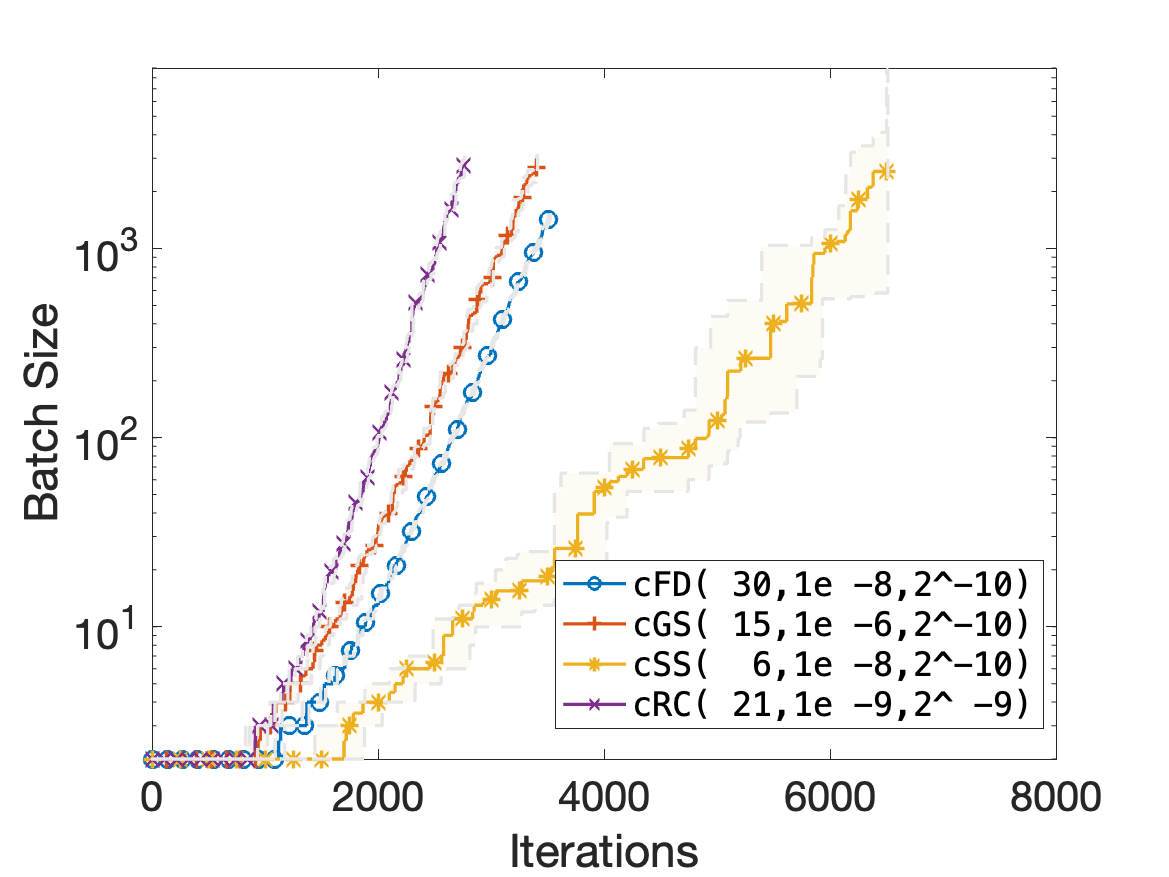}
  \caption{Batch Size}
\end{subfigure}
\begin{subfigure}{0.33\textwidth}
  \centering
  \includegraphics[width=1.1\linewidth]{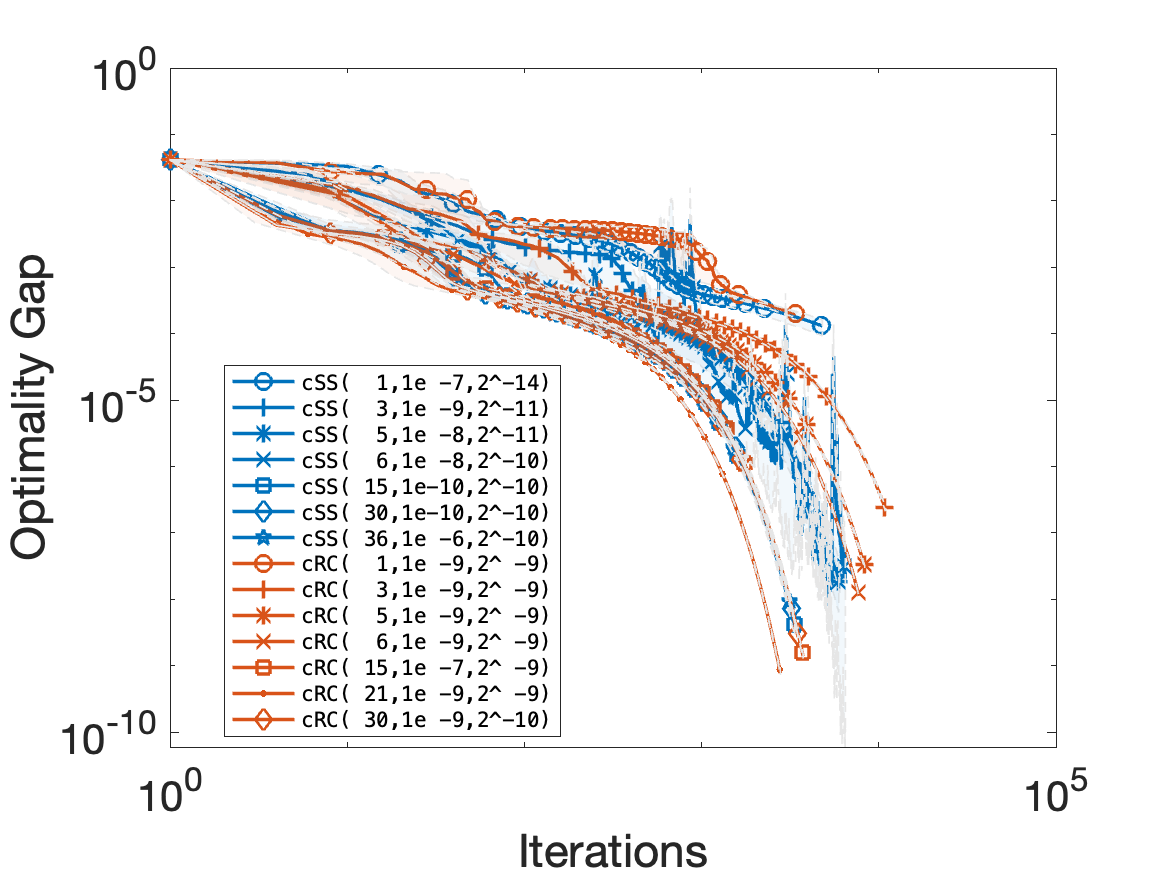}
  \caption{Comparison of cSS and cRC}
\end{subfigure}
\caption{\small{Performance of different gradient estimation methods using the tuned hyperparameters on the Chebyquad function with relative error and $ \sigma = 10^{-3} $.}}
\end{figure}

\begin{figure}[H]
\centering
\begin{subfigure}{0.33\textwidth}
  \centering
  \includegraphics[width=1.1\linewidth]{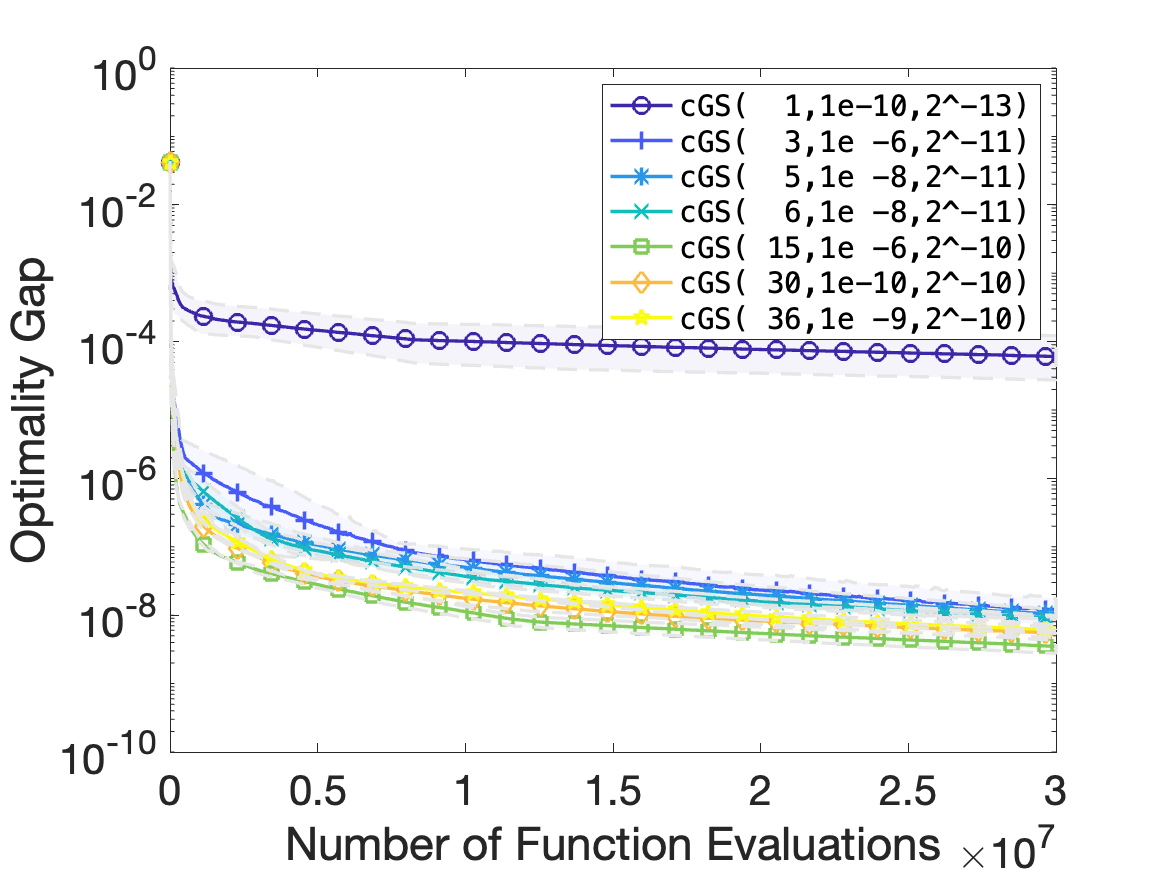}
  \caption{Performance of cGS}
\end{subfigure}%
\begin{subfigure}{0.33\textwidth}
  \centering
  \includegraphics[width=1.1\linewidth]{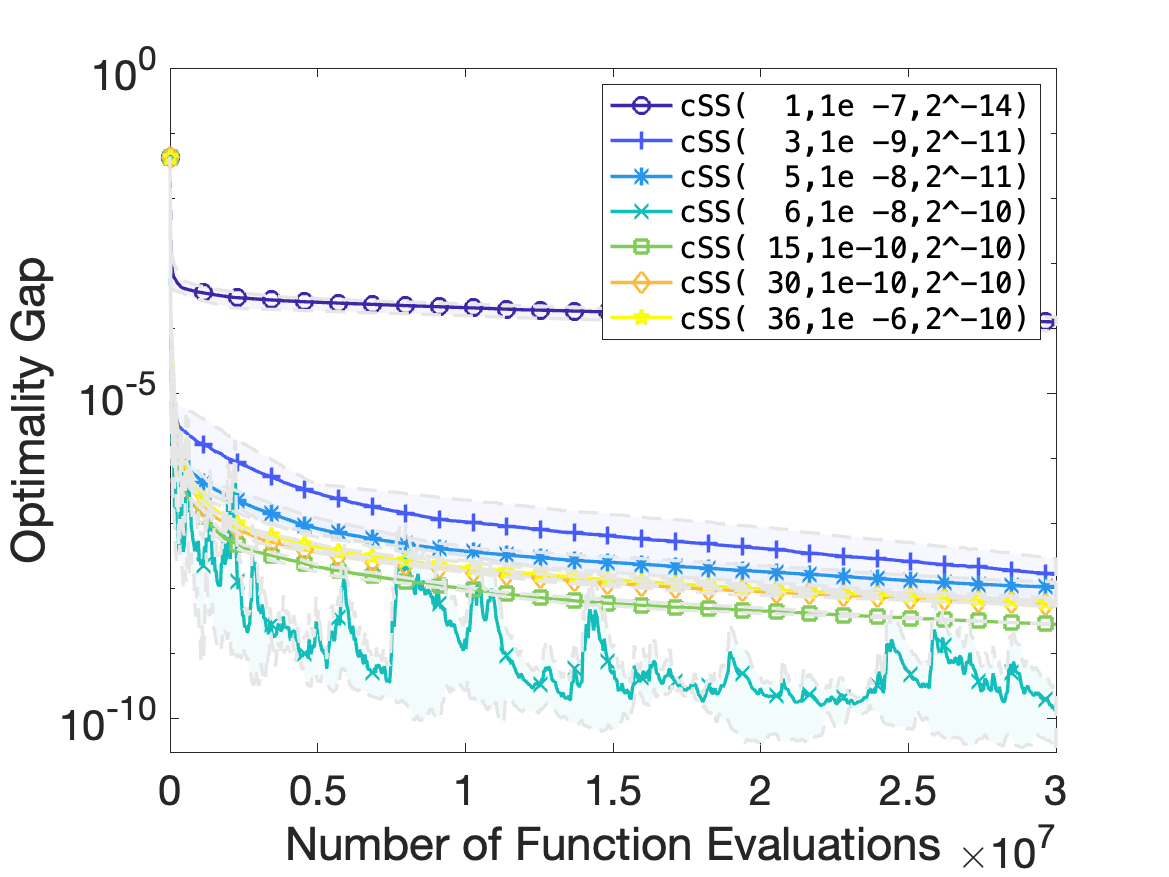}
  \caption{Performance of cSS}
\end{subfigure}%
\begin{subfigure}{0.33\textwidth}
  \centering
  \includegraphics[width=1.1\linewidth]{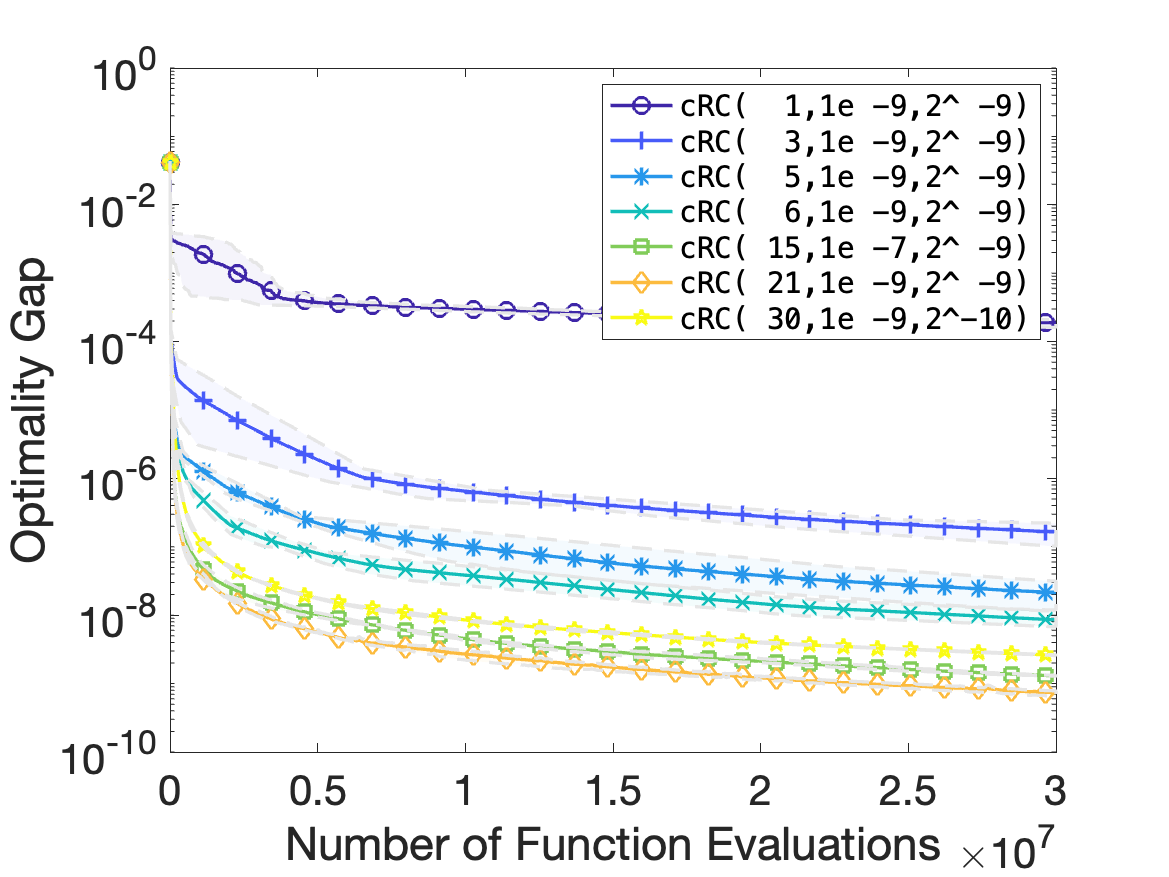}
  \caption{Performance of cRC}
\end{subfigure}
\caption{The effect of number of directions $N$ on the performance of different randomized gradient estimation methods on the Chebyquad function with relative error and $\sigma = 10^{-3}$. The sampling radius $\nu$ and step size $\alpha$ are tuned for each method and $N$ combination to achieve the best performance.}
\end{figure}

\newpage
\paragraph{Chebyquad Function $(d = 30, p = 45)$ with Absolute Error, $\sigma = 10^{-3}$}

\begin{figure}[H]
\centering
\begin{subfigure}{0.33\textwidth}
  \centering
  \includegraphics[width=1.1\linewidth]{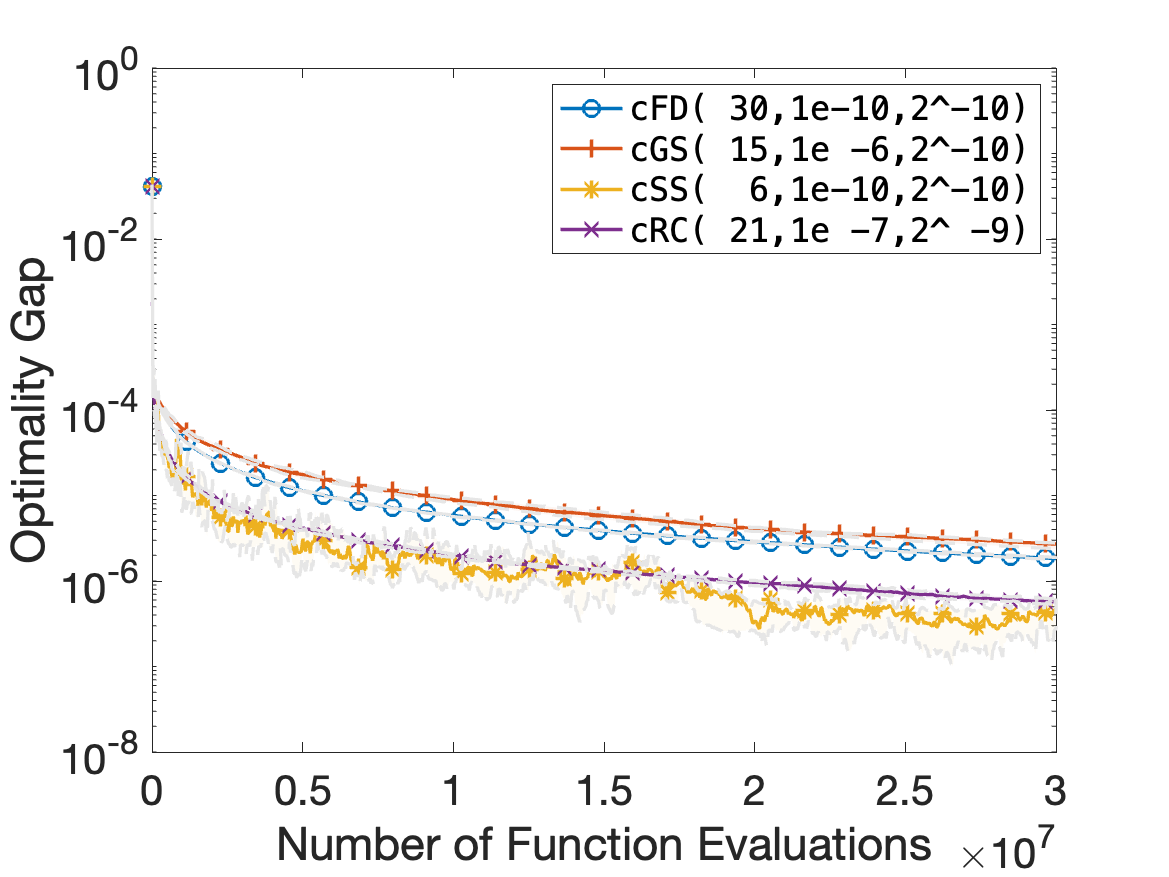}
  \caption{Optimality Gap}
\end{subfigure}%
\begin{subfigure}{0.33\textwidth}
  \centering
  \includegraphics[width=1.1\linewidth]{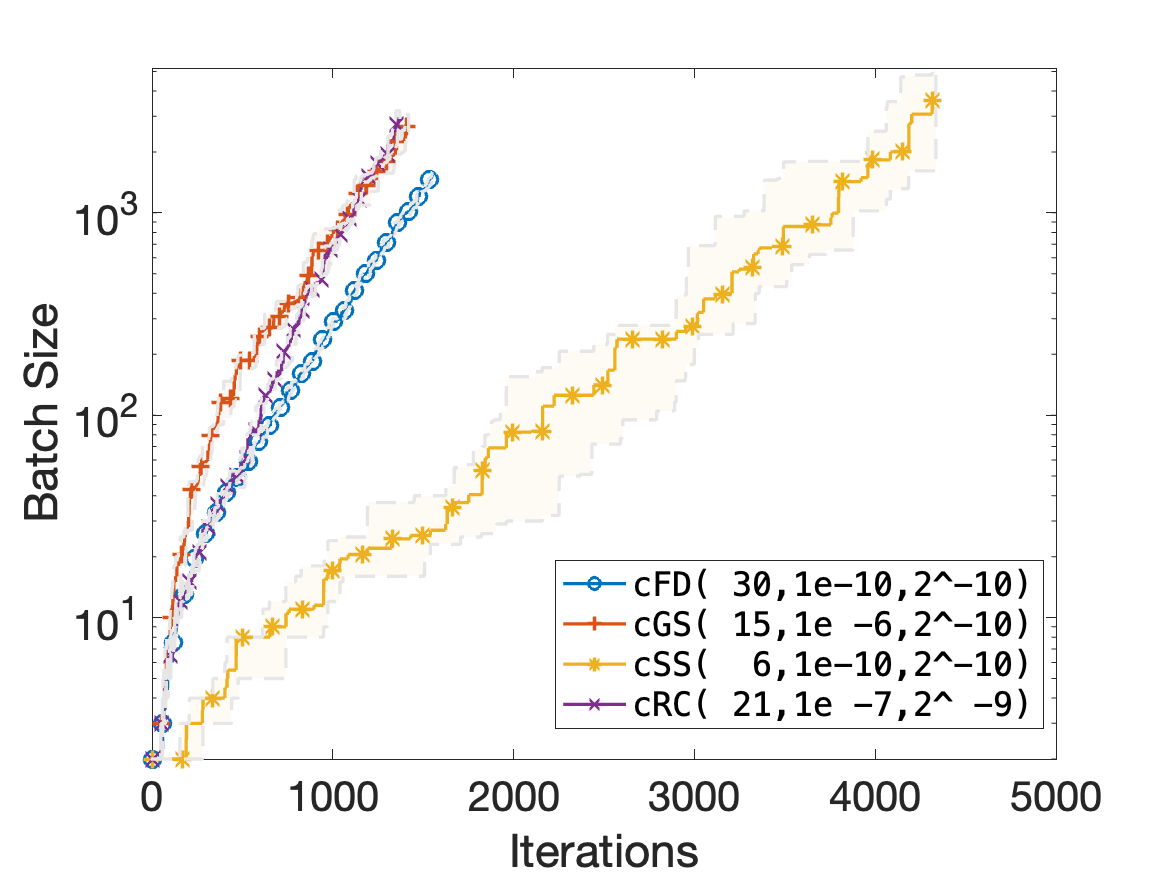}
  \caption{Batch Size}
\end{subfigure}
\begin{subfigure}{0.33\textwidth}
  \centering
  \includegraphics[width=1.1\linewidth]{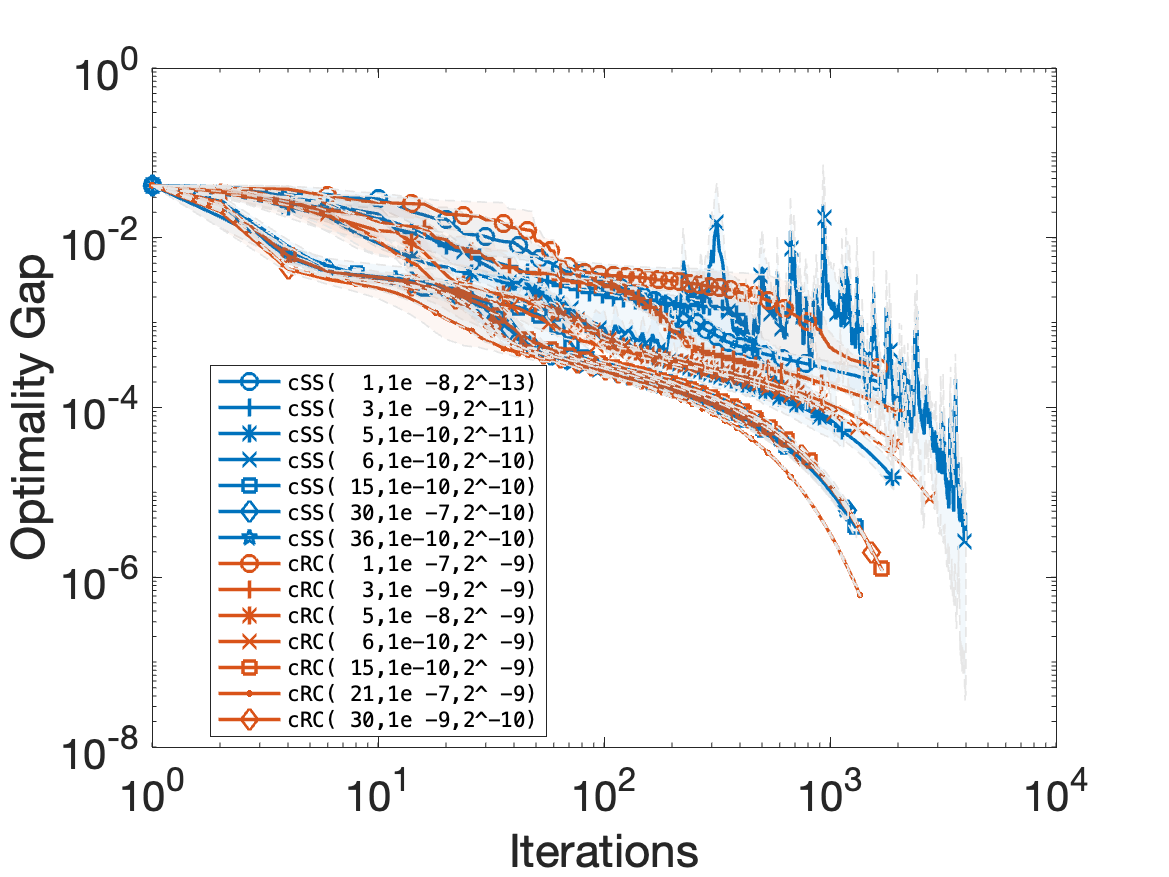}
  \caption{Comparison of cSS and cRC}
\end{subfigure}
\caption{Performance of different gradient estimation methods using the tuned hyperparameters on the Chebyquad function with absolute error and $ \sigma = 10^{-3} $.}
\end{figure}

\begin{figure}[H]
\centering
\begin{subfigure}{0.33\textwidth}
  \centering
  \includegraphics[width=1.1\linewidth]{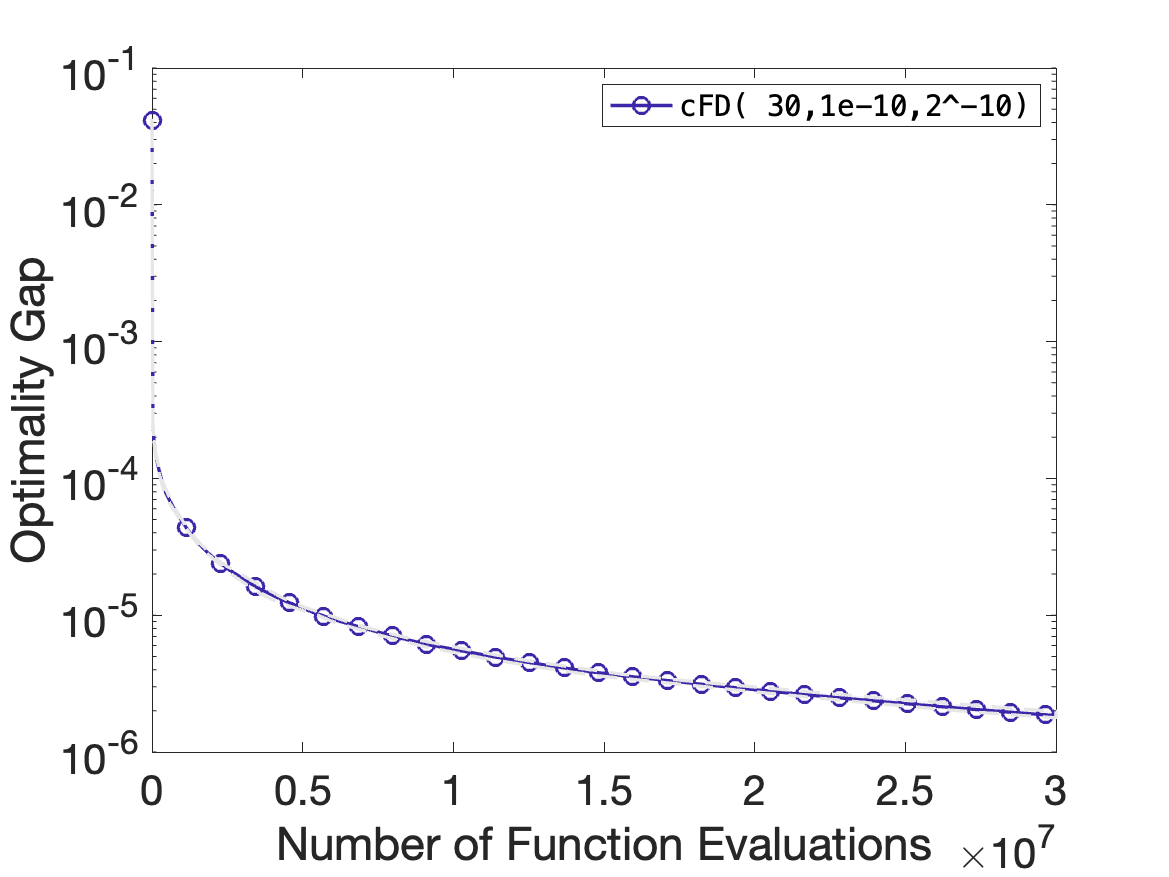}
  \caption{Performance of cGS}
\end{subfigure}%
\begin{subfigure}{0.33\textwidth}
  \centering
  \includegraphics[width=1.1\linewidth]{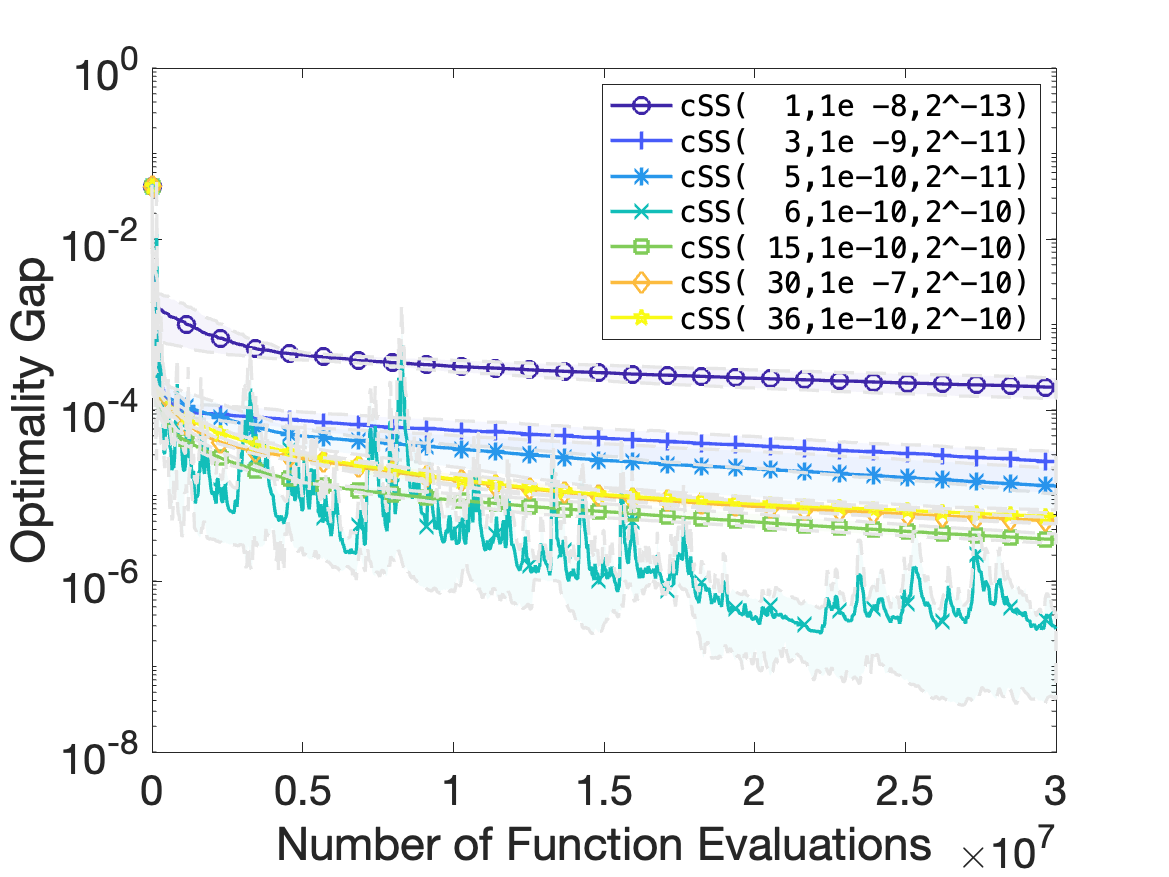}
  \caption{Performance of cSS}
\end{subfigure}
\begin{subfigure}{0.33\textwidth}
  \centering
  \includegraphics[width=1.1\linewidth]{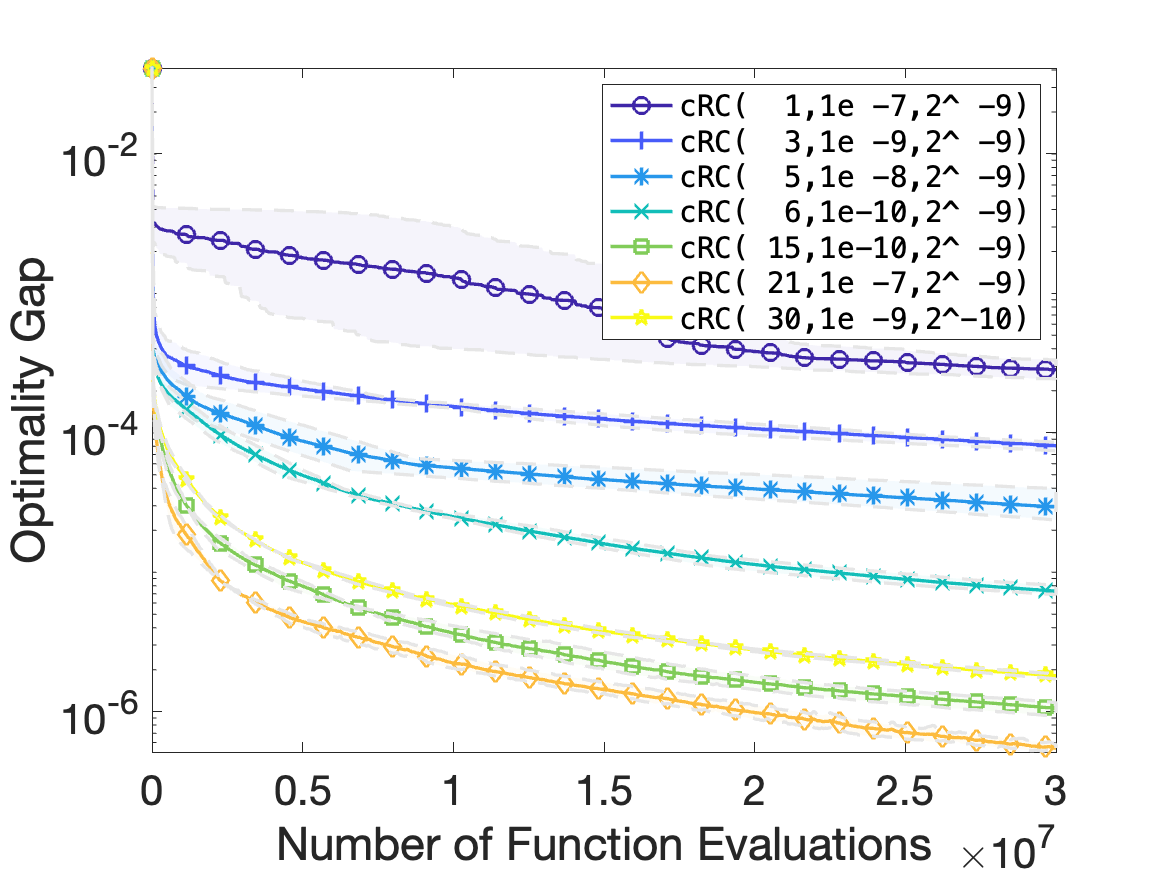}
  \caption{Performance of cRC}
\end{subfigure}
\caption{The effect of number of directions $N$ on the performance of different randomized gradient estimation methods on the Chebyquad function with absolute error and $ \sigma = 10^{-3} $. The sampling radius $\nu$ and step size $\alpha$ are tuned for each method and $N$ combination to achieve the best performance.}
\end{figure}

\newpage
\paragraph{Chebyquad Function $(d = 30, p = 45)$ with Relative Error, $\sigma = 10^{-5}$}

\begin{figure}[H]
\centering
\begin{subfigure}{0.33\textwidth}
  \centering
  \includegraphics[width=1.1\linewidth]{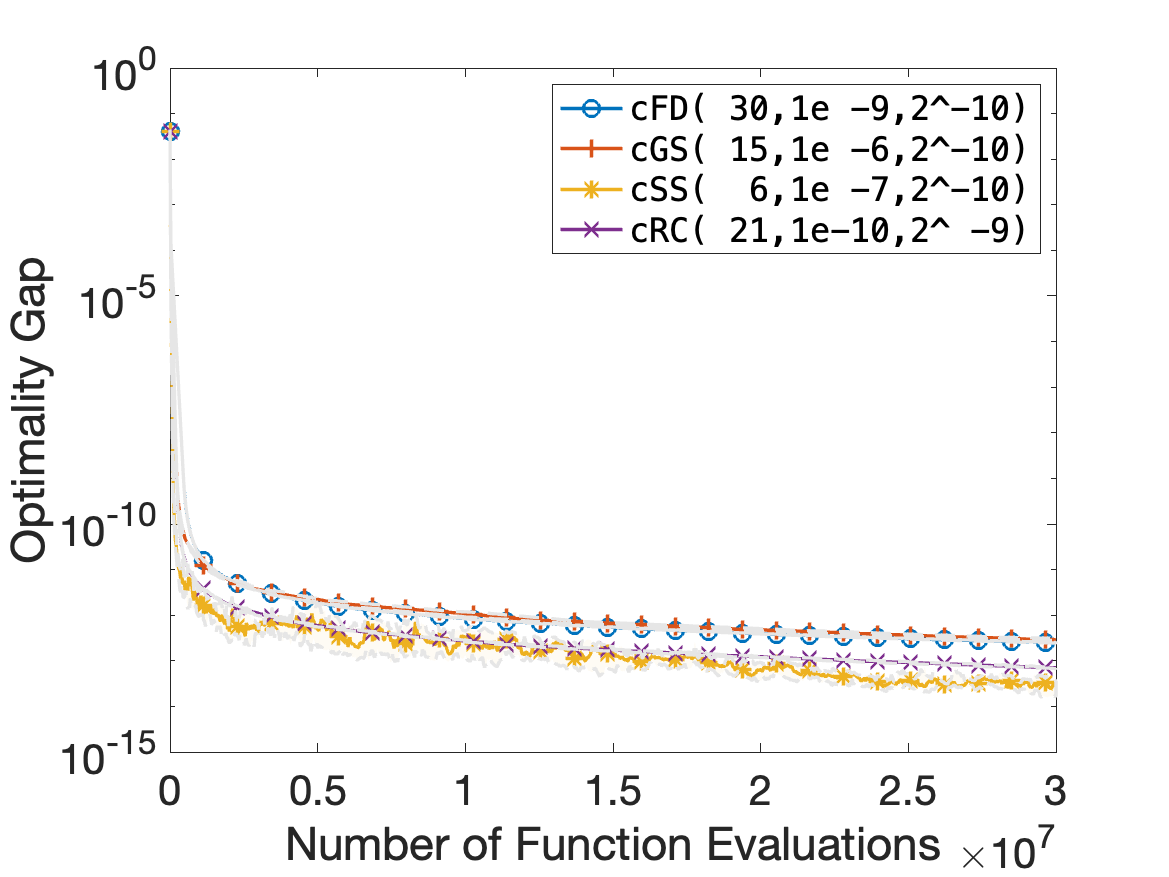}
  \caption{Optimality Gap}
\end{subfigure}%
\begin{subfigure}{0.33\textwidth}
  \centering
  \includegraphics[width=1.1\linewidth]{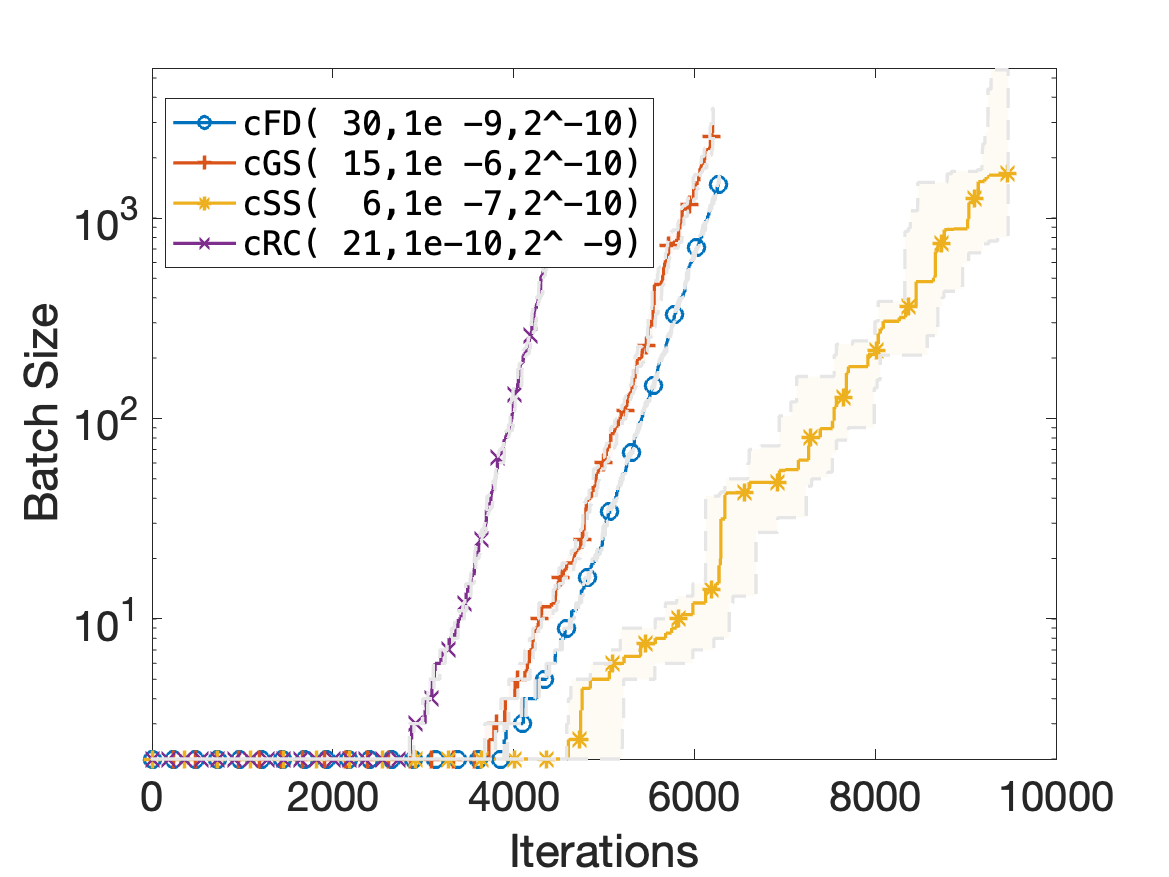}
  \caption{Batch Size}
\end{subfigure}
\begin{subfigure}{0.33\textwidth}
  \centering
  \includegraphics[width=1.1\linewidth]{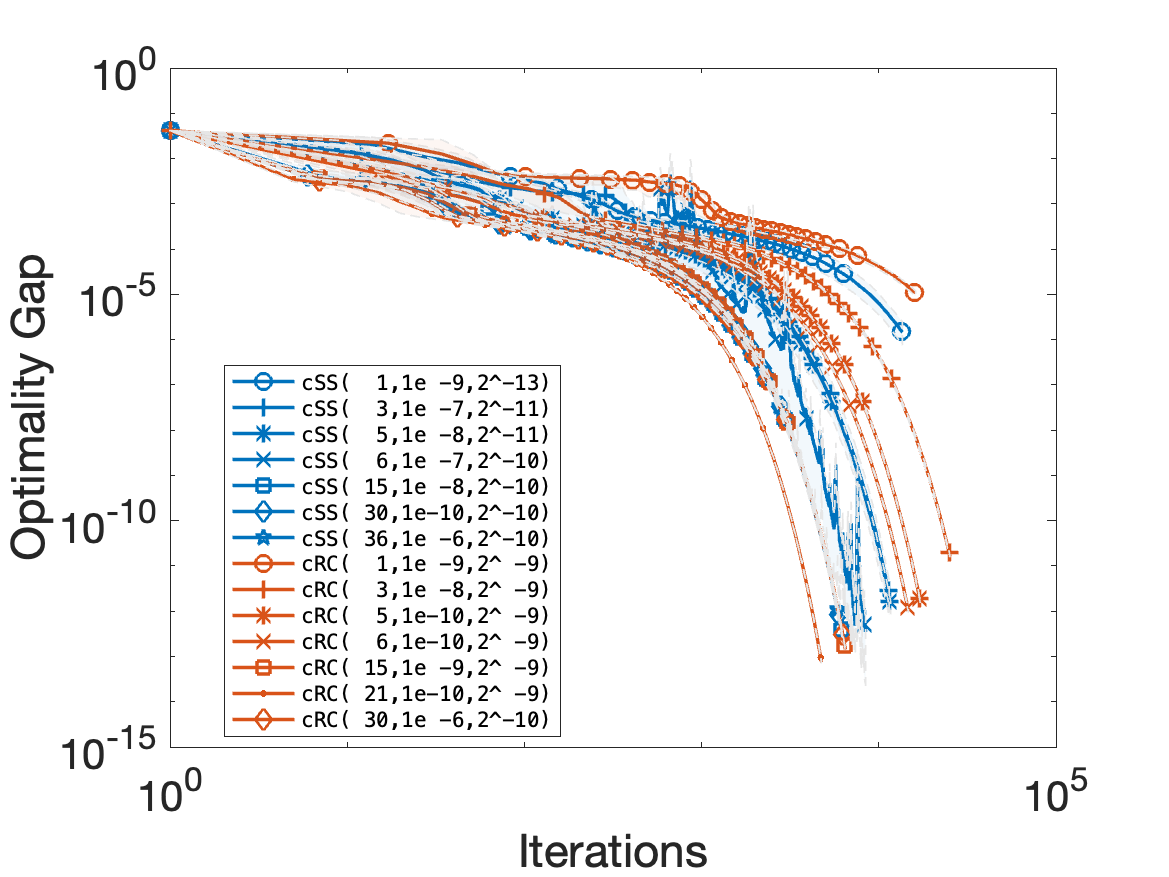}
  \caption{Comparison of cSS and cRC}
\end{subfigure}
\caption{\small{Performance of different gradient estimation methods using the tuned hyperparameters on the Chebyquad function with relative error and $ \sigma = 10^{-5} $.}}
\end{figure}

\begin{figure}[H]
\centering
\begin{subfigure}{0.33\textwidth}
  \centering
  \includegraphics[width=1.1\linewidth]{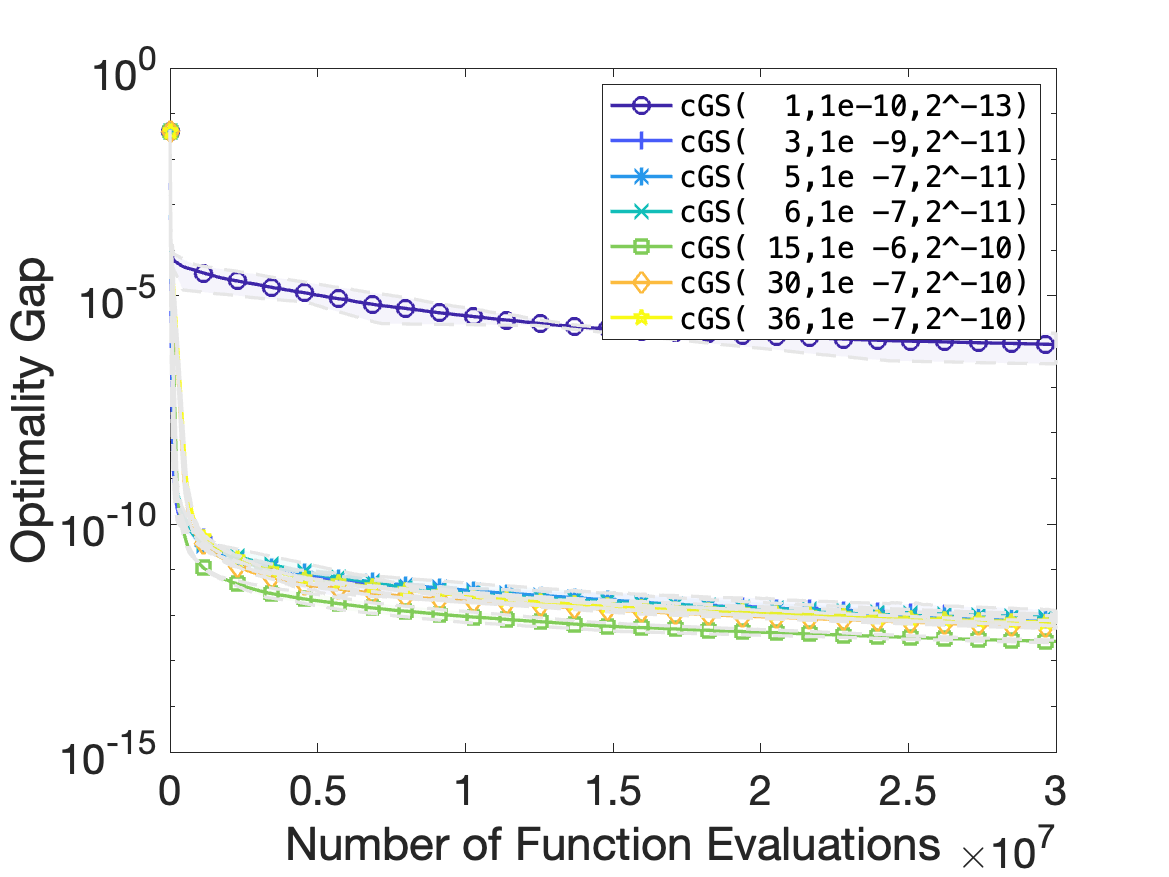}
  \caption{Performance of cGS}
\end{subfigure}%
\begin{subfigure}{0.33\textwidth}
  \centering
  \includegraphics[width=1.1\linewidth]{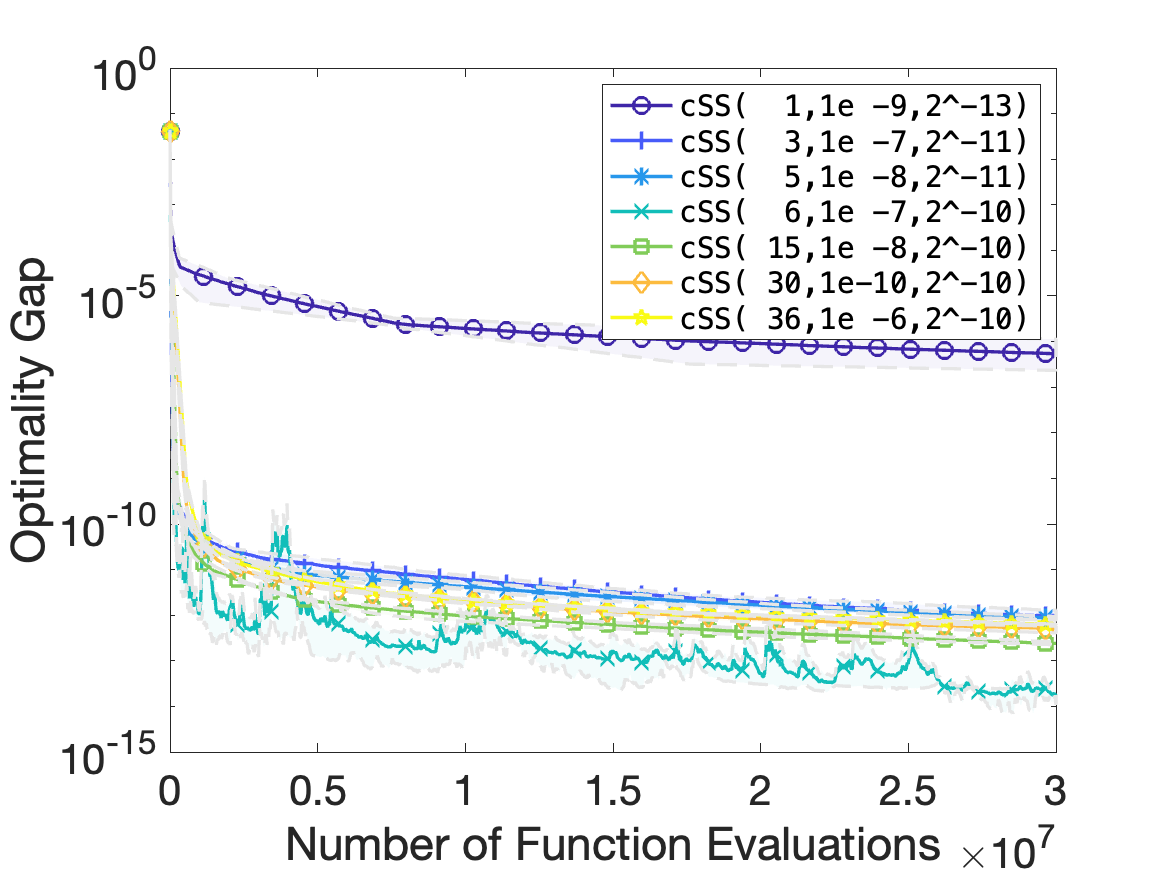}
  \caption{Performance of cSS}
\end{subfigure}%
\begin{subfigure}{0.33\textwidth}
  \centering
  \includegraphics[width=1.1\linewidth]{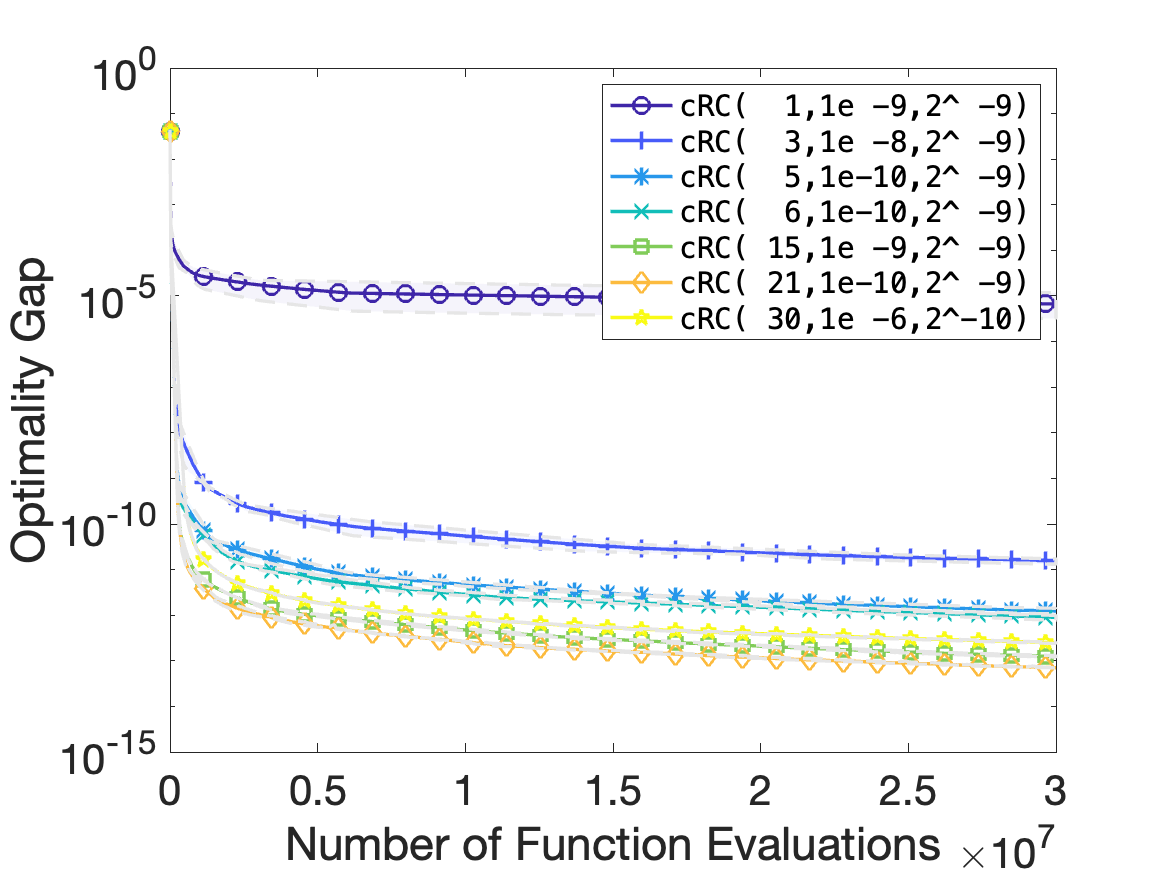}
  \caption{Performance of cRC}
\end{subfigure}
\caption{The effect of number of directions $N$ on the performance of different randomized gradient estimation methods on the Chebyquad function with relative error and $ \sigma = 10^{-5} $. The sampling radius $\nu$ and step size $\alpha$ are tuned for each method and $N$ combination to achieve the best performance.}
\end{figure}

\newpage
\paragraph{Chebyquad Function $(d = 30, p = 45)$ with Absolute Error, $\sigma = 10^{-5}$}

\begin{figure}[H]
\centering
\begin{subfigure}{0.33\textwidth}
  \centering
  \includegraphics[width=1.1\linewidth]{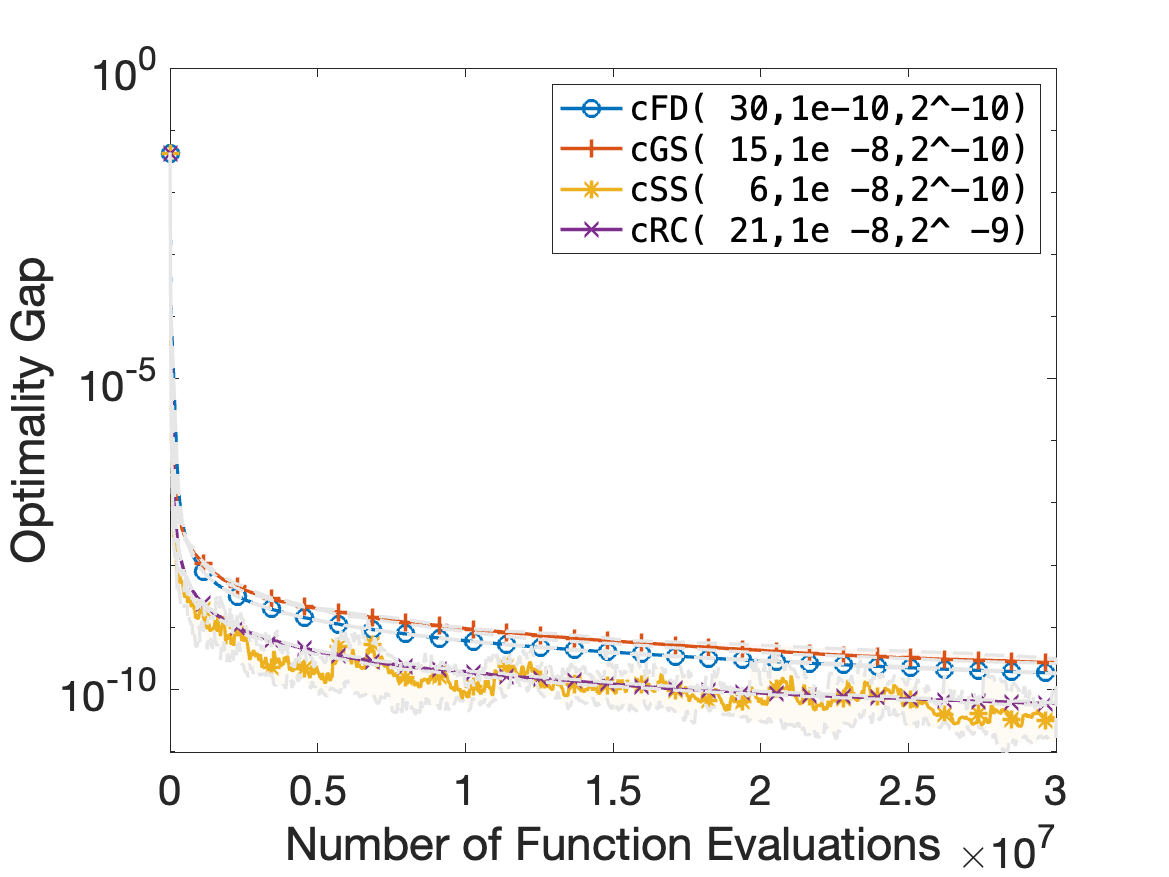}
  \caption{Optimality Gap}
\end{subfigure}%
\begin{subfigure}{0.33\textwidth}
  \centering
  \includegraphics[width=1.1\linewidth]{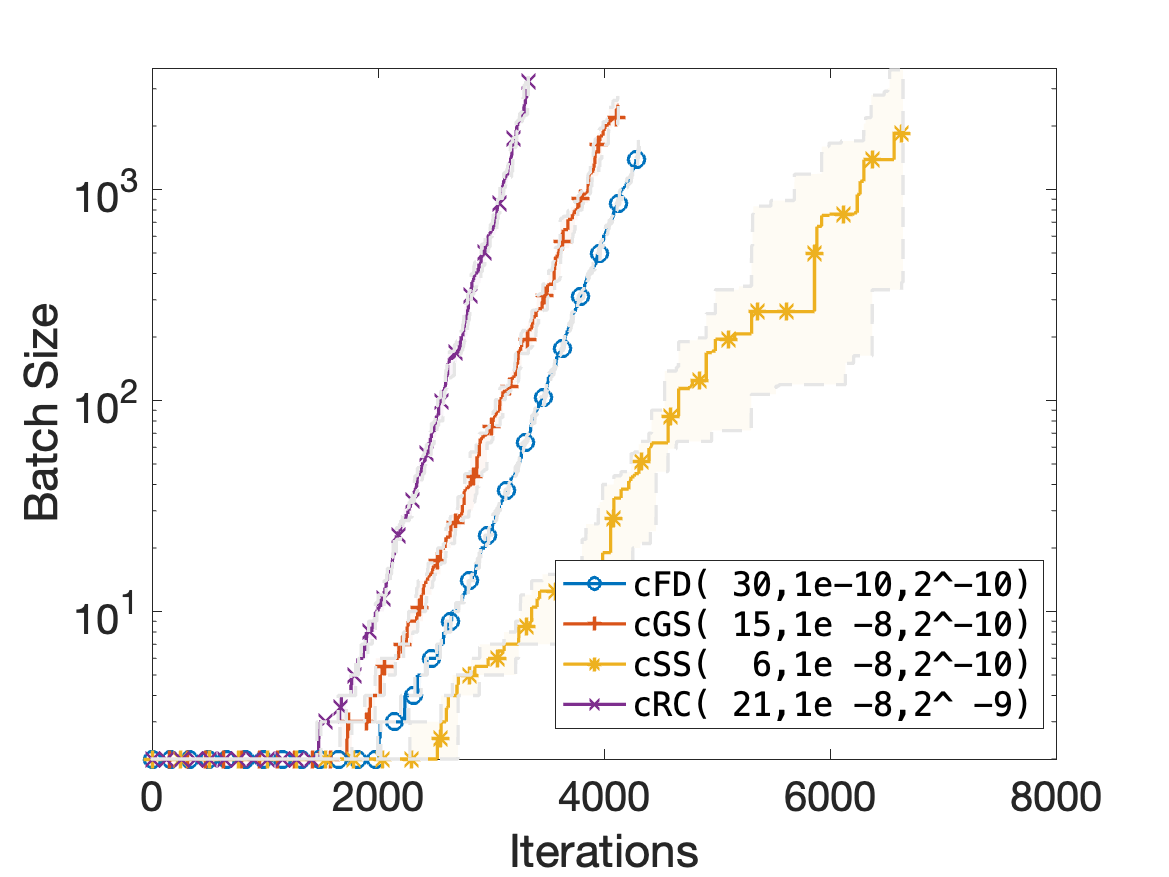}
  \caption{Batch Size}
\end{subfigure}
\begin{subfigure}{0.33\textwidth}
  \centering
  \includegraphics[width=1.1\linewidth]{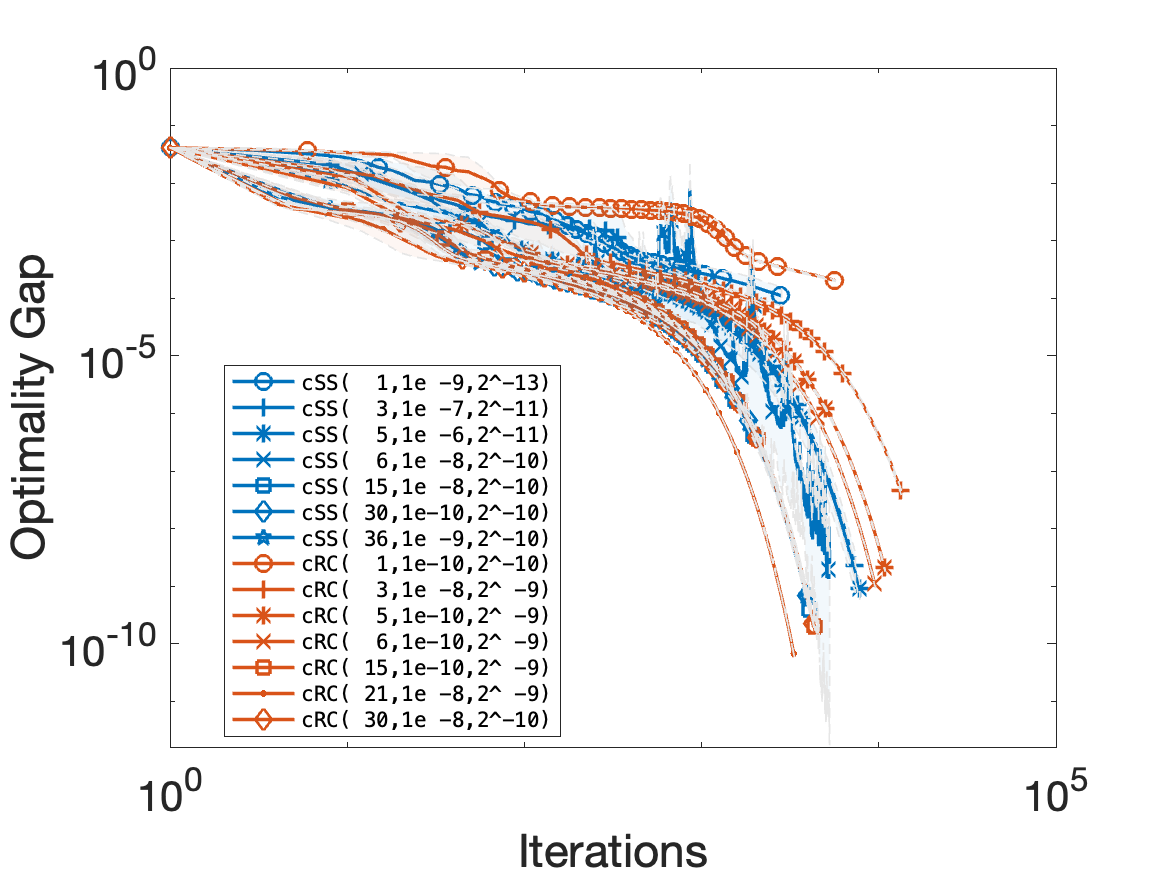}
  \caption{Comparison of cSS and cRC}
\end{subfigure}
\caption{Performance of different gradient estimation methods using the tuned hyperparameters on the Chebyquad function with absolute error and $ \sigma = 10^{-5} $.}
\end{figure}

\begin{figure}[H]
\centering
\begin{subfigure}{0.33\textwidth}
  \centering
  \includegraphics[width=1.1\linewidth]{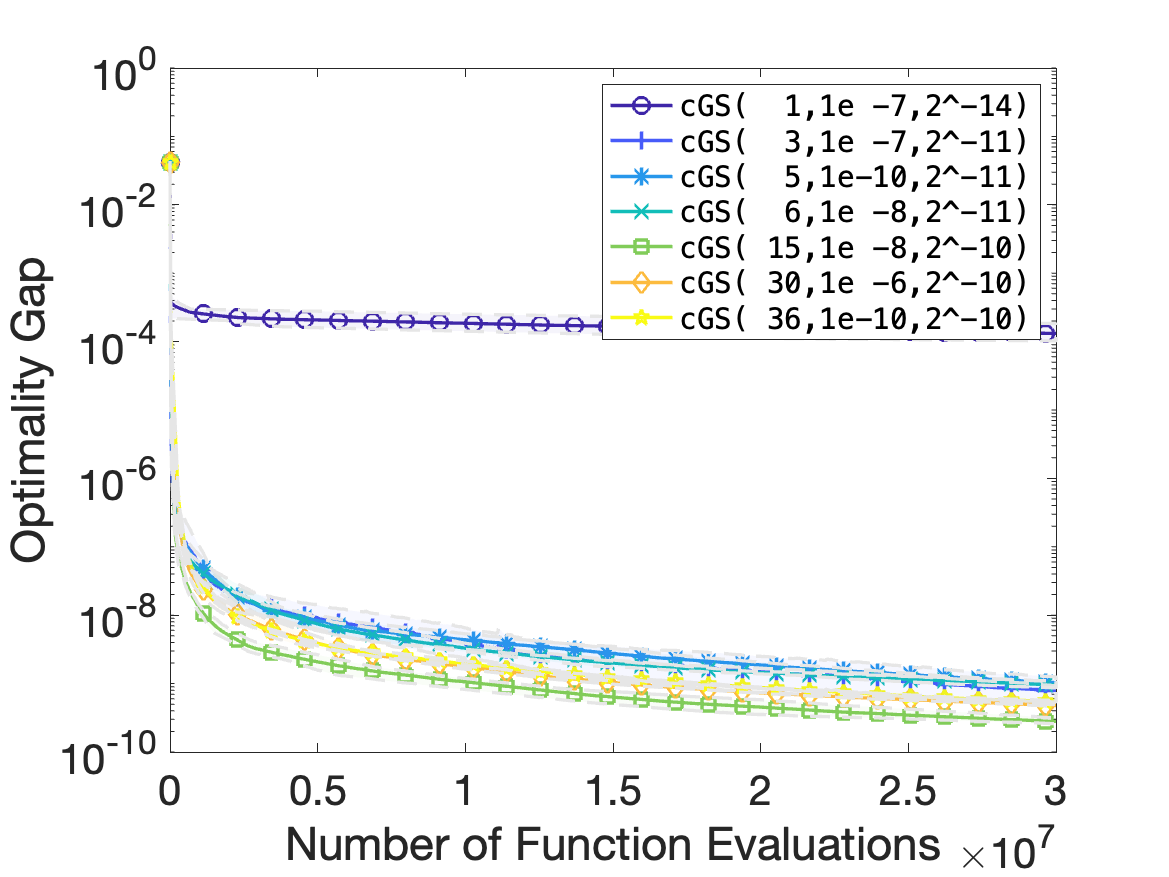}
  \caption{Performance of cGS}
\end{subfigure}%
\begin{subfigure}{0.33\textwidth}
  \centering
  \includegraphics[width=1.1\linewidth]{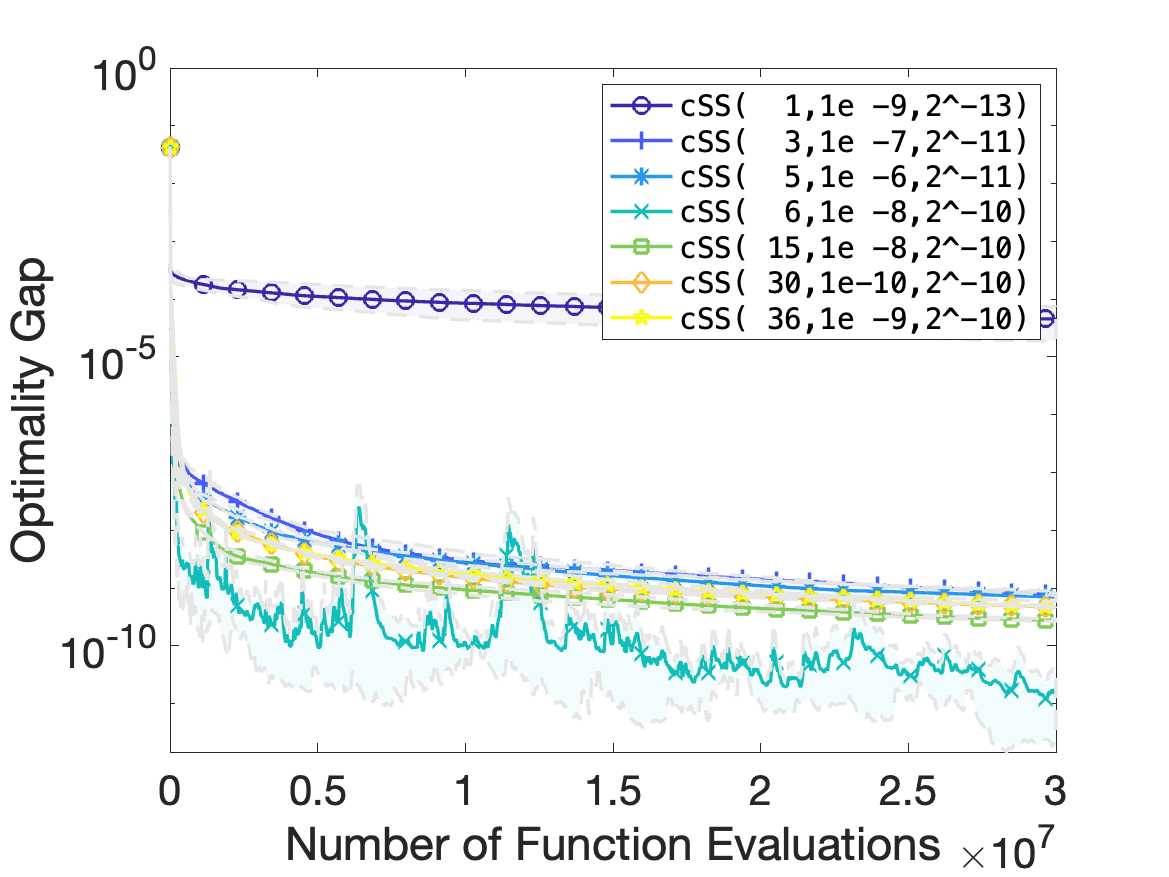}
  \caption{Performance of cSS}
\end{subfigure}
\begin{subfigure}{0.33\textwidth}
  \centering
  \includegraphics[width=1.1\linewidth]{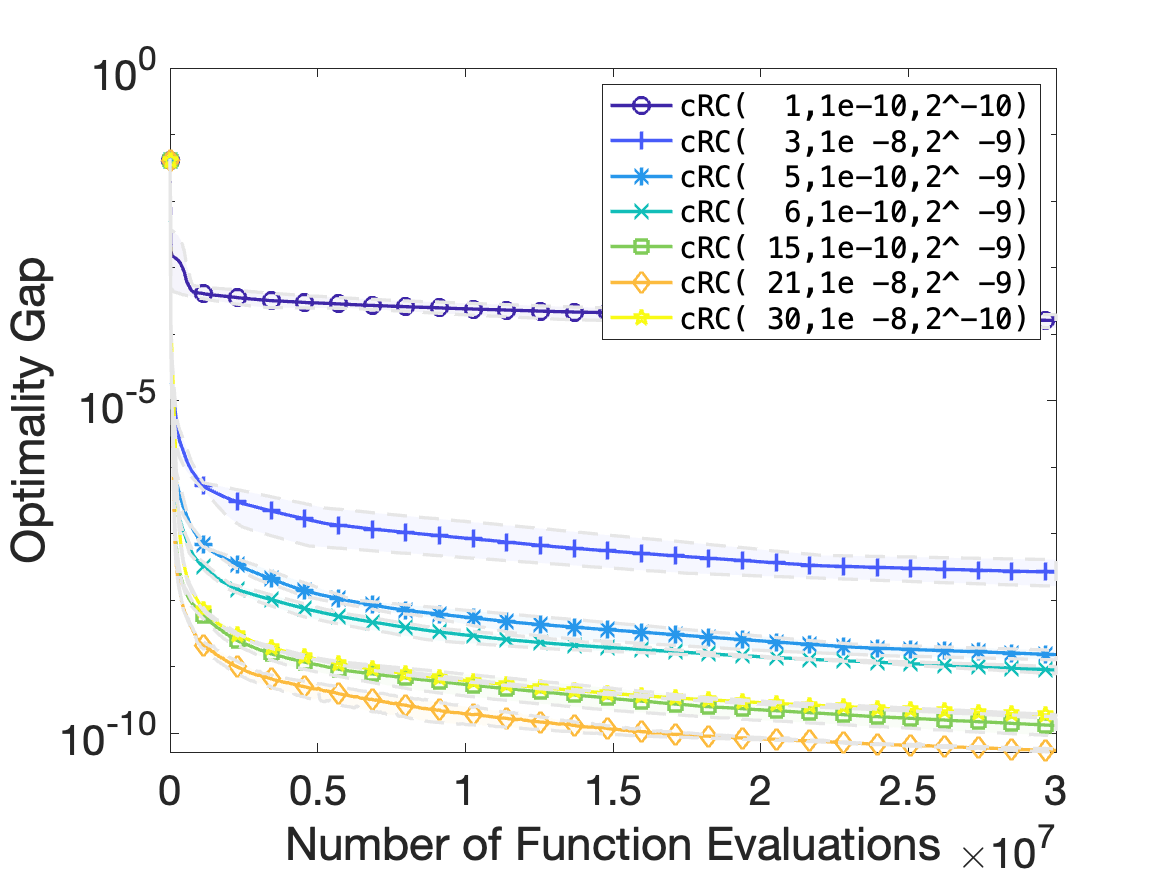}
  \caption{Performance of cRC}
\end{subfigure}
\caption{ The effect of number of directions $N$ on the performance of different randomized gradient estimation methods on the Chebyquad function with absolute error and $ \sigma = 10^{-5} $. The sampling radius $\nu$ and step size $\alpha$ are tuned for each method and $N$ combination to achieve the best performance.}
\end{figure}

\newpage
\paragraph{Osborne Function $(d = 11, p = 65)$ with Relative Error, $\sigma = 10^{-3}$}

\begin{figure}[H]
\centering
\begin{subfigure}{0.33\textwidth}
  \centering
  \includegraphics[width=1.1\linewidth]{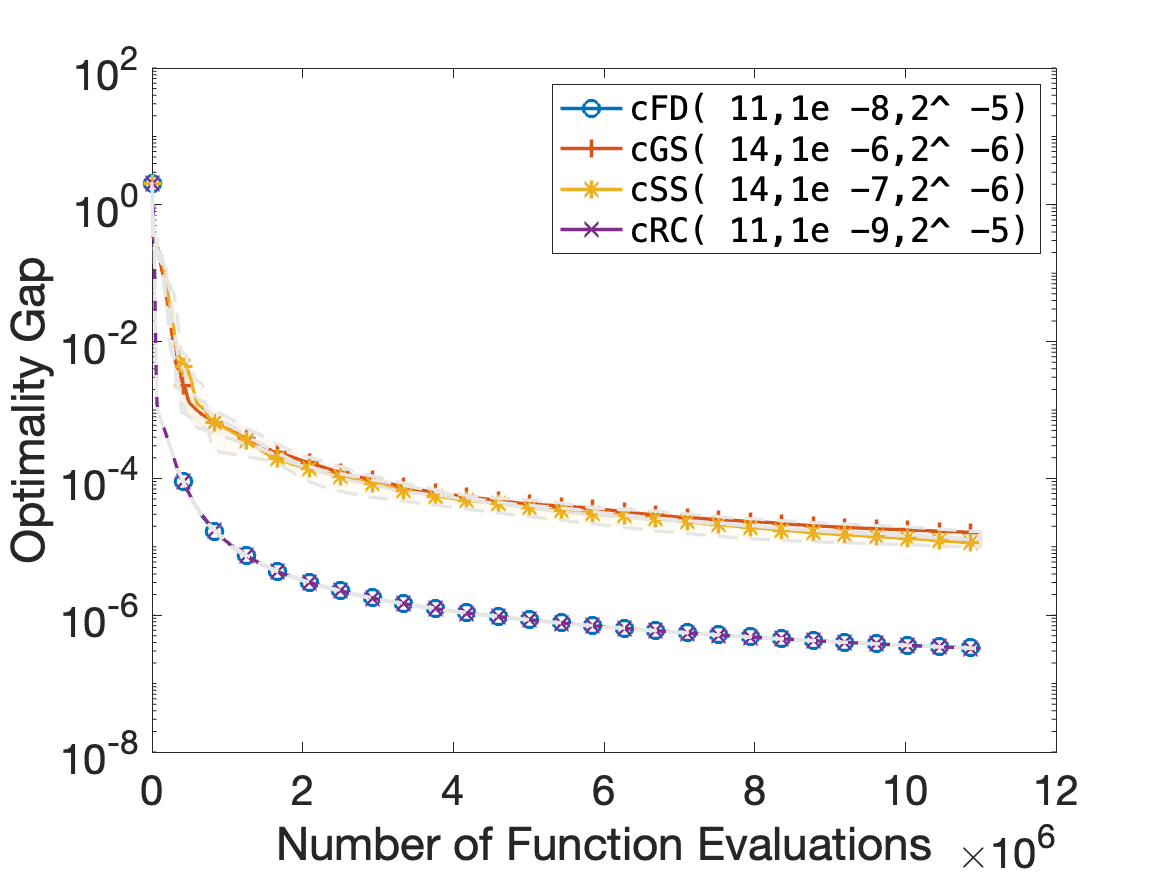}
  \caption{Optimality Gap}
\end{subfigure}%
\begin{subfigure}{0.33\textwidth}
  \centering
  \includegraphics[width=1.1\linewidth]{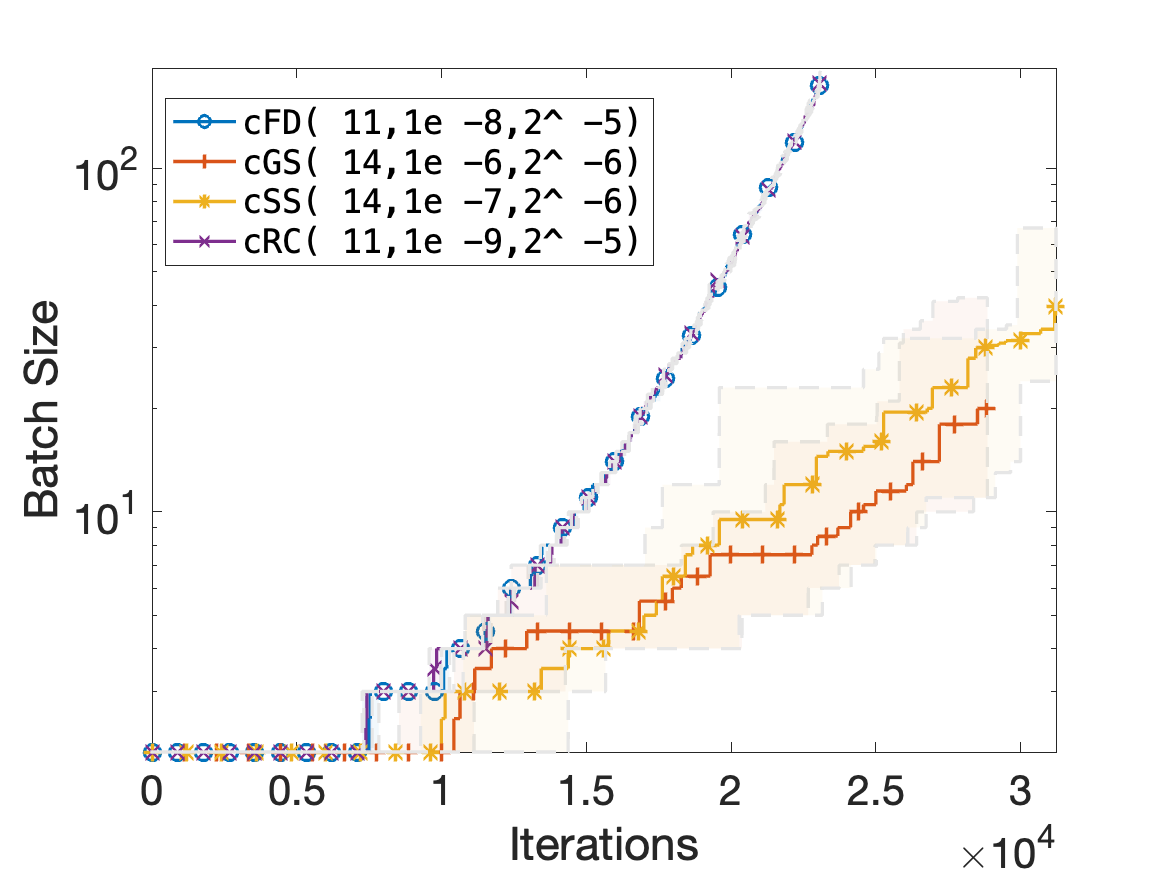}
  \caption{Batch Size}
\end{subfigure}
\begin{subfigure}{0.33\textwidth}
  \centering
  \includegraphics[width=1.1\linewidth]{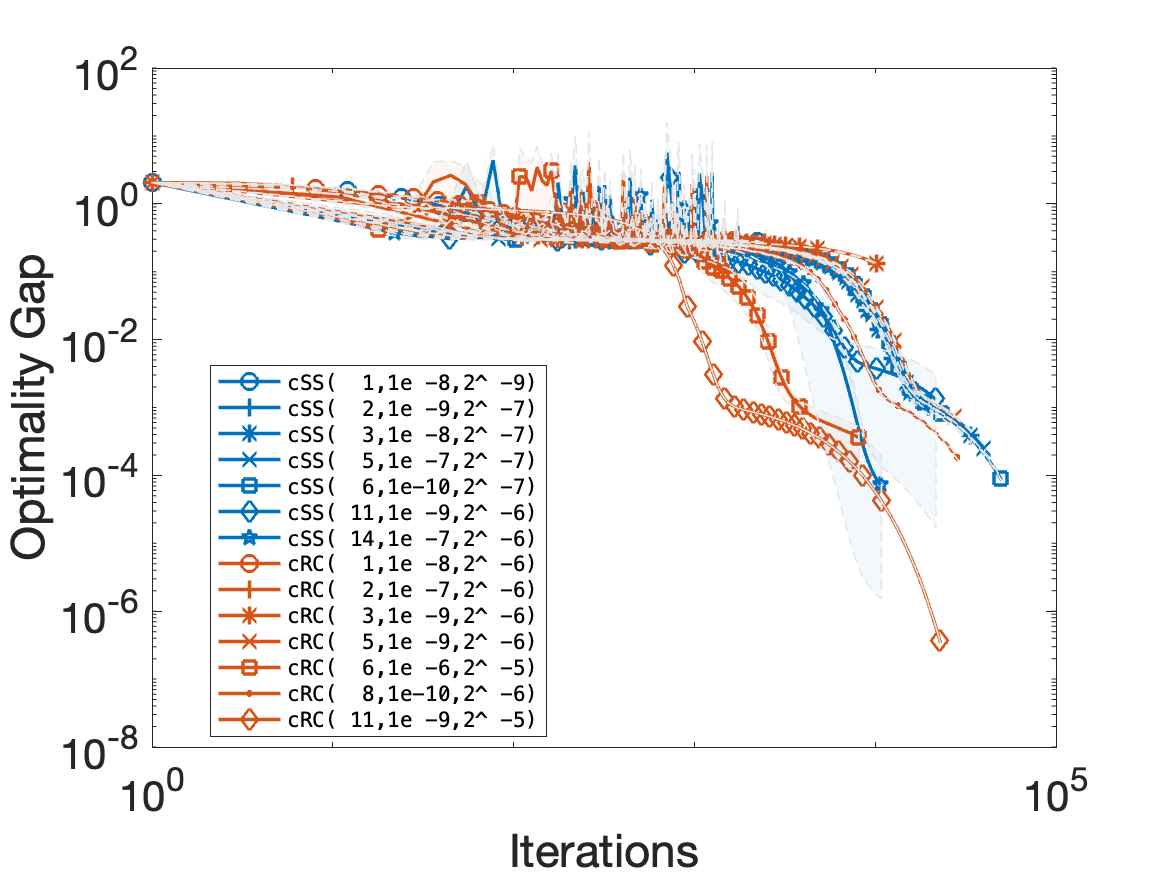}
  \caption{Comparison of cSS and cRC}
\end{subfigure}
\caption{Performance of different gradient estimation methods using the tuned hyperparameters on the Osborne function with relative error and $ \sigma = 10^{-3} $.}
\end{figure}

\begin{figure}[H]
\centering
\begin{subfigure}{0.33\textwidth}
  \centering
  \includegraphics[width=1.1\linewidth]{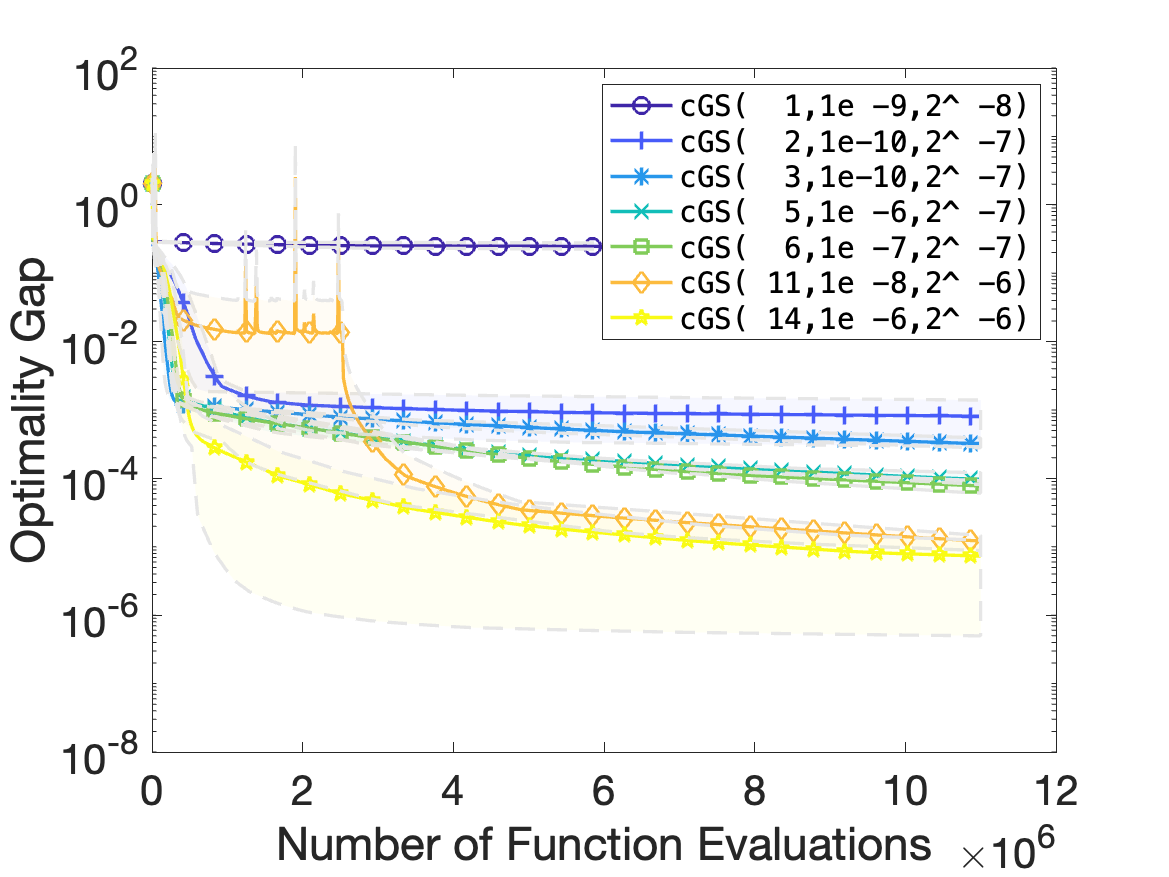}
  \caption{Performance of cGS}
\end{subfigure}%
\begin{subfigure}{0.33\textwidth}
  \centering
  \includegraphics[width=1.1\linewidth]{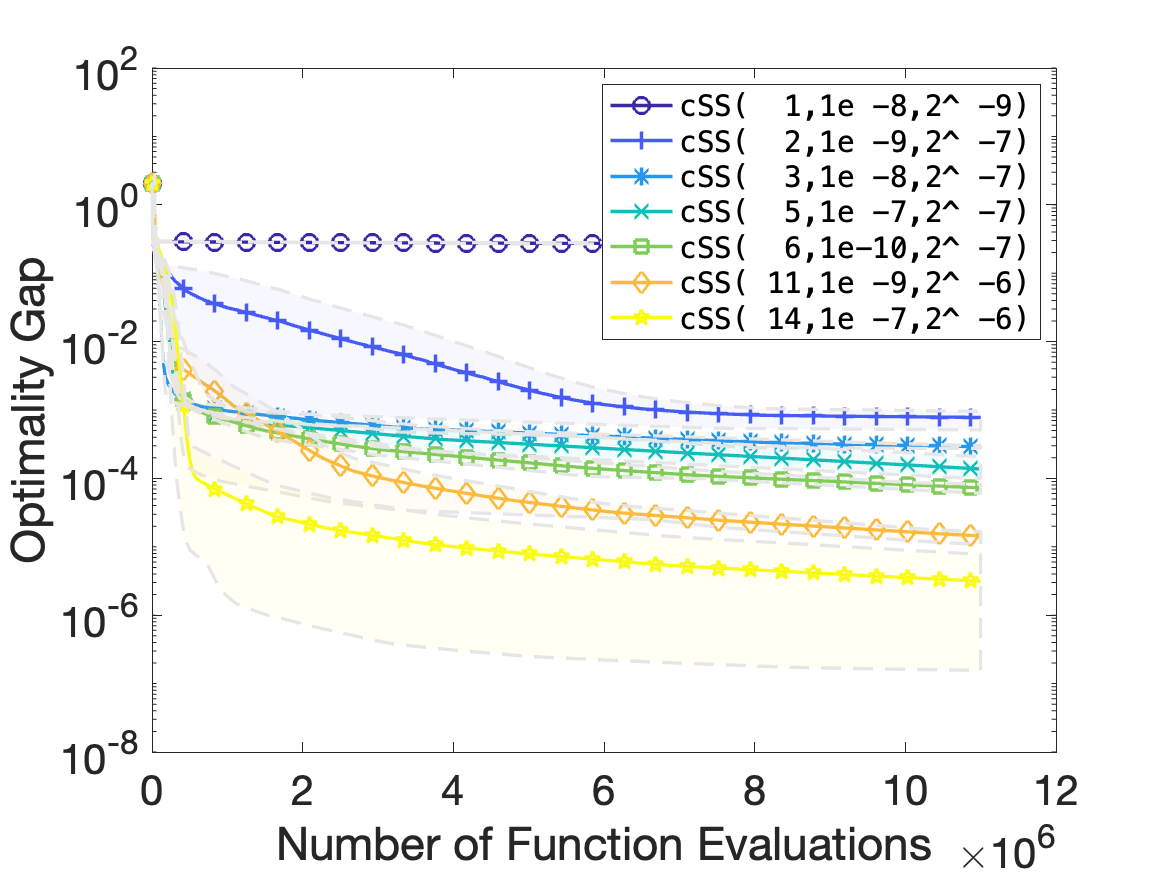}
  \caption{Performance of cSS}
\end{subfigure}%
\begin{subfigure}{0.33\textwidth}
  \centering
  \includegraphics[width=1.1\linewidth]{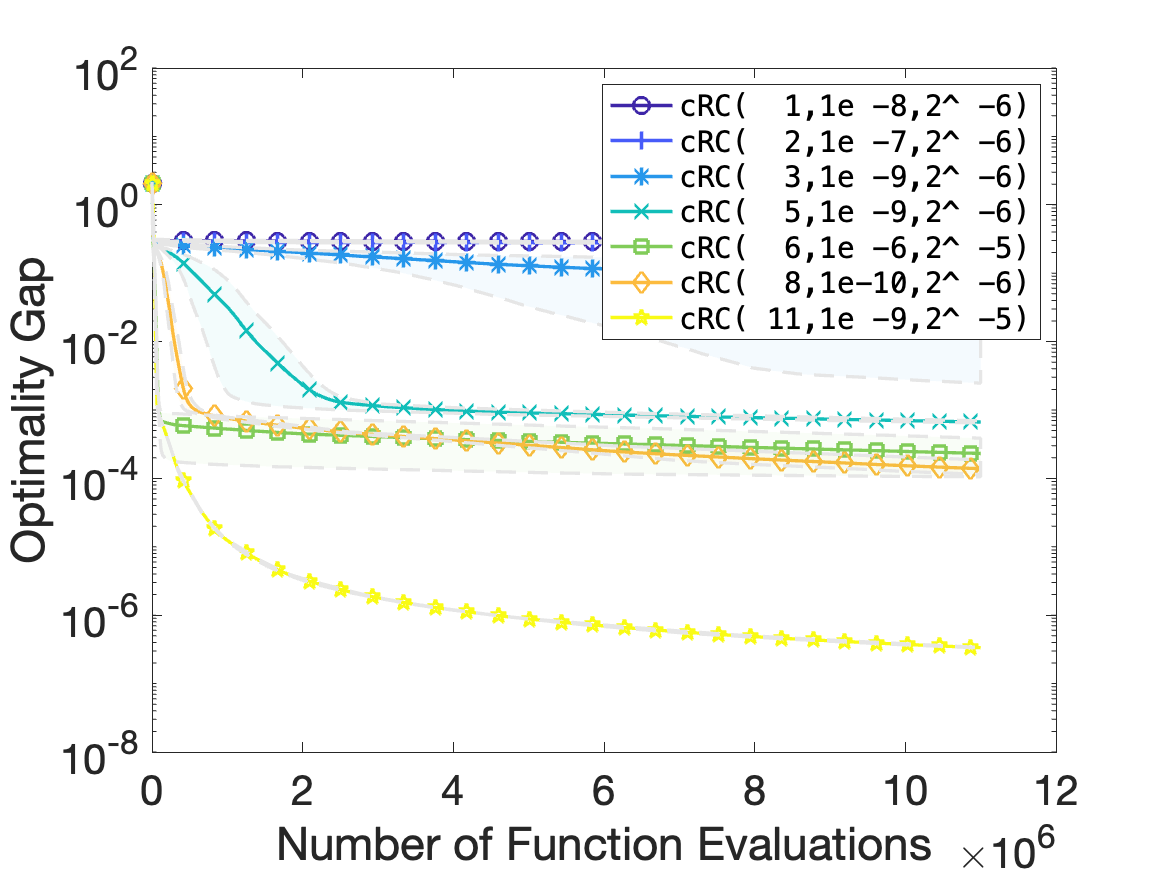}
  \caption{Performance of cRC}
\end{subfigure}
\caption{The effect of number of directions $N$ on the performance of different randomized gradient estimation methods on the Osborne function with relative error and $ \sigma = 10^{-3} $. The sampling radius $\nu$ and step size $\alpha$ are tuned for each method and $N$ combination to achieve the best performance.}
\end{figure}

\newpage
\paragraph{Osborne Function $(d = 11, p = 65)$ with Absolute Error, $\sigma = 10^{-3}$}

\begin{figure}[H]
\centering
\begin{subfigure}{0.33\textwidth}
  \centering
  \includegraphics[width=1.1\linewidth]{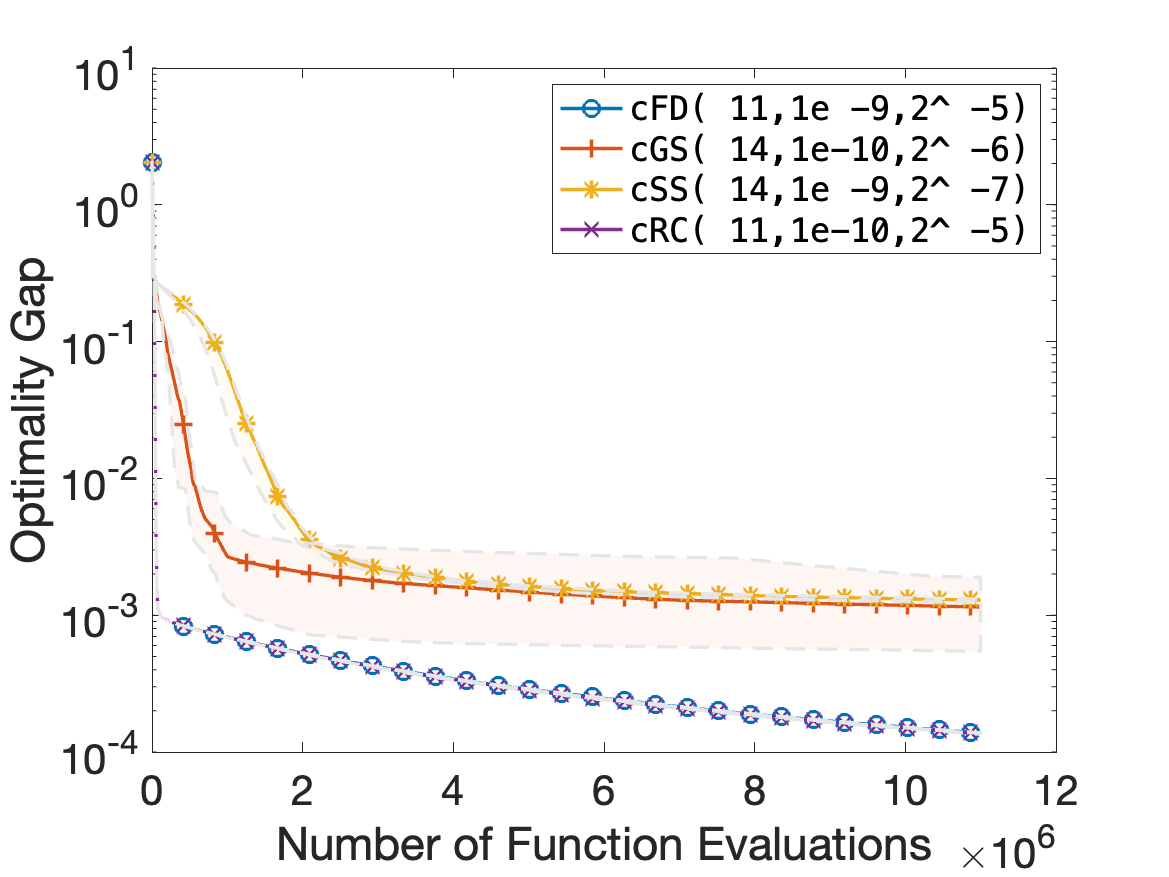}
  \caption{Optimality Gap}
\end{subfigure}%
\begin{subfigure}{0.33\textwidth}
  \centering
  \includegraphics[width=1.1\linewidth]{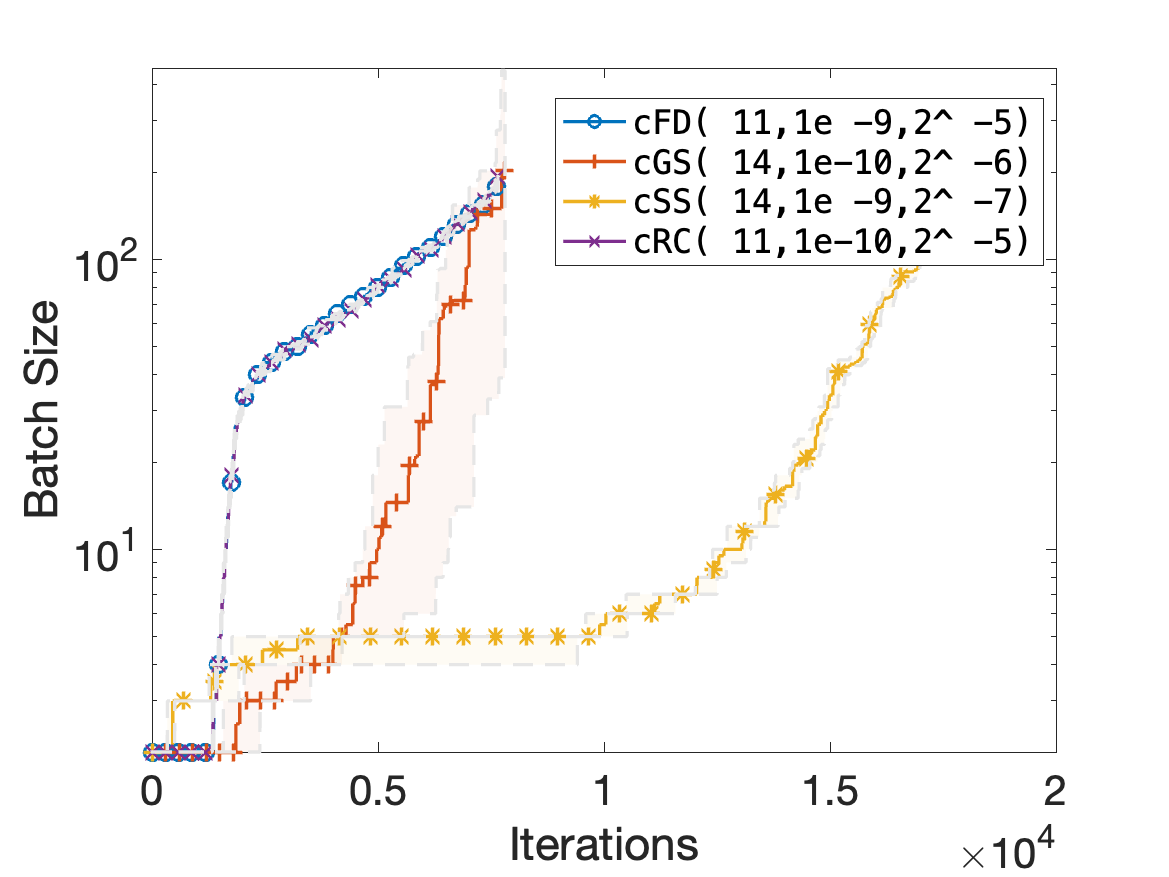}
  \caption{Batch Size}
\end{subfigure}
\begin{subfigure}{0.33\textwidth}
  \centering
  \includegraphics[width=1.1\linewidth]{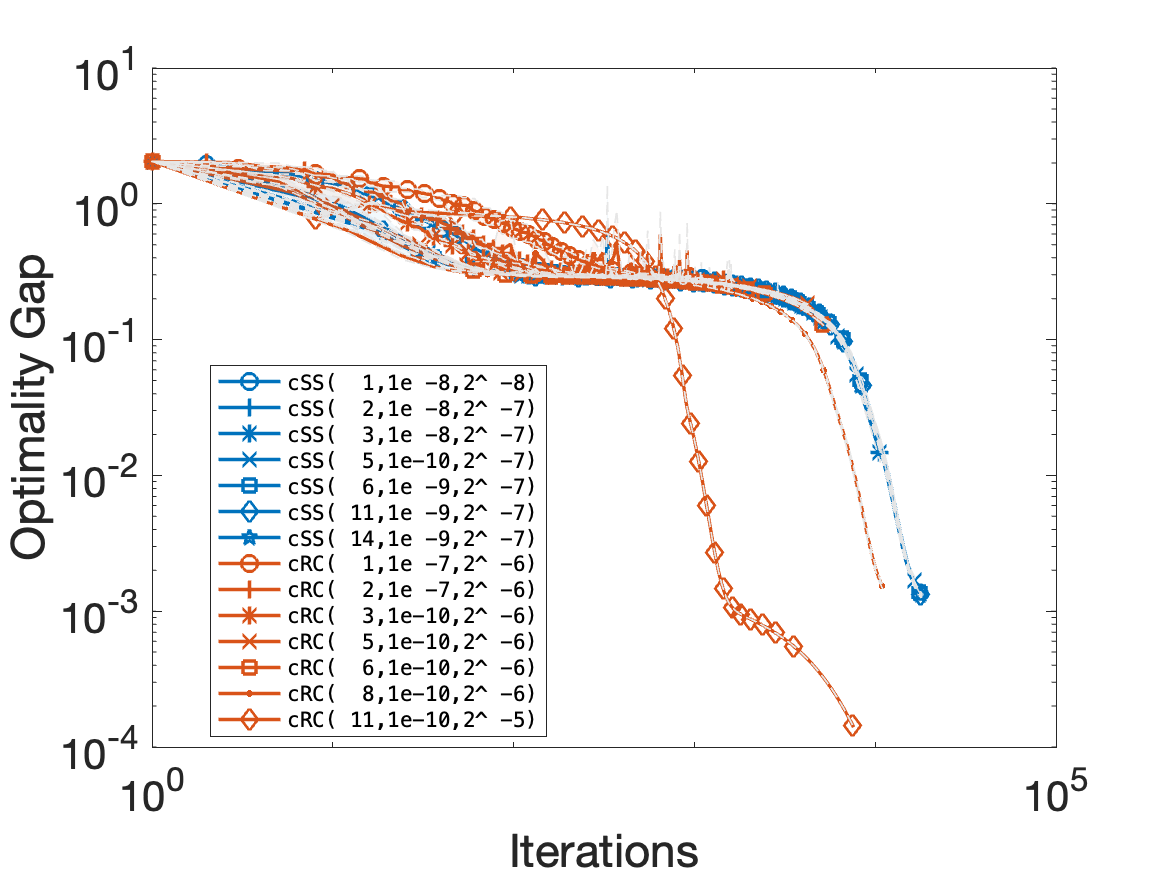}
  \caption{Comparison of cSS and cRC}
\end{subfigure}
\caption{Performance of different gradient estimation methods using the tuned hyperparameters on the Osborne function with absolute error and $ \sigma = 10^{-3} $.}
\end{figure}

\begin{figure}[H]
\centering
\begin{subfigure}{0.33\textwidth}
  \centering
  \includegraphics[width=1.1\linewidth]{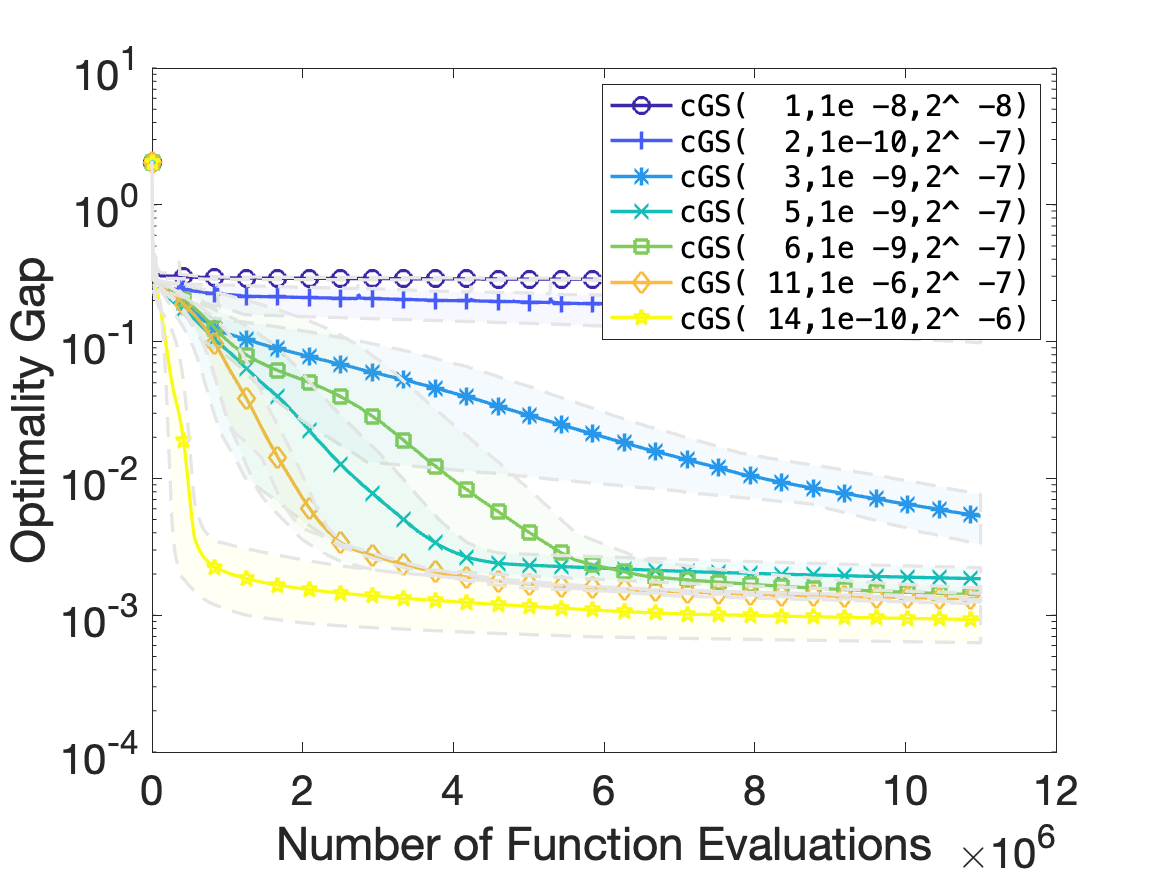}
  \caption{Performance of cGS}
\end{subfigure}%
\begin{subfigure}{0.33\textwidth}
  \centering
  \includegraphics[width=1.1\linewidth]{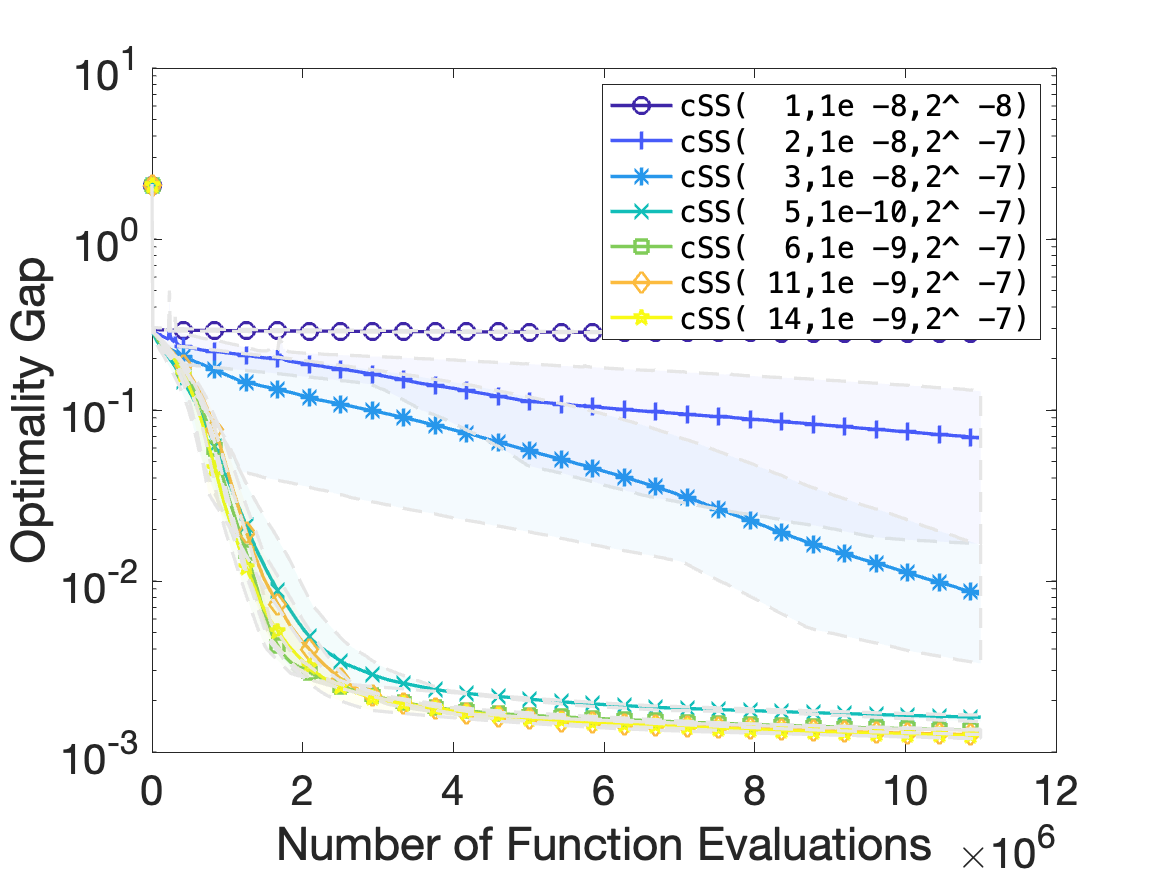}
  \caption{Performance of cSS}
\end{subfigure}%
\begin{subfigure}{0.33\textwidth}
  \centering
  \includegraphics[width=1.1\linewidth]{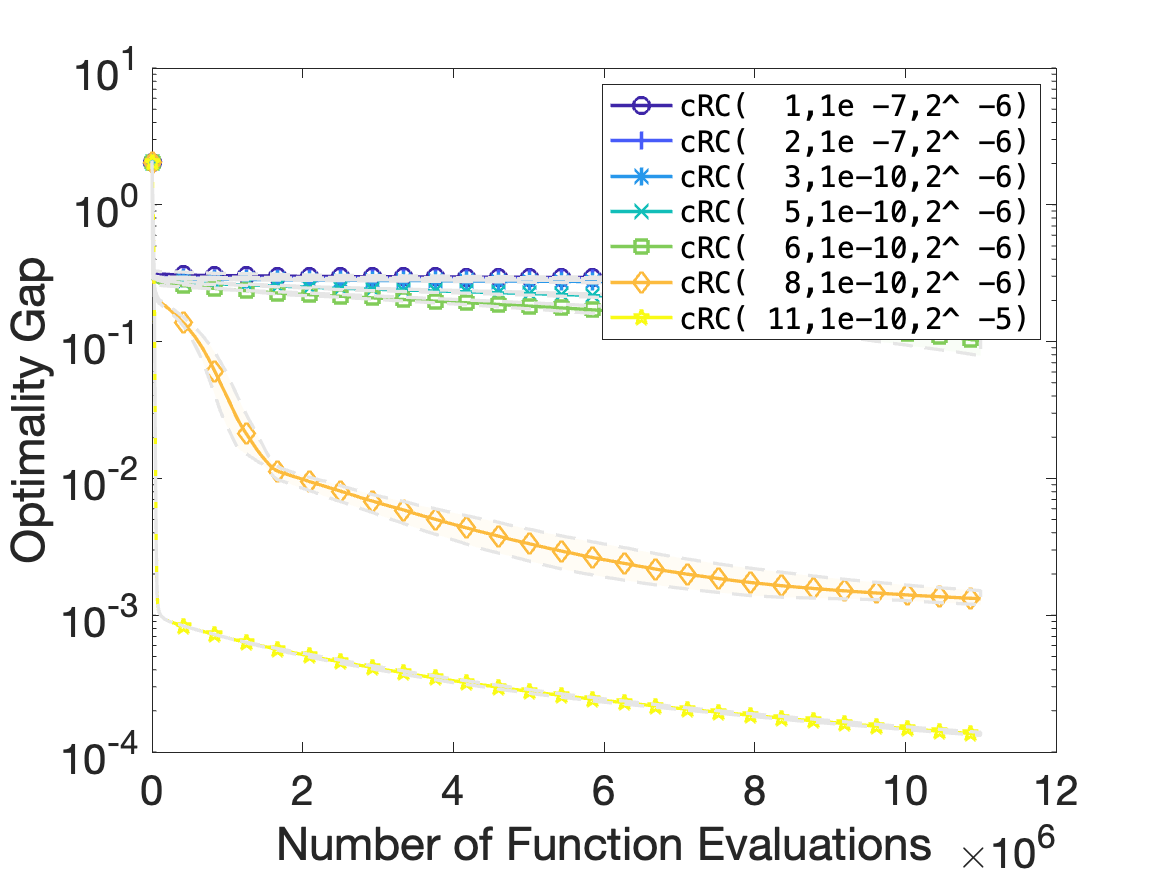}
  \caption{Performance of cRC}
\end{subfigure}
\caption{The effect of number of directions $N$ on the performance of different randomized gradient estimation methods on the Osborne function with absolute error and $ \sigma = 10^{-3} $. The sampling radius $\nu$ and step size $\alpha$ are tuned for each method and $N$ combination to achieve the best performance.}
\end{figure}

\newpage
\paragraph{Osborne Function $(d = 11, p = 65)$ with Relative Error, $\sigma = 10^{-5}$}

\begin{figure}[H]
\centering
\begin{subfigure}{0.33\textwidth}
  \centering
  \includegraphics[width=1.1\linewidth]{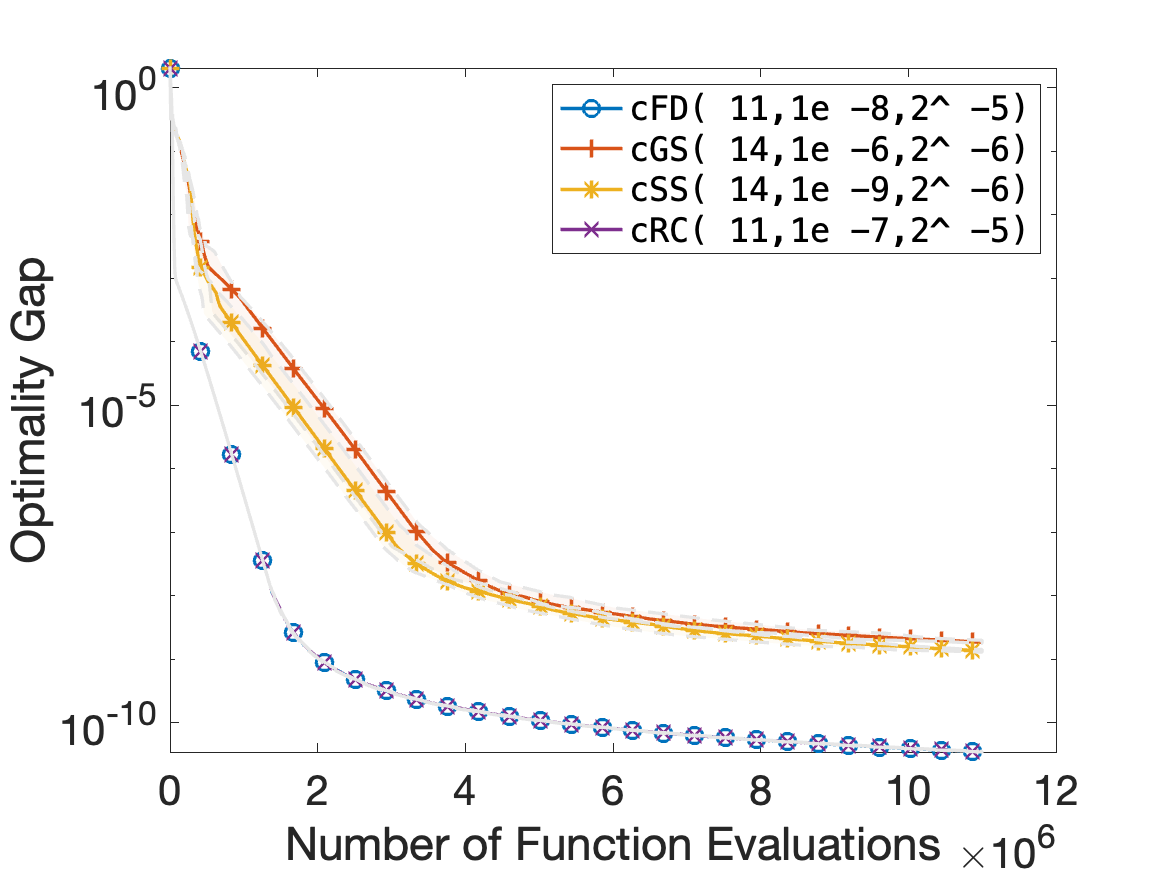}
  \caption{Optimality Gap}
\end{subfigure}%
\begin{subfigure}{0.33\textwidth}
  \centering
  \includegraphics[width=1.1\linewidth]{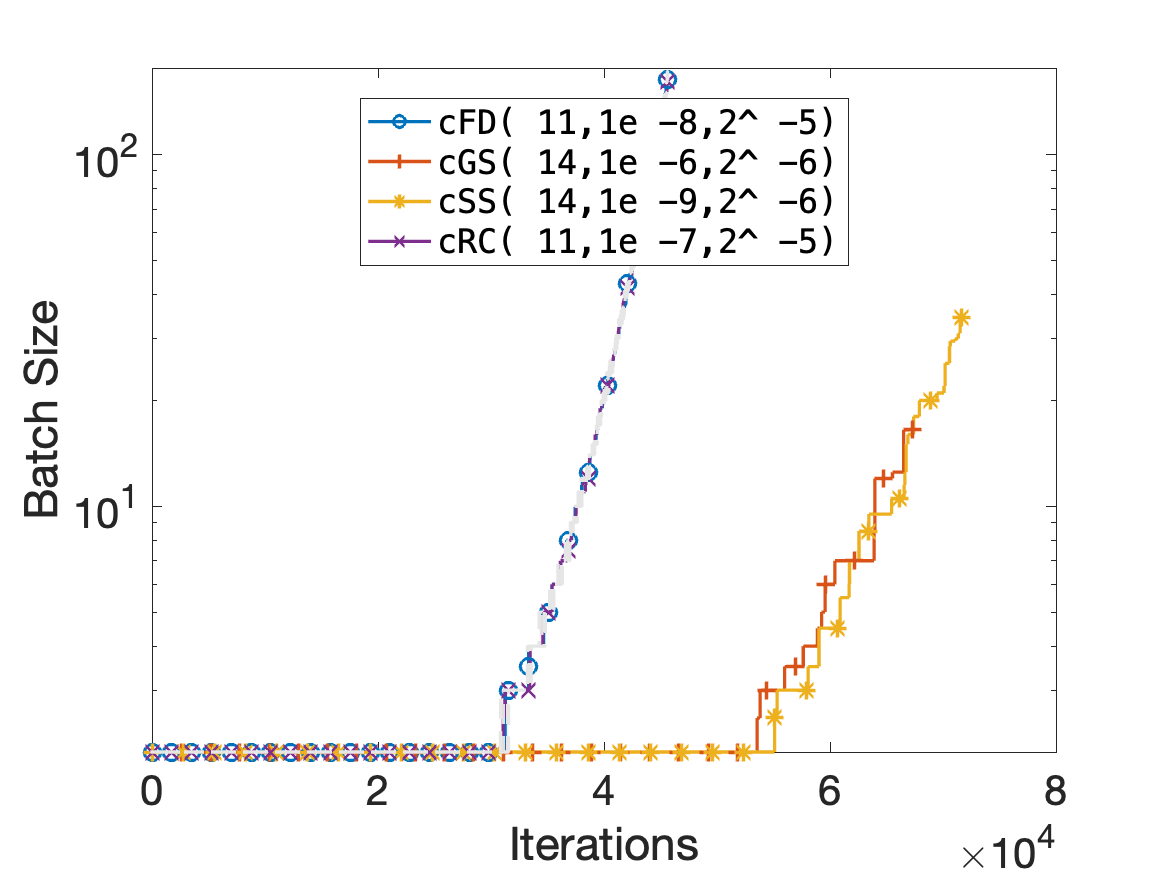}
  \caption{Batch Size}
\end{subfigure}
\begin{subfigure}{0.33\textwidth}
  \centering
  \includegraphics[width=1.1\linewidth]{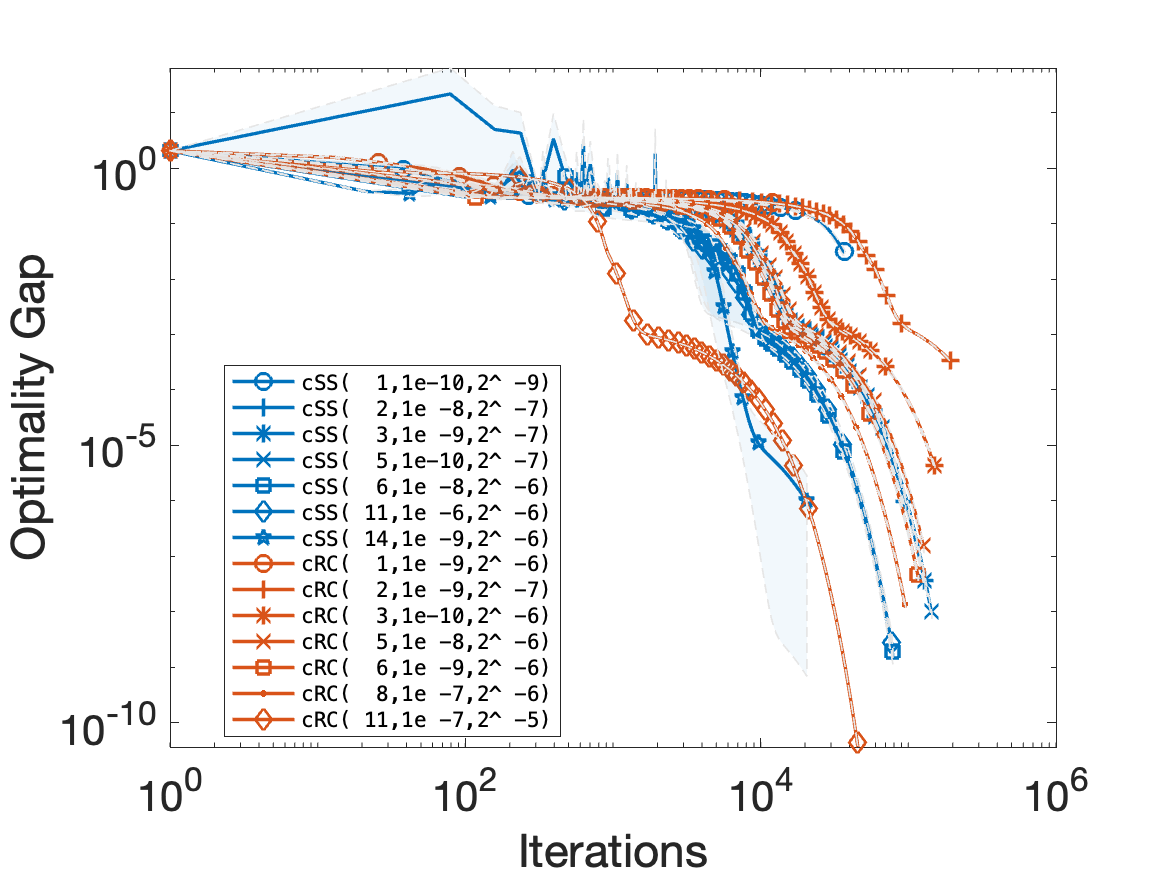}
  \caption{Comparison of cSS and cRC}
\end{subfigure}
\caption{Performance of different gradient estimation methods using the tuned hyperparameters on the Osborne function with relative error and $ \sigma = 10^{-5} $.}
\end{figure}

\begin{figure}[H]
\centering
\begin{subfigure}{0.33\textwidth}
  \centering
  \includegraphics[width=1.1\linewidth]{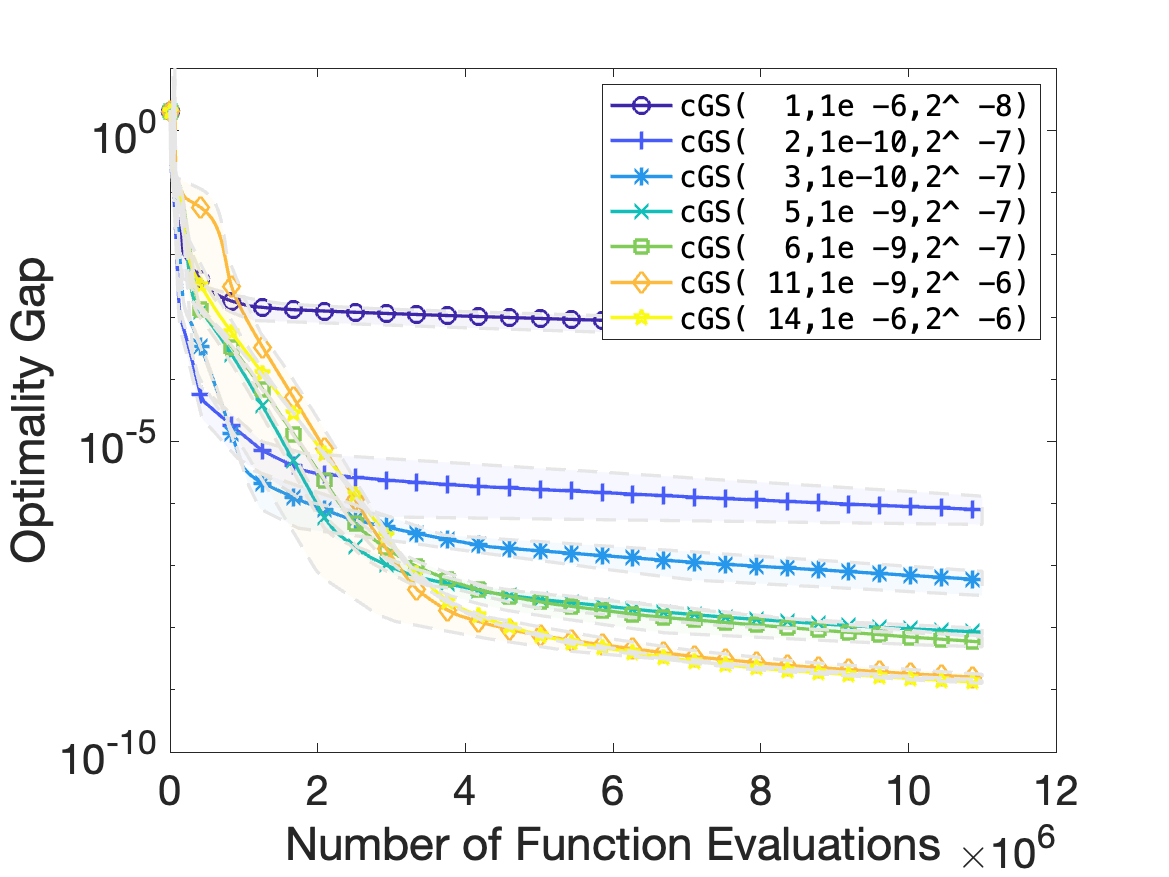}
  \caption{Performance of cGS}
\end{subfigure}%
\begin{subfigure}{0.33\textwidth}
  \centering
  \includegraphics[width=1.1\linewidth]{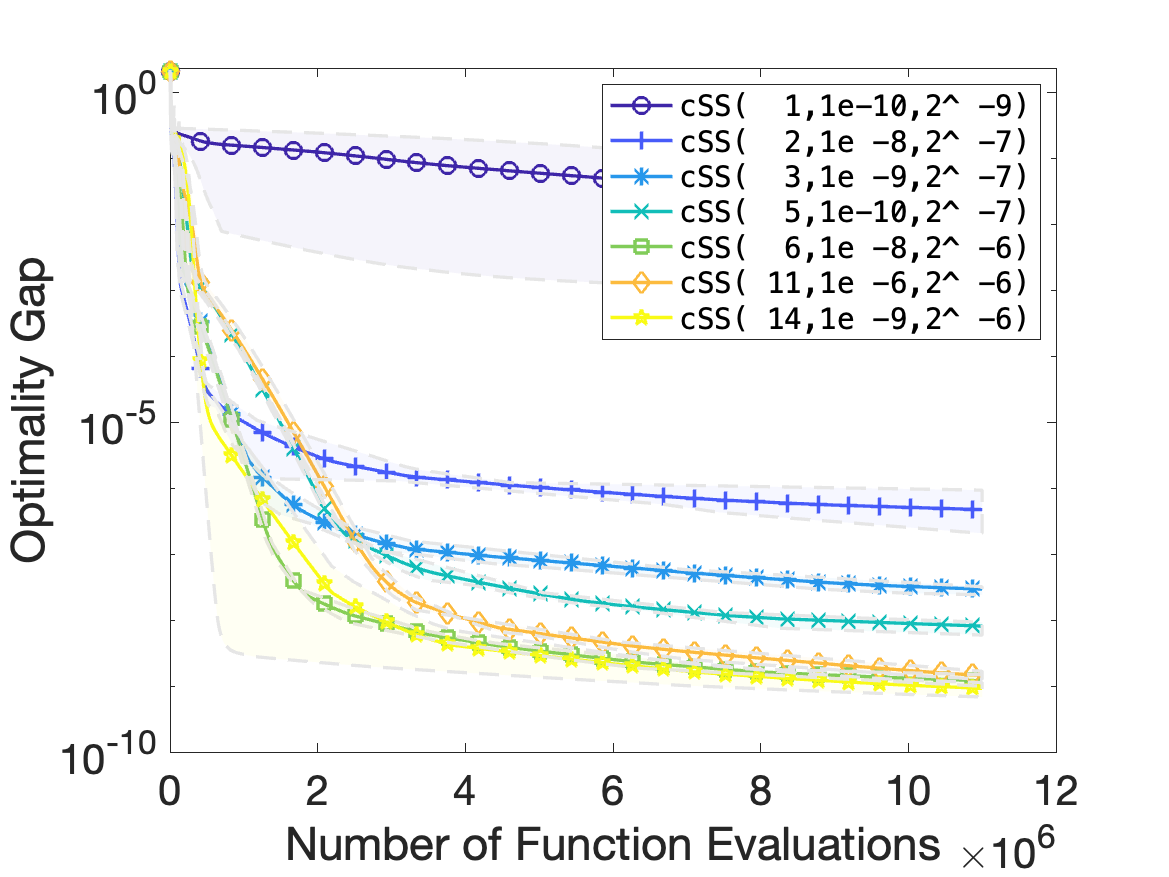}
  \caption{Performance of cSS}
\end{subfigure}%
\begin{subfigure}{0.33\textwidth}
  \centering
  \includegraphics[width=1.1\linewidth]{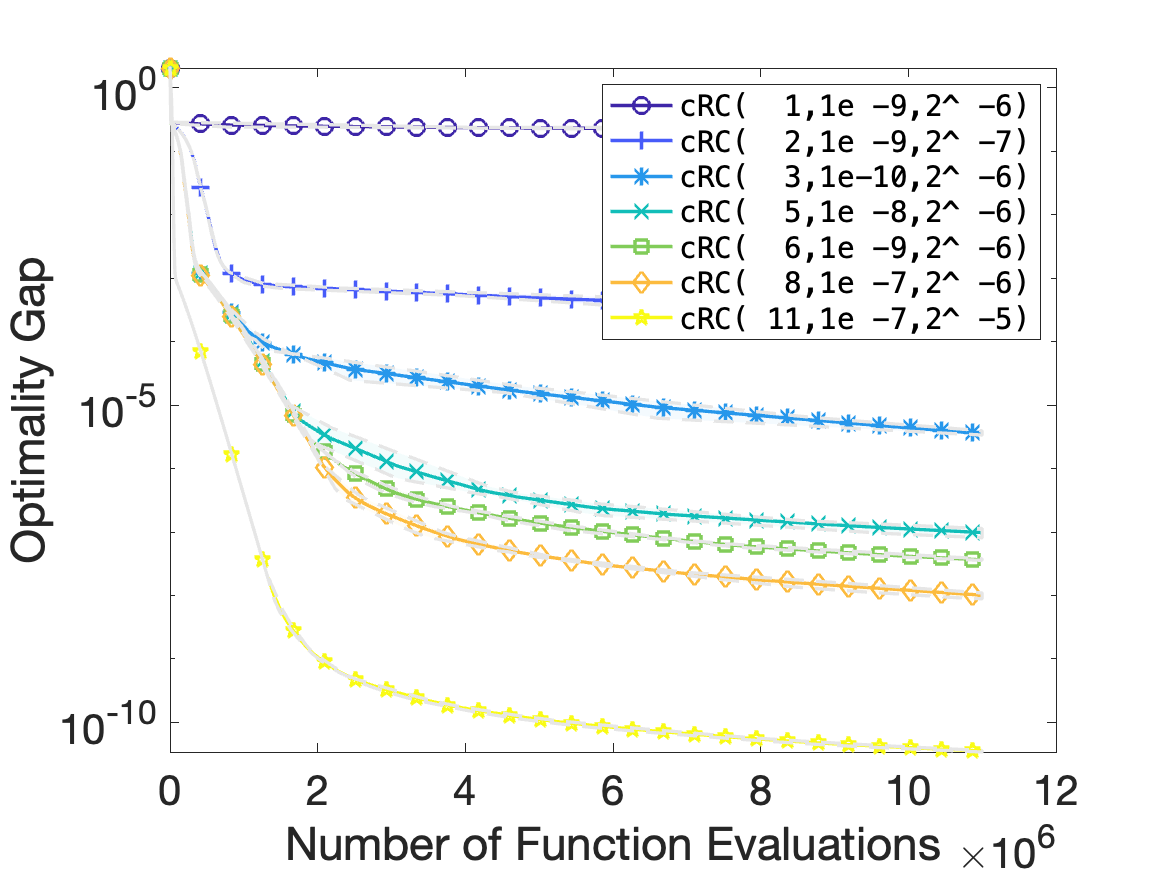}
  \caption{Performance of cRC}
\end{subfigure}
\caption{The effect of number of directions $N$ on the performance of different randomized gradient estimation methods on the Osborne function with relative error and $ \sigma = 10^{-5} $. The sampling radius $\nu$ and step size $\alpha$ are tuned for each method and $N$ combination to achieve the best performance.}
\end{figure}

\newpage
\paragraph{Osborne Function $(d = 11, p = 65)$ with Absolute Error, $\sigma = 10^{-5}$}

\begin{figure}[H]
\centering
\begin{subfigure}{0.33\textwidth}
  \centering
  \includegraphics[width=1.1\linewidth]{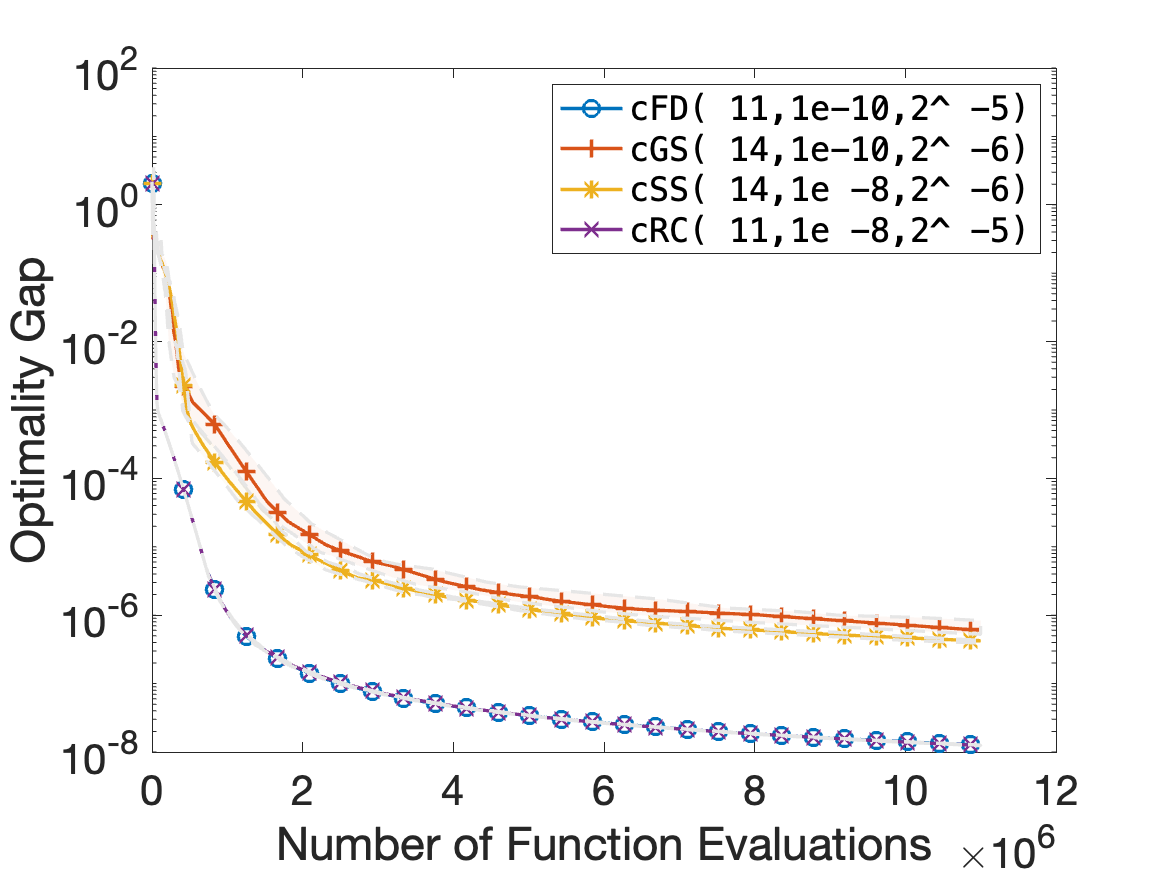}
  \caption{Optimality Gap}
\end{subfigure}%
\begin{subfigure}{0.33\textwidth}
  \centering
  \includegraphics[width=1.1\linewidth]{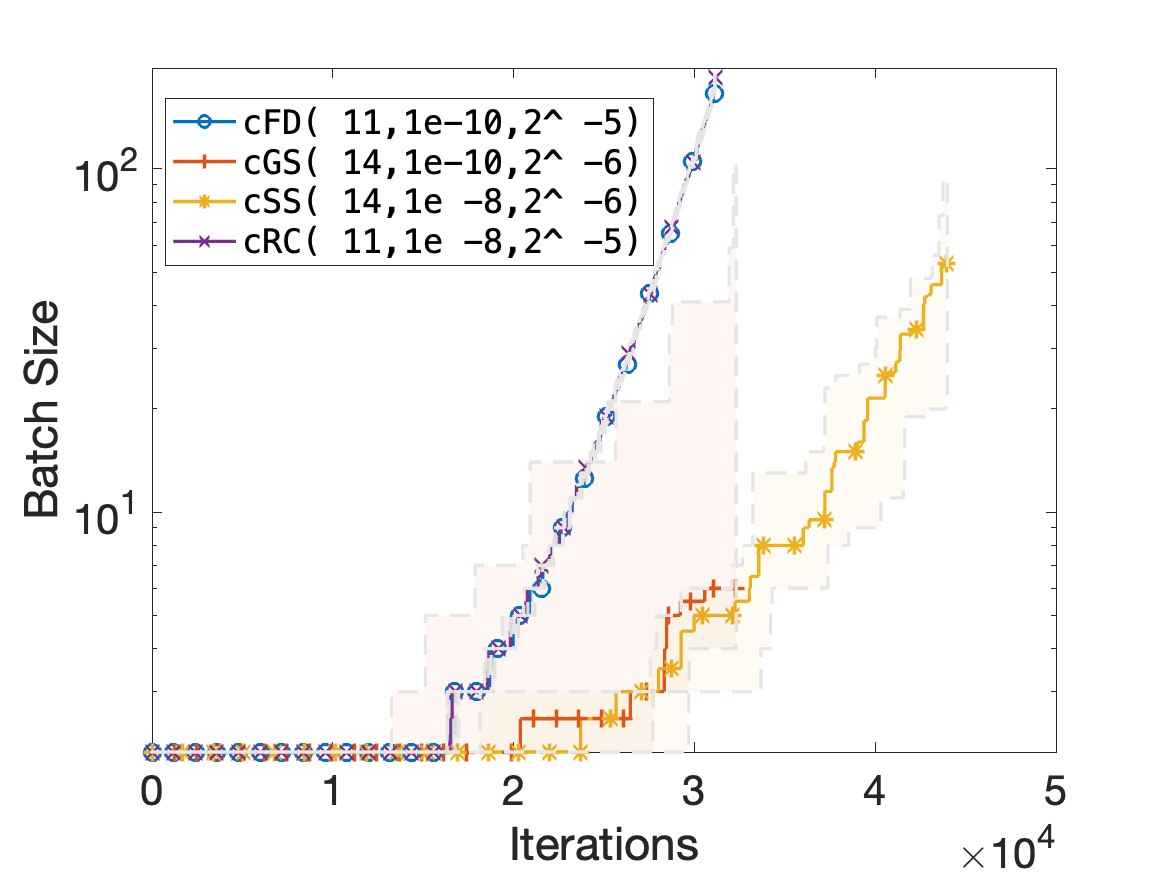}
  \caption{Batch Size}
\end{subfigure}
\begin{subfigure}{0.33\textwidth}
  \centering
  \includegraphics[width=1.1\linewidth]{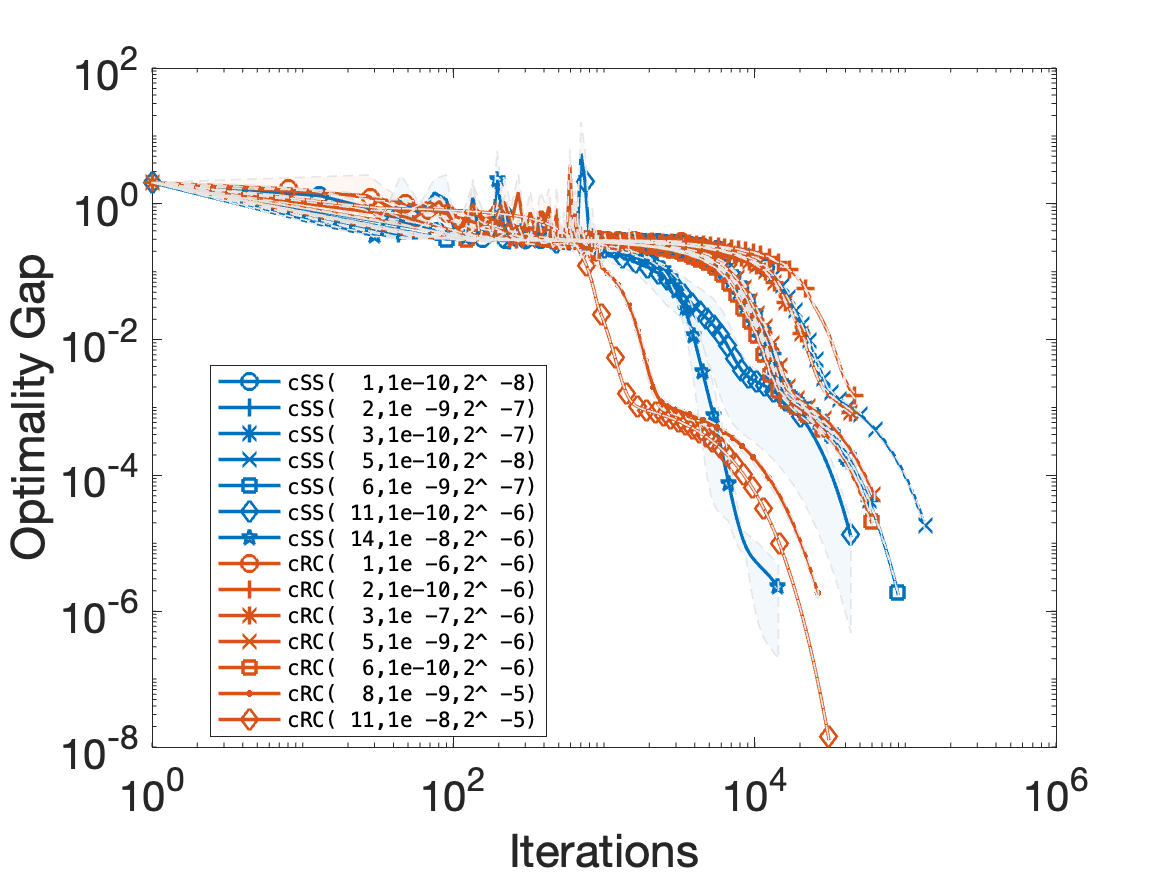}
  \caption{Comparison of cSS and cRC}
\end{subfigure}
\caption{Performance of different gradient estimation methods using the tuned hyperparameters on the Osborne function with absolute error and $ \sigma = 10^{-5} $.}
\end{figure}

\begin{figure}[H]
\centering
\begin{subfigure}{0.33\textwidth}
  \centering
  \includegraphics[width=1.1\linewidth]{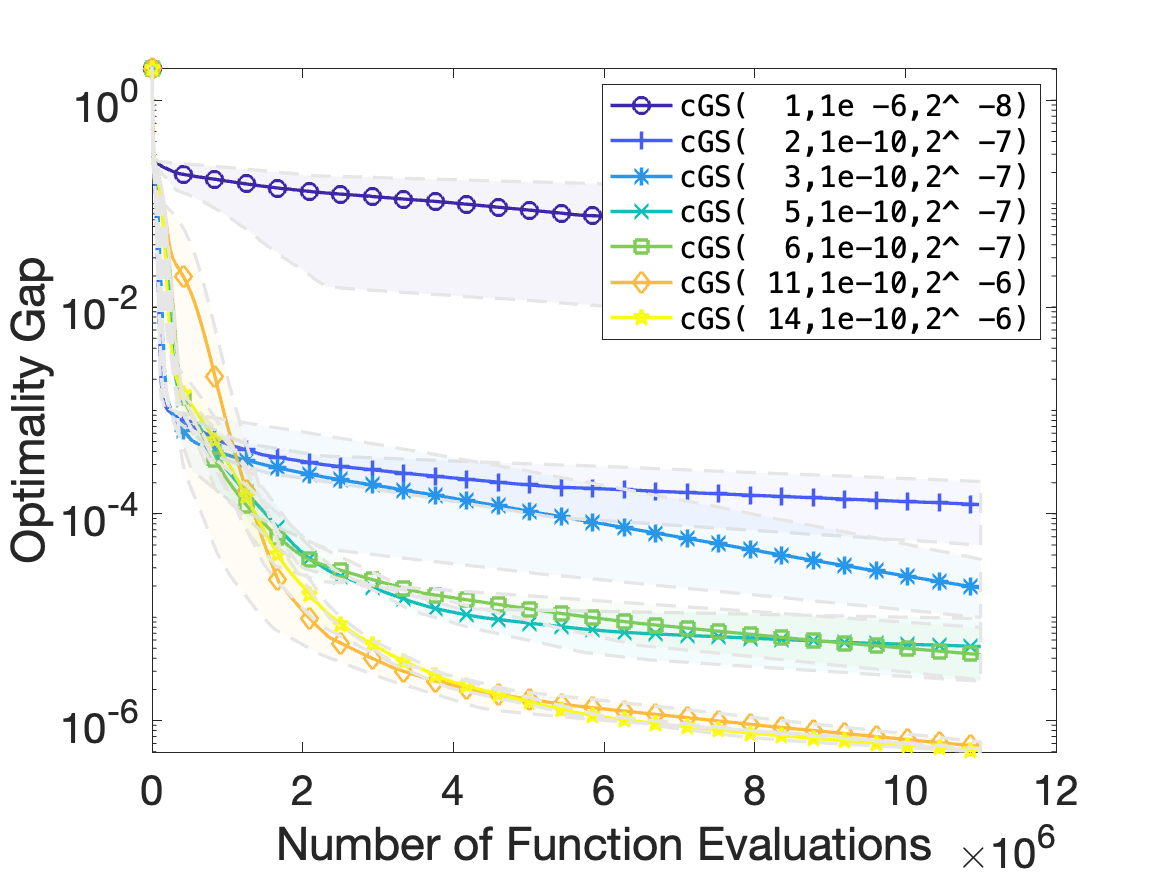}
  \caption{Performance of cGS}
\end{subfigure}%
\begin{subfigure}{0.33\textwidth}
  \centering
  \includegraphics[width=1.1\linewidth]{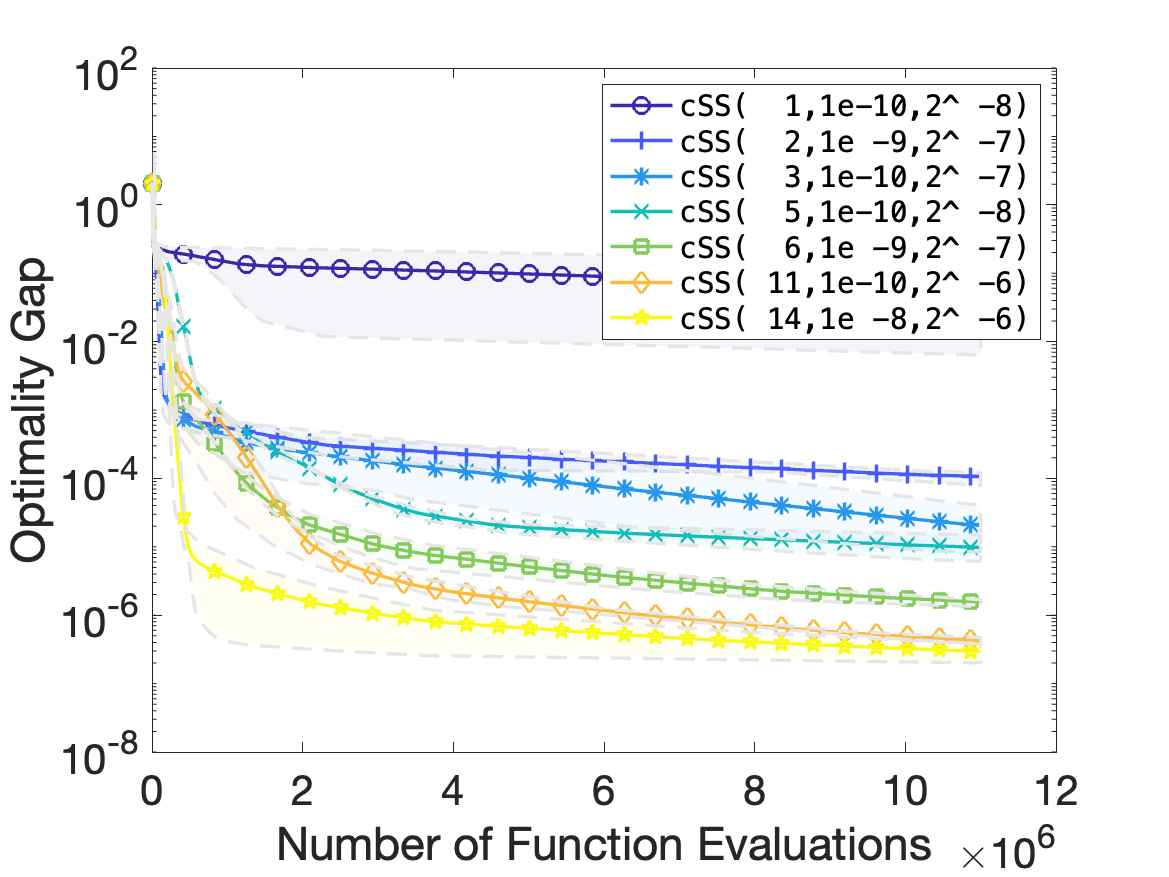}
  \caption{Performance of cSS}
\end{subfigure}%
\begin{subfigure}{0.33\textwidth}
  \centering
  \includegraphics[width=1.1\linewidth]{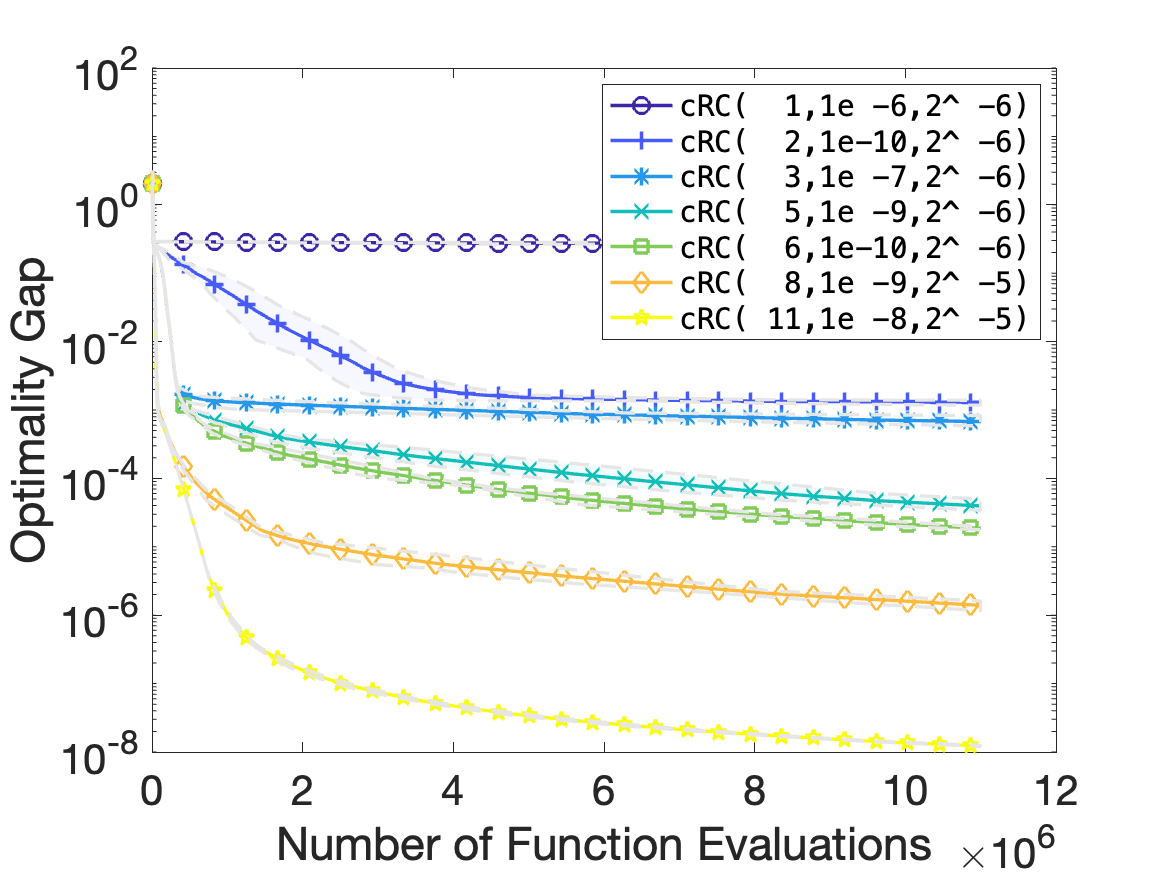}
  \caption{Performance of cRC}
\end{subfigure}
\caption{The effect of number of directions $N$ on the performance of different randomized gradient estimation methods on the Osborne function with absolute error and $ \sigma = 10^{-5} $. The sampling radius $\nu$ and step size $\alpha$ are tuned for each method and $N$ combination to achieve the best performance.}
\end{figure}

\newpage
\paragraph{Bdqrtic Function $(d = 50, p = 92)$ with Relative Error, $\sigma = 10^{-5}$}

\begin{figure}[H]
\centering
\begin{subfigure}{0.33\textwidth}
  \centering
  \includegraphics[width=1.1\linewidth]{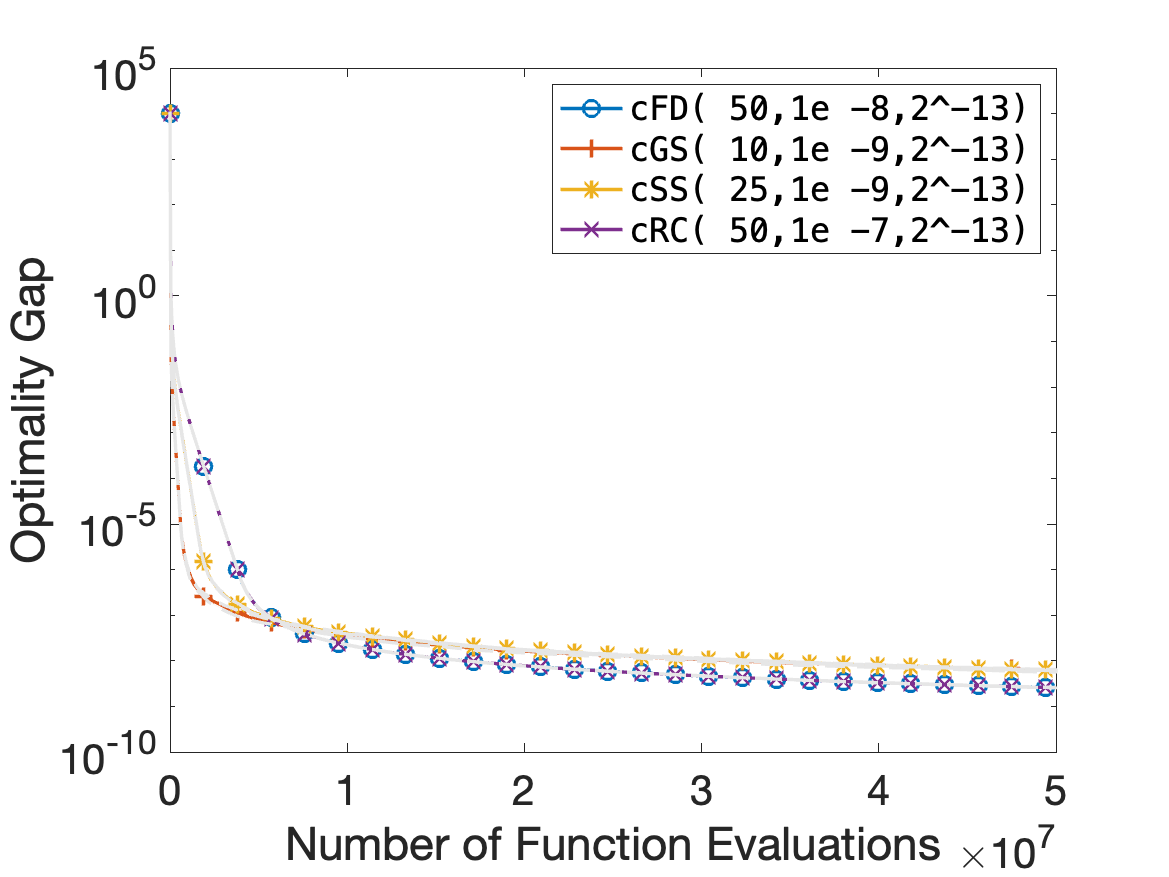}
  \caption{Optimality Gap}
\end{subfigure}%
\begin{subfigure}{0.33\textwidth}
  \centering
  \includegraphics[width=1.1\linewidth]{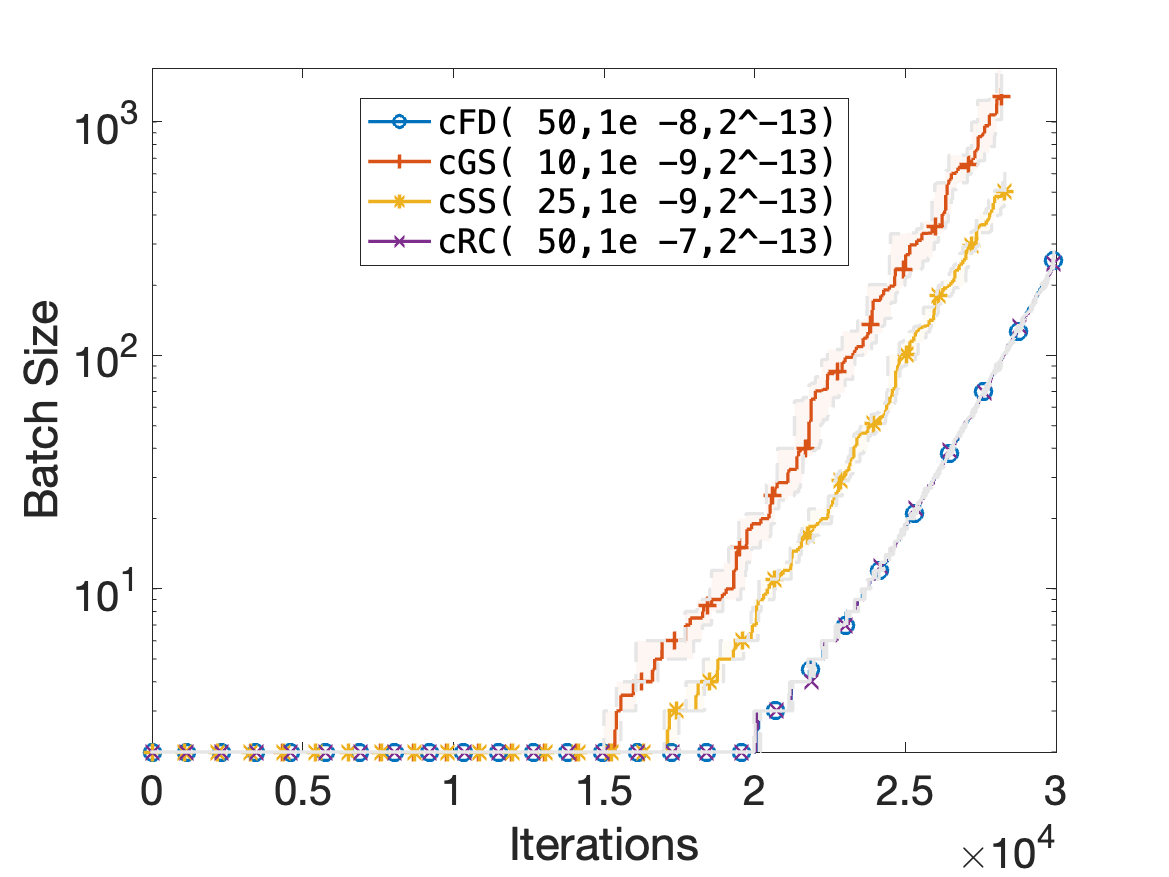}
  \caption{Batch Size}
\end{subfigure}
\begin{subfigure}{0.33\textwidth}
  \centering
  \includegraphics[width=1.1\linewidth]{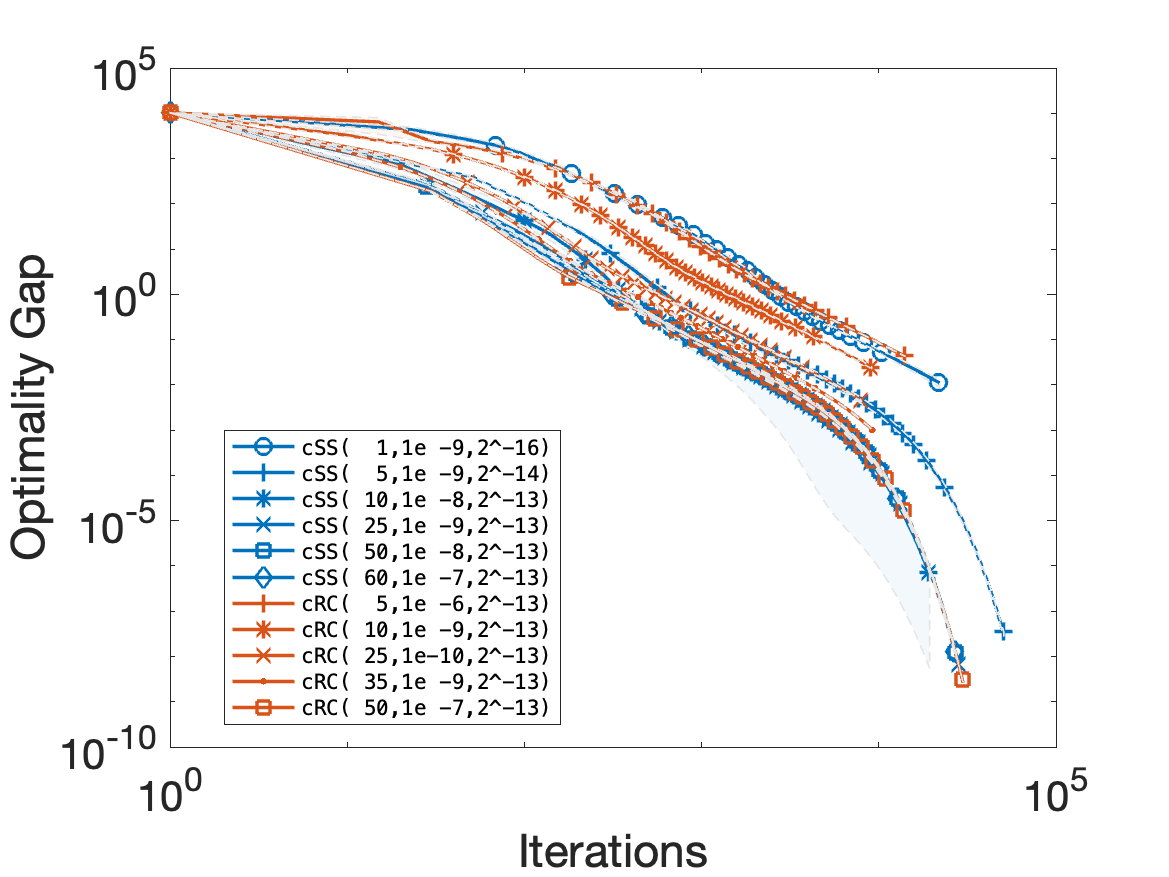}
  \caption{Comparison of cSS and cRC}
\end{subfigure}
\caption{Performance of different gradient estimation methods using the tuned hyperparameters on the Bdqrtic function with relative error and $ \sigma = 10^{-5} $.}
\end{figure}

\begin{figure}[H]
\centering
\begin{subfigure}{0.33\textwidth}
  \centering
  \includegraphics[width=1.1\linewidth]{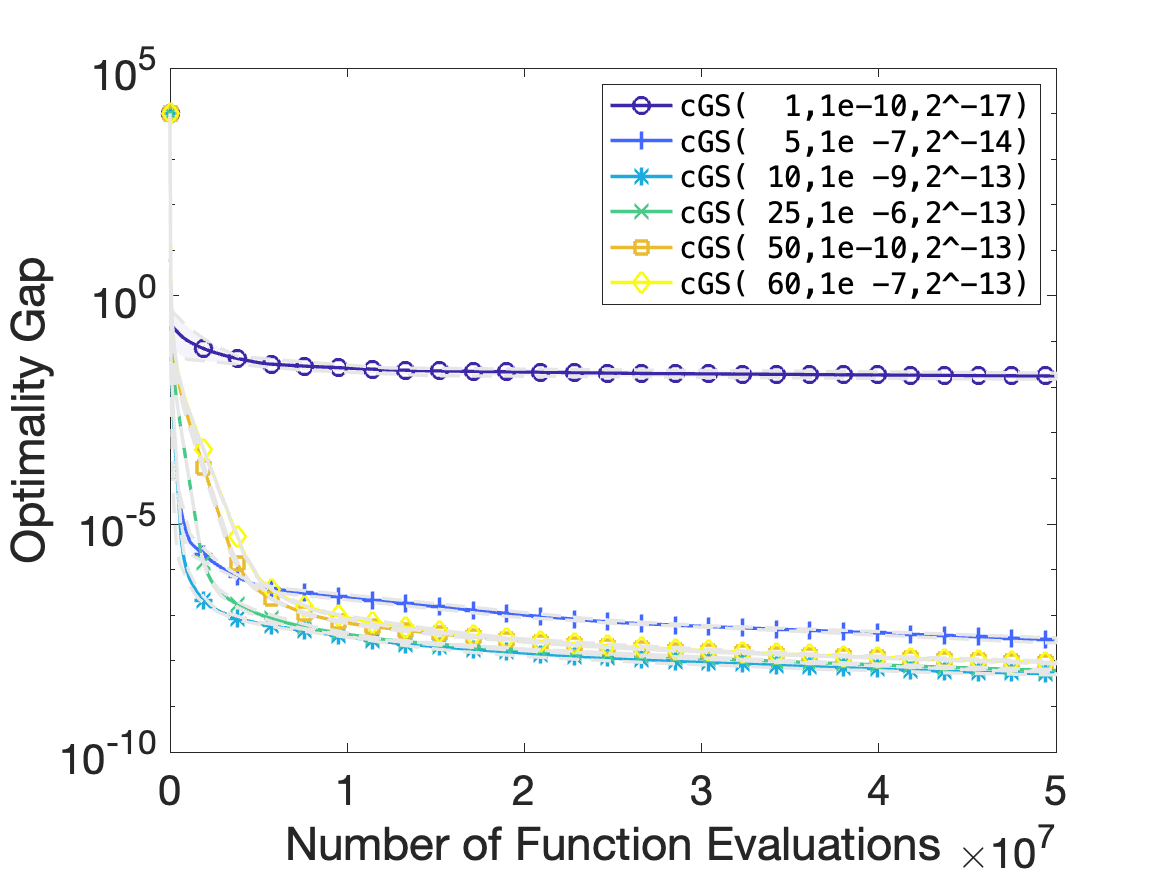}
  \caption{Performance of cGS}
\end{subfigure}%
\begin{subfigure}{0.33\textwidth}
  \centering
  \includegraphics[width=1.1\linewidth]{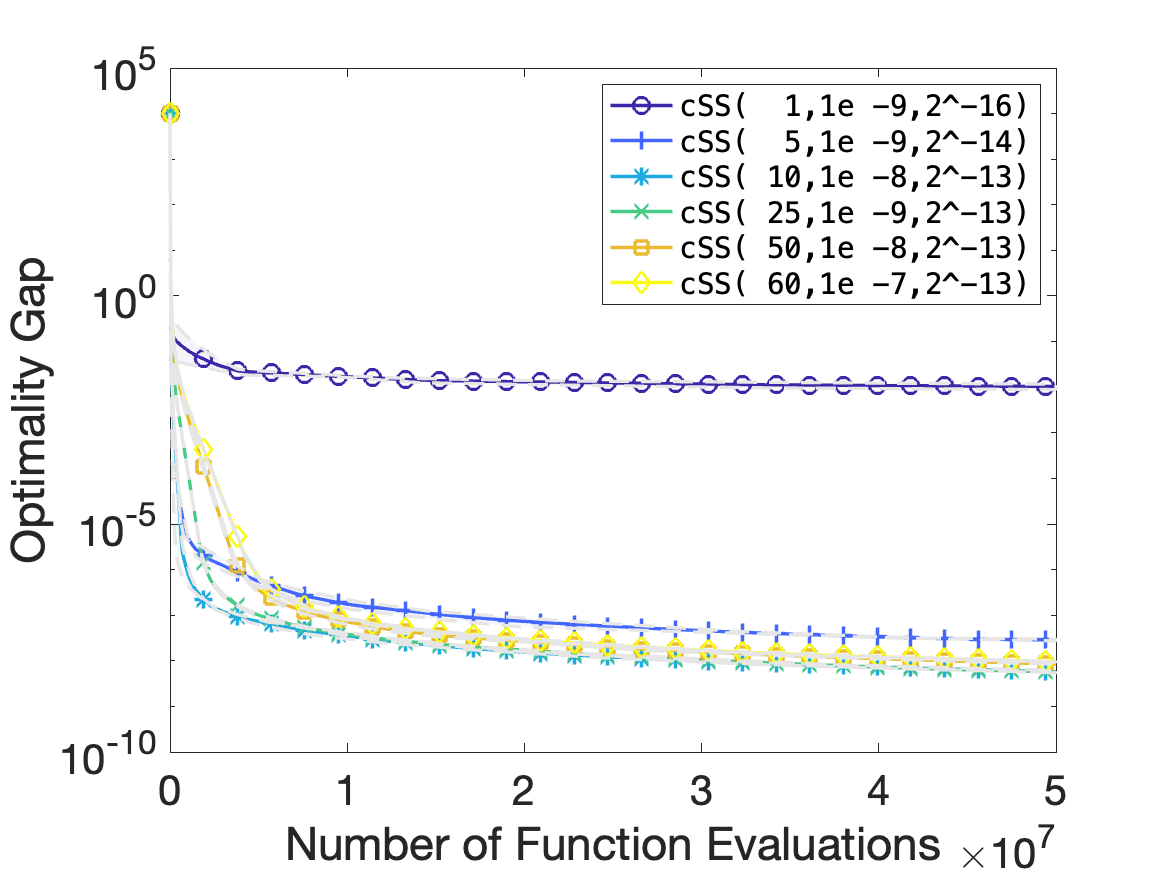}
  \caption{Performance of cSS}
\end{subfigure}%
\begin{subfigure}{0.33\textwidth}
  \centering
  \includegraphics[width=1.1\linewidth]{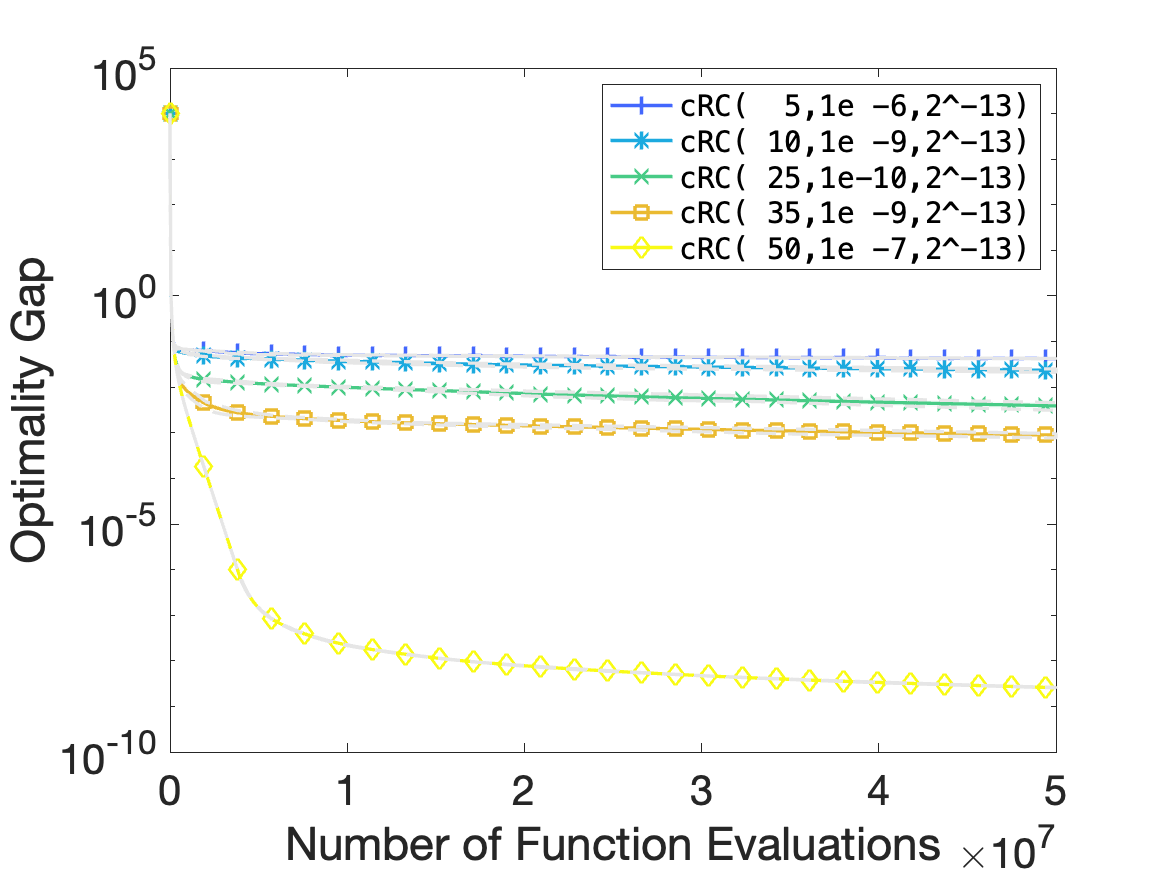}
  \caption{Performance of cRC}
\end{subfigure}
\caption{The effect of number of directions $N$ on the performance of different randomized gradient estimation methods on the Bdqrtic function with relative error and $ \sigma = 10^{-5} $. The sampling radius $\nu$ and step size $\alpha$ are tuned for each method and $N$ combination to achieve the best performance.}
\end{figure}

\newpage
\paragraph{Bdqrtic Function $(d = 50, p = 92)$ with Absolute Error, $\sigma = 10^{-5}$}

\begin{figure}[H]
\centering
\begin{subfigure}{0.33\textwidth}
  \centering
  \includegraphics[width=1.1\linewidth]{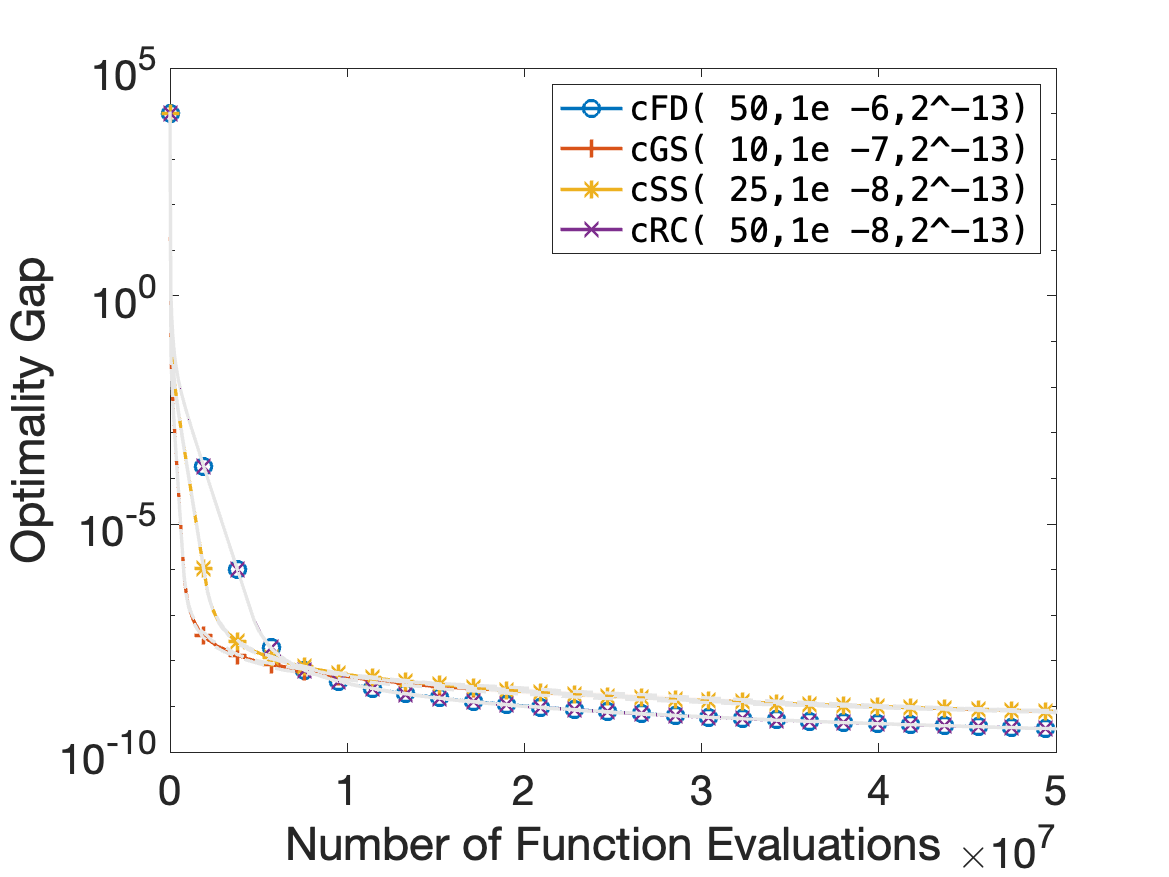}
  \caption{Optimality Gap}
\end{subfigure}%
\begin{subfigure}{0.33\textwidth}
  \centering
  \includegraphics[width=1.1\linewidth]{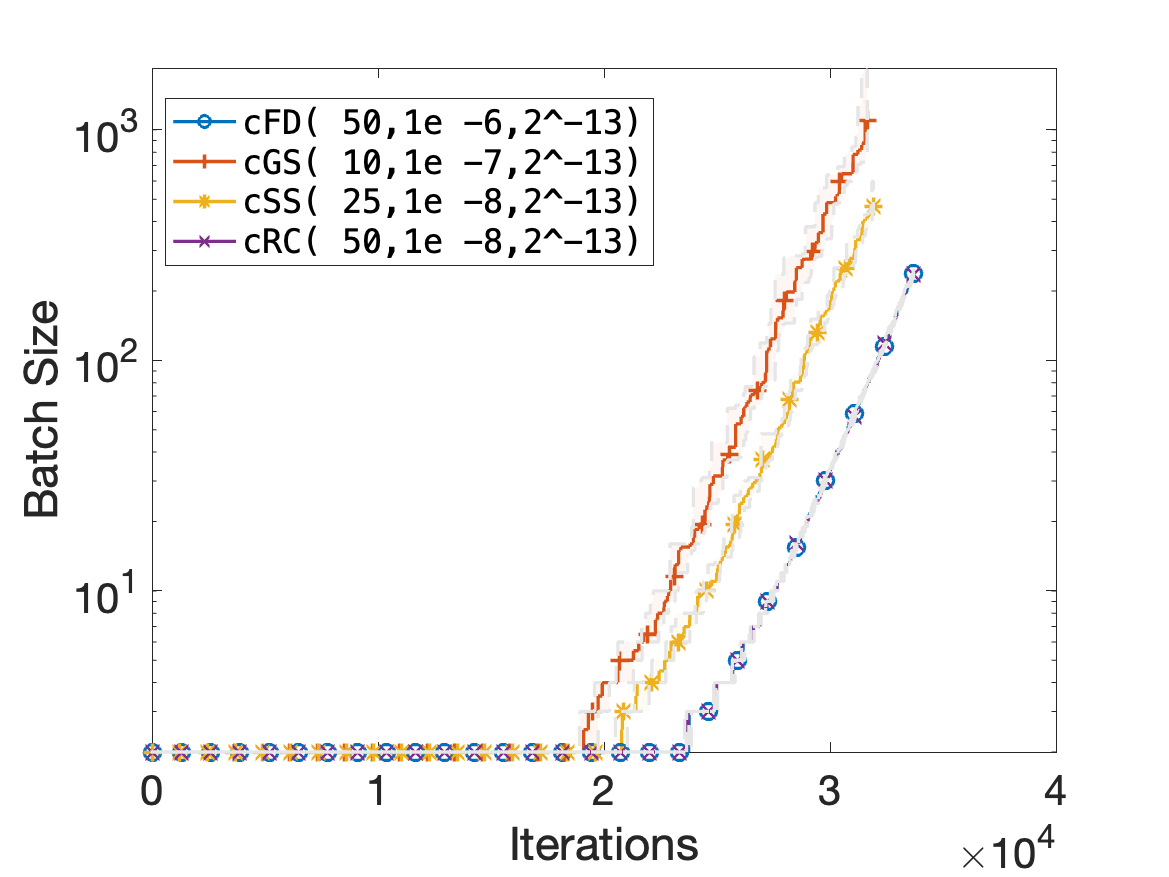}
  \caption{Batch Size}
\end{subfigure}
\begin{subfigure}{0.33\textwidth}
  \centering
  \includegraphics[width=1.1\linewidth]{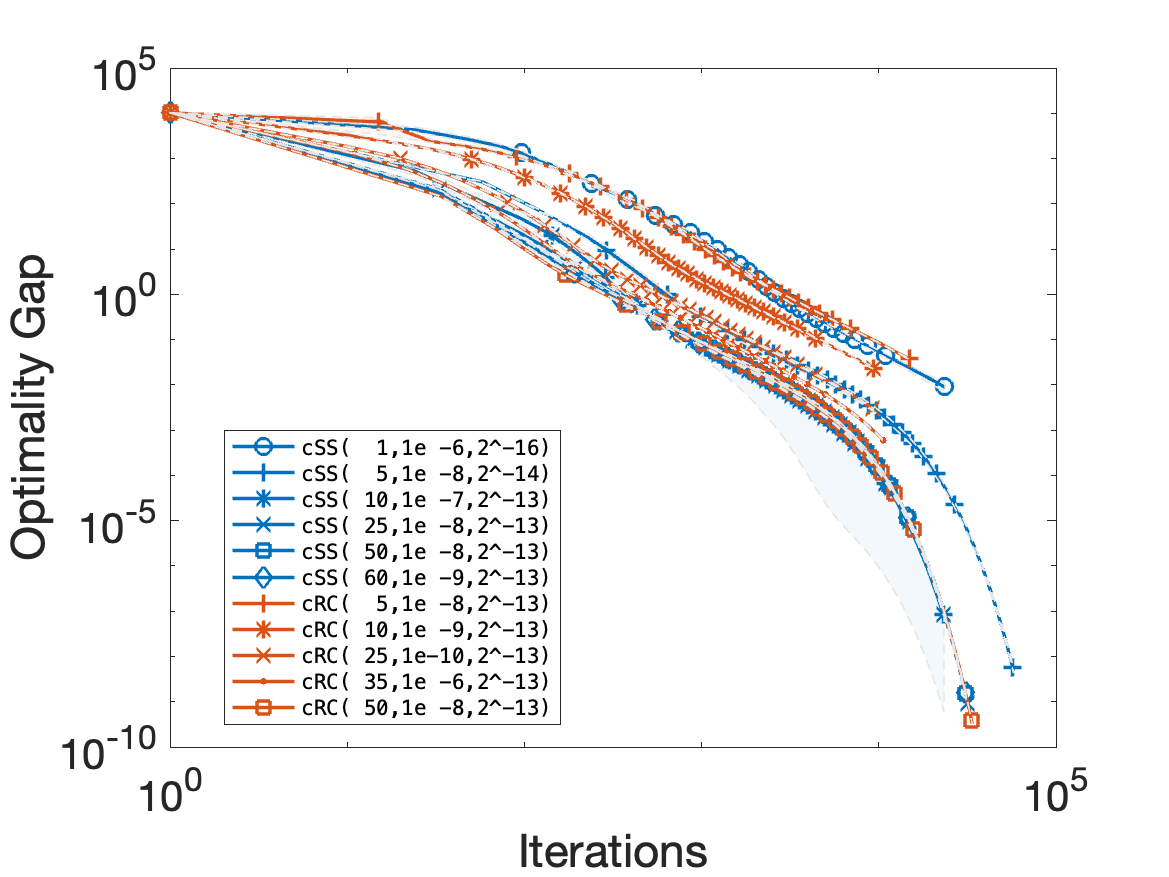}
  \caption{Comparison of cSS and cRC}
\end{subfigure}
\caption{Performance of different gradient estimation methods using the tuned hyperparameters on the Bdqrtic function with absolute error and $ \sigma = 10^{-5} $.}
\end{figure}

\begin{figure}[H]
\centering
\begin{subfigure}{0.33\textwidth}
  \centering
  \includegraphics[width=1.1\linewidth]{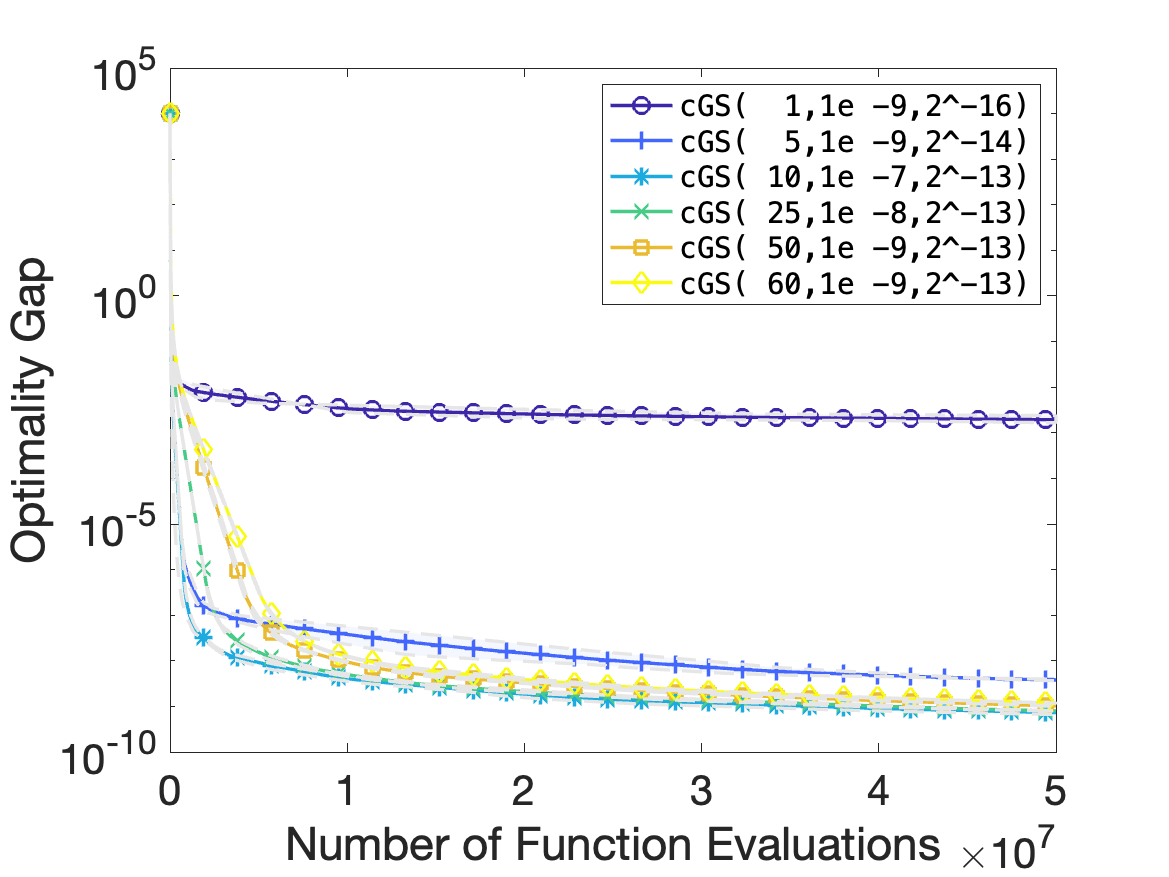}
  \caption{Performance of cGS}
\end{subfigure}%
\begin{subfigure}{0.33\textwidth}
  \centering
  \includegraphics[width=1.1\linewidth]{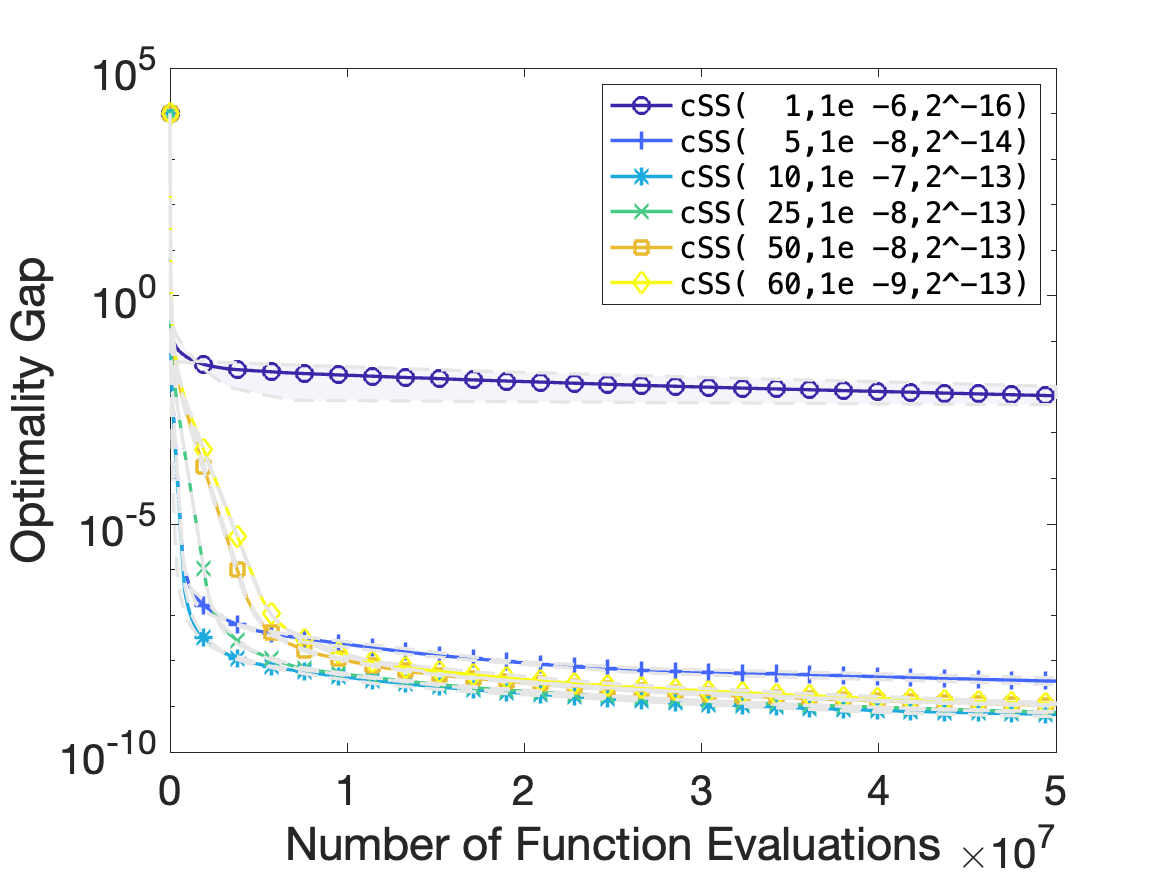}
  \caption{Performance of cSS}
\end{subfigure}%
\begin{subfigure}{0.33\textwidth}
  \centering
  \includegraphics[width=1.1\linewidth]{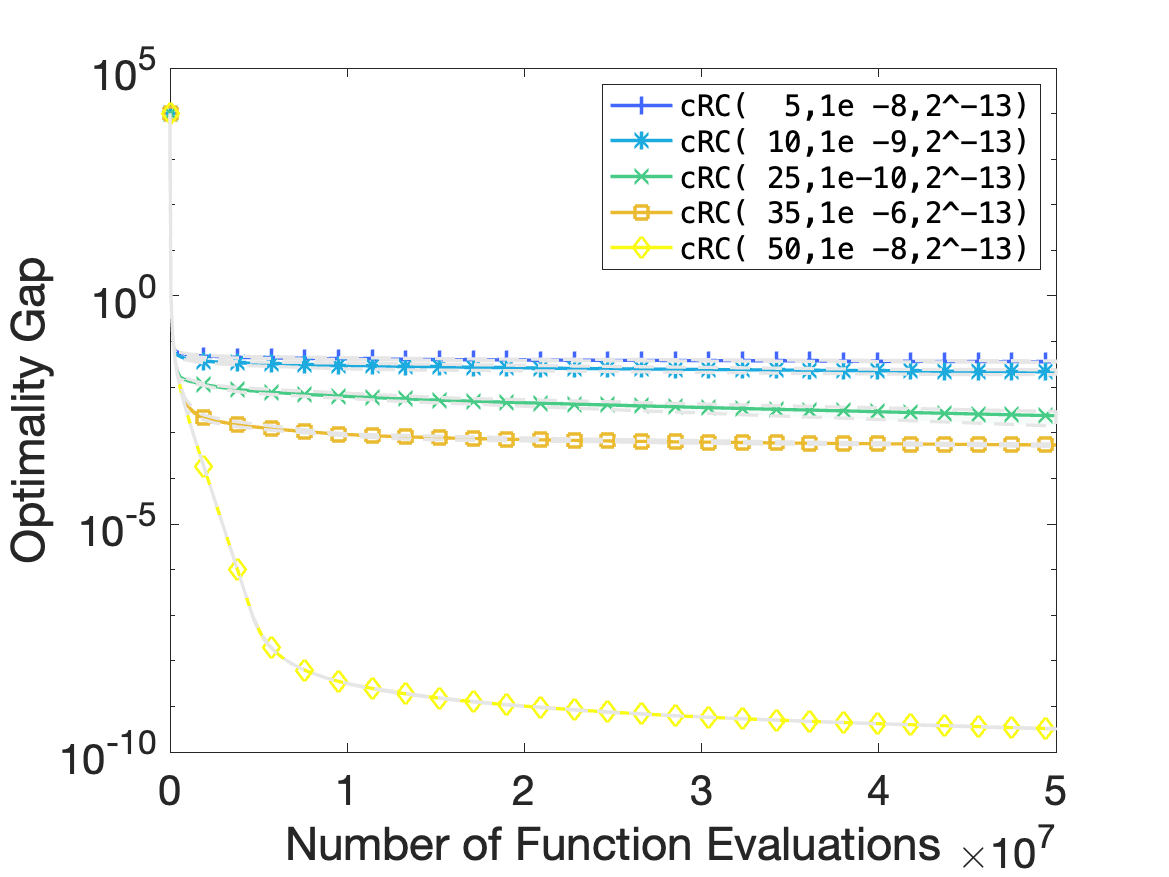}
  \caption{Performance of cRC}
\end{subfigure}
\caption{The effect of number of directions $N$ on the performance of different randomized gradient estimation methods on the Bdqrtic function with absolute error and $ \sigma = 10^{-5} $. The sampling radius $\nu$ and step size $\alpha$ are tuned for each method and $N$ combination to achieve the best performance.}
\end{figure}

\newpage
\paragraph{Cube Function $(d = 20, p = 30)$ with Relative Error, $\sigma = 10^{-3}$}

\begin{figure}[H]
\centering
\begin{subfigure}{0.33\textwidth}
  \centering
  \includegraphics[width=1.1\linewidth]{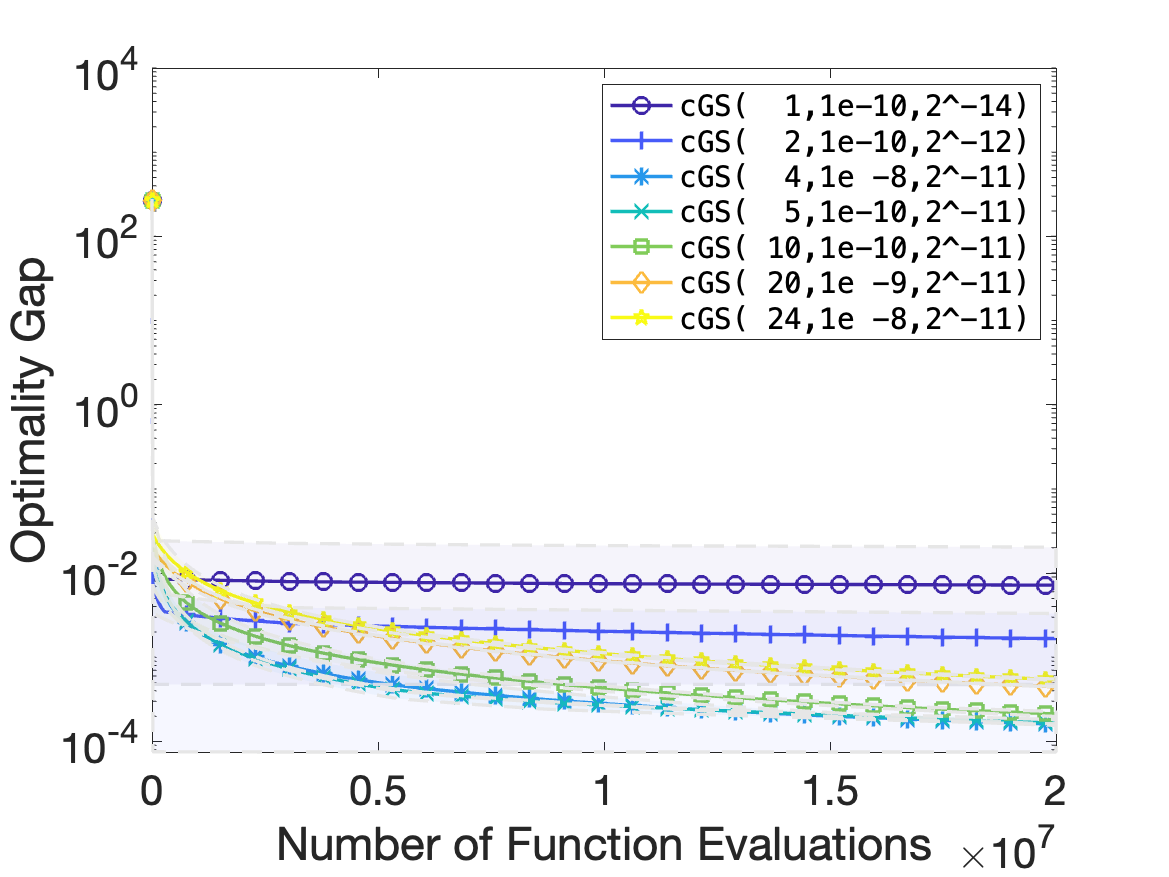}
  \caption{Performance of cGS}
\end{subfigure}%
\begin{subfigure}{0.33\textwidth}
  \centering
  \includegraphics[width=1.1\linewidth]{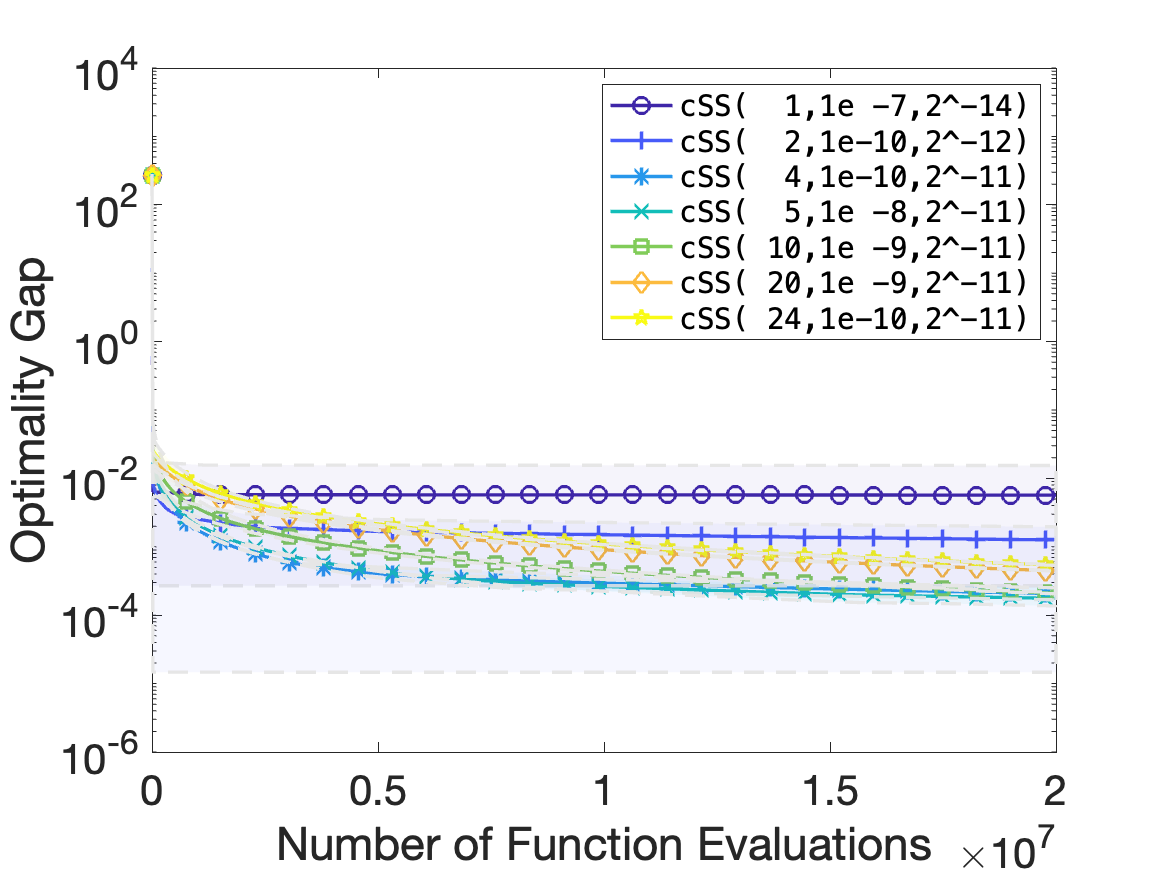}
  \caption{Performance of cSS}
\end{subfigure}%
\begin{subfigure}{0.33\textwidth}
  \centering
  \includegraphics[width=1.1\linewidth]{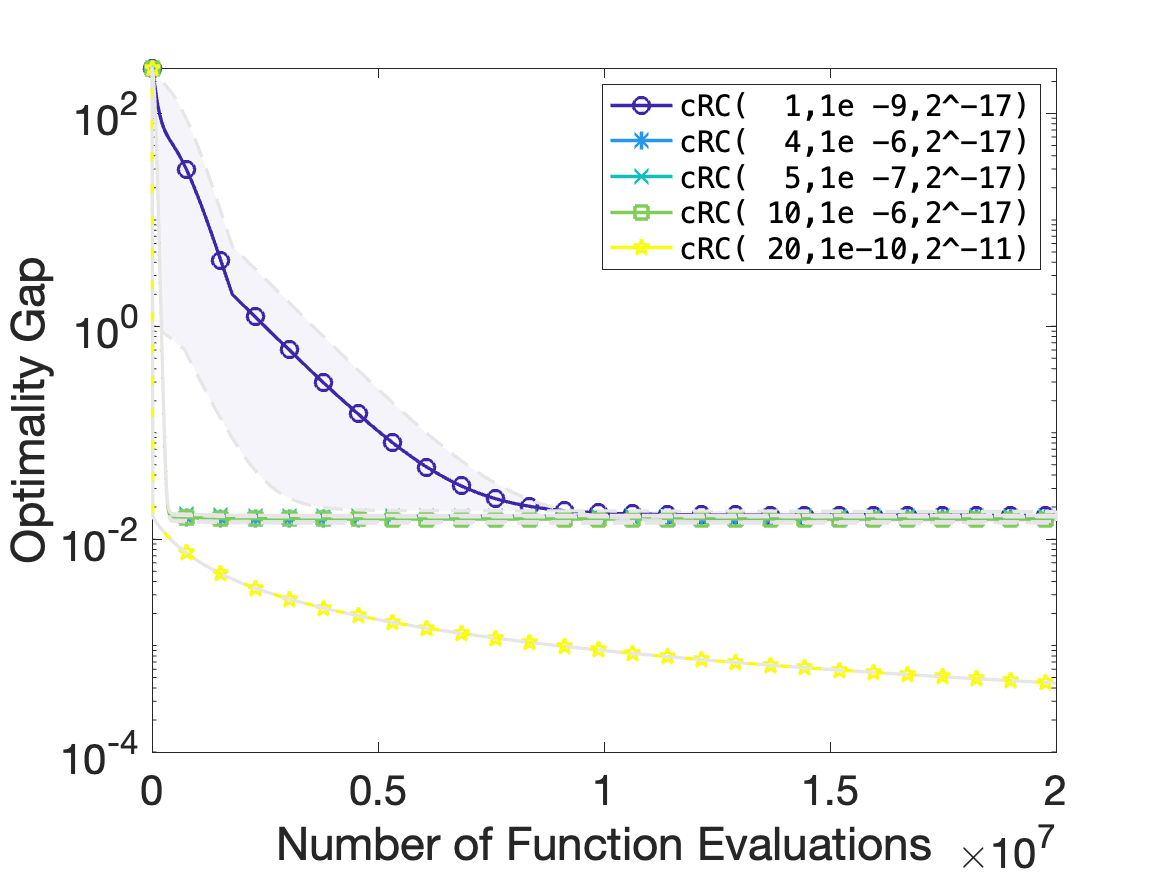}
  \caption{Performance of cRC}
\end{subfigure}
\caption{The effect of number of directions $N$ on the performance of different randomized gradient estimation methods on the Cube function with relative error and $ \sigma = 10^{-3} $. Sampling radius $\nu$ and step size $\alpha$ are tuned for each method and $N$ combination to achieve the best performance.}
\end{figure}

\paragraph{Cube Function $(d = 20, p = 30)$ with Absolute Error, $\sigma = 10^{-3}$}

\begin{figure}[H]
\centering
\begin{subfigure}{0.33\textwidth}
  \centering
  \includegraphics[width=1.1\linewidth]{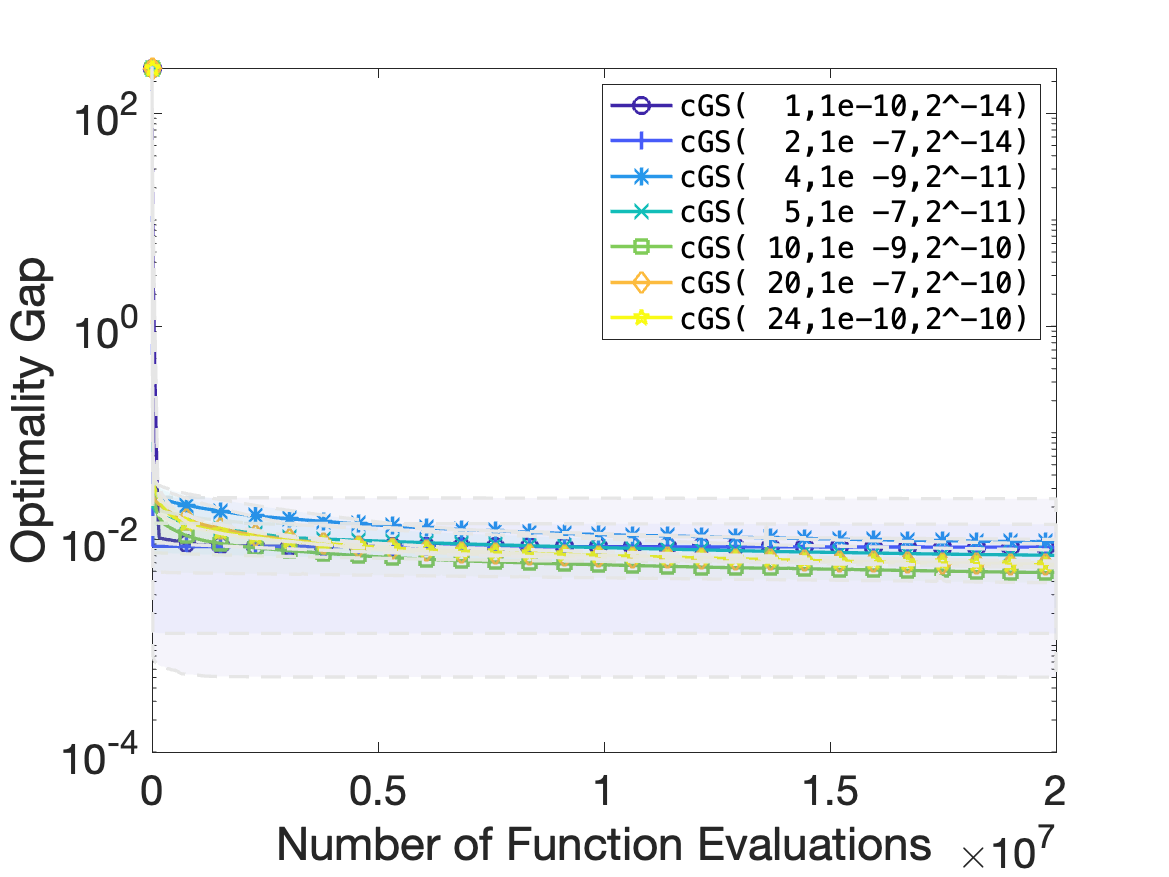}
  \caption{Performance of cGS}
\end{subfigure}%
\begin{subfigure}{0.33\textwidth}
  \centering
  \includegraphics[width=1.1\linewidth]{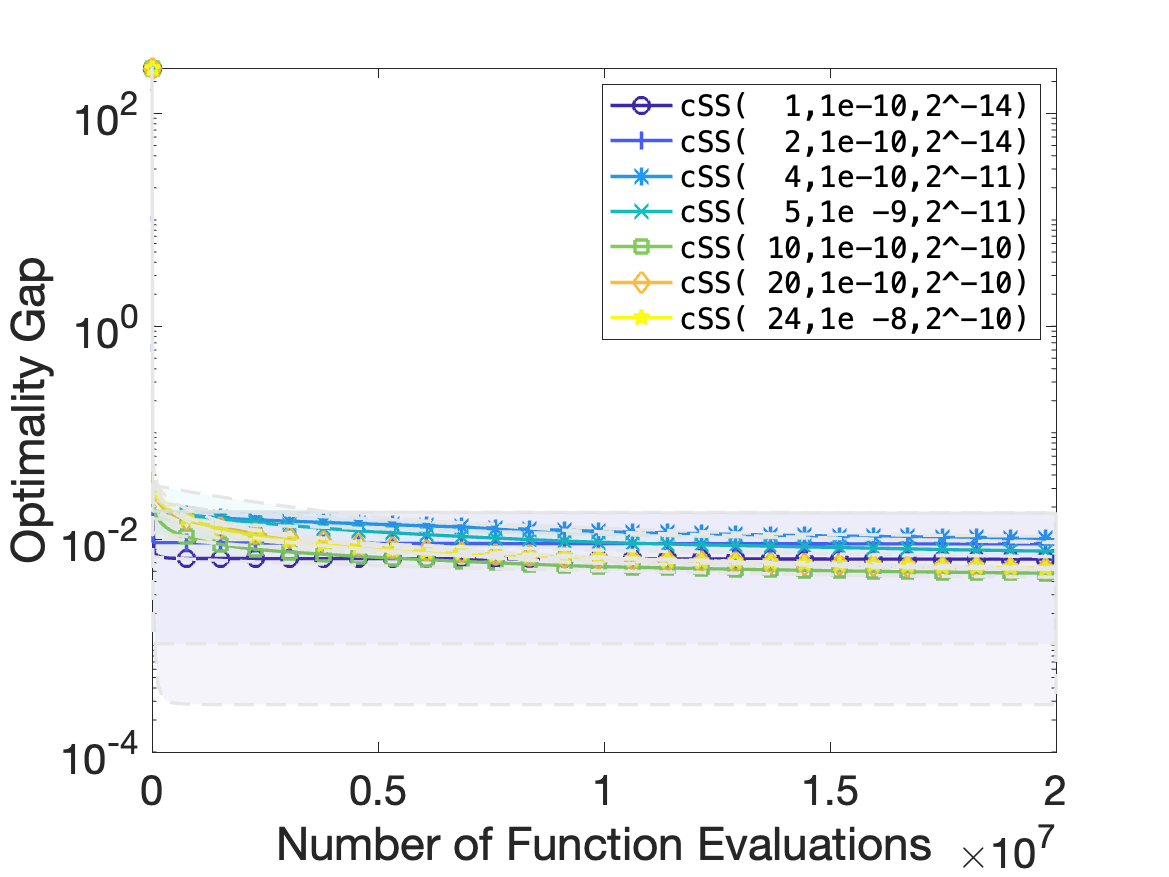}
  \caption{Performance of cSS}
\end{subfigure}%
\begin{subfigure}{0.33\textwidth}
  \centering
  \includegraphics[width=1.1\linewidth]{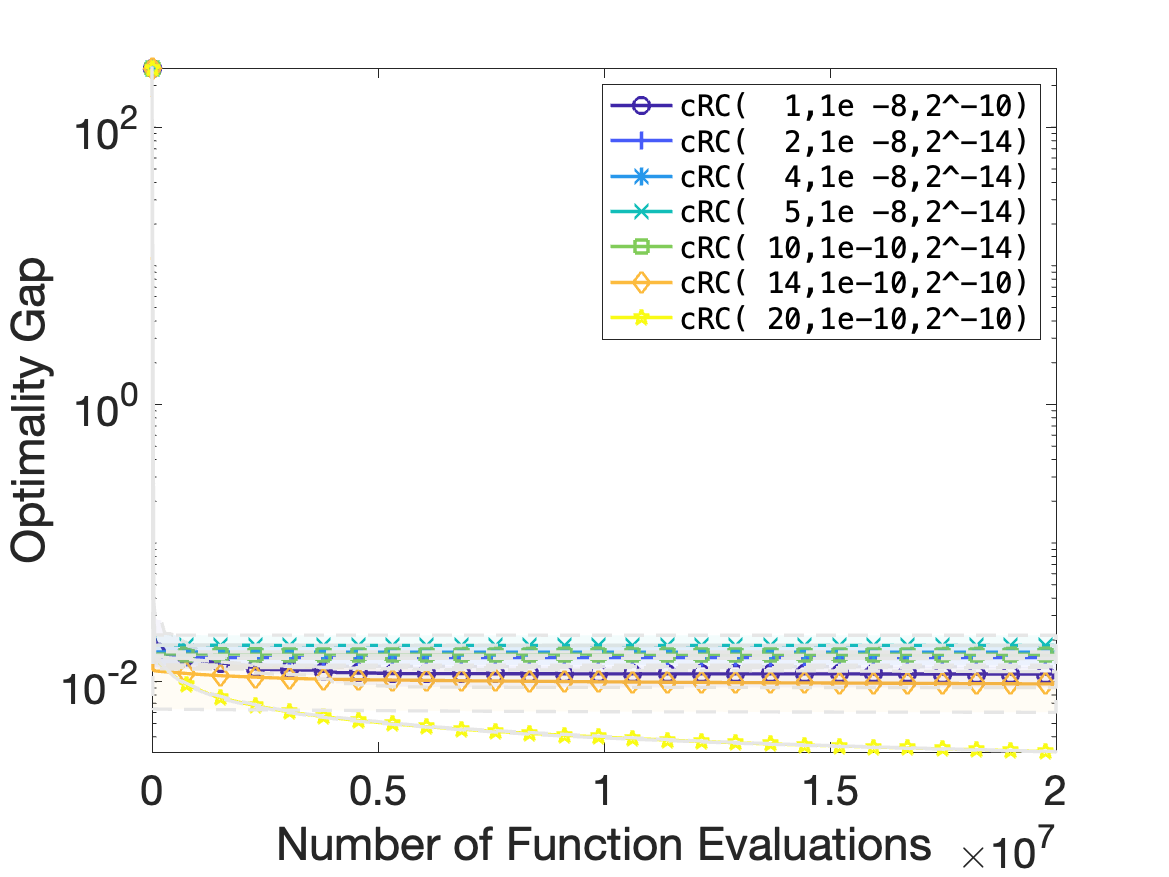}
  \caption{Performance of cRC}
\end{subfigure}
\caption{The effect of number of directions $N$ on the performance of different randomized gradient estimation methods on the Cube function with absolute error and $ \sigma = 10^{-3} $. The sampling radius $\nu$ and step size $\alpha$ are tuned for each method and $N$ combination to achieve the best performance.}
\end{figure}

\newpage
\paragraph{Cube Function $(d = 20, p = 30)$ with Relative Error, $\sigma = 10^{-5}$}

\begin{figure}[H]
\centering
\begin{subfigure}{0.33\textwidth}
  \centering
  \includegraphics[width=1.1\linewidth]{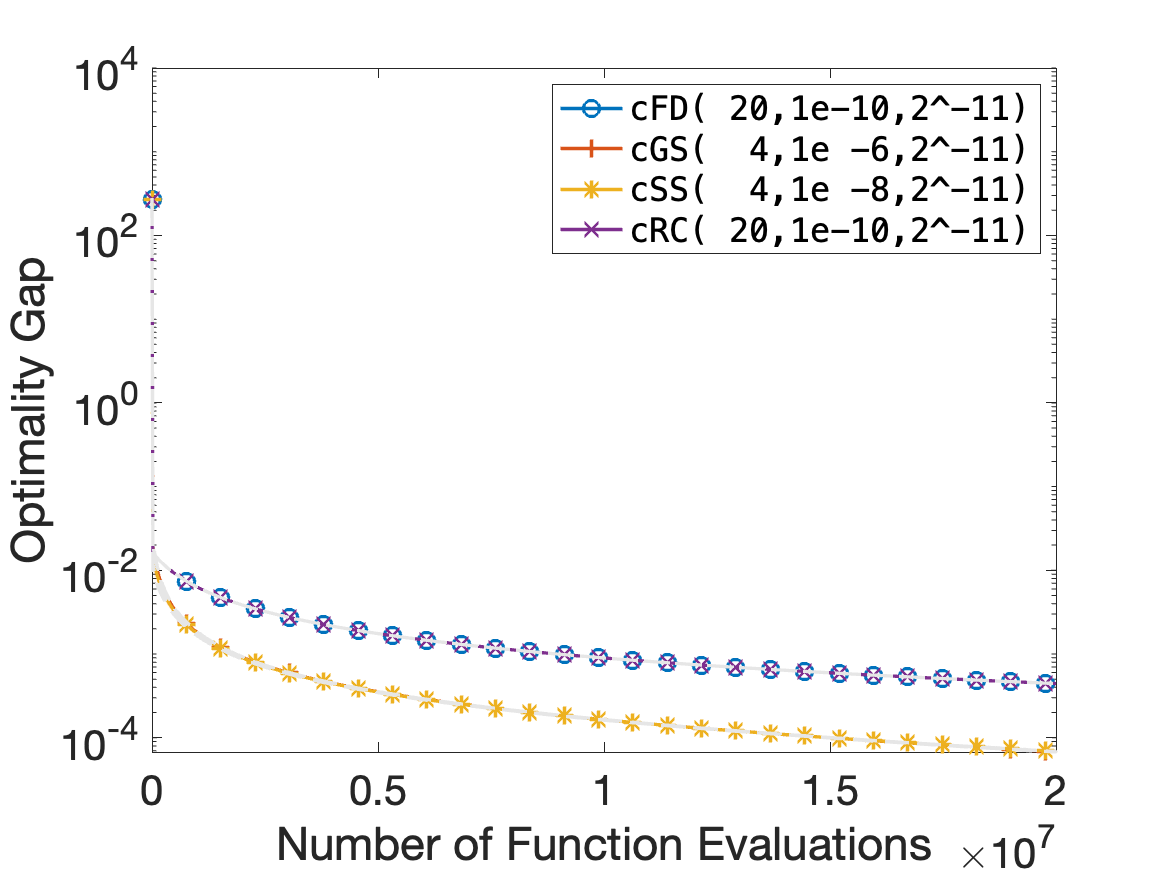}
  \caption{Optimality Gap}
\end{subfigure}%
\begin{subfigure}{0.33\textwidth}
  \centering
  \includegraphics[width=1.1\linewidth]{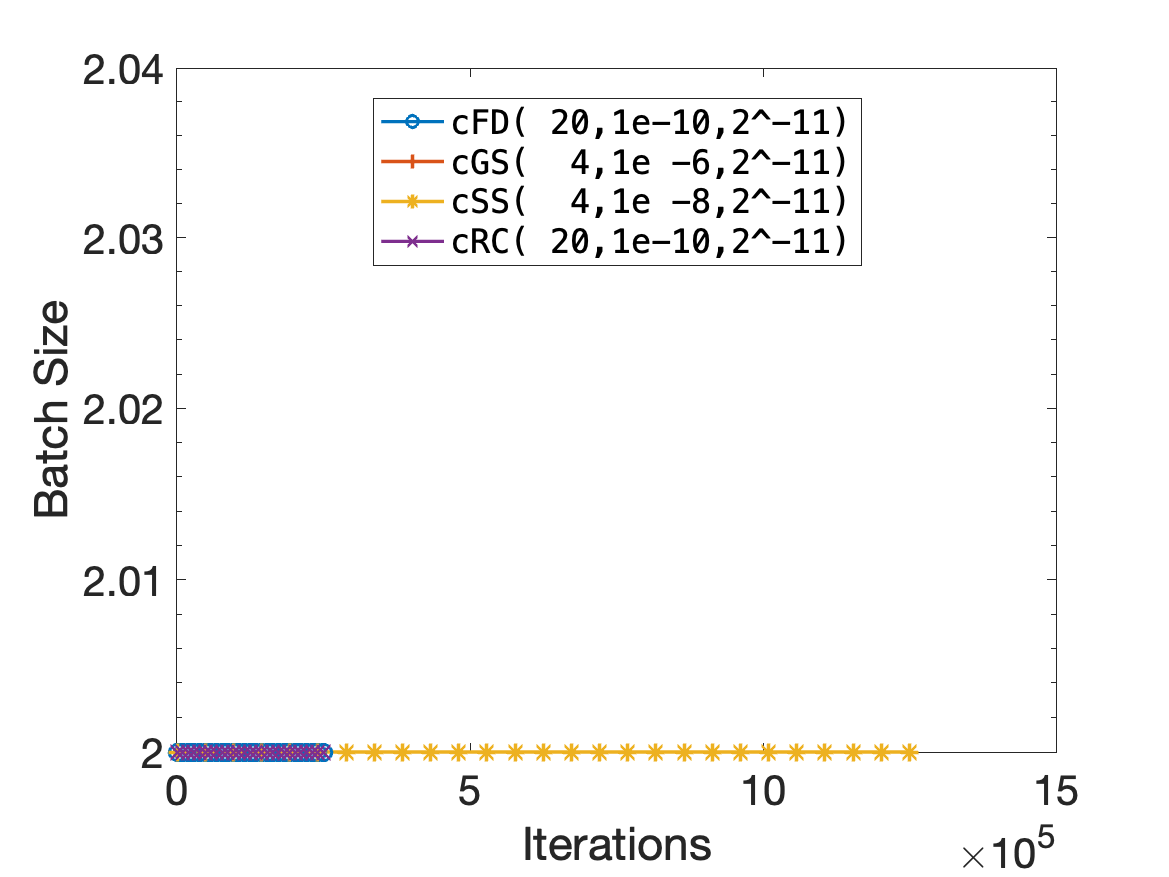}
  \caption{Batch Size}
\end{subfigure}
\begin{subfigure}{0.33\textwidth}
  \centering
  \includegraphics[width=1.1\linewidth]{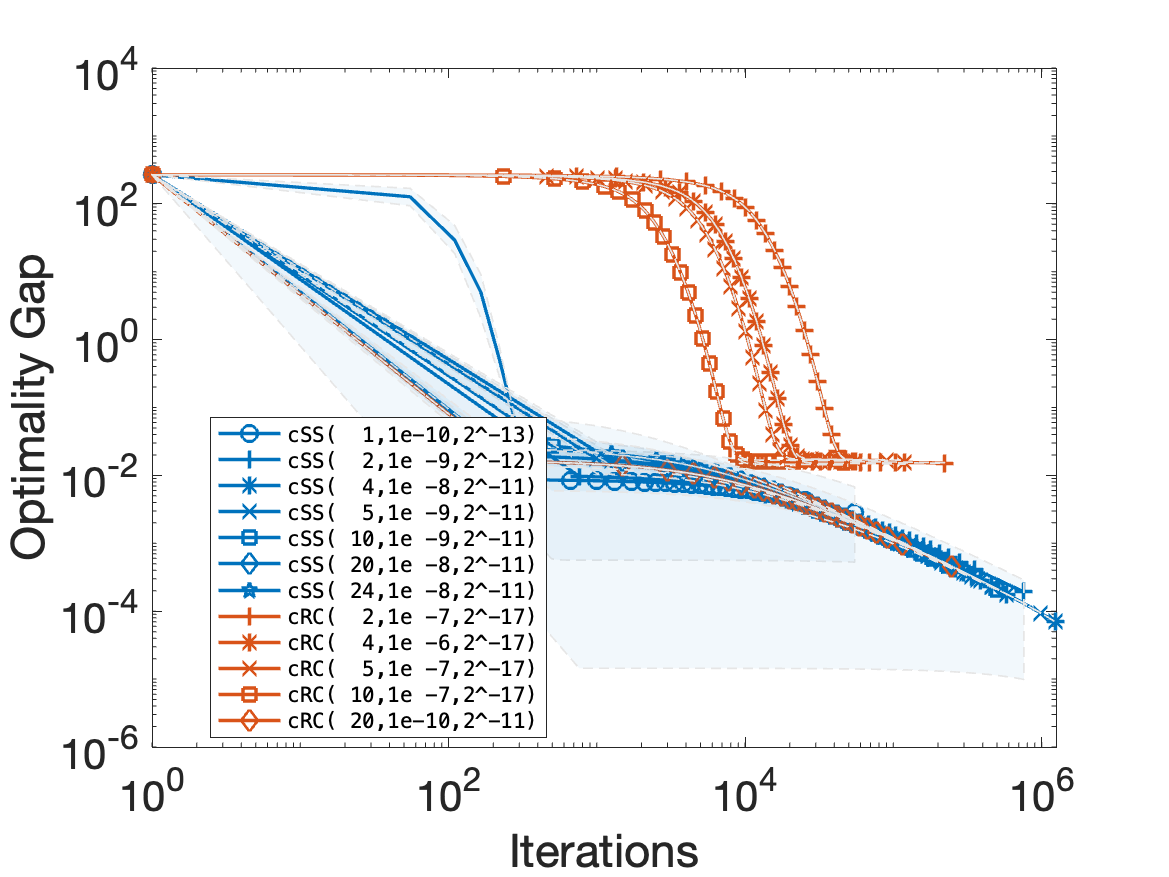}
  \caption{Comparison of cSS and cRC}
\end{subfigure}
\caption{Performance of different gradient estimation methods using the tuned hyperparameters on the Cube function with relative error and $ \sigma = 10^{-5} $.}
\end{figure}

\begin{figure}[H]
\centering
\begin{subfigure}{0.33\textwidth}
  \centering
  \includegraphics[width=1.1\linewidth]{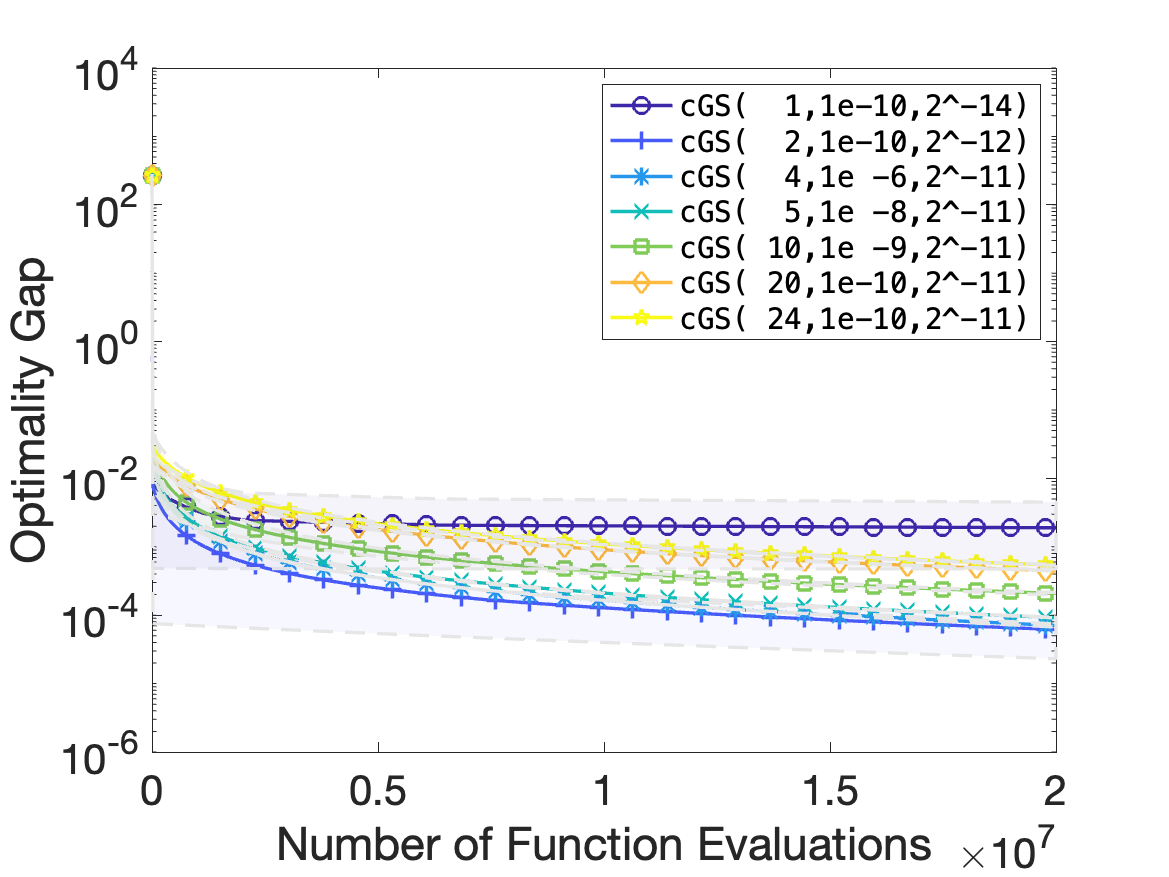}
  \caption{Performance of cGS}
\end{subfigure}%
\begin{subfigure}{0.33\textwidth}
  \centering
  \includegraphics[width=1.1\linewidth]{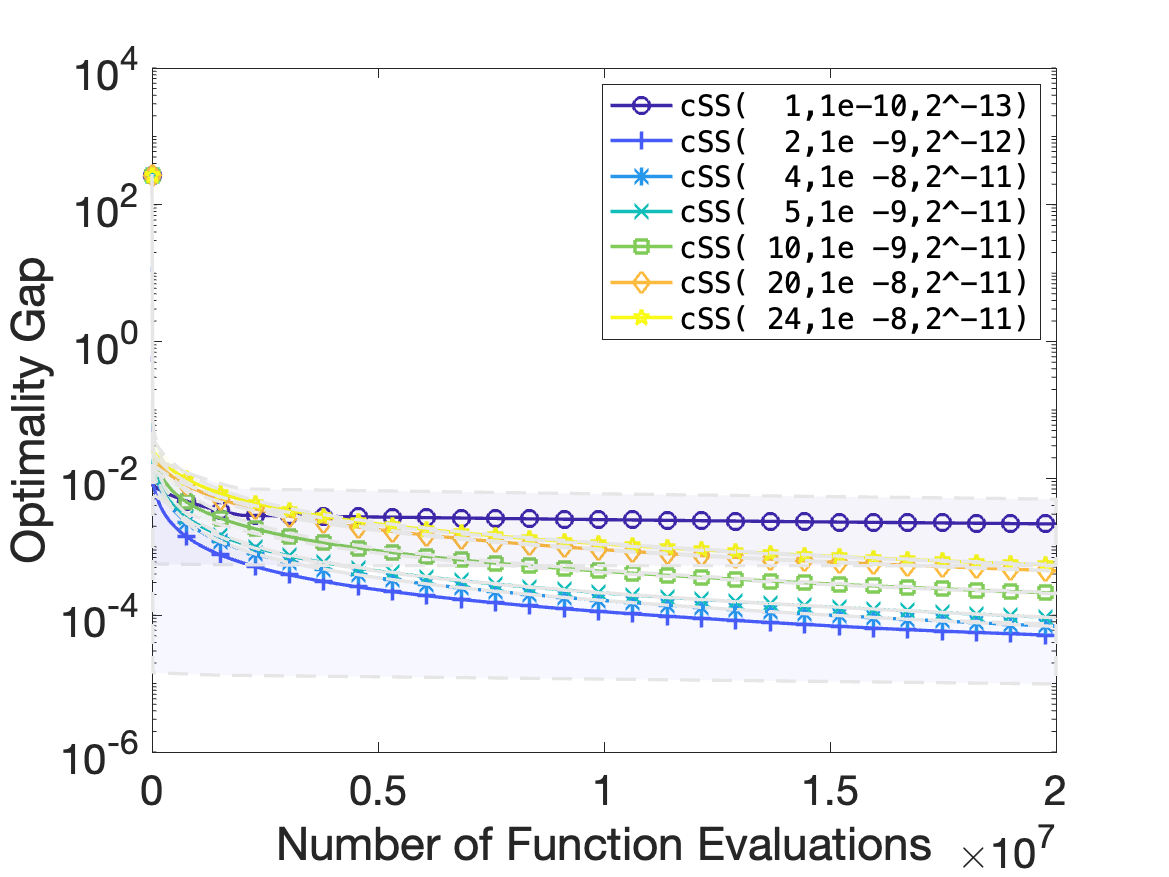}
  \caption{Performance of cSS}
\end{subfigure}%
\begin{subfigure}{0.33\textwidth}
  \centering
  \includegraphics[width=1.1\linewidth]{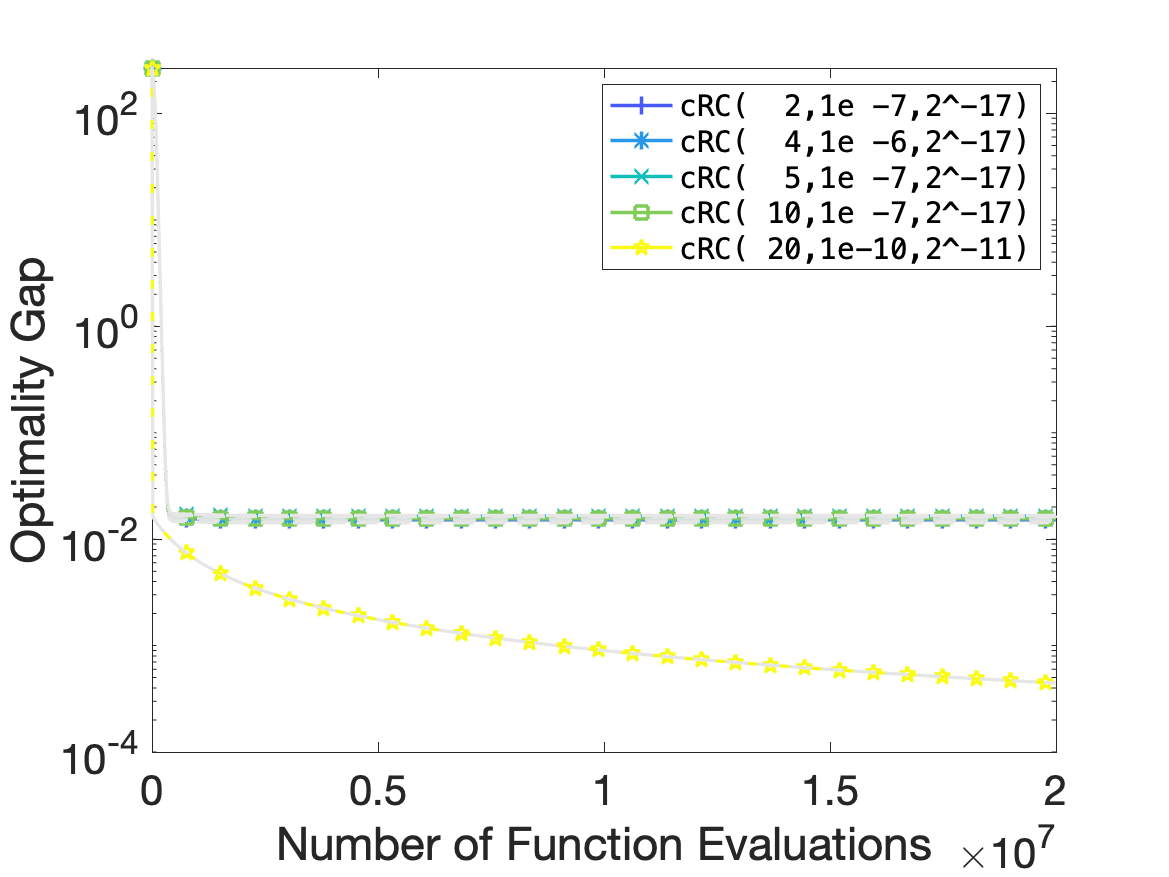}
  \caption{Performance of cRC}
\end{subfigure}
\caption{The effect of number of directions $N$ on the performance of different randomized gradient estimation methods on the Cube function with relative error and $ \sigma = 10^{-5} $. The sampling radius $\nu$ and step size $\alpha$ are tuned for each method and $N$ combination to achieve the best performance.}
\end{figure}

\newpage
\paragraph{Cube Function $(d = 20, p = 30)$ with Absolute Error, $\sigma = 10^{-5}$}

\begin{figure}[H]
\centering
\begin{subfigure}{0.33\textwidth}
  \centering
  \includegraphics[width=1.1\linewidth]{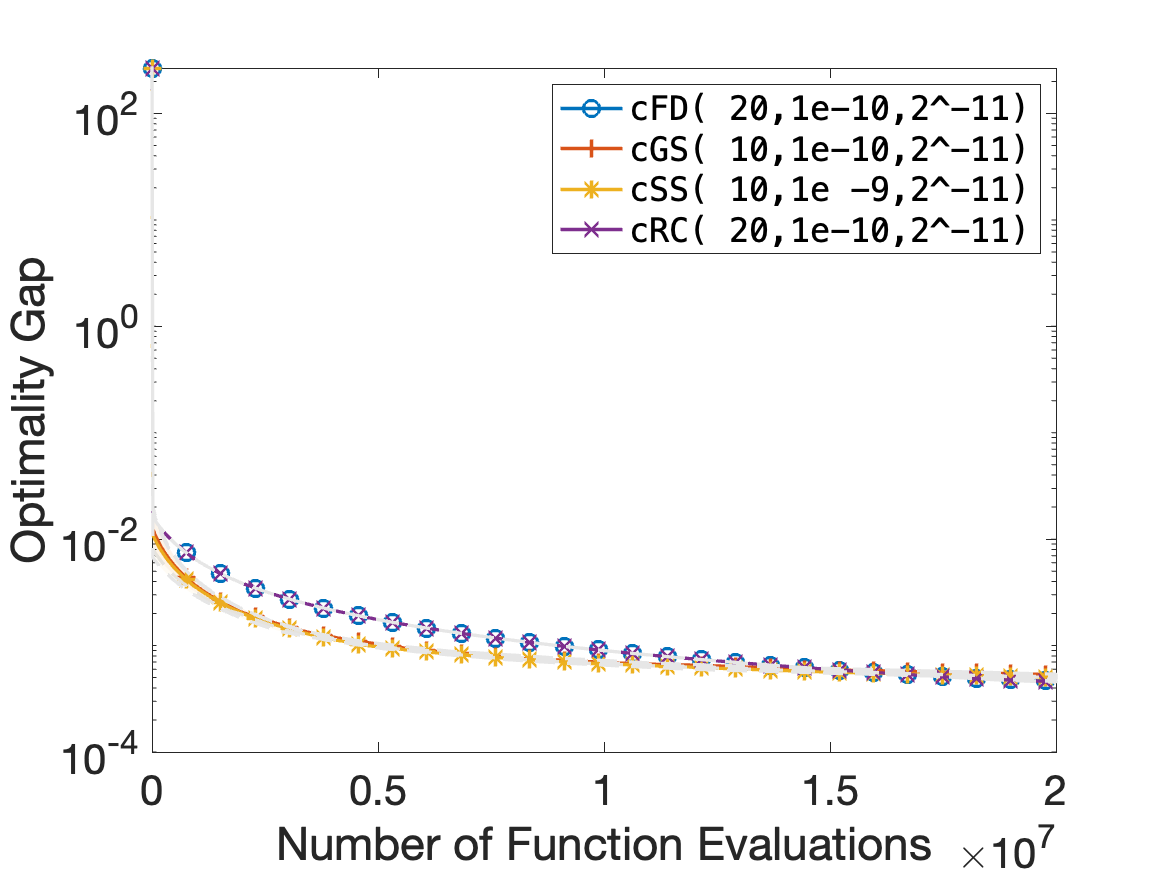}
  \caption{Optimality Gap}
\end{subfigure}%
\begin{subfigure}{0.33\textwidth}
  \centering
  \includegraphics[width=1.1\linewidth]{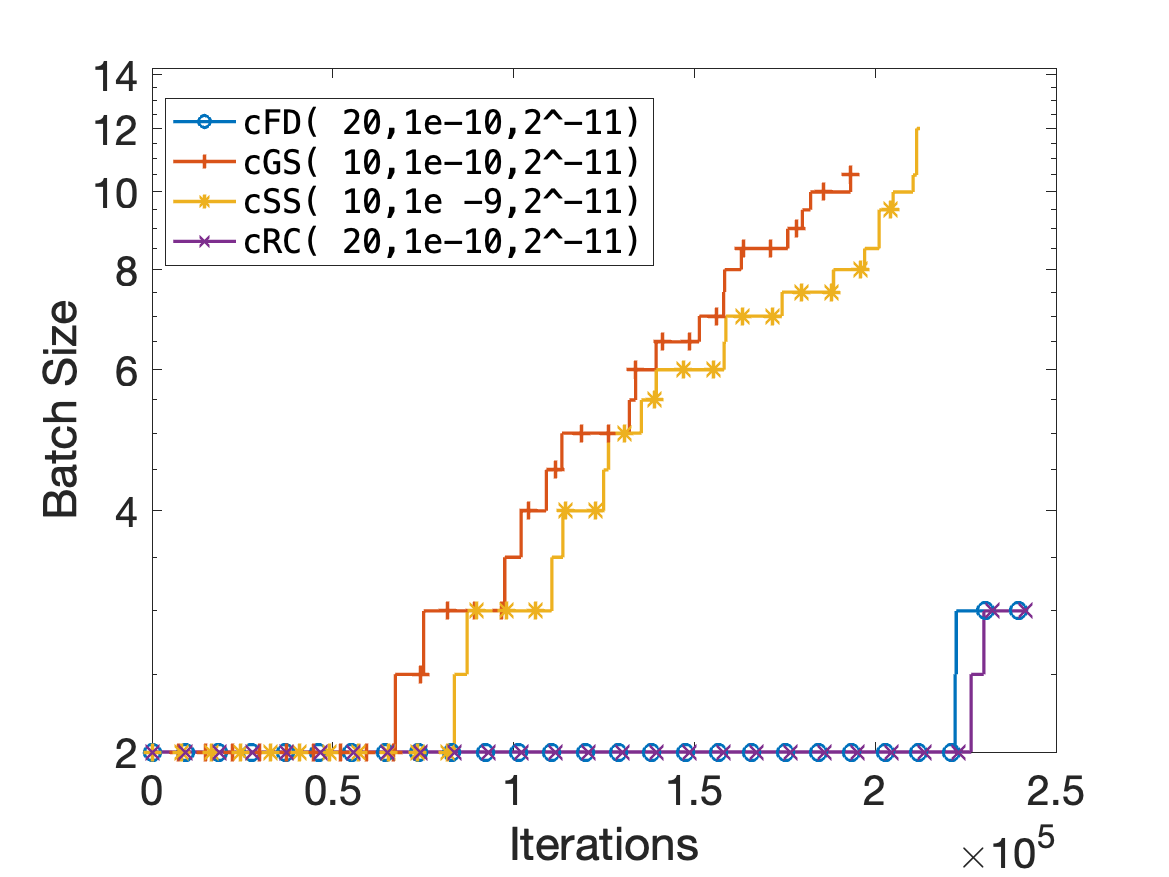}
  \caption{Batch Size}
\end{subfigure}
\begin{subfigure}{0.33\textwidth}
  \centering
  \includegraphics[width=1.1\linewidth]{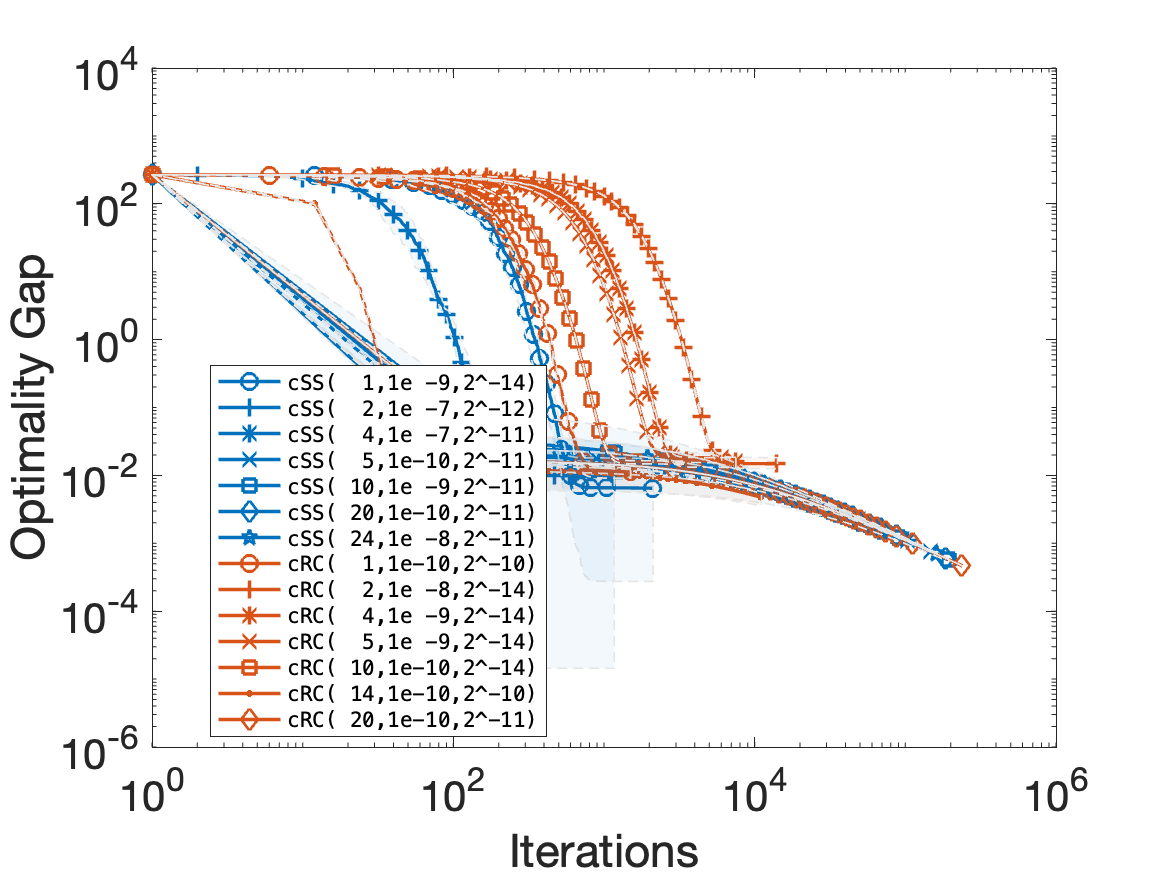}
  \caption{Comparison of cSS and cRC}
\end{subfigure}
\caption{Performance of different gradient estimation methods using the tuned hyperparameters on the Cube function with absolute error and $ \sigma = 10^{-5} $.}
\end{figure}

\begin{figure}[H]
\centering
\begin{subfigure}{0.33\textwidth}
  \centering
  \includegraphics[width=1.1\linewidth]{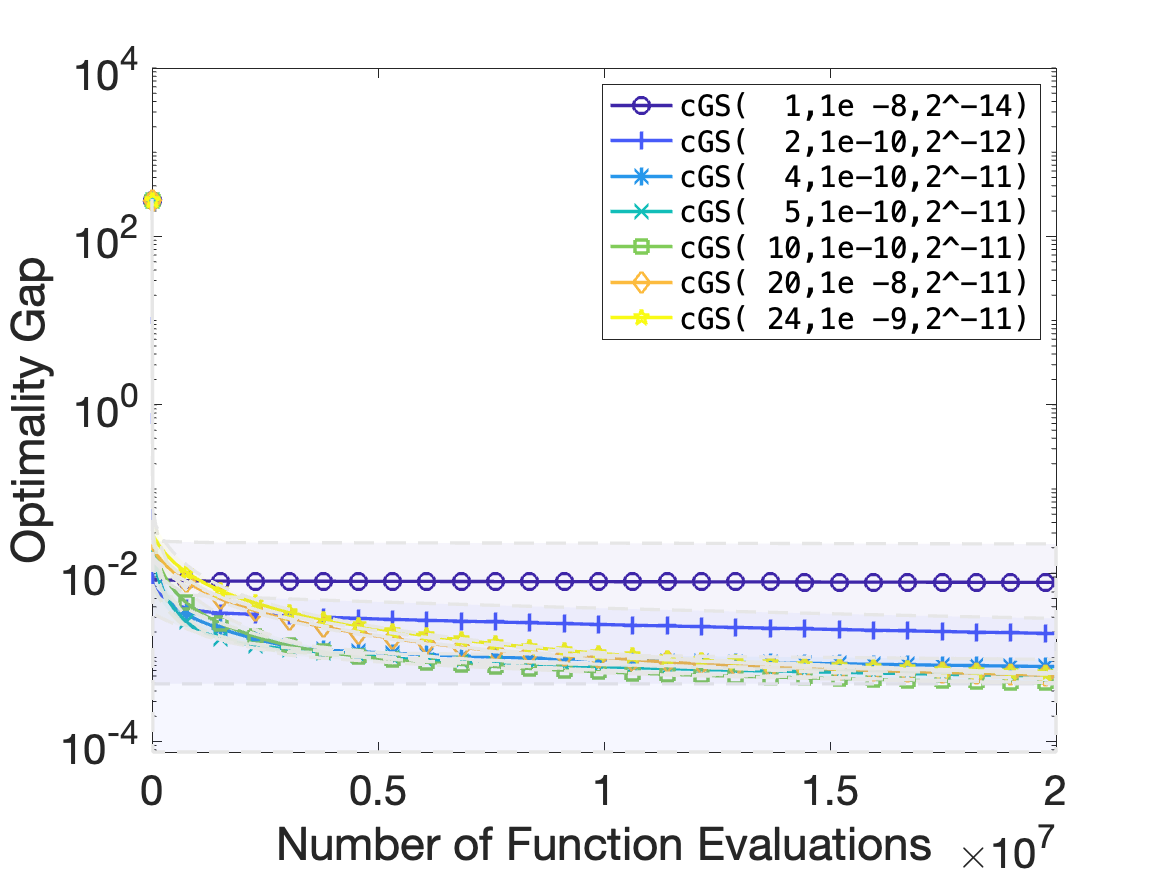}
  \caption{Performance of cGS}
\end{subfigure}%
\begin{subfigure}{0.33\textwidth}
  \centering
  \includegraphics[width=1.1\linewidth]{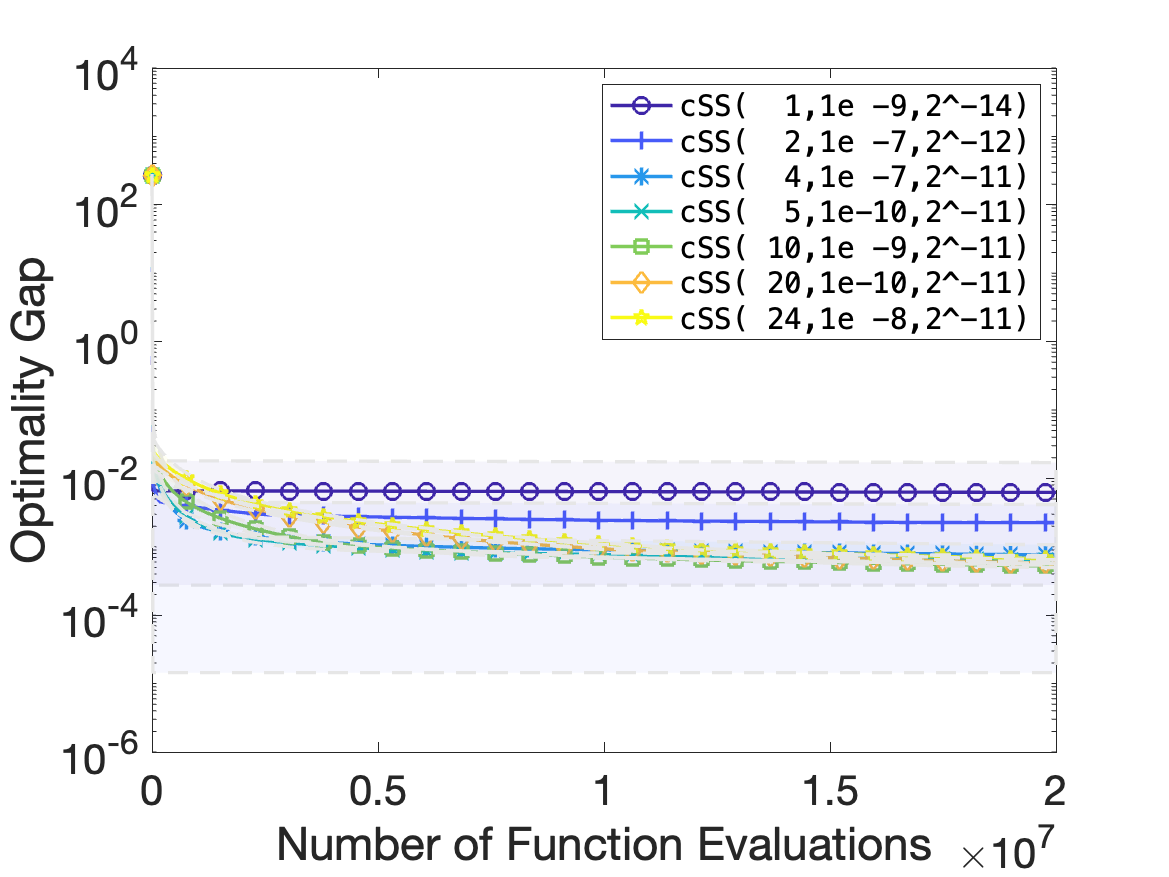}
  \caption{Performance of cSS}
\end{subfigure}%
\begin{subfigure}{0.33\textwidth}
  \centering
  \includegraphics[width=1.1\linewidth]{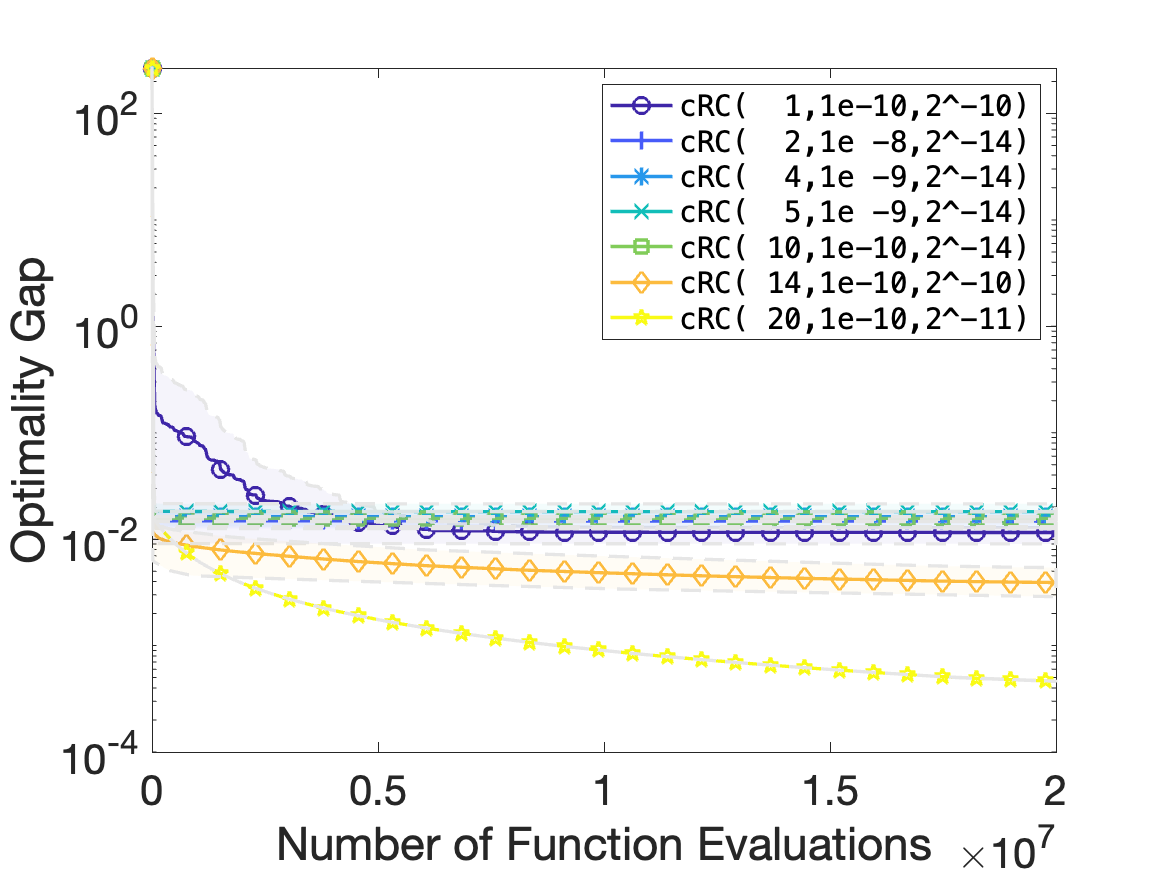}
  \caption{Performance of cRC}
\end{subfigure}
\caption{The effect of number of directions $N$ on the performance of different randomized gradient estimation methods on the Cube function with absolute error and $ \sigma = 10^{-5} $. The sampling radius $\nu$ and step size $\alpha$ are tuned for each method and $N$ combination to achieve the best performance.}
\end{figure}

\newpage
\subsubsection{Logistic Regression Problems}
\label{sec:add_plots_CFD_logreg}

For the binary classification problem, we consider minimization of the logistic loss with $\ell_2$ regularization:
\begin{equation}%
    F(x) = \frac{1}{N_{\text{data}}} \sum_{i = 1}^{N_{\text{data}}} \log(1 + \exp(-z^i x^T y^i)) + \frac{\lambda}{2} \|x\|^2,
\end{equation}
where
$(y_i,z_i) \in \R^d\times\{-1,1\}$ are the training data for $i=1,2,\cdots, N_{\text{data}}$, and $\lambda = 1/N_{\text{data}}$ is the regularization parameter. We considered the classification data sets \texttt{mushroom} and \texttt{MNIST} from the LIBSVM collection \citep{10.1145/1961189.1961199}, where $N_{\text{data}} = 5,500$ and $d = 112$, and $N_{\text{data}} = 60,000$ and $d = 784$, respectively. We chose the initial sample size $|S_0|$ to be $10\%$ of $N_{\text{data}}$. We set the total function evaluation budget to be $100 d N_{\text{data}}$. \\

\paragraph{\texttt{mushroom} data set}

\begin{figure}[H]
    \centering
    \begin{subfigure}{0.33\textwidth}
        \centering
        \includegraphics[width=1.1\linewidth]{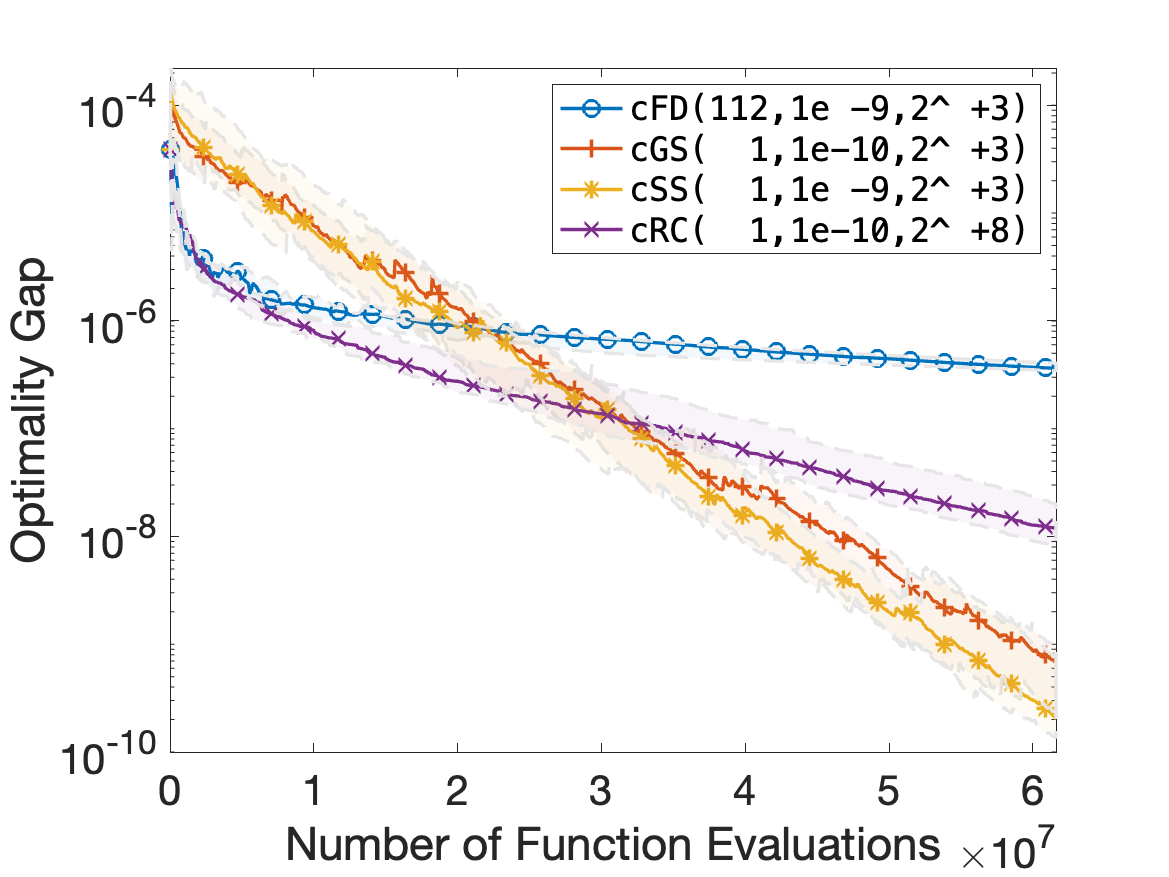}
        \caption{Optimality Gap}
    \end{subfigure}%
    \begin{subfigure}{0.33\textwidth}
        \centering
        \includegraphics[width=1.1\linewidth]{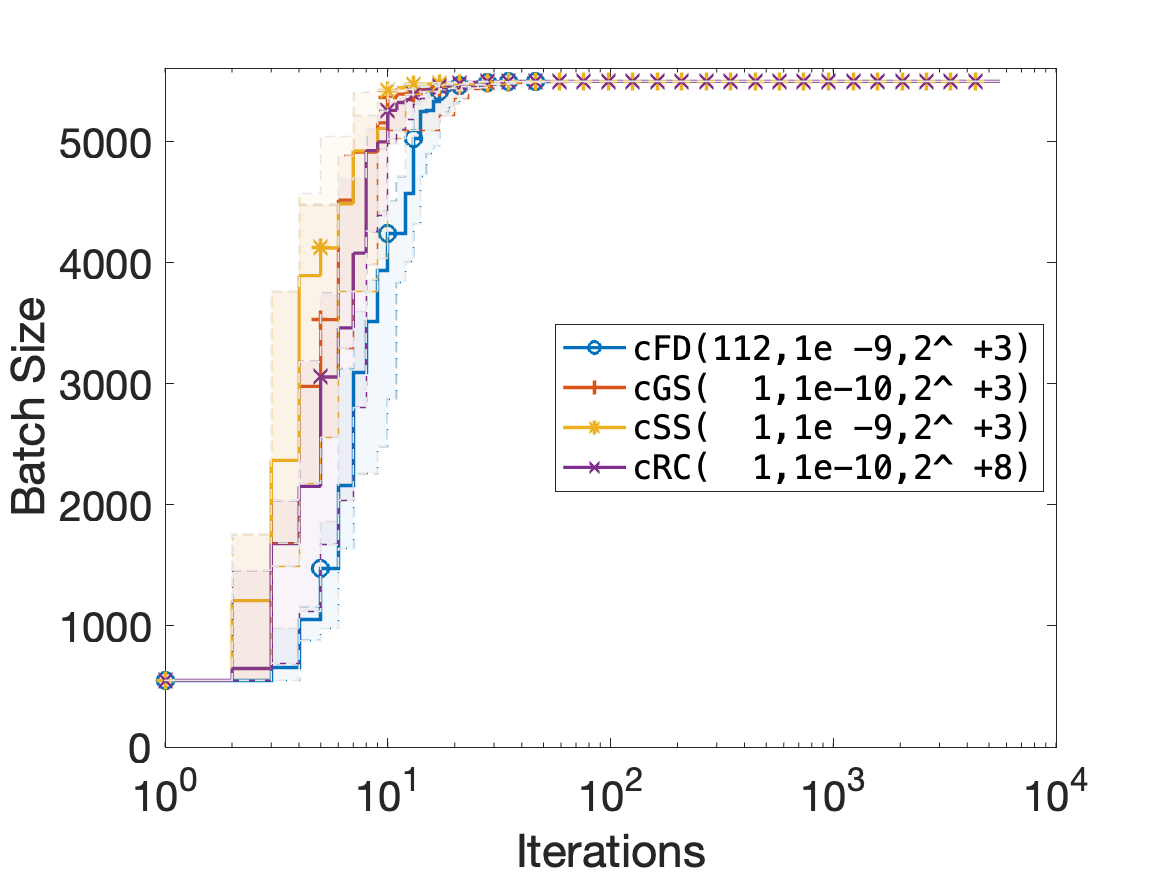}
        \caption{Batch Size}
    \end{subfigure}
    \begin{subfigure}{0.33\textwidth}
        \centering
        \includegraphics[width=1.1\linewidth]{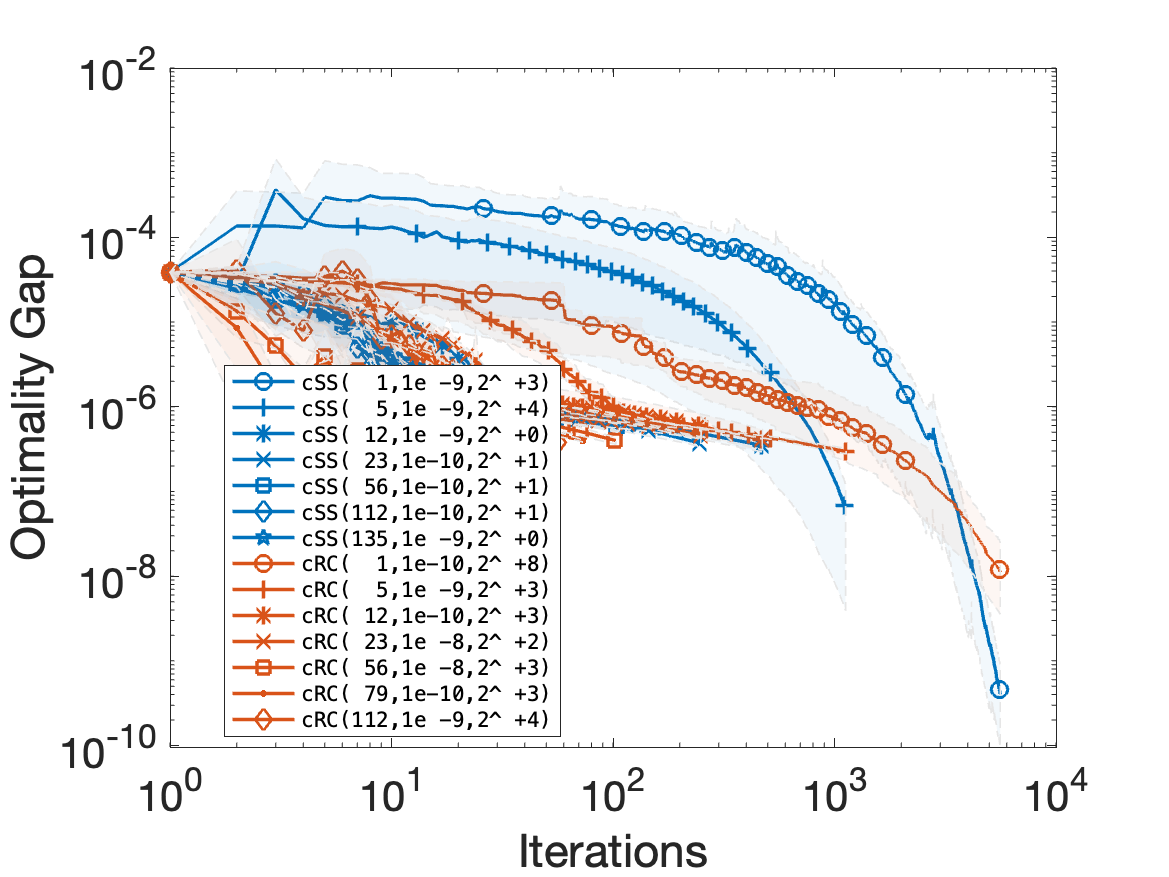}
        \caption{Comparison of cSS and cRC}
    \end{subfigure}
    \caption{Performance of different gradient estimation methods using the tuned hyperparameters on the \texttt{mushroom} data set.}
\end{figure}

\begin{figure}[H]
\centering
    \begin{subfigure}{0.33\textwidth}
        \includegraphics[width=1.1\linewidth]{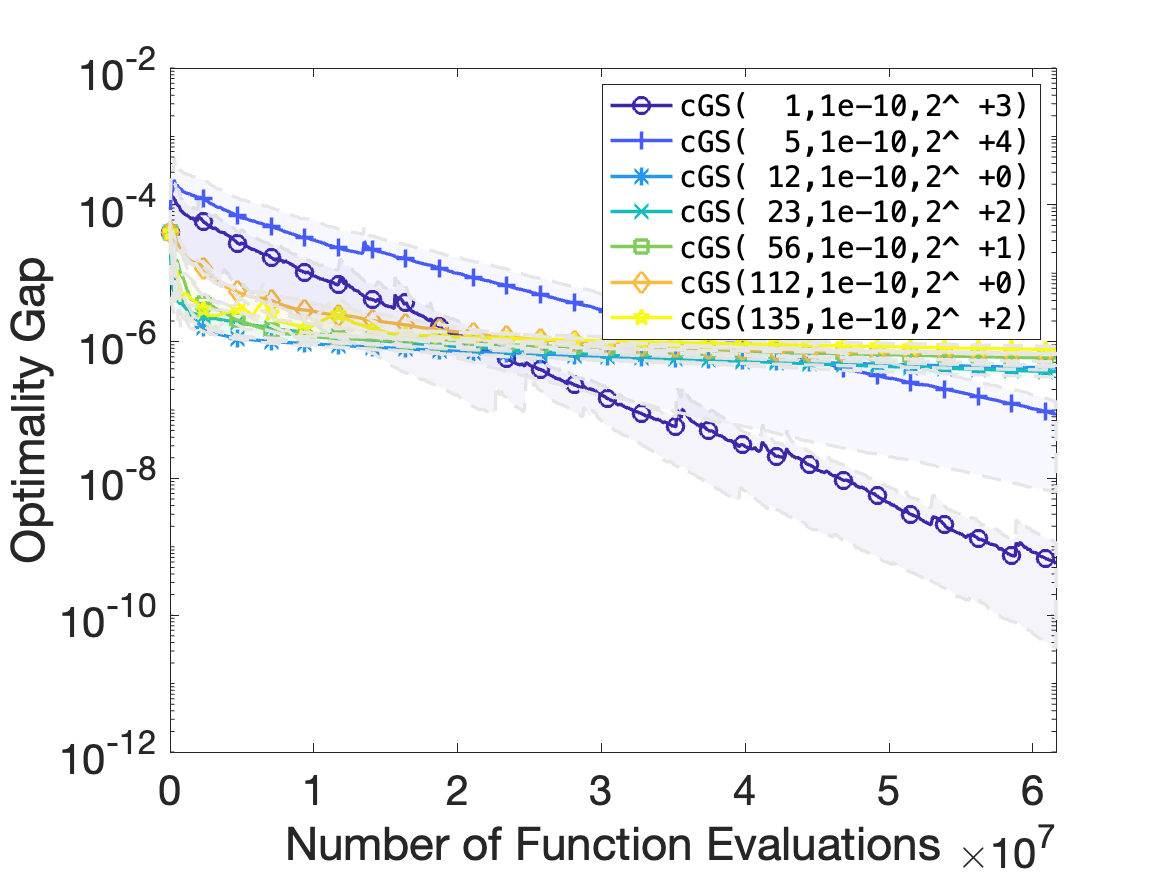}
        \caption{Performance of cGS}
    \end{subfigure}%
    \begin{subfigure}{0.33\textwidth}
        \centering
        \includegraphics[width=1.1\linewidth]{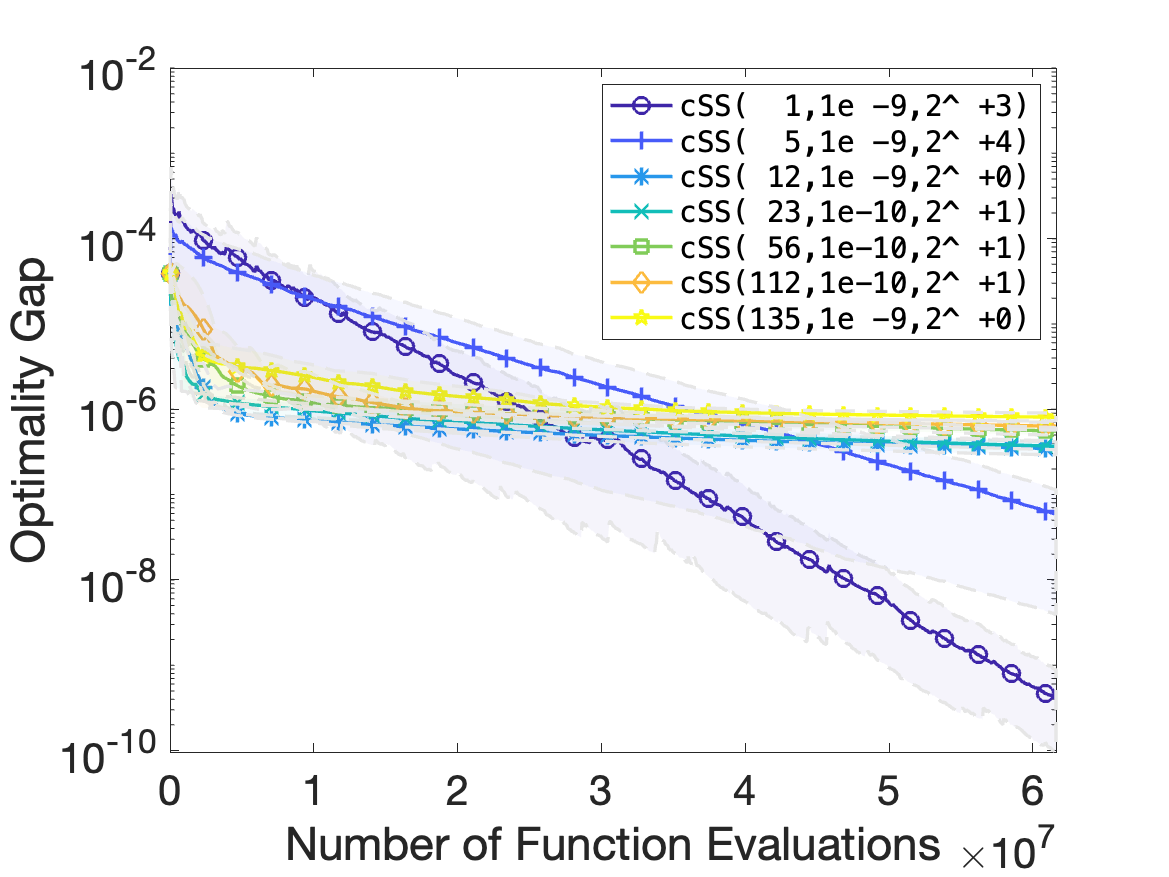}
        \caption{Performance of cSS}
    \end{subfigure}
    \begin{subfigure}{0.33\textwidth}
        \centering
        \includegraphics[width=1.1\linewidth]{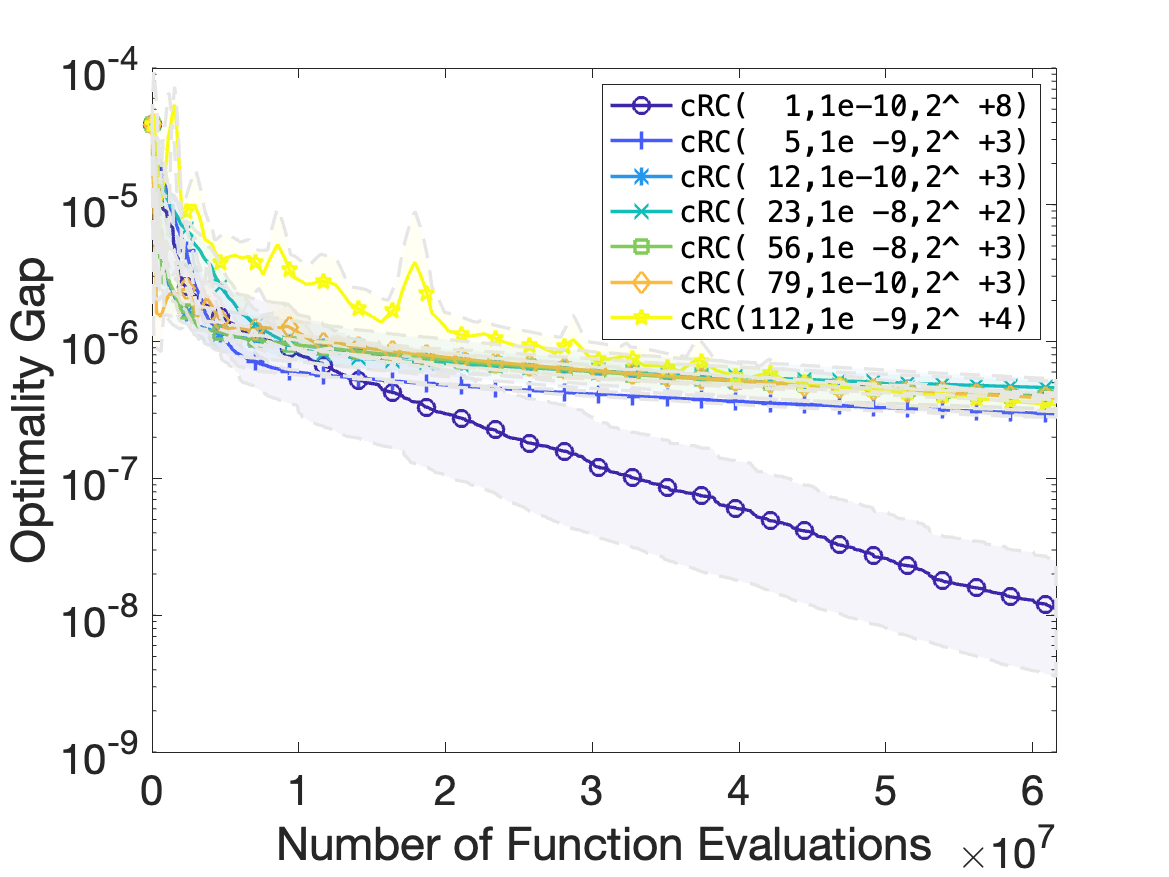}
        \caption{Performance of cRC}
    \end{subfigure}
\caption{The effect of number of directions $N$ on the performance of different randomized gradient estimation methods on the \texttt{mushroom} data set. The sampling radius $\nu$ and step size $\alpha$ are tuned for each method and $N$ combination to achieve the best performance.
}
\end{figure}

\newpage
\paragraph{\texttt{MNIST} data set}

\begin{figure}[H]
    \centering
    \begin{subfigure}{0.33\textwidth}
        \centering
        \includegraphics[width=1.1\linewidth]{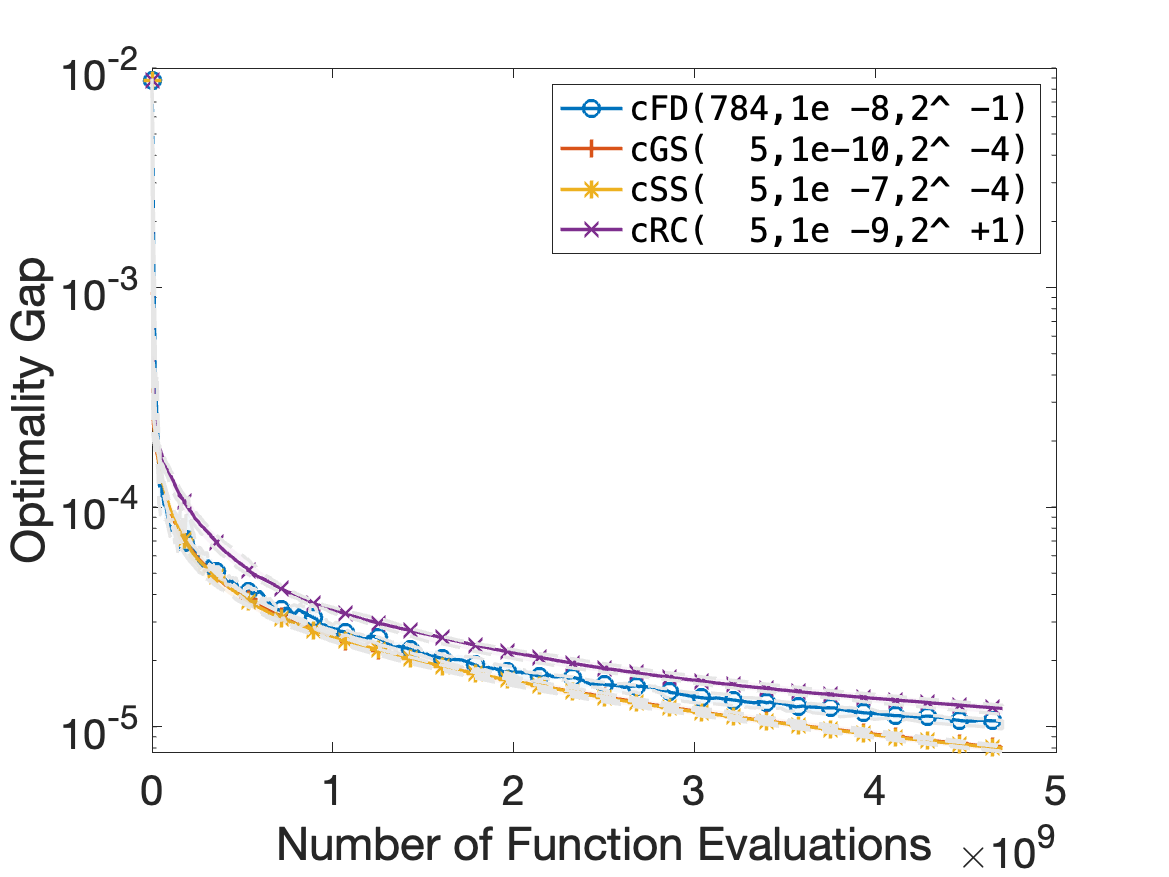}
        \caption{Optimality Gap}
    \end{subfigure}%
    \begin{subfigure}{0.33\textwidth}
        \centering
        \includegraphics[width=1.1\linewidth]{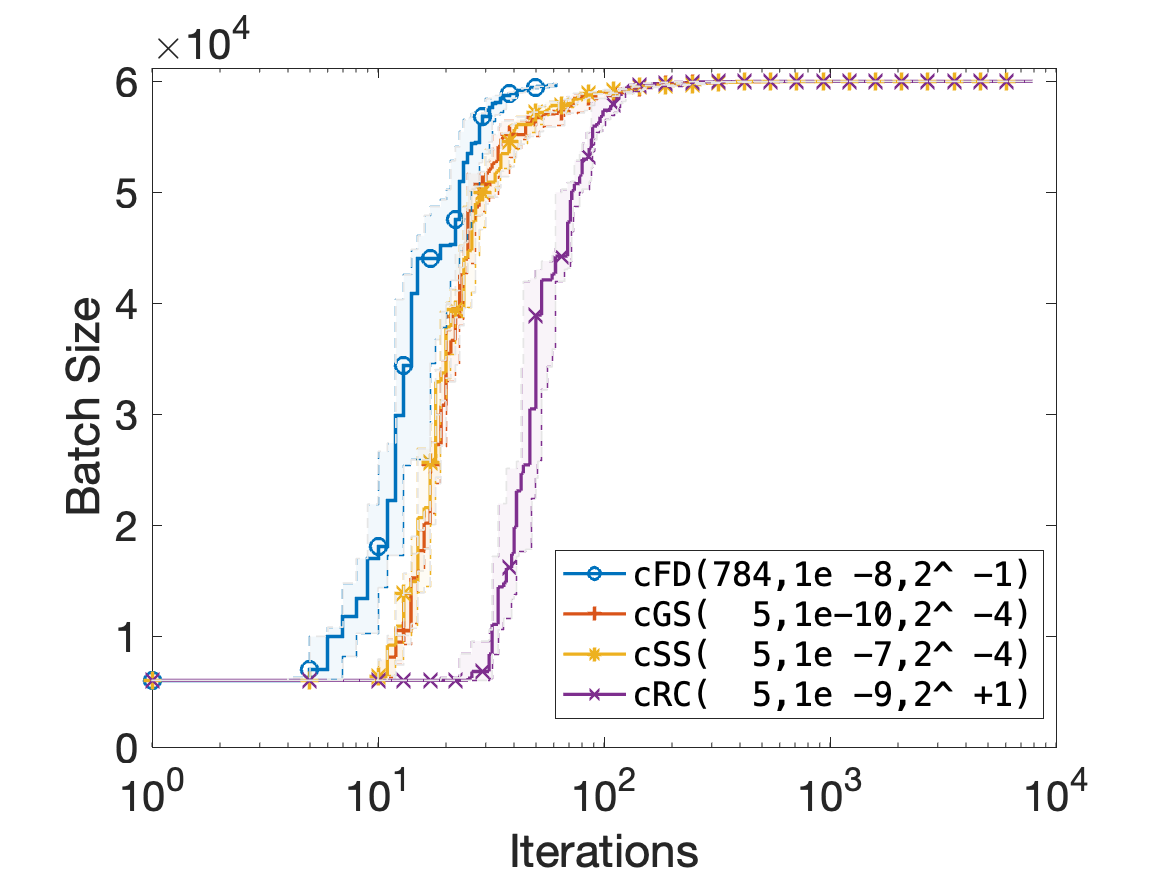}
        \caption{Batch Size}
    \end{subfigure}
    \begin{subfigure}{0.33\textwidth}
        \centering
        \includegraphics[width=1.1\linewidth]{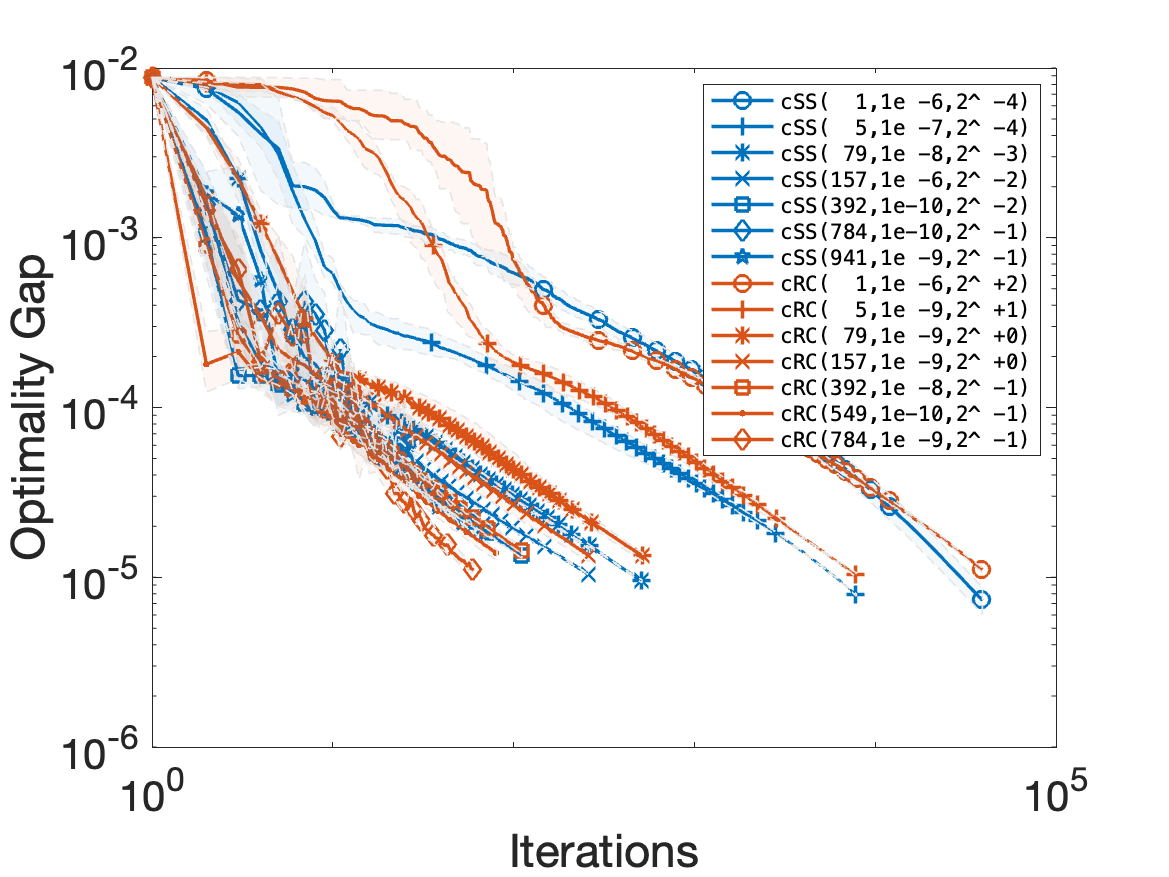}
        \caption{Comparison of cSS and cRC}
    \end{subfigure}
    \caption{Performance of different gradient estimation methods using the tuned hyperparameters on the \texttt{MNIST} data set.}
\end{figure}

\begin{figure}[H]
\centering
    \begin{subfigure}{0.33\textwidth}
        \includegraphics[width=1.1\linewidth]{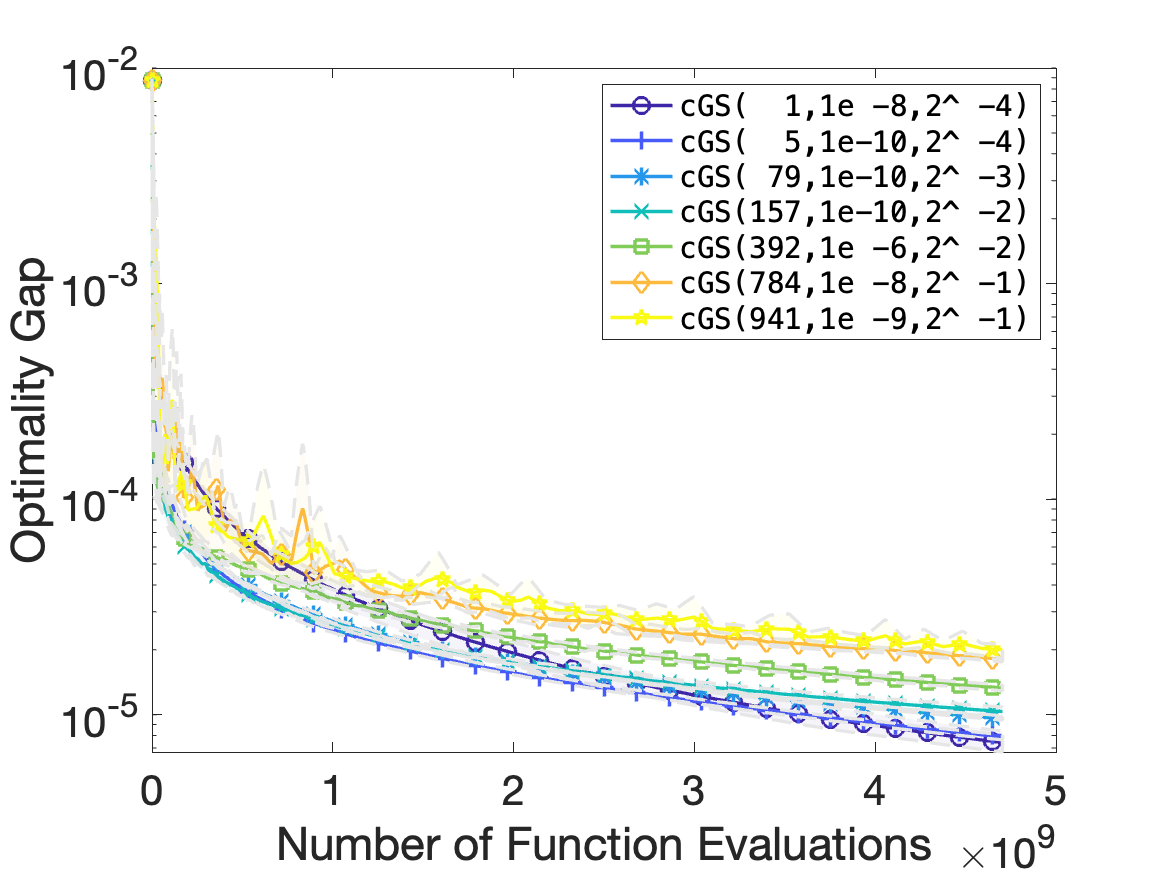}
        \caption{Performance of cGS}
    \end{subfigure}%
    \begin{subfigure}{0.33\textwidth}
        \centering
        \includegraphics[width=1.1\linewidth]{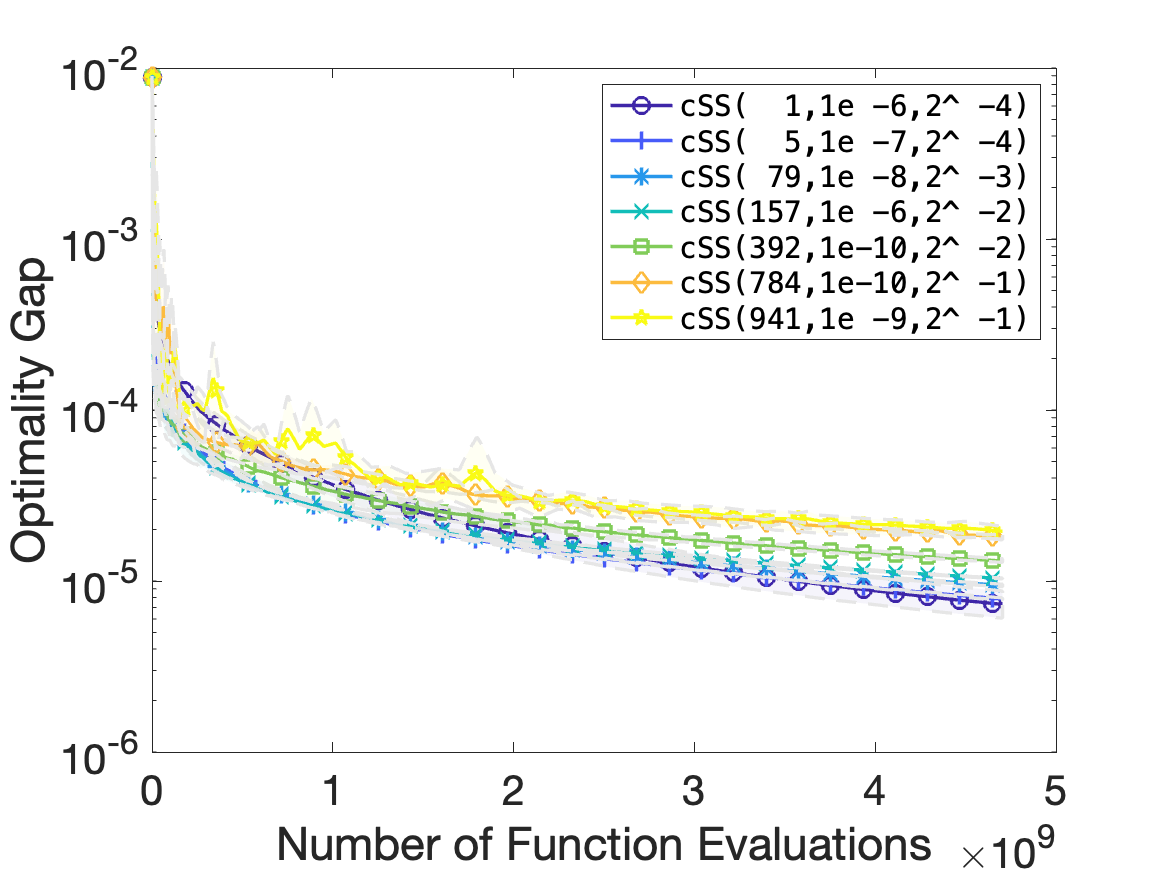}
        \caption{Performance of cSS}
    \end{subfigure}
    \begin{subfigure}{0.33\textwidth}
        \centering
        \includegraphics[width=1.1\linewidth]{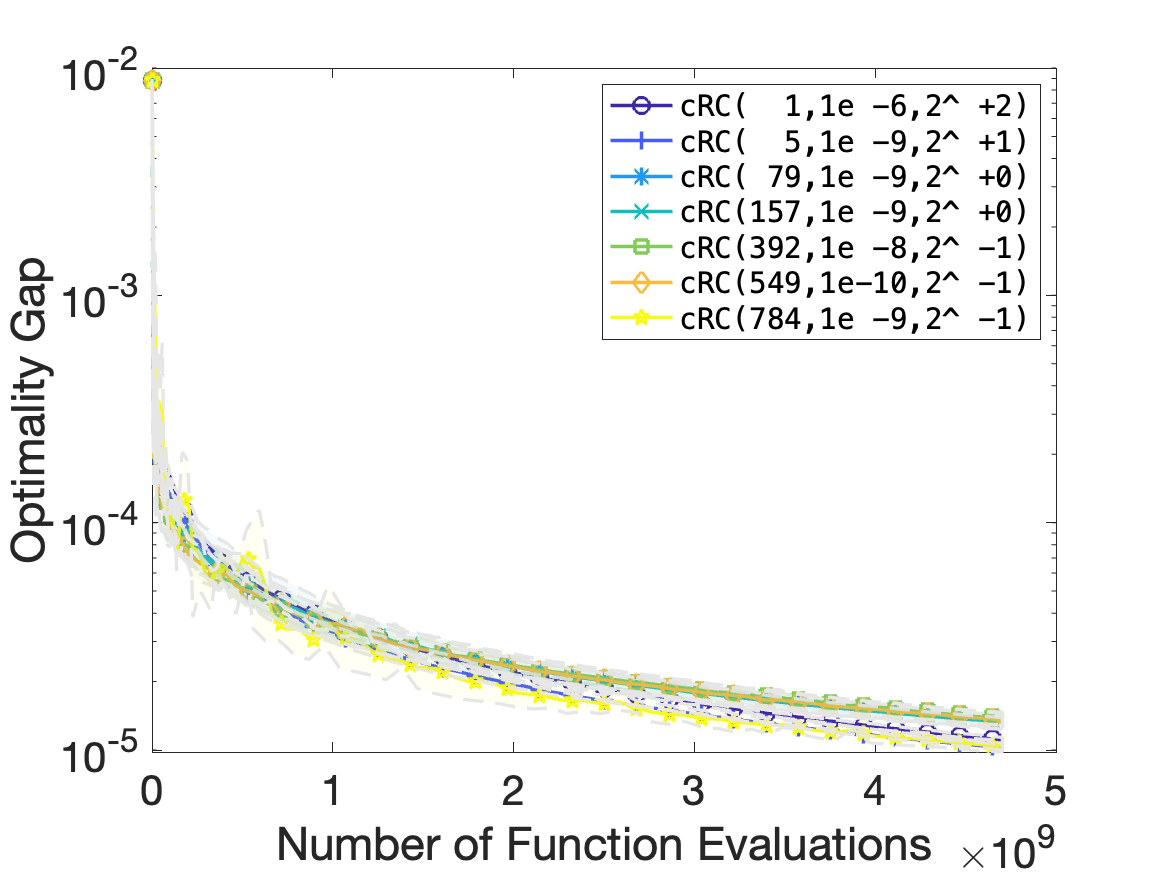}
        \caption{Performance of cRC}
    \end{subfigure}
    \caption{The effect of number of directions $N$ on the performance of different randomized gradient estimation methods on the \texttt{MNIST} data set. The sampling radius $\nu$ and step size $\alpha$ are tuned for each method and $N$ combination to achieve the best performance.
    }
\end{figure}

\end{document}